\immediate\write18{makeindex clt.nlo -s nomencl.ist -o clt.nls}
\documentclass[11pt]{article}

\usepackage{titlesec}
\titleformat{\subsection}[runin]{\normalfont\bfseries}{\thesubsection.}{.5em}{}[.]\titlespacing{\subsection}{0pt}{2ex plus .1ex minus .2ex}{.8em}
\titleformat{\subsubsection}[runin]{\normalfont\itshape}{\thesubsubsection.}{.3em}{}[.]\titlespacing{\subsubsection}{0pt}{1ex plus .1ex minus .2ex}{.5em}
\titleformat{\paragraph}[runin]{\normalfont\itshape}{\theparagraph.}{.3em}{}[.]\titlespacing{\paragraph}{0pt}{1ex plus .1ex minus .2ex}{.5em}

\usepackage[labelfont=sc,font=small,labelsep=period]{caption}
\usepackage{enumerate}

\usepackage[titletoc,title]{appendix}

\usepackage[letterpaper, hmargin=1.1in, top=1in, bottom=0.9in, footskip=0.6in]{geometry}

\usepackage{hyperref}

\makeatletter\def\SetFigFont#1#2#3#4#5{\small}

\usepackage{amsmath} 
\usepackage{amssymb}
\usepackage{amsfonts}
\usepackage{latexsym}
\usepackage{amsthm}
\usepackage{amsxtra}
\usepackage{amscd}
\usepackage{dsfont}
\usepackage{mathrsfs}
\usepackage{bm}
\usepackage{stmaryrd}

\usepackage{graphicx, color}
\usepackage[nottoc,notlof,notlot]{tocbibind} 
\usepackage{cite}

\numberwithin{equation}{section}
\numberwithin{figure}{section}

\theoremstyle{plain} 
\newtheorem{theorem}{Theorem}[section]
\newtheorem*{theorem*}{Theorem}
\newtheorem{lemma}[theorem]{Lemma}
\newtheorem*{lemma*}{Lemma}
\newtheorem{corollary}[theorem]{Corollary}
\newtheorem*{corollary*}{Corollary}
\newtheorem{proposition}[theorem]{Proposition}
\newtheorem*{proposition*}{Proposition}
\newtheorem{definition}[theorem]{Definition}
\newtheorem*{definition*}{Definition}

\newtheorem*{conjecture*}{Conjecture}

\newtheorem*{example*}{Example}
\newtheorem{remark}[theorem]{Remark}
\newtheorem*{remark*}{Remark}

\renewcommand{\b}[1]{\boldsymbol{\mathrm{#1}}} 
\renewcommand{\cal}{\mathcal}

\newcommand{\wt}{\widetilde}

\usepackage{ifthen}
\def\comment#1{\ifthenelse{\isodd{\value{page}}}{\marginpar{\raggedright\scriptsize{\textcolor{darkred}{#1}}}}{\marginpar{\raggedleft\scriptsize{\textcolor{darkred}{#1}}}}}

\renewcommand{\P}{\mathbb{P}}
\newcommand{\E}{\mathbb{E}}
\newcommand{\R}{\mathbb{R}}
\newcommand{\C}{\mathbb{C}}
\newcommand{\D}{\mathbb{D}}
\newcommand{\N}{\mathbb{N}}
\newcommand{\Z}{\mathbb{Z}}

\newcommand{\T}{\mathbb{T}}
\newcommand{\rd}{{\rm d}}
\newcommand{\OO}{{\rm O}}
\newcommand{\oo}{{\rm o}}

\newcommand{\ee}{\mathrm{e}}
\newcommand{\ii}{\mathrm{i}}
\newcommand{\dd}{\mathrm{d}}

\renewcommand{\leq}{\leqslant}
\renewcommand{\geq}{\geqslant}
\renewcommand{\epsilon}{\varepsilon}

\newcommand{\qq}[1]{[\![{#1}]\!]}

\newcommand{\p}[1]{({#1})}
\newcommand{\pb}[1]{\bigl({#1}\bigr)}
\newcommand{\pB}[1]{\Bigl({#1}\Bigr)}
\newcommand{\pbb}[1]{\biggl({#1}\biggr)}

\newcommand{\pa}[1]{\left({#1}\right)}

\newcommand{\qB}[1]{\Bigl[{#1}\Bigr]}

\newcommand{\qa}[1]{\left[{#1}\right]}

\newcommand{\abs}[1]{\lvert #1 \rvert}
\newcommand{\absb}[1]{\bigl\lvert #1 \bigr\rvert}

\newcommand{\absa}[1]{\left\lvert #1 \right\rvert}

\newcommand{\norm}[1]{\lVert #1 \rVert}

\newcommand{\normB}[1]{\Bigl\lVert #1 \Bigr\rVert}

\newcommand{\norma}[1]{\left\lVert #1 \right\rVert}
\newcommand{\normt}[1]{{\norm{#1}}}

\DeclareMathOperator{\supp}{supp}
\DeclareMathOperator{\re}{Re}

\DeclareMathOperator{\dist}{dist}

\newcommand{\half}{\frac{1}{2}}
\newcommand{\hatmueps}{\hat\mu^{(\varepsilon)}}
\newcommand{\be}{\begin{equation}}
\newcommand{\eeq}{\end{equation}}
\newcommand{\bea}{\begin{align}}

\renewcommand{\thefootnote}{\arabic{footnote}}

\makeatletter
\let\@fnsymbol\@arabic
\makeatother

\date{}
                    
\usepackage[refpage]{nomencl}
\usepackage{xpatch}

\newlength{\nomitemorigsep}
\makenomenclature
\renewcommand\nomgroup[1]{%
\itemsep\nomitemorigsep
  \ifthenelse{\equal{#1}{C}}{%
   \item[\textbf{Hamiltonian, energy}] }{%
    \ifthenelse{\equal{#1}{A}}{%
     \item[\textbf{Interaction}]}{%
      \ifthenelse{\equal{#1}{D}}{%
        \item[\textbf{Measure}]}{%
          \ifthenelse{\equal{#1}{B}}{%
           \item[\textbf{Potential}]}{{%
        \ifthenelse{\equal{#1}{E}}{%
         \item[\textbf{Partition function}]}{{%
        \ifthenelse{\equal{#1}{F}}{%
         \item[\textbf{Other Symbols }]}%
                    {{}}}}}}}}} \itemsep\nomitemsep}

\newcommand{\nomHam}[1][]{\nomenclature[C#1]}
\newcommand{\nomInt}[1][]{\nomenclature[A#1]}
\newcommand{\nomMea}[1][]{\nomenclature[D#1]}
\newcommand{\nomPot}[1][]{\nomenclature[B#1]}
\newcommand{\nomPar}[1][]{\nomenclature[E#1]}
\newcommand{\nomOth}[1][]{\nomenclature[F#1]}

\usepackage{tocloft}
\begin{document}
\title{\Large\bf The two-dimensional Coulomb  plasma: \\ quasi-free approximation and central limit theorem\vspace{0.3cm}}

\renewcommand{\thefootnote}{\arabic{footnote}}

\author{\normalsize Roland Bauerschmidt\footnote{University of Cambridge, Statistical Laboratory, DPMMS, {\tt rb812@cam.ac.uk}.} \and
\normalsize Paul Bourgade\footnote{New York University,  
partially supported by NSF grant DMS-1513587,  {\tt bourgade@cims.nyu.edu.}} \and
\normalsize Miika Nikula\footnote{Harvard University, Center of Mathematical Sciences and Applications, {\tt minikula@cmsa.fas.harvard.edu}.} \and
\normalsize Horng-Tzer Yau\footnote{Harvard University, Department of Mathematics, 
partially supported by NSF grant   DMS-1606305, 1855509 and a Simons Investigator award, {\tt htyau@math.harvard.edu.}}}

\maketitle

\vspace{-0.5cm}
\begin{abstract}
For the two-dimensional one-component Coulomb plasma,
we derive an asymptotic expansion of  
the free energy up to order $N$, the number of particles of the gas, 
with an effective error bound $N^{1-\kappa}$ for some constant $\kappa > 0$.
This expansion is based on approximating the Coulomb gas by a quasi-free Yukawa gas. 
Further, we prove that the fluctuations of the
linear statistics are given by a Gaussian free field at any positive temperature.
Our proof of this central limit theorem uses a loop equation for the Coulomb gas, 
the free energy asymptotics, and rigidity bounds on the local density fluctuations of
the Coulomb gas,  which we obtained in a previous paper.
\end{abstract}

\tableofcontents

\section{Introduction and main results}
\label{sec:intro}

\subsection{One-component plasma}
\label{sec:intro-ocp}

The two-dimensional one-component Coulomb plasma (OCP) is a Gibbs  measure 
on the configurations of $N$ charges $\b z=(z_1,\dots, z_N) \in \C^N$.
Given an external potential $V: \C \to \R \cup \{+\infty\}$,
the Hamiltonian of this measure is defined by
\nomHam[04]{$H_{N, V}^G$}{Hamiltonian for interaction $G$ and external potential $V$}%
\nomPot{$V$}{external potential, for the Coulomb of Yukawa gas, on the plane or torus}%
\nomInt[02]{$G$}{generic two-body interaction}%
\begin{equation} \label{e:Hdef}
  H_{N, V}^G(\b z)
  = N \sum_{j} V(z_j) +  \sum_{j \neq k} G(z_j,z_k)
\end{equation}
where $G(z_j,z_k) = \cal C(z_j-z_k)$ is the two-dimensional Coulomb potential, 
\begin{equation}\label{3}
\cal C(z_j-z_k) = -\log |z_j-z_k|, 
\end{equation}
\nomInt[01]{$\cal C$}{2d Coulomb interaction}%
characterized by $\Delta \log |\cdot| = 2 \pi \delta_0$ as distributions
and $\sum_{i \neq j} = 2 \sum_{i < j}$.
The Coulomb plasma 
is our main interest, but throughout the paper we will also consider other symmetric interactions $G(z_j,z_k)$.
The associated canonical Gibbs measure 
at the inverse temperature $\beta>0$ is defined by
\begin{equation} \label{e:Pdef}
  P_{N, V, \beta}^G (\rd \b z) = \frac{1}{Z_{N, V, \beta}^G} \ee^{-\beta H_{N, V}^G( \b z)} \, m^{\otimes N}(\rd  \b z),
\end{equation}
\nomMea[01]{$m$}{Lebesgue measure on $\C$ or on the torus}%
\nomMea[08]{$P_{N, V, \beta}^G$}{Gibbs measure for interaction $G$, external potential $V$, inverse temperature $\beta$}%
\nomPar{$Z_{N, V, \beta}^G$}{associated to Hamiltonian $H_{N,V}^{G}$, at inverse temperature $\beta$}%
where $m$ denotes the Lebesgue measure on $\C$, and $Z_{N,V,\beta}^G$ the normalization constant.
Here we have assumed that $V$ has sufficient growth at infinity, so that the latter is well-defined.
We will follow  the convention that when $G = \cal C$ then 
we will omit the superscript $\cal C$ whenever there is no confusion.
Similar conventions apply to other subscripts, i.e., we will often omit $N$ and $\beta$.

Throughout the paper, we will use the terms Coulomb plasma, Coulomb gas, and OCP to refer to the measure $P_{N,V,\beta}$.
This model has connections with a variety of models in mathematical physics and probability theory.
For $\beta=1$, it describes the eigenvalues density for some measures on non-Hermitian random matrices \cite{MR0173726,MR1174692}. In particular,  for quadratic $V$ the complex vector
$\b z$ is distributed like the spectrum of a matrix with complex Gaussian entries. Moreover, the properties of this two-dimensional gas are known to be 
related to the fractional quantum Hall effect:
for $\beta=2s+1$, with $s$ integer,   $P_{N, V, \beta}$ is the density obtained from Laughlin's guess for wave functions of fractional fillings of type $(2s+1)^{-1}$ \cite{PhysRevLett.50.1395}. Finally, an important problem is the  crystallization of the two-dimensional Coulomb gas for small temperature
\cite{MR604143,JSP28325}.

The Coulomb plasma is a system with two scales:
the microscopic scale describing distances comparable to the typical interparticle distance $N^{-1/2}$
and the macroscopic scale describing distances of order $1$.
At the macroscopic scale, the empirical particle measure concentrates around a limiting density that
is described by classical potential theory, which we now describe. 
For potentials $V$ that are lower semicontinuous and satisfy the growth condition
\begin{equation} \label{e:Vgrowth}
  \liminf_{|z| \to \infty} \big( V(z) - (2+\epsilon) \log |z| \big) > -\infty
\end{equation}
for some $\epsilon > 0$, 
it is well known (see e.g. \cite{MR1485778}) that there exists a compactly supported equilibrium measure $\mu_V$
that is the unique minimizer of the convex energy functional
\begin{equation} \label{e:IV}
  \cal I_V(\mu) = 
  \iint \log \frac{1}{|z-w|} \,\mu(\rd z)\,\mu(\rd w) + \int V(z) \, \mu(\rd z)
\end{equation}
\nomHam[09]{$\cal I_V$}{energy functional with Coulomb interation and external potential $V$}%
over the set of probability measures on $\C$.
The unique minimizer $\mu_V$ is  supported 
on a compact set  $S_V$ 
\nomOth[17]{$S_V$}{equilibrium set, support of $\mu_V$}%
and, assuming that $V$ is smooth, it has the  density 
\nomMea[02]{$\mu_V$}{equilibrium measure for external potential $V$ and Coulomb interaction, minimizer of $\cal I_V$}%
\begin{equation}\label{eqn:rhoV}
\rho_V=\frac{1}{4\pi}\Delta V\mathds{1}_{S_V}
\end{equation}
\nomMea[09]{$\rho_V$}{density of $\mu_V$}%
with respect to the Lebesgue measure $m$.
We write $I_V = \cal I_V(\mu_V)$ for the minimum of $\cal I_V$.
\nomHam[08]{$I_V$}{minimum of $\cal I_V$}%
For ${\b z} \in \C^N$, the empirical measure is defined by
\begin{equation*}
  \hat\mu =  \hat\mu^{\b z} =  \frac{1}{N} \sum_{j} \delta_{z_j}.
\end{equation*}
\nomMea[04]{$\hat \mu$}{empirical measure}%
For arbitrary $\beta \in (0,\infty)$,
it is  well-known that $\hat\mu \to \mu_V$ vaguely in probability as $N\to\infty$,
with $\hat\mu$ distributed under $P_{N, V, \beta}$. 
In \cite{MR3694026} {(see also \cite{MR3719060})}
we have proved two stronger estimates for the Coulomb gas,
which in can be summarized as follows.
 For $b>0$ and $k \in \N$, we introduce the norms
\nomOth[02]{$b$}{mesoscopic scale for test function, also noted $N^{-s}$}
\nomOth[11]{$N^{-s}$}{mesoscopic scale for test function, also noted $b$}
\begin{equation}\label{normt}
  \normt{f}_{\infty,k,b} =   \sum_{j=0}^k b^j  \norm{ \nabla^j f }_\infty,
  \qquad
  \normt{f}_{\infty,k} = \normt{f}_{\infty,k,1}
  .
\end{equation}
\nomOth[26]{$\normt{\cdot}_{\infty,k,b}$}{$\infty$-norm on scale $b$ up to $k$-th derivative}%
\nomOth[25]{$\normt{\cdot}_{\infty,k}$}{$\infty$-norm up to $k$-th derivative}%
Note that the boundedness of $\normt{f}_{\infty,k,b}$ means that $f$ is smooth at scale $b$.
We typically take $b = N^{-s}$ for some $s\in [0,1/2)$,
and assume that $f$ is supported in a disk of radius of order $b$.
The first estimate proved in \cite{MR3694026} is a \emph{local law} that asserts that 
for any smooth $f$  supported in a disk of radius $b=N^{-s}$ ($s\in[0,1/2)$) centered at
some point $z_0$ in the bulk (i.e., interior) of $S_V$ (and the function $f$ supported in the bulk when $s=0$),
we have
\begin{align}\label{eqn:local}
  \frac{1}{N} \sum_{j=1}^N f(z_j) - \int f(z) \, \mu_V(\rd z)
  &   =
    \OO(\log N) \pa{N^{-1-2s} \|\Delta f\|_\infty + N^{-\half-s} \|\nabla f\|_{ L^2}}
    \nonumber\\
    & = \OO(\log N)N^{-1/2-s} \|f\|_{\infty,2,N^{-s}}
  \,,
\end{align}
where $\normt{f}_{ L^2} = (\int |f|^2 \, \rd m)^{1/2}$ is the $L^2$-norm of $f$,
with very high probability.
A stronger  estimate, which we shall call \emph{rigidity}, asserting that 
\begin{equation} \label{e:rigidity}
  \sum_{j=1}^N f(z_j) - N\int f(z) \, \mu_V(\rd z)
  =\OO(N^\epsilon) \|f\|_{\infty,4,N^{-s}}
\end{equation}
with very high probability, also holds under the same assumptions. 

The main result of this paper is the identification of the random error term in the above rigidity estimate.
It is given by the Gaussian free field with a nonzero mean.

\subsection{Main results}
Our main results are the following two theorems.
In addition to the condition \eqref{e:Vgrowth},
the global potential $V$ is always assumed  to satisfy 
\begin{equation}\label{Vcondition1} 
V \in  \mbox{{$\mathscr{C}^{5}$} on a neighborhood of $S_V=\supp \mu_V$, \;\;
$\alpha_0 \leq \Delta V(z) \leq \alpha_0^{-1}$ for all $z \in S_V$}
\end{equation}
for some constant $\alpha_0 >0$.
We assume that the boundary of $S_V$ is piecewise $\mathscr{C}^{1}$,
or more precisely that $\partial S_V$ is a finite union of $\mathscr{C}^1$ curves.
The prototypical example is $V(z)=|z|^2$ in which case $S_V$ is a disk, and
the convergence $\hat\mu \to \mu_V$ is known as the circular law in random matrix theory.

\begin{theorem} \label{freeasy} 
There exists a constant $\zeta^{\cal C}_\beta  \in\mathbb{R}$ such that,
for any external potential $V$ satisfying the conditions \eqref{e:Vgrowth} and \eqref{Vcondition1},
\nomOth[07]{$\zeta^{\cal C}_\beta$}{universal constant in partition function second order  asymptotics}%
for any $\kappa<1/24$,
\begin{equation*}
\frac{1}{\beta N} \log  \int  \ee^{-\beta H_{V}} \, m^{\otimes N}(\rd\b z)
=
- N   I_V
+ \frac 1 2  \log N 
+ \zeta^{\cal C}_{\beta}
+ \Big (\frac 1 2 -  \frac 1 \beta \Big ) \int \rho_V \log \rho_V \, \rd m
+ \OO(N^{-\kappa}).
\end{equation*}
\end{theorem} 

A similar result, as a limiting statement instead of a quantitative error bound,
and with $\zeta^{\cal C}_{\beta}$ characterized via a large deviation principle,
was previously proved in \cite{MR3735628}.
For our application to the proof of Theorem~\ref{clt} below, a quantitative
error bound is essential.
In addition, we will provide a physical interpretation of $\zeta^{\cal C}_{\beta}$
as the \emph{residual free energy} of the Coulomb
(or technically a long-range Yukawa) gas on the torus;
see Theorems~\ref{thm:Z-torus} and \ref{FA}.

For the statement of Theorem~\ref{clt}, 
we require the following additional  definitions. 
\nomOth[23]{$X^f_V$}{linear statistics of function $f$ centered with $\mu_V$}%
\nomOth[24]{$Y_{V}^f$}{limiting shift for the expectation of $X^f_V$}%
For any  function $f$ with support in $S_V$, 
let
\begin{align}\label{eqn:Xf}
X^f_V &= \sum_j f(z_j) - N \int f \,\rd \mu_V,\\
Y_{V}^f&=     \frac { 1} { 4 \pi}   \int \Delta  f \log \Delta {V} \,\rd m =  \frac { 1} { 4 \pi}   \int \Delta  f(z) \log \rho_{V}(z) \,m(\rd z).\label{eqn:Yn}
\end{align}
In the following theorem, $f : \C \to \R$ is supported on a disk with radius $b=N^{-s}$ for a fixed scale $s\in[0,1/2)$,  and $\normt{f}_{\infty,5,b} \leq C < \infty$ uniformly in $N$. We also assume that the support of $f$
satisfies ${\rm dist}({\rm supp}(f),S_V^c)>\varepsilon$ for some $\varepsilon>0$ uniformly in $N$.
(Indeed, the last condition  can be  relaxed to $\varepsilon=N^{-1/4+c}$ for arbitrarily small $c$,
i.e., $f$ still supported in the bulk).

\begin{theorem} \label{clt}
Suppose that $V$ satisfies the condition \eqref{e:Vgrowth} and \eqref{Vcondition1},
and that $f$ has support in a ball of radius $b=N^{-s}$ with the above conditions.
Then there exists $\tau_{0} = \tau_{0}(s) >0$ such that
for any $0<\tau<\tau_0$ and $0<\lambda\ll (Nb^2)^{1-2\tau}$, we have
\begin{equation*}
\frac{1}{\beta \lambda} \log \E \left(\ee^{- \beta \lambda  \left(X_{V}^f - \big (\frac{1}{\beta}
   -\frac{1}{2}  \big )  Y_{V}^f \right)} \right) = \frac {\lambda}  {8 \pi}   \int  | \nabla  f (z) |^2 \, m(\rd z)   
   +     \OO( (Nb^2)^{-\tau}).
\end{equation*}
Here the expectation is with respect to $P_{N, V, \beta}^{\cal C}$. 
\end{theorem}

Note that  $\lambda$ is allowed to be very large in this theorem;
this provides strong error estimates for the  Gaussian convergence.
This central limit theorem is noteworthy due to the absence of normalization: fluctuations of $X_{V}^f$ are only of order one, due to repulsion, but still Gaussian.
For the purpose of establishing the central limit theorem for $X_{V}^f$,
it suffices to take $\lambda$ to be of order one (independent of $N$).

Finally, a result similar to Theorem~\ref{clt} was obtained simultaneously and independently in \cite{MR3788208}.

\subsection{Related results}

The study of one- and two-dimensional Coulomb and log-gases has attracted considerable attention recently, see e.g. \cite{MR2641363} 
for many aspects of these probability measures in connection with statistical physics.
The subject of our work, abnormally small Gaussian charge fluctuations of the one-component plasma, was first predicted in the late 1970s (see \cite{MR1239571} and the references therein).

In dimension two, in the special case $\beta=1$, the central limit theorem was first proved for the Ginibre ensemble, i.e. for quadratic external potential $V$  \cite{MR2095933,MR2361453}. These results were extended to more general $V$ by combining tools from determinantal point processes and the loop equation approach \cite{MR3342661,MR2817648}. In particular, in the latter works the determinantal structure was used to prove local isotropy of the point process, an
important a priori estimate necessary to the loop equation approach.
For general inverse temperature $\beta$, the determinantal structure does not hold; nevertheless an expansion of the partition function and correlation functions 
was predicted in \cite{MR1986427,MR2932659,MR2240466}.
The expansion of the partition function up to order $N$ was rigorously obtained in \cite{MR3735628}
(along with a corresponding large deviation principle for a tagged point process);
see also the related earlier works \cite{MR3455593,MR3353821,MR3309890};
in addition, see also \cite{MR3693658}.
Still for the two-dimensional Coulomb gas at any temperature,
a local density \cite{MR3824950,MR3719060,MR3694026} was recently proved, together with abnormally small charge fluctuations  in the sense of rigidity \cite{MR3694026}, see (\ref{e:rigidity}).
Other recent results in this direction include \cite{MR3687124,MR3324139,MR3742814,MR3907950}.

For the $\log$-gas on the line, much more is known. Indeed, in dimension one the Selberg integrals are often a good starting point to evaluate partition functions, and anisotropy does not cause any trouble in the analysis of loop equations. For general $\beta$ and $V$, full expansions of the partition function and correlators were predicted in \cite{MR2118807},
proved at first orders in \cite{MR3063494} and at all orders in \cite{MR3010191,1303.1045}. 
A natural analogue of the rigidity \eqref{e:rigidity} is also known to hold for log-gases on the real line \cite{MR3192527}.
Still for the $\log$-gas in dimension $1$,
the central limit theorem was first discovered on the circle for $\beta=2$ in \cite{MR975365},
and on the real line for any $\beta$ in \cite{MR1487983}.
For test functions supported on a mesoscopic scale, the local central limit theorem was  proved
on the circle for some compact groups in \cite{MR1797877},
for general $\beta$ ensembles with quadratic $V$ in \cite{MR3541852}
and for general $V$ in \cite{MR3865662}.

For expansions at high temperatures, and exponential decay of microscopic correlations,
in closely related models of Coulomb gases, see \cite{MR691043,MR574172}.
For results on crystallization in the one-dimensional one-component Coulomb plasma, see \cite{kunz1974one,BraLieAppl,MR597033}.
Further results on Coulomb systems in statistical mechanics are reviewed in \cite{MR1722991,MR2641363}.

\subsection{Proof sketch}

In Section~\ref{sec:Z-torus}, we first prove that an extended version
of Theorem~\ref{freeasy} holds for Yukawa gases on a torus.
The essence is to show that the constant $\zeta^{\cal C}_\beta$, to be called the residual free energy,
can be identified independently of the range of the Yukawa interaction.   This fact is then used
in Section~\ref{sec:quasifree} to  establish an expansion of the free energy of the Coulomb gas  up to order $N^{1-\kappa}$.
The main idea is to approximate the Coulomb gas first by a short-range Yukawa gas, and then by a quasi-free Yukawa gas. Roughly speaking, 
a Yukawa gas with range $\ell \ll 1$ can be viewed, for the purpose of computing free energy, as an ideal gas consisting of 
independent squares 
of size $b$ satisfying  $1\gg b \gg \ell$ and with the gas inside each square being a  periodic Yukawa gas with range $\ell$.  Since  this gas is an ideal gas over  a distance longer than a  mesoscopic scale $b$,
we call it a quasi-free approximation.

The Yukawa approximation to the Coulomb gas is a well-known tool in the study of the quantum Coulomb
gas, see,  e.g., \cite{MR1017745, MR937769}. However, the precision needed here is far beyond the previous results.  
Following the traditional approaches in free energy estimates, 
we will prove  the  free energy expansion of the Coulomb gas  by establishing 
a lower and an upper bound.  The proof of the upper bound,  contained in Section~\ref{sec:quasifreeub}, consists of the standard argument 
of counting two-body Yukawa interactions in neighboring squares and uses only that the density of 
Coulomb gas is bounded for all  scales $\ge N^{-1/2+ \epsilon}$ by \cite{MR3694026}. 
The lower bound turns out to be much more difficult than the upper bound. The Yukawa gas used in the approximation of the Coulomb gas 
is constructed from the Yukawa gas on periodic squares,
so the resulting Yukawa gas on the plane breaks the translational and rotational invariances of the Coulomb gas. 
The translational invariance is easy to restore by averaging over the ``grid'' of the squares. The rotational invariance, however, is hard to recover and the effects of breaking 
it has to be estimated precisely. We remark that in the quantum Coulomb gas, the lower bound of the free energy was proved \cite{MR0339751} by carefully maintaining the Coulomb rotational invariance. This was possible due to  the use of the  ``Swiss cheese'' approximation.
In our setting, we are forced to maintain the square approximation since  the limits  of the residual free energy were established only for squares. The key estimate which allows us to control the breaking of the rotational invariance 
is contained in Section~\ref{sec:decoupling}, where  we estimate  the energy distortion resulting  from embedding torus into the Euclidean space.
This estimate uses the rigidity estimate of the periodic  Yukawa gas, a parallel version of the 
 rigidity estimate 
established in \cite{MR3694026} for the Coulomb gas.   
Using the estimates from Section~\ref{sec:decoupling},  we complete the proof of the lower bound of the free energy in Section~\ref{sec:quasifreelb}.

Another difficulty in establishing Theorem \ref{freeasy} is the surface energy of a Coulomb gas.
The typical inter-particle distance of this gas is  $N^{-1/2}$, therefore 
the total Coulomb energy  for  particles within a distance  $N^{-1/2}$ 
to  the boundary of the support of the equilibrium measure is  of order $N$.
To see this, note that the number of surface particles, i.e., particles with distance 
of order  $N^{-1/2}$ to the boundary, is of order $\sqrt N$. Thus their Coulomb interaction energy is of order $N$.
Theorem~\ref{freeasy} requires to capture 
these interaction energies up to order $N^{1-\kappa}$.
In other words, the leading term in the energy associated with the  charges near the boundary of the support
of the Coulomb gas  has to be identified. 
Our idea is to use an ideal gas approximation for a boundary layer and then switch to a Yukawa approximation 
for interior particles.
We will explain this idea in Section~\ref{sec:quasifree}.

In Section~\ref{sec:clt}, we first prove that the central limit theorem holds after subtracting
a random term, the \emph{local angle term}. From this result and the asymptotic expansion of the free energy 
for the Coulomb gas, Theorem \ref{freeasy}, we obtain  that  the angle term does in fact vanish in a large deviation sense.
We thus prove Theorem~\ref{clt} for a test  function $f$ with macroscopic support.  
For test functions with support on a mesoscopic scale~$b$,
we proceed via conditioning to a  disk of radius~$2b$.
This conditioning procedure was used in \cite{MR3694026}; it has the advantage 
of  reformulating  the question into a problem on the natural scale $b$. 

Throughout the paper,
we will extensively use the local density and rigidity estimates for the Yukawa gas and Coulomb gas with additional angular interaction, in a form similar to (\ref{eqn:local}) and (\ref{e:rigidity}).
In Appendices~\ref{app:Yrigi}--\ref{app:yukawa}, we therefore extend the estimates of \cite{MR3694026} 
to the Yukawa gas and the Coulomb gas with angle term.
The rigidity estimates, to be proved in Appendix~\ref{app:Yrigi}, 
use estimates of the local laws in Appendix~\ref{app:yukawa}. We reverse the logic order because 
the proofs of the local laws in Appendix~\ref{app:yukawa} are technical and use extensively conventions from \cite{MR3694026}.

\subsection{Notation}
We use the usual Landau $\OO$ and $\oo$ symbols.
For $N$-dependent quantities $A,B \geq 0$, we write $A \ll B$ when there exists
$\epsilon>0$ and $N_0 \geq 0$ such that $A \leq N^{-\epsilon} B$ for $N \geq N_0$.
For an event $E$, we say that $E$ holds with high probability if
for all $D>0$ there is $N_D$ such that $\P(E) \geq 1-N^{-D}$ for $N \geq N_D$.
For random variables $A$ and $B$, we write $A\prec B$ if for any $\epsilon>0$ the event
$|A|\leq N^\epsilon|B|$ holds with high probability.
\newcommand{\Oinfty}{\OO(N^{-\infty})}
\nomOth[27]{$\Oinfty$}{subpolynomially small error term}%
We use the notation $A=\Oinfty$ 
to denote that $A$ is subpolynomially small:
for every $D>0$ one has $|A| \leq N^{-D}$ for all $N\geq N_D$
with probability at least $1-N^{-D}$ (if $A$ is a random variable).

\section{Preliminaries} 
\label{sec:prelim}

We begin with the definitions of the Coulomb and Yukawa gas ensembles,
and we give a summary of the potential theory that we require,
as well as of the estimates on the local density.

\subsection{Coulomb and Yukawa potentials}
\label{sec:prelim-yukawa}

We will identify $\R^2$ and $\C$ and usually write $z$ and $w$ for its elements.
The two-dimensional Coulomb potential is $C(z) = -\log |z|$, satisfying
$-\Delta C = 2\pi \delta_0$ as distributions.
The Yukawa potential with range $\ell>0$ is the solution to $(-\Delta+1/\ell^2)Y^\ell = 2\pi \delta_0$.
Explicitly, the two-dimensional Yukawa potential is given by the formula
\begin{equation} 
  Y^{\ell} (z)
  := \frac{1}{4\pi} \int_{\R^2} \ee^{- \ii p \cdot z} \int_0^\infty \ee^{- t (p^2+ 1/\ell^2)/2} \, \rd t \, \rd p
  =  \int_1^\infty  \ee^{- a ( s  + 1/s) } \, \frac {\rd s}s =: g(a) \label{Yg},
  \quad a = \frac {|z|}{2\ell},
\end{equation} 
\nomInt[09]{$Y^{\ell}$}{Yukawa interaction on $\mathbb{C}$ with range $\ell$}%
\nomOth[06]{$\ell$}{range of the Yukawa gas}%
where $p\cdot z$ denotes the Euclidean inner product on $\R^2$.
From this formula, note that $Y^{\ell}(z)$ is pointwise positive and positive definite,
and that there is an absolute constant $Y_0$ such that 
\begin{equation} \label{Ya}
Y^{\ell}(z)
\begin{cases}  \sim - \log |z| + \log \ell + Y_0+ \OO(|z|/\ell) & \mbox  { if }  |z|/\ell \leq 1\\
  \le  C_1 \ee^{ - C_2|z|/\ell}      & \mbox  { if }   |z|/\ell \ge  1.
\end{cases} 
\end{equation}
Indeed, the asymptotic relation can be checked with constant $Y_0 = \log 2+ \vartheta$ from
\begin{equation} \label{e:asg}
  g(a) = \vartheta- \log a + \OO(a), \quad \vartheta = \int_0^\infty \pa{\ee^{-s}-1_{s<1}} \frac{\rd s}{s}, 
\end{equation}
where $g$ was defined in \eqref{Yg}.
In particular, up to the constant $Y_0 + \log \ell$,
the two-dimensional Coulomb potential $-\log |z|$ is the limit $\ell\to\infty$ of $Y^{\ell}(z)$.
We denote by $\T$ the two-dimensional unit torus $\C/\Z^2$.
\nomOth[19]{$\T$}{unit torus}
\nomInt[05]{$U^{\ell}$}{periodic Yukawa interaction on $\mathbb{T}$ with range $\ell$}%
For $\ell > 0$, the Yukawa interaction of range $\ell$ on $\T$ is defined by
\begin{equation} \label{e:Uell}
  U^\ell(z)
  = \sum_{n\in \Z^2} Y^\ell(z+n).
\end{equation}

\subsection{Ensembles}
\label{sec:prelim-ensembles}

We now define the Coulomb gas and its perturbed versions on the plane $\C$, and 
the Yukawa gas on the torus $\T$.

\paragraph{Coulomb (and Yukawa) gas on the plane}

Remember that for a one-particle potential $V: \C\to \R \cup\{+\infty\}$
and the two-particle interaction $G: \C \times \C \setminus \triangle \to \R$ on $\C$,
where $\triangle = \{(z,w) \in \C^2: z=w\}$,
we define the $N$-particle Hamiltonian by
\begin{equation} \label{e:Hdef-bis}
  H_{N,V}^G(\b z) = N \sum_{j} V(z_j) +  \sum_{j \neq k} G(z_j,z_k),
  \qquad (\b z \in \C^N),
\end{equation}
and the corresponding Gibbs measures at inverse temperature $\beta>0$ by
\begin{equation} \label{e:Pdef-bis}
  P_{N,V,\beta}^G (\rd \b z) = \frac{1}{Z_{N,V, \beta}^G} \ee^{-\beta H_{N,V}^G( \b z)} \, m^{\otimes N}(\rd  \b z),
\end{equation}
where $Z_{N,V,\beta}^G$ is the partition function.
The Coulomb interaction is obtained by taking $G(z,w)=\mathcal{C}(z-w)$ to be the Coulomb potential
and we omit the argument $G$ in that case;
the Yukawa interaction of range $\ell$ is obtained with $G(z,w)=Y^\ell(z-w)$
and we then write $\ell$ instead of $Y^\ell$ in the superscript.
For the Coulomb case, we sometimes use the convention $\ell=\infty$.

On the plane, we only use the Yukawa potential as a regularization
of the Coulomb potential, with $\ell \geq N^2$, in which case it is for all
of our purposes equivalent to a Coulomb potential.

\paragraph{Yukawa gas on the torus}

Similarly, for $\ell>0$ and for a potential $V: \T \to \R$,
the $N$-particle Hamiltonian of the periodic Yukawa gas on $\T$ is defined by
\begin{equation}\label{Hell}
  H_{N,V}^\ell(\b z)
  = N \sum_j V(z_j) + \sum_{j\neq k} U^\ell(z_j-z_k),
  \qquad (\b z \in \T^N),
\end{equation}
where $U^\ell$ was defined in \eqref{e:Uell} and we here use the abbreviation $H^\ell_{N} = H^\ell_{N,0}$
The corresponding probability measures are again defined as in \eqref{e:Pdef-bis},
with $m$ now the Lebesgue measure on $\T$.
\nomHam[05]{$H_{N,V}^\ell$}{Hamiltonian for the Yukawa interaction on $\mathbb{T}$ with range $\ell$ and external potential $V$}%

On the torus, we  use the Yukawa potential with short range
compared to the side length of the torus (but still large with respect to
the interparticle spacing), i.e., $N^{-1/2} \ll \ell \ll 1$.

\paragraph{Perturbed Coulomb gas on the plane}

We will also consider perturbations of the Coulomb gas on the plane,
for which the two-particle interaction
takes the form
\begin{equation} \label{e:C-perturb-def}
  G(z,w) = \mathcal{C}(z-w) + t\tilde G(z,w),
\end{equation}
\nomInt[03]{$\tilde G$}{generic perturbation of the interaction}%
with $t\in\R$, and where we assume that the perturbation $\tilde G$ satisfies, for some $\theta>0$,
\begin{equation} \label{e:C-perturb}
  |\tilde G(z,w)| \leq 1, \quad |\tilde G(z,w)| \leq \ee^{-|z-w|^2/(2\theta^2)}.
\end{equation}
The perturbed Coulomb gas 
will be used only in Section~\ref{sec:clt}.
We therefore suggest the reader to skip this material until it is used in Section~\ref{sec:clt}.

\subsection{Potential theory}

We define variational functionals
for the Yukawa potential with external potential $V$ on probability measures $\mu$ on $\C$ by
\begin{equation} \label{e:Iyukawa}
  \cal I^\ell_V(\mu )
  = \int V(z) \, \mu(\rd z)
  + \int Y^\ell(z-w) \, \mu (\rd z) \, \mu( \rd w).
\end{equation}
\nomHam[10]{$\cal I^\ell_V$}{energy functional with Yukawa interaction with range $\ell$  and external potential $V$}%
In the definition of the variational functional for the Coulomb interaction,
the Yukawa potential $Y^\ell$ is replaced by the Coulomb potential $C$,
and we then again omit the superscript $\ell$.
Moreover, we use the analogous definition for the variational functional of the Yukawa gas
on the torus $\T$, where $Y^\ell$ is replaced by $U^\ell$.
We always make the following assumptions:
\begin{enumerate}
\item
  The set $\Sigma_V = \{z: V(z) < \infty \}$ has positive logarithmic capacity;
  see \cite[Section~I.1]{MR1485778}.
\item
  The potential $V$ is locally in $C^{1,1}$ and, for the full plane,
  it satisfies the growth condition
  \begin{equation} \label{e:Qgrowth}
    \liminf_{|z| \to \infty} (V(z) - \varepsilon \log |z|) > -\infty.
  \end{equation}
  In the Yukawa case, we assume $\varepsilon > 0$,
  whereas in the Coulomb gas we assume that $\varepsilon>2$.
  In the case of the torus, the growth assumption is trivial.
\end{enumerate}
For a probability measure $\mu$ on $\C$ respectively $\T$,
define the Yukawa potential by 
\nomPot{$Y^\ell_{\mu}$}{Yukawa potential with range $\ell$ associated to a measure $\mu$, on the plane}%
\nomPot{$U^\ell_{\mu}$}{Yukawa potential with range $\ell$ associated to a measure $\mu$, on the torus}%
\begin{equation*}
  Y^\ell_{\mu}(z) 
  = \int Y^\ell(z-w)  \, \mu(\rd w),
  \quad \text{respectively} \quad
  U^\ell_{\mu}(z) 
  = \int U^\ell(z-w)  \, \mu(\rd w),
\end{equation*}
and again we use analogous notation in the Coulomb case.
The following standard result gives the existence and uniqueness of the equilibrium measure for the Yukawa and Coulomb gas. 
Let $P(\Sigma_V)$ be the set of probability measures supported in $\Sigma_V$.
We write $m = 1/\ell$ and use the convention $m^2=0$ for the Coulomb case.
\nomOth[08]{$m$}{inverse of the Yukawa interaction, $m=1/\ell$}%

\begin{theorem} \label{thm:Yeqmeasure}
  Consider the Yukawa potential of range $\ell$ on $\C$ or the Coulomb potential on $\C$ (with the convention $\ell=\infty$).
  Suppose satisfies assumption (i)--(ii) above.
  \nomMea[03]{$\mu_V^\ell $}{equilibrium meaure for external potential $V$ and Yukawa interaction with range $\ell$, minimizer of $\cal I_V^\ell$}%
  Then there exists a unique $\mu_V^\ell \in P(\Sigma_V)$ such that
  \begin{equation} \label{defI}
  \cal   I^\ell_V(\mu_V^\ell) = \inf \{ \cal I^\ell_V(\mu): \mu \in P(\Sigma_V) \}.
  \end{equation}
  The support $S_V^\ell = \supp \mu_V^\ell $ is bounded (uniformly in $\ell$)
  \nomOth[18]{$S_V^\ell$}{equilibrium set, support of $\mu_V^\ell$}%
  and of positive capacity, and $\cal I_V^\ell (\mu_V^\ell) < \infty$. Furthermore, 
  the energy-minimizing measure $\mu_V^\ell$ may be characterized as the unique
  element $\mu$ of $P(\Sigma_V)$ for which 
  there exists a constant $c_V \in \R$ such that
  Euler-Lagrange equation
  \begin{align}
    \label{e:EL}
    { Y^\ell_{\mu}} + \tfrac12 V = c_V & \quad \text{q.e.\ in $S_V^\ell$} \quad \text{and}\\
    { Y^\ell_{\mu}} + \tfrac12 V \geq c_V & \quad \text{q.e.\ in $\C$} \nonumber
  \end{align}
  holds. 
  The equilibrium measure $\mu^\ell_V$ in the set $S_V^\ell$ is given by
  \begin{equation}
    \label{e:equilibriumdensity}
    \mu_V^\ell
    = \frac{1}{4\pi} (\Delta V + m^2( 2c_V- V))
    =  \frac{1}{4\pi} \left( (\Delta-m^2) V +  2m^2    c_V\right),
  \end{equation}
  where the Laplacian is understood in the distributional sense.
  The same statement holds for the Yukawa potential on the torus $\T$ with $Y_\mu^\ell$ replaced by $U_\mu^\ell$.
\end{theorem}

The proof is identical to that of the Coulomb case; see e.g.\ \cite{MR1485778}.
Also, by the same argument, under the assumption that $V$ satisfies \eqref{e:Vgrowth},
the support of $\mu_V^\ell$ is compact uniformly in $\ell$.

\medskip
In the case of the Yukawa gas on the torus with $V=0$, by translation-invariance,
the unique minimizer in \eqref{defI} is the uniform probability measure on $\T$.
Hence the minimum energy of the variational functional for the Yukawa gas on the unit torus
is simply given by
\begin{equation} \label{e:Iell-torus}
  \inf_{\mu \in P(\T)}  \int U^\ell(z-w) \, \mu (\rd z) \, \mu(\rd w)
  =
  2\pi\ell^2.
\end{equation}
We will use this fact in Section~\ref{sec:Z-torus}.

\subsection{Local density estimates}

From now on, we always assume that $V$ satisfies the assumptions
of Theorem~\ref{thm:Yeqmeasure}.
The \emph{local density estimates} stated in the following theorems imply that,
for any disk $B$ of radius $r \gg N^{-1/2}$ (and the respective support assumptions),
the number of particles in $B$ is of order $r^2$ with high probability under the
respective ensemble.
For their statements,
given a test function $f: \C \to \R$, we denote the linear statistic
centered by the equilibrium measure by
\begin{equation} \label{e:Xfdef-app}
  X_f
  =
  \sum f(z_k) - N \int f(z) \, \mu_V^\ell(\rd z)
  =
  N \int f(z) \, \tilde\mu(\rd z),
\end{equation}
where $\hat\mu = \frac{1}{N} \sum_j \delta_{z_j}$ denotes the empirical measure, and
\nomMea[05]{$\tilde \mu$}{difference between empirical measure and equilibrium measure}%
\begin{equation}
  \tilde \mu = \hat\mu - \mu_V^\ell.
\end{equation}%
The following two theorems will be proved in the Appendix~\ref{app:yukawa}.

\begin{theorem}[Local density for the torus] \label{thm:YTdensity}
Consider the Yukawa gas on the torus $\T$ with Hamiltonian~\eqref{Hell}
and assume  the potential $V$ satisfies (i) above and $\supp(\mu_V)=\T$.
For any $f: \T \to \R$ supported in a disk of radius $b\gg N^{-1/2}$,
\begin{equation}
\label{e:YTdensity}
|X_f |
\prec \sqrt{N b^2 (f,(-\Delta+m^2)f)} + b^2 \|\Delta f\|_\infty.
\end{equation}
In particular,
for any disk $B \subset \T$ with radius $b \gg N^{-1/2}$, with high probability,
we have
\begin{equation}
\label{e:YTdensity2}
N\hat\mu(B) = \OO(Nb^2).
\end{equation}
\end{theorem}

\begin{theorem}[Local density estimate on  the plane] \label{thm:Cdensity}
Suppose that $V$ satisfies the conditions \eqref{e:Vgrowth} and \eqref{Vcondition1}.
Consider either the Coulomb gas on $\C$ with potential $V$ and Hamiltonian~\eqref{e:Hdef-bis},
the perturbed Coulomb gas  in \eqref{e:C-perturb} with $|t| \theta^2 N \leq 1$,
or the Yukawa gas with range $\ell \geq N^2$.
Then for any $f: \C\to \R$ supported in a disk of radius $b \gg N^{-1/2}$
that is contained in $S_V$ and has  a distance  $\gg N^{-1/4} + t^{1/4}$
to $\partial S_V$,
\begin{equation}
\label{e:Cdensity}
| X_f|
\prec \sqrt{Nb^2 (f,-\Delta f)} +  b^2\|\nabla^2 f\|_\infty
= \OO(\sqrt{Nb^2}) \|f\|_{\infty,2,b}.
\end{equation}
In particular, for any disk $B \subset S_V$ with radius $b \gg N^{-1/2}$
and distance $\gg N^{-1/4}  + t^{1/4}$ to $\partial S_V$, with high probability,
\begin{equation}
\label{e:Cdensity2}
N\hat\mu(B) = \OO(Nb^2).
\end{equation}
Moreover, if $D = \{z\in S_V: \dist(z,\partial S_V) \leq b'\}$  with $b' \gg N^{-1/4}$ then,
with high probability,
\begin{equation} \label{e:Cdensity3}
N\hat\mu(D) = \OO(Nb').
\end{equation}
\end{theorem}

\subsection{Rigidity estimates}

In addition to the local density estimates of the previous subsection,
for the Yukawa gas on the torus,
we also need the stronger \emph{rigidity estimates} 
given by the following theorems.
These theorems are proved in Appendix~\ref{app:Yrigi},
again following the method of \cite{MR3694026}.

\begin{theorem}[Rigidity estimate for Yukawa gas on the torus] \label{thm:Yrigi}
Consider the Yukawa gas on the unit torus of range $\ell$.
Let $s \in (0,\frac12)$, and assume that $N^{-1/2} \ll \ell \ll 1$ and that $V=0$.
For any sufficiently smooth $f : \T \to \R$ supported in a ball of radius $b=N^{-s}$,
\begin{equation}
\label{e:Yrigi} 
 |X_f| \prec \left(\frac{b}{\ell}+1\right)^2
 \normt{f}_{\infty,3, b}.
\end{equation}
\end{theorem}

In the regime that  $b/\ell \le  1$, this estimate  improves the previous local density estimate by about a factor $1/(\sqrt N b)$ with a price 
of taking one more derivative in the test function $f$. 

As a corollary, we obtain the following proposition which estimates functions of two points.
The proposition, proved in Appendix \ref{app:Yrigi},
is a direct application of the rigidity estimate just stated and Taylor expansion.  
To state the estimate, for any sufficiently smooth function $g: \T \times \T \to \R$, we denote
\begin{equation}\label{eqn:nablaj}
  g^{(j)}_{ {\rm B}_t}(z,w)=\sup_{(x,y)\in{\rm B}_t(z)\times {\rm B}_t(w)}|\nabla^j g(x,y)|,
\end{equation}
where ${\rm B}_t(z)$ is the Euclidean ball of radius $t$ centered at $z$ and
$|\nabla^j g(x,y)|$
is the maximum over all partial derivatives of $g$ of order $j$.

\begin{proposition} \label{prop:Grigi}
Consider the Yukawa gas on the unit torus of range $\ell$.
Assume that $N^{-1/2} \ll \ell \ll 1$ and that $V=0$.
Fix $N^{-1/2} \ll s \ll 1$.
Then for any smooth function $g$ on $\T \times \T$ 
and any fixed $p\in\mathbb{N}$,
\begin{equation} \label{e:Grigi}
  N^2 \iint g(z, w) \, \tilde\mu(\rd z) \, \tilde\mu(\rd w)
  \prec
  \pa{\frac { 1} {s^4 } +  \frac {1} { \ell^4}}
  \sum_{j=0}^{p-1} s^j \|\nabla^j g\|_1 
  N^2 s^p\|g^{(p)}_{{\rm B}_s}\|_1 
\end{equation}
where $\|\cdot\|_1$ is the $L^1$-norm on $\T\times \T$ and $\|\nabla^j g\|_1 = \||\nabla^j g|\|_1$.
\end{proposition}

Notice that, besides explicit factors,  $t$ only appears in the error term $g^{(p)}_{{\rm B}_t}$.
We usually choose $t$ to be  slightly smaller than the scale that the function $g$ is smooth on.

\subsection{ Conditioned local density estimates}
\label{sec:prelim-cond}

To prove the mesoscopic versions of the central limit theorem,
in addition to the above local density estimates,
we need conditioned versions of these.
These and can be skipped on the first reading.  

To state the conditioned estimates, we first recall
the local conditioning from \cite[Section~5]{MR3694026}.
 We first focus on the Coulomb gas on the plane and comment on the changes
for the Yukawa gas on the torus afterwards.
Let $B \subset \C$ be a  disk of radius $b$ contained in $S_V$,
and consider the Coulomb gas obtained by conditioning on all of the particles outside $B$.
More precisely, for a particle configuration $z\in \C^N$,
let $M = M(z)$ denote the number of particles in $B$,
let $(\tilde z_1, \tilde z_2, \dots, \tilde z_M)$ denote the collection of particles inside $B$,
and let $(\hat z_1, \hat z_2, \dots, \hat z_{N-M})$ denote the particles outside $B$.
The Hamiltonian $H_{N,V}$ may then be written as
\begin{equation}
  H_{N,V}(z)
  = \sum_{j \neq k} \log \frac{1}{|\tilde z_j-\tilde z_k|}
  + N \sum_{j} \pB{ V(\tilde z_j) - V_o(\tilde z_j| \hat z) }
  + E(\hat z),
\end{equation}
where
\begin{equation}
  V_o(w|\hat z) = -\frac{2}{N} \sum_{k} \log \frac{1}{|w-\hat z_k|},
                  \qquad
  E(\hat z) = \sum_{j \neq k}  \log \frac{1}{|\hat z_j-\hat z_k|} + N \sum_{j} V(\hat z_j)
  .
\end{equation}
The term $E(\hat z)$ is independent of the particles in $B$
and is thus irrelevant for the conditioned measure.
For any configuration of external particles $\hat z \in (\C \setminus B)^{N-M}$
and $z \in \C$, we write
\begin{align}
  \label{e:Wdef}
  W(w|\hat z) &= \begin{cases}
    \frac{N}{M} (V(w) - V_o(w|\hat z)) & (w \in B),\\
    +\infty& (w \not\in B),
  \end{cases}
 \\
  \label{e:Pconddef}
  P_{N,V,\beta}(\rd w | \hat z) &=  P_{M(\hat z),W( \cdot | \hat z),\beta}(\rd w).
\end{align}
The Coulomb gas given by the potential $W(\cdot|\hat z)$ is the
conditional gas inside $B$, given the external configuration $\hat z$.
Here we have used the convention of the measure $ P_{N,V,\beta}(\rd w | \hat z)$
in \eqref{e:Pdef}; this convention  also explains the normalization factor $N/M$ in \eqref{e:Wdef}.
In \cite{MR3694026}, it was proven that under our assumptions on $V$
the conditional potential satisfies the following properties.
First, since $V_o(\cdot|\hat z)$ is harmonic in $B$
we have 
\begin{equation}\label{eqn:Delta}
\mu_W = \frac{\Delta W(z)}{4 \pi} = \frac{N}{M} \mu_V
\end{equation}
in the interior of the support $S_W \subset B$
(where $\mu_W$ and its support $S_W$ are defined by minimization of
the Coulomb version of \eqref{e:Iyukawa}).
For any function $f$ that has compact support in $S_W$, we thus have
\begin{equation} \label{e:muVmuW}
M \int f \,\rd \mu_W = N \int f \, \rd \mu_V.
\end{equation}
Finally, from \cite[Sections~5-6]{MR3694026},
we know that the measure $\rd\mu_W$ may be expressed as $\tfrac{N}{M} \mathds{1}_{S_W} \rd \mu_V  + v  \,\rd s$,
where $\rd s$ is the length measure on $\partial B$, $v \in L^\infty(\partial B)$,
and that the following properties hold.
These properties are verified in the proof of \cite[Theorem~6.1]{MR3694026}.

  The same definitions and properties apply in the Yukawa case when the Coulomb potential is replaced by the Yukawa potential
  (the analogues of \cite[Sections~5-6]{MR3694026} are proved in Appendix~\ref{app:yukawa}), when
  \eqref{eqn:Delta} is replaced by the Yukawa density of the form \eqref{e:equilibriumdensity},
  and when \eqref{e:muVmuW} is restricted to test functions with $\int f \, \rd m = 0$.

\begin{lemma} \label{lem:Wproperties}
Consider the perturbed Coulomb gas on the plane as in Theorem~\ref{thm:Cdensity}
or the Yukawa gas on the torus as in Theorem~\ref{thm:YTdensity}.
For any $s \in (0,\frac12)$,
there exists a constant $ \tau > 0$ such that the following statements hold with probability
at least $1 - \ee^{-N r^2}$  for $r=N^{-s}$:
\begin{align*}
\hspace{-3cm}{\it (i)}\hspace{3cm}&M = N\mu_V(B)(1+\OO(M^{-\tau})),\\
\hspace{-3cm}{\it (ii)}\hspace{3cm}&S_W \supset \{ z \in B: d(z,\partial B) > M^{-\tau} r \},\\
\hspace{-3cm}{\it (iii)}\hspace{3cm}&\mu_W(\partial B) = \int v \, \rd s \leq M^{-\tau},\\
\hspace{-3cm}{\it (iv)}\hspace{3cm}&\|v\|_\infty \leq \OO(1/r).
\end{align*}
\end{lemma}

In particular, any disk $B$ in the lemma satisfies the following 
\emph{good boundary conditions}:

\begin{definition}[Good boundary conditions]  \label{good boundary}
Fix a scale $N^{-1/2 + \epsilon} \le r \ll 1$.
Let  $B$ be a disk of radius  $r$,
let $P(\cdot|{ \hat z})$ be the conditional law (with the particles $ \hat z$ outside $B$ fixed)
of the Coulomb gas induced on $z \in B^{M({ \hat z})}$ where $M({ \hat z})$ is the number of particles  contained in $B$,
and let $W(\cdot |\hat z)$ be the corresponding potential (with $W(\cdot |\hat z)=+\infty$ outside $B$).
We say that the boundary condition $ \hat z$ of the conditional law are \emph{good boundary conditions} if the following properties hold.
The  equilibrium measure associated to $W = W(\cdot|\hat z)$
of the conditional measure 
can be decomposed as $\mu_{ W}(\rd z) = \rho_{ W}(z) \, m(\rd z) + v(z) \, \rd s$,
where $\rd s$ is the length measure on $\partial S_{ W}$ and $S_{ W} \subset B$. 
Furthermore, there exists a disk $\Omega$ of 
 radius   $r(1-N^{-\tau})$ for some $\tau > 0$ 
such that   the equilibrium measure satisfies the bounds 
\begin{equation} \label{e:quasifree-meso-cond}
  \sum_{k=0}^3  r^k  \|(\nabla^k\rho_{ W})1_{\Omega}\|_\infty \leq K,
  \quad
  \frac 1 {| S_W|}  \int_\Omega \rho_{ W}(z) \, m(\rd z) \geq 1-{ M}^{-a},
  \quad
  \|v\|_\infty \leq { M}^{A}
\end{equation}
for some constants $a>0, A \geq 0,K>0$.
\end{definition}

\begin{theorem} \label{thm:YTdensity-cond}
  In the setting of Theorem~\ref{thm:YTdensity},
  let $B$ be a disk of radius $r$ with good boundary conditions, and write $n=Nr^2$.
  Then  good boundary conditions
  in the sense of Definition~\ref{good boundary} hold with high probability  under the original measure.
  Furthermore, 
  for any disk $B' \subset  S_W$ with radius $ N^{-1/2} \ll b \ll r$
  and distance $\epsilon r$ to $\partial S_W$,
  with high probability  under the conditioned measure,
  the conditioned version of \eqref{e:YTdensity} holds { (where $X_f$ is defined with respect to $V$)}:
  \begin{equation}
    \label{e:YTdensity-cond}
    |X_f|
    \prec \sqrt{N b^2 (f,(-\Delta+m^2)f)} + b^2 \|\Delta f\|_\infty.
  \end{equation}
\end{theorem}

\begin{theorem} \label{thm:Cdensity-cond}
  In the setting of Theorem~\ref{thm:Cdensity},
  let $B$ be a disk of radius $r$ with good boundary conditions, and write $n=Nr^2$ and $t=N^{-2\sigma}$.
  Then  good boundary conditions
  in the sense of Definition~\ref{good boundary} hold with high probability  under the original measure.
  Furthermore, 
  for any disk $B' \subset  S_W$ with radius $ N^{-1/2} \ll b \ll r$
  and distance $\gg (n^{-1/4}  +n^{-\sigma/2}) r$ to $\partial S_W$,
  with high probability
  under the conditioned measure,
  the conditioned version of \eqref{e:Cdensity} holds:
  \begin{equation}
    \label{e:Cdensity-cond}
    |X_f|
    \prec \sqrt{Nb^2 (f,-\Delta f)} +  b^2\|\nabla^2 f\|_\infty
    .
  \end{equation}
\end{theorem}

\section{Free energy of the torus}
\label{sec:Z-torus}

We start with proving a version of Theorem~\ref{freeasy} for the Yukawa gas on the torus.
This outlines the strategy for the proof of Theorem~\ref{freeasy} in a simplified context
and also constructs the constant $\zeta$ in Theorem~\ref{freeasy}.

\subsection{Main result}

Recall the definition of the Yukawa gas on the unit torus from Section~\ref{sec:prelim-ensembles}
and also that the minimum energy of the variational functional for the Yukawa gas on the unit
torus is given by $2\pi\ell^2$ by \eqref{e:Iell-torus}.
We denote the $N$-particle partition function of the Yukawa gas on the unit torus with range $\ell$ by
\begin{equation*}
  Z^{(\ell)}_{N} 
  = \int_{\T^N} \ee^{-\beta H_{N}^\ell(\b z)} \, m(\rd\b z),
\end{equation*}
\nomPar{$Z^{(\ell)}_{N}$}{associated to Hamiltonian $H_N^{(\ell)}$, at inverse temperature $\beta$}%
where $H_{N}^\ell$ was defined in \eqref{Hell}.
The main result of this section is the following theorem,
namely a version of Theorem~\ref{freeasy} for the Yukawa gas on the torus.

\begin{theorem} \label{thm:Z-torus}
  There exists a $\beta$ dependent constant $\zeta$, the \emph{residual free energy} of
  the Yukawa gas on the torus,
  such that for any $\sigma>0$ there is $\kappa>0$ such that if $N^{-1/2+\sigma}\leq \ell\ll 1$,
  \begin{equation} \label{e:Z-torus}
    \frac{1}{\beta} \log Z_{N}^{(\ell)} = -2\pi \ell^2 N^2 + N \log \ell + \frac12 N \log N
    + N \zeta + \OO(N^{1-\kappa}).
  \end{equation}
More precisely, $\OO(N^{1-\kappa})$ is $N^{\varepsilon}\OO(N^{7/8}+N^{1-2\sigma})$.
\end{theorem}

\begin{remark}
The above statement holds without the assumption $\ell \ll 1$.
Since this generalization is not needed for our application,
we restrict to this slightly simplified case.
\end{remark}

To prove Theorem~\ref{thm:Z-torus}, we define
the specific  residual free energy  in a system of  $N$ particles with interaction range $\ell$ by
\begin{equation}
  \label{zetadef} 
  \zeta^{(\ell)}(N)
  =
  \frac{1}{N} \xi^{(\ell)}(N) -  \frac 1 2  \log  N,
  \quad
  \xi^{(\ell)}(N)
  = \frac{1}{\beta} \log Z^{(\ell)}_N
  + 2\pi \ell^2 N^2  
  - N \log \ell
  .
\end{equation}
\nomPar{$\zeta^{(\ell)}(N)$}{torus residual free energy, a normalized version of $\log Z_N^{(\ell)}$}
\nomPar{$\xi^{(\ell)}(N)$}{a normalized version of $\log Z_N^{(\ell)}$}
In this notation,
Theorem~\ref{thm:Z-torus} asserts that
$\zeta^\ell(N) = \zeta + \OO(N^{1-\kappa})$ whenever $\ell \geq N^{-1/2+\sigma}$.

Along this section and in Section~\ref{sec:quasifree},
we will repeatedly use the Jensen inequality in the form 
\begin{equation}\label{jensen} 
\log \int \ee^{-B}+ \E^B(B-A) \le  \log \int \ee^{-A} \le \log \int \ee^{-B}+ \E^A(B-A)  ,  
\end{equation}
where $\E^A X = \frac {\int \ee^{-A} X}{\int \ee^{-A}}$ and integration is with respect
\nomMea[00]{$\E^A$}{expectation for the Gibbs measure associated to a Hamiltonian $A$}
to a fixed measure.

\subsection{Continuity of the residual free energy} 
In the following Lemmas~\ref{on} and \ref{xic}, it is
proved that $\zeta^{(\ell)}(N)$ is almost independent 
of the range $\ell$ provided that $\ell\gg N^{-1/2}$, and that $\zeta^{(\ell)}(N)$ depends only weakly on the number of particles $N$.

\begin{lemma}\label{on}
  For any $\sigma \in (0,\frac12)$  and $\nu$ and $\omega$ such that $N^{-1/2+\sigma} \leq \nu \leq \omega  \ll 1$,
  the following inequality holds:
  \begin{equation} \label{zetas}
    \OO(N^{-2\sigma +\epsilon})\leq \zeta^{(\omega)}(N) - \zeta^{(\nu)}(N)\leq
    \Oinfty,
  \end{equation}
  where the notation $\Oinfty$ was defined at the end of Section~\ref{sec:intro}.
\end{lemma}

\begin{proof} 
We start with the upper bound on $\zeta^{(\omega)}-\zeta^{(\nu)}$.
By Jensen's inequality,
\begin{equation} \label{e:HRellJensen}
  \frac{1}{\beta} \log \int \ee^{- \beta H_{N}^{\omega} (\b z)}  \, m(\rd\b z)
  \le
  \frac{1}{\beta}  \log \int \ee^{- \beta  H_{N}^\nu(\b z) }  \, m(\rd\b z)
  - \E^{H^\omega_N} [H^\omega_N - H^\nu_N]
  .
\end{equation}
Let $L_\omega^\nu(z) = U^\omega(z) - U^\nu(z)$.
Then $L_\omega^\nu(0)= \log(\omega/\nu)  + \Oinfty$ 
since $U^\ell(0) = Y^\ell(0) + O(e^{-c/\ell})$ and $\nu \leq \omega \ll 1$.
Since $L_\omega^\nu$ is positive definite,
as can be verified by representing it in Fourier space, we also have 
\nomHam[12]{${\b L}_\omega^\nu$}{${\b L}_\omega^\nu
  = \int L^\nu_\omega (z-w) \, \tilde\mu(\rd w) \, \tilde\mu(\rd z)$}%
\begin{equation}
  {\b L}_\omega^\nu
  = \int L^\nu_\omega (z-w) \, \tilde\mu(\rd w) \, \tilde\mu(\rd z)
  \geq 0,   \quad {\rm for } \; \nu \leq \omega \ll 1
  .
\end{equation}
Together with $\int U^\ell(z) \, m(\rd z) = 2\pi \ell^2$ by \eqref{e:Iell-torus}, we have the estimate %
\nomInt[04]{$L_\omega^\nu$}{difference between Yukawa interactions $U^{\ell}$ with ranges $\ell=\omega$ and $\ell=\nu$}%
\begin{equation} \label{e:LOmega-torus}
  H^\omega_N-  H^\nu_N
  = \sum_{j \neq k} L^\nu_\omega (z_j-z_k) 
 = 2\pi(\omega^2-\nu^2)N^2 - N \log(\omega/\nu)  + N^2 {\b L}_{\omega}^\nu + \Oinfty.
\end{equation}

By the definition \eqref{zetadef},
this proves  that $\zeta^{(\omega)}(N) - \zeta^{(\nu)}(N) \leq \Oinfty$.

For the lower bound, we use the Jensen inequality and \eqref{e:LOmega-torus} to obtain
\begin{multline}  \label{e:on-lb-Jensen}
  \frac{1}{\beta} \log \int \ee^{-\beta H_N^\omega(\b z)} \, m(\rd\b z)
  \\\ge 
  \frac{1}{\beta} \log \int  \ee^{-\beta H_N^{\nu} (\b z) } \, \, m(\rd \b z)
  -
  2\pi(\omega^2-\nu^2)N^2
  +
  N\log(\omega/\nu) 
  -
  N^2 \E^{H_N^{\nu}} \b L^{\nu}_{\omega} + \Oinfty.
\end{multline}
We apply the two-point rigidity estimate \eqref{e:Grigi} with $g(z,w)=L_{\omega}^{\nu}(z,w)$, $\ell=\nu$, $t=\nu N^{-\epsilon}$ and $p=2/\epsilon$. Note that this choice of $g$ satisfies
\begin{gather*}
  t^j \|\nabla^j g\|_{1}
  \le C_j t^j \nu^{-j} \omega^2 \le C_j \omega^2,
  \\
  t^p \|g_{{\rm B}_t}^{(p)}\|_{1}
  \le C_p t^p \nu^{-p} \omega^2
  \le C_p N^{-p\epsilon} \omega^2
  \leq C_p N^{-2}.
\end{gather*}
Therefore  \eqref{e:Grigi}  gives   
\begin{equation}
  N^2 \E^{H_N^{\nu}} \b L^\nu_{\omega}
  \leq
  N^{4\epsilon} \OO \p{\omega^2\nu^{-4} }.
\end{equation}

Replacing $\epsilon$ by $\epsilon/4$,
we have thus proved that
\begin{equation} \label{e:zetas-crude}
  \zeta^{(\omega)}(N) - \zeta^{(\nu)}(N) \geq N^\epsilon \OO(\omega^2\nu^{-4}).
\end{equation}
This estimate can be improved to give the lower bound stated in \eqref{zetas} as follows.
For any fixed $\epsilon>0$ small, we 
choose $(\nu_i)_{i=1}^k$ such that
$\nu_1 = \omega$, $\nu_k = \nu$ and $1 \leq \nu_{j}/\nu_{j+1} \leq N^{\varepsilon}$.
Since $\omega \leq 1$ by assumption,
there exists an admissible choice of $k$ depending on $\epsilon$ but not on $N$. Then
\eqref{e:zetas-crude} with $(\omega,\nu)$ replaced by $(\nu_j,\nu_{j+1})$ gives
\begin{equation} 
  \zeta^{(\omega)}(N) - \zeta^{(\nu)}(N)
  =
  \sum_{j=1}^{k-1} (  \zeta^{(\nu_j)}(N) - \zeta^{(\nu_{j+1})}(N))
= \sum_{j=1}^{k-1} N^\epsilon \OO(\nu_j^2\nu_{j+1}^{-4}) = N^\epsilon \OO(\nu^{-2}).
\end{equation}

Since $\nu^{-2} \leq N^{1-2\sigma}$ by assumption,
this completes the proof of the lemma. 
\end{proof}

We record the following rough bound on the partition function.
  
\begin{lemma} \label{lem:xinlogn}
  The torus residual free energy satisfies
  \begin{equation} \label{e:xinlogn}
    \xi^{(\gamma)}(n) = \OO(n \log n).
  \end{equation}
\end{lemma}

\begin{proof}
This bound follows by smearing out the point charges into charge densities and using the positive definiteness for the upper bound and  Jensen's inequality for the lower bound.
This is  a standard argument and therefore we omit the details.
The interested reader can look into \cite[Proposition~4.1]{MR3694026} or \eqref{freeY}.
\end{proof}

Using the above  bound on the partition function,
we obtain the following estimate for its dependence on $n$.

\begin{lemma}\label{xic}
The torus residual free energy satisfies 
\begin{equation}\label{xi0}
\zeta^{(\gamma)}(n) - \zeta^{(\gamma)} (m)
= \OO \pa{  |m-n|\frac{\log(n+m)}{n+m} } .
\end{equation}
\end{lemma} 

\begin{proof} 
We  now prove the following more precise version of \eqref{e:xinlogn}:
\begin{equation}  \label{e:xi0pf}
 \xi^{(\gamma)} (n)  +  2\pi\gamma^2 - \log \gamma \le  \xi^{(\gamma)}  (n+1) 
 \le  \xi^{(\gamma)}  (n)
+ \OO(\log n).
\end{equation}
In particular, $\xi^{(\gamma)}(n+1)-\xi^{(\gamma)}(n) = \OO(\log n)$.
This implies the claim as follows.
If $ n \le m \le 2 n$, then
\begin{equation}
  \zeta^{(\gamma)}(m) - \zeta^{(\gamma)} (n)
  = \sum_{k=n}^m \OO\pa{\frac{\log k}{k}}
  = \OO \pa{  |m-n|\frac{\log m}{n} }
  = \OO \pa{  |m-n|\frac{\log(n+m)}{n+m} } .
\end{equation}
On the other hand, if $m \ge 2 n$, then already \eqref{e:xinlogn} implies
\begin{equation}
  \zeta^{(\gamma)}(m) - \zeta^{(\gamma)} (n)
  = \OO(\log n) + \OO(\log m)
  = \OO \pa{  |m-n|\frac{\log(n+m)}{n+m} }.
\end{equation}
This proves the claim for $n \leq m$.
The case $n>m$ follows by exchanging the roles of $n$ and $m$.

It remains to prove \eqref{e:xi0pf}. We start with the lower bound.
By Jensen's inequality, 
\begin{align*} 
\log \frac {  \int e^{ - \beta \sum_{i \neq j;  i, j = 1 }^{n+1}   U^\gamma  (z_i-z_j)   }\, m^n(\rd \b z) } 
{ \int \ee^{- \beta \sum_{i \neq j; i, j = 1}^{n}  U^\gamma  (z_i-z_j) }\, m^n(\rd \b z) }
\ge  - 2 \beta \E_n^\gamma   \sum_{j=1 }^{n}  U^\gamma  (z_{n+1}-z_j), 
\end{align*} 
where $m^n(\rd \b z) = \prod_{j=1}^n m(\rd  z_j)$. In the following, we omit the superscript whenever it is obvious.
Integrating both sides over $z_{n+1}$, and again using Jensen's inequality, we get 
\begin{equation*}
\log Z_{n+1} \ge  \int m(\rd z_{n+1}) \log   \int \ee^{- \beta \sum_{i \neq j;  i, j = 1 }^{n+1}   
U^\gamma  (z_i-z_j)   }m(\rd \b z)  \ge \log Z_n  - (2 n \beta) (2\pi \gamma^2).
\end{equation*}
By the definition of $\xi^{(\gamma)}(n)$, it follows that
\begin{align*}
 \xi^{(\gamma)}  (n+1)
  &=  2\pi\gamma^2 (n+1)^2 - (n+1) \log \gamma + \frac{1}{\beta} \log   Z_{n+1} (\beta) \\
  &  \ge 2\pi\gamma^2 (n+1)^2 - 2 n (2\pi\gamma^2) + \frac{1}{\beta} \log Z_n(\beta  )  - (n+1)\log \gamma 
  = \xi^{(\gamma)}  (n) + 2\pi\gamma^2 - \log \gamma 
    .
\end{align*} 
This gives the lower bound in \eqref{e:xi0pf}.

For the upper bound, we set $\hat H_k = \sum_{i \not = j, i, j \not = k }^{n+1}  U^\gamma  (z_i-z_j)$.
Then, by  H\"older's inequality, it follows that
\begin{equation}\label{ZZ}
Z_{n+1} (\beta)
  = \int \exp \Big [ -\frac { \beta} { n-1} \sum_{k= 1  }^{n+1}  \hat H_k   \Big ]  m (\rd \b z)
  \le \int \ee^{-\beta  \frac { n+1} { n-1} \hat H_k }  m  (\rd \b z) = Z_n(\beta  \frac { n+1} { n-1}) 
    .
\end{equation} 
Since $\xi^{(\gamma)}(n) = \OO(n \log n)$, we have
\begin{equation*}
\frac{1}{\beta} \log  \int \ee^{ -\beta H_n} \, m(\rd \b z)
= - 2\pi\gamma^2 n^2 + \OO(n \log n), \quad  H_n
= \sum_{i \neq j }^{n} U^\gamma  (z_i-z_j).
\end{equation*}
 Using this estimate and   the convexity of the function $t \to \log \int \ee^{- t H_n} \, m(\rd \b z)$, we have 
$$
- \E_n^{\gamma, \beta}  H_n \le   \log  \int \ee^{ - (\beta+ 1)  H_n   } \, m(\rd \b z) - \log  \int \ee^{ - \beta  H_n   } \, m(\rd \b z)
\le  - 2\pi\gamma^2 n^2 + \OO(n \log n) .
$$
Using  \eqref{ZZ} and integrating the relation  $\partial_\beta \log Z_n(\beta  )  = - \E_n^{\gamma, \beta}  H_n$,   we therefore get 
\begin{align*} 
 \log   Z_{n+1} (\beta)
  \le \log Z_n\left(\beta  \frac { n+1} { n-1}\right) 
  &= \log Z_n(\beta  ) 
    - \int_{\beta}^{\beta  \frac { n+1} { n-1}}   \E_n^{\gamma, s}  H_n \, \rd s
    \\
  &\le  \log Z_n(\beta  ) 
    - 2\pi\gamma^2 \frac { 2n^2\beta} { n-1}  
    + \OO(\log n).
\end{align*} 
In summary, we have proved that 
\begin{align*}
  \xi^{(\gamma)}  (n+1) 
  &=  2\pi\gamma^2(n+1)^2  - (n+1) \log \gamma
    + \frac{1}{\beta} \log   Z_{n+1} (\beta)  
  \\ &  \le  2\pi\gamma^2 (n+1)^2  +  \frac{1}{\beta} \log Z_n(\beta  ) 
       -  2\pi\gamma^2 \frac {2n^2} { n-1}   - n \log \gamma - \log \gamma+ \OO(\log n)  
  \\ & = \xi^{(\gamma)}  (n) 
+ \OO(\log n),
\end{align*} 
which is the upper bound in \eqref{e:xi0pf}.
\end{proof}

\subsection{Scaling relation}

In the remainder of this section, we will often consider the Yukawa gas on 
a torus of side length $b$.
Let $\T^{(b)}$ denote the torus of side length $b$, i.e., the square $[b/2,b/2)^2$
  with horizontal respectively vertical sides identified.
The Yukawa potential on $\T^{(b)}$,   under this identification to the square $[b/2,b/2)^2$,   is defined by
\nomOth[21]{$\T^{(b)}$}{torus of side length $b$}
\nomOth[03]{$b$}{torus side length}

\begin{equation} \label{e:Ubell}
  U_b^\ell(z) = U_b^{(\gamma)}(z) = \sum_{n \in (b\Z)^2} Y^\ell(z+n),
\end{equation}
where $Y^\ell$ is the full plane Yukawa potential defined in \eqref{Yg}
\nomInt[07]{$U^{\ell}_b$}{periodic Yukawa interaction on $\mathbb{T}^{(b)}$ with range $\ell$}%
and $\gamma = \ell/b$ denotes the relative interaction range from now on. 
We denote the corresponding partition function of the $n$-particle Yukawa gas by $Z_{b,n}^{(\gamma)}$, and set
\begin{equation} \label{e:xib}
  \xi^{(\gamma)}_b(n)
  = \frac{1}{\beta} \log Z^{(\gamma)}_{b,n}  +  2\pi \gamma^2 n^2 
  - n \log \ell,
  \qquad
  Z_{b,n}^{(\gamma)} = 
  \int_{(\T^{(b)})^{n}} \ee^{-\beta \sum_{i \neq j} U_b^\ell  (w_i- w_j)   } \, m(\rd \b w).
\end{equation}
\nomOth[05]{$\gamma$}{relative interaction range, $\gamma=\ell/b$}%
\nomPar{$ \xi^{(\gamma)}_b(n)$}{a normalized version of $\log Z_{b,n}^{(\gamma)}$}%
\nomPar{$Z_{b,n}^{(\gamma)}$}{associated to Hamiltonian with interaction $U^\ell_b$, at inverse temperature $\beta$ ($\gamma=\ell/b$)}%

From now on, we adopt the following convention for $z-w$ in $\T^{(b)}$ including the case $b = 1$.

\begin{definition}\label{minus}
For $z, w \in \T^{(b)}$, we always choose the representative for $z-w$ (which is only defined modulo $(b\Z)^2$)
to be in $[-b/2,b/2)^2$.
\end{definition}

For later use, we record the following scaling relation.

\begin{lemma} \label{lem:scaling}
For any $K>0$,
\begin{align} 
  \xi_{Kb}^{(\gamma)}  (n)   = \Big ( \frac 1 \beta  -   \frac 1 2 \Big )  n \log  K^2 +  \xi_b^{(\gamma)} (n).
 \label{xi6}
\end{align}
In particular, by choosing $K = b^{-1}$,
with the definition of $\zeta$ from \eqref{zetadef},
\begin{equation}\label{xizeta}
 \xi^{(\gamma)}_{b} (n)  
=  n \zeta^{(\gamma)} (n )  +  \frac {n} 2   \log  n
+ n \Big ( \frac 1 2 - \frac 1 \beta  \Big )    \log  b^{-2}.
\end{equation}
\end{lemma}

\begin{proof}
By definition of the Yukawa potential \eqref{e:Ubell},
we have $U^{K\ell}_{Kb} (Kr) = U^{\ell}_b (r)$. 
Therefore, by changing variables to $z= w K$, 
\begin{multline*} 
  \frac{1}{\beta} Z_{n,Kb}^{(\gamma)}
  =
  \frac{1}{\beta} \log \int_{(\T^{(Kb)})^n} \ee^{-\beta \sum_{i \neq  j}^n U^{K\ell}_{Kb}  (z_i- z_j)  } \, m(\rd\b z)
  \\
  = 
  \frac{1}{\beta}   \log \int_{(\T^{(b)})^n} \ee^{-\beta \sum_{i \neq j} U^\ell_b  (w_i- w_j)   } \,  m(\rd\b w)
  +
  \frac{1}{\beta}  n \log K^2
  =
  \frac{1}{\beta}   \log Z_{n,b}^{(\gamma)}
  +
  \frac{1}{\beta}  n \log K^2
  , 
\end{multline*}
where the term with $\log K^2 $ comes from the scaling factor in the Jacobian. 
With $\gamma = \ell/b$ and using the definition \eqref{e:xib} of $\xi$, we have the  rescaling identity 
\begin{equation*}
   \xi_{Kb}^{(\gamma)}  (n) 
  =  2\pi \gamma^2 n^2 
    - n \log {K \ell} + \frac{1}{\beta}  n \log K^2 
    + \frac{1}{\beta}  \log Z_{n,b}^{(\gamma)} 
  = \Big ( \frac 1 \beta  -   \frac 1 2 \Big )  n \log  K^2 +  \xi_b^{(\gamma)} (n)
\end{equation*}
as claimed.
\end{proof}

\subsection{Quasi-free approximation}
\label{sec:qf-def-torus}

To prove Theorem~\ref{thm:Z-torus}
we  first replace the interaction range $\ell$ by
$N^{-1/2+\sigma}$ for an arbitrary fixed $\sigma>0$.
 By Lemma~\ref{on}, this replacement contributes an error $N^\epsilon O(N^{1-2\sigma})$
  to \eqref{e:Z-torus}. From now on, we therefore assume that $\ell=N^{-1/2+\sigma}$.

In the following, we always parametrize the unit torus $\T$ by the square $[-1/2,1/2)^2$.
For a parameter $b \ll 1$ such that $1/b$ and $Nb^2$ are both integers,
we then divide the unit torus into a grid of (small) squares $\alpha$ of side length $b$.
To be concrete, we center the grid such that the small square containing $0 \in [1/2,1/2)^2$ has $0$ as its center.
We denote the set of these squares by $\b S_0$.
\nomOth[01]{$\alpha$}{index of the squares}

For $\ell \ll b \ll 1$,
the \emph{quasi-free Yukawa interaction} is obtained from the Yukawa interaction by,
roughly speaking,
removing the interaction  between particles  in 
a small square with particles outside that square
and replacing the interaction between particles in the same
square by a periodic one.
More precisely, we denote by $\b n = (n_\alpha)$ a \emph{particle profile}, \nomOth[09]{$\b n = (n_\alpha)$}{particle profile, assignment of number of particles to squares $\alpha$}
by which we mean an assignment of a number of particles $n_\alpha \in \N$
to each square $\alpha$,
with the constraint $\sum_\alpha n_\alpha = N$ where sums over $\alpha$
are always over $\alpha \in \b S_0$. 
We associate a torus  $\T_\alpha$ of side length $b$ to each square $\alpha$.
The tori $\T_\alpha$ are of course all identical and equal to $\T^{(b)}$,
but we keep the index $\alpha$ to emphasize the connection with the square it is associated to,
and label elements in $\T_\alpha$ by  $(\alpha, z)$ with $z \in \T^{(b)}$.
For $v=(\alpha,z) \in \T_\alpha$ we write $U_b^\ell(v) = U_b^\ell(z)$
where $U_b^\ell(v)$ is the periodic Yukawa interaction on $\T^{(b)}$ defined in \eqref{e:Ubell}.
\nomOth[20]{$\T_\alpha$}{torus of side length $b$ associated to the square $\alpha$}

For $n_\alpha \in \N$,
we define 
\begin{align}\label{Hhat}
  \hat H_\alpha(\b v)
  = \sum_{i \neq j} U^\ell_b (v_i-v_j)
  \qquad  (\b v \in \T_\alpha^{n_\alpha}).
\end{align}
\nomInt[06]{$U_\alpha^\ell$}{periodic Yukawa interaction on the torus $\mathbb{T}_\alpha$ with range $\ell$}%
\nomHam[06]{$\hat H_\alpha$}{Hamiltonian for the interaction $U_\alpha^\ell$}%
Given a particle profile $\b n$,
the \emph{quasi-free free energy} with particle profile $\b n$ is defined by
\nomPar{$F(\b n)$}{quasi-free free energy for particle profile $\b n$}
\begin{equation} \label{e:Fdef-torus}
  F(\b n)
  =
  \frac{1}{\beta} \log \binom{N}{\b n}
  +
  \frac{1}{\beta} \sum_{\alpha} \log \int_{\T_\alpha^{n_\alpha}} \ee^{-\beta \hat H_{\alpha}(\b u) } \, m(\rd\b u) 
\end{equation}
where the term $
  {N \choose \b n} = \frac{N!}{\prod_\alpha n_\alpha!} $
arises as the number of ways to distribute $N$ particles into groups of sizes
$(n_\alpha)$ with $\sum_\alpha n_\alpha = N$.

The name quasi-free is motivated by the fact that particles in different squares do not interact, i.e., their contribution is additive.
\emph{The  following two propositions show that its free energy is a good approximation to that
of the original Yukawa gas.}
\nomOth[10]{$\b {\bar n} = (\bar n_\alpha)$}{expected particle profile}
To state the second proposition, denote by
$\b {\bar n} = (\bar n_\alpha)$ with $\bar n_\alpha = \bar n = Nb^2$
the expected number of particles in the square $\alpha$.

\begin{proposition}[Upper Bound]\label{prop:qf-ub-torus}
Assume that $\ell \ll b \ll 1$. Then 
\begin{equation} \label{e:qf-ub-torus}
  \frac{1}{\beta} \log \int \ee^{- \beta H^\ell(\b z)} \, m(\rd\b z)
  \leq
  \frac{1}{\beta} \log \sum_{\b n} \ee^{\beta F(\b n)}
  + N^\epsilon \OO(N^2 \ell^3 b^{-1})
  .
\end{equation} 
\end{proposition}

\begin{proposition}[Lower Bound] \label{prop:qf-lb-torus}
{ Assume that $\ell \ll b \ll 1$. Then}
\begin{equation}  \label{e:qf-lb-torus}
  \frac{1}{\beta} \log \int \ee^{- \beta H^\ell(\b z)} \, m(\rd\b z)
  \ge
  F(\bar {\b n})
  + N^\epsilon \OO(N^2 \ell^3 b^{-1})
  .
\end{equation} 
\end{proposition}

We will prove the two propositions in the next two sections.

\subsection{Upper bound: proof of Proposition~\ref{prop:qf-ub-torus}}

By translation invariance, instead of working with the grid of squares $\b S_0$ centered at $0$,
we can equivalently consider the shifted grid 
consisting of square $u+\alpha$ with $\alpha \in \b S_0$ 
and center $u \in [-b/2,b/2)^2$. The center of $u+\alpha$ is $c(u+\alpha) = u+c(\alpha)$.
Given this choice of origin $u$ and a square $\alpha \in \b S_0$, we define
\begin{equation} \label{Phidef}
  \Phi_\alpha \equiv \Phi_\alpha^u : u+\alpha \to \T^{(b)}
  \quad \mbox {the natural embedding from the square $u+\alpha$ into the torus $\T^{(b)}$, }
\end{equation}
mapping the boundary of $u+\alpha$ to a vertical and a horizontal line in $\T^{(b)}$.
More precisely, using the coordinates $c(u+\alpha)+[-b/2,b/2)^2$ on the square $u+\alpha \subset \T$
and the coordinates $[-b/2,b/2)^2$ on the torus $\T^{(b)}$, we set
\begin{equation} \label{Phidef2}
  \Phi_\alpha^u(z) = z-c(u+\alpha).
\end{equation}
For $z,w$ in the original unit torus $\T$, we define the quasi-free pair interaction
through the embeddings $\Phi_\alpha$ by
\nomInt[10]{$\tilde Y_u^\ell$}{sum of periodic Yukawa interaction on tori $\mathbb{T}_\alpha$, with origin $u$}%
\nomOth[12]{$\Phi_\alpha^u$}{embedding of the square alpha, shifted by $u$, in $\mathbb{T}^{(b)}$}%
\begin{equation} \label{e:tildeYellu-torus}
\tilde Y_u^\ell(z,w)
= \sum_{\alpha  \in \b S_0} { U^{\ell}_b}(\Phi_\alpha^u(z)-\Phi_\alpha^u(w)) \mathds{1}_{z\in u+\alpha} \mathds{1}_{w\in u+\alpha} \qquad (z,w \in \T).
\end{equation}
The interaction $\tilde Y_u^\ell$ is in fact very  simple:
we divide the unit torus into a grid of cubes of side length $b$ with the grid centered at $u$. 
Then for two particles in the same small square $\alpha$, we view them as two points on the torus $\T^{(b)}$
interacting via the torus Yukawa potential $U^{\ell}_b$.
For two particles in different small squares, the interaction vanishes.

\nomHam[07]{$\tilde H^\ell_u$}{Hamiltonian associated to interaction $\tilde Y_u^\ell$}%
The corresponding Hamiltonian $\tilde H^\ell_u$ with pair interaction $\tilde Y^\ell_u$ is
\begin{equation} \label{e:qf-H-torus}
\tilde H^\ell_u(\b z)
= \sum_{i \neq j} \tilde Y_u^\ell(z_i,z_j) \qquad (\b z\in \T^N).
\end{equation} 
The choice of origin $u \in [-b/2,b/2)^2$ was arbitrary and we will eventually
average of this choice. We set $\E^uf(u) = \frac{1}{b^2} \int_{[-b/2,b/2)^2} \rd u f(u)$
and define the function $\bar Y$ by
\begin{equation}  \label{e:gdef-torus}
  \bar Y(z,w) =
  \E^u
  \tilde Y^\ell_u(z,w) \qquad (z,w\in \T).
\end{equation}
\nomInt[08]{$\bar Y$}{interaction $Y_u^\ell$ averaged over $u$}%
By Jensen's inequality and then averaging over $u$,
\begin{equation} \label{e:qf-ub-torus-pf-jensen}
  \frac{1}{\beta}   \log \int  \ee^{- \beta  H^\ell(\b z)}   \, m(\rd \b z)    
  \le
  \frac{1}{\beta}  \E^u \log   \int  \ee^{- \beta   \tilde  H^{\ell}_u(\b z) }   \, m(\rd \b z)
  +   \E^u \E^{H^\ell} (\tilde H^{\ell}_u- H^\ell).
\end{equation}  
The second term on the right-hand side is
$\E^{H^\ell}  \sum_{i \neq j} [\bar Y(z_i,z_j) - U^\ell(z_i - z_j)]$.
This expression is estimated in Lemma~\ref{lem:gYub-torus} below.
Its proof follows by counting particles using the local density estimate.

In preparation for the proof, we first state some  estimates on $\bar Y$.
These estimates are stated in terms of the function
$g: \C \to \R$ defined with $z = x+\ii y$ by
\begin{equation} \label{e:by}
  g(z) = \frac {(b-|x|)_+(b-|y|)_+}{b^2}  Y^\ell\pa{\sqrt{ |x|_b^2 + |y|_b^2}},
  \qquad
  |x|_b = |x| \wedge (b-|x|)_{+},
\end{equation}
where we write $Y(r)$ for $Y(z)$ with $|z|=r$.
By definition, $g$ is supported in $[-b,b]^2$. Assuming an identification of $\T$ with $[-1/2, 1/2)^2$, $g$ can be extended to a function 
on $\T$. Thus  
 $g(z-w)$ is well-defined as  a function on $\T\times \T$.

\begin{lemma} \label{lem:gb-torus}
Assume that $\ell \ll b$.
Then
\begin{equation} \label{e:gb-torus1}
  \bar Y(z,w) = g(z-w) + \OO(\ee^{-cb/\ell}) \qquad (z,w\in\T)
\end{equation}
and 
\begin{equation} \label{e:gb-torus2}
  U^\ell(z-w) = g(z-w) + \OO(\ell/b) \qquad (z,w\in \T).
\end{equation}
\end{lemma}

\begin{proof}
We first verify \eqref{e:gb-torus1}, i.e., we evaluate $\bar Y(z,w)$.
If $|z-w|_\infty > b$  then the points $z$ and $w$ are necessarily in two different squares and $\tilde Y_u^\ell(z,w)=0$
by \eqref{e:tildeYellu-torus}. Thus we can assume that $|z-w|_\infty \leq b$.
We write $z-w = x+\ii y$ with $x,y \in [-b,b]$ 

For fixed such $z,w$, the probability under the average over $u$ that $z$ and $w$ lie in the same square
is given by $(b-|x|)/b \times (b-|y|)/b$.
For $z,w$ in the same square, we have $U^\ell_b(\Phi(z)-\Phi(w))=Y^\ell(\sqrt{x^2+y^2})+\OO(\ee^{-cb/\ell})$
where the error is from the sum over the periods in the definition of $U^\ell_b$.
We have thus proved \eqref{e:gb-torus1}, i.e.,
\begin{align*}
  \bar Y(z,w)
  = \E^u \tilde Y_{u}^\ell(z,w)
     = \frac{(b-|x|)(b-|y|)}{b^2} Y^\ell \pb{\sqrt{x^2+y^2}}
    + \OO(\ee^{-c/\gamma})
    .
\end{align*}

To verify \eqref{e:gb-torus2}, first assume that $|x| \vee |y| \leq b$. Then,
by the definition \eqref{e:by} and using the exponential decay $Y^\ell(z-w) = O(\ee^{-c|z-w|/\ell})$, 
\begin{equation}
  g(z-w)-Y^\ell(z-w)
  = O(|z-w|/b) Y^\ell(z-w) = O(\ell/b)
  .
\end{equation}
On the other hand, if $|x|\vee |y| \geq b$ then $g(x,y) =0$ and the claim follows from
\begin{equation}
  |g(z-w)-Y^\ell(z-w)| = Y^\ell(z-w) = O(\ee^{-c|z-w|/\ell}) \leq O(\ee^{-cb/\ell}) \leq O(\ell/b),
\end{equation}
using the assumption $\ell \ll b$. This completes the proof.
\end{proof}

\begin{lemma} \label{lem:gYub-torus}
Assume that $\ell \ll b \ll 1$. Then
\begin{equation}
  \label{e:gYub1-torus}
  \E^{H^\ell} \sum_{i, j} 
  [\bar Y(z_i-z_j) - U^\ell(z_i - z_j)] 
  =
  N^\epsilon \OO(N^2 \ell^3b^{-1})
  .
\end{equation}
\end{lemma}

\begin{proof}
We use the local density for the Yukawa gas, Theorem~\ref{thm:YTdensity},
implying that any square in $\T$ of diameter $r \gg N^{-1/2}$ 
contains $\OO(Nr^2)$ particles with high probability.
In addition, $Y^\ell(z_i-z_j) \leq \ee^{-cN^\epsilon}$ if  $|z_i-z_j| \ge \ell N^\epsilon$.
Thus $\bar Y(z_i-z_j)+  U^\ell(z_i - z_j)  \leq \Oinfty$ 
in this case and 
the sum over the contributions of these terms in \eqref{e:gYub1-torus} is again of order $\Oinfty$.
Therefore, we can assume that $|z_i-z_j| \le \ell N^\epsilon$ for all $i, j$ from now on. 
By using \eqref{e:gb-torus1}, we can replace $\bar Y$ by $g$. 
As a consequence,
\begin{equation*} 
  \E^{H^\ell} \sum_{i, j} [g(z_i-z_j) - U^\ell(z_i - z_j)] 
  = \OO( N^\epsilon N (N \ell^2) (\ell /b)) 
\end{equation*}
since each of the at most $N$ particles $z_i$ interacts with $\OO(N^\epsilon N\ell^2)$
particles $z_j$, and the difference $U^\ell(z-w)-g(z-w)$ is of order $\ell/b$ by Lemma~\ref{lem:gb-torus}.
This proves \eqref{e:gYub1-torus}.
\end{proof}

\begin{proof}[Proof of Proposition~\ref{prop:qf-ub-torus}]
By \eqref{e:qf-ub-torus-pf-jensen} and \eqref{e:gYub1-torus},
\begin{equation} 
  \frac{1}{\beta}   \log \int  \ee^{- \beta  H^\ell(\b z)}   \, m(\rd \b z)    
  \le
  \frac{1}{\beta}  \E^u \log   \int  \ee^{- \beta   \tilde  H^{\ell}_u(\b z) }   \, m(\rd \b z)
  + O(N^\epsilon N^2 \ell^3 b^{-1}).
\end{equation}  
By the definitions \eqref{e:tildeYellu-torus}--\eqref{e:qf-H-torus},
after partitioning in the number of particles in each square,
the integral in the first term on the right-hand side factorizes over the squares
and therefore is
\begin{equation*}
  \int  \ee^{- \beta   \tilde  H^{\ell}_u(\b z) }   \, m(\rd \b z)
  =
  \sum_{\b n} \binom{N}{\b n} \prod_\alpha \int_{\T_\alpha^{n_\alpha}} \ee^{- \beta   \hat H_\alpha(\b u) }   \, m(\rd \b u)
  = \sum_{\b n} \ee^{\beta F(\b n)},
\end{equation*}
where we used the definition \eqref{e:Fdef-torus}.
This completes the proof.
\end{proof}

\subsection{Lower bound: proof of Proposition~\ref{prop:qf-lb-torus}}
\label{sec:qf-lb-torus}

To obtain a lower bound on the partition function,
 we use the coordinates $[-1/2,1/2)^2$ for $\T$ and the grid $\b S$ centered at $0$.
We can then restrict the particle numbers in all squares $\alpha$ to their mean $\bar n = \bar n_\alpha = Nb^2$.
{(Although $\bar n_\alpha$ is independent of $\alpha$,
  we often keep the index for analogy with Section~\ref{sec:quasifree}.)}
Thus we define the indicator function
\begin{equation} \label{chi1-torus}
  \hat\chi(\b z)
  =
  \prod_\alpha \b 1 \Big (n_\alpha(\b z) = \bar n_\alpha \Big )
  ,
\end{equation}
where, for a particle configuration $\b z \in \T^N$,
we define $\b n (\b z) = (n_\alpha(\b z))$ to be the particle profile
associated to the configuration $\b z$, i.e.,
$n_\alpha(\b z)$ is the number of particles $z_j \in \alpha$.
Trivially,
\begin{equation} 
  \label{e:qf-lb-torus-trivial}
  \frac{1}{\beta} \log \int \ee^{-\beta H^\ell(\b z)} \, m(\rd\b z)
  \ge \frac{1}{\beta} \log \int \ee^{-\beta H^\ell(\b z) } \, \hat\chi(\b z) \, m(\rd \b z)
  .
\end{equation}
Ordering the squares in $\b S$ arbitrarily as $\alpha_1, \alpha_2, \dots$,
we write $\tilde \chi(\b z)$ for $\hat \chi(\b z)$ multiplied by
the indicator function of the event in which the 
particles $z_1, \dots, z_{\bar n}$ are in $\alpha_1$, the particles
$z_{\bar n+1}, \dots, z_{2\bar n}$ are in $\alpha_2$, and so forth. Then
\begin{equation}  \label{e:qf-lb-torus-binom}
  \frac{1}{\beta} \log \int \ee^{-\beta H^\ell(\b z) } \, \hat\chi(\b z) \, m(\rd \b z)
  =
  \frac{1}{\beta} \log \binom{N}{\b n} +
  \frac{1}{\beta} \log \int \ee^{-\beta H^\ell(\b z) } \, \tilde\chi(\b z) \, m(\rd \b z),
\end{equation}
where $\binom{N}{\b n}$ is the combinatorial factor for dividing $N$ particles into small cubes of size $n_1, n_2, \ldots$

To estimate the integral on the right-hand side,
we choose maps
\begin{equation} \label{e:Psidef-torus}
  \Psi_\alpha : \T^{(b)} \to \alpha
\end{equation}
that embed the torus $\T^{(b)}$ injectively into the square $\alpha$.
Note that the maps $\Psi_\alpha$ go in the opposite direction of the maps $\Phi_\alpha$ in \eqref{Phidef}
used in the upper bound.
Such an embedding is necessarily discontinuous along a horizontal and a vertical line in the image.
We will choose the maps $\Psi_\alpha$ randomly by averaging over the positions of the discontinuity lines.
The center where the two discontinuity lines meet can be parametrized by a point $u \in [-b/2,b/2)^2$.
Using the coordinates $[1/2,1/2)^2$ on the unit torus and recalling that  $\b S$ is the grid of size $b$ with center $0$,
we can  parametrize the square $\alpha \in \b S$ by $c(\alpha)+[-b/2,b/2)^2$ where $c(\alpha) \in (b\Z)^2$ is the center of $\alpha$.
Using $[-b/2,b/2)^2$ as coordinates for $\T^{(b)}$,
we define $\Psi_\alpha^u$  by
\begin{equation} \label{e:Psidef-choice-torus}
  \Psi_\alpha^u(w) = c(\alpha)+[u+w], \qquad (w \in [-b/2,b/2)^2),
\end{equation}
where $[z]$ the representative of $z\in \C$ in $[-b/2,b/2)^2$ modulo $(b\Z)^2$.
Clearly, the maps $\Psi_\alpha^u$ have Jacobian $|\rd\Psi_\alpha^u| = 1$.
We write $\E^\Psi$ for the average over $u$.

\nomOth[14]{$\Psi_\alpha^u$}{a map from $\mathbb{T}_\alpha$ to $\alpha$, with discontinuity lines having origin $u$, declared to be flat}%

Given the particle profile $\b n = \b {\bar n} = (\bar n_\alpha)$ and $\hat H_\alpha$  defined in \eqref{Hhat}, set
\begin{equation}
  \hat H(\b v) = \sum_\alpha \hat H_\alpha(\b v^\alpha).
\end{equation}
Let $\omega_\alpha$ be the probability measure of the Yukawa gas on $(\T^{(b)})^{n_\alpha}$ and $\b\omega$ their product:
\begin{equation}
  \omega_\alpha(\rd \b v^\alpha) = \frac{1}{Z_\alpha} \ee^{- \beta \hat H_\alpha (\b v^\alpha)  }   m(\rd \b v^\alpha),
  \qquad
  \b \omega = \prod_\alpha \omega_\alpha.
\end{equation}
Moreover, given the maps $\Psi_\alpha$, define $\Psi$ by
\begin{equation}
\Psi : \prod_\alpha {(\T^{(b)})^{n_\alpha}} \to \T^N, \quad
\Psi (\{\b v\}) = ( \{ \Psi_\alpha (\b v^\alpha)  \} ) \in \T^N,
\end{equation}
where configurations in the image of $\Psi$ have $n_\alpha=\bar n_\alpha$ particles in the square $\alpha$ (in some fixed order that is irrelevant),
and $\Psi^* \b \omega = \prod_\alpha \Psi_\alpha^* \b \omega_\alpha$ is a measure on such configurations of $N$ particles in $\T$.

For such a configuration $\b z$,  we write $\b z^\alpha$ for the vector of particles in the square $\alpha$.
Then defining $\hat H_\Psi(\b z) = \sum_\alpha \hat H_\alpha \circ \Psi_\alpha^{-1}(\b z^\alpha)$,
we have by Jensen's inequality,
\begin{equation}
\frac{1}{\beta} \log \int \ee^{-\beta H^\ell(\b z)} \tilde\chi(\b z) \, m(\rd \b z)
\ge
\frac{1}{\beta} \log \int \ee^{-\beta \hat H_\Psi(\b z)} \tilde\chi(\b z) \, m(\rd \b z)
+ \E^{\Psi^* \b \omega} (\hat H_\Psi-H^\ell)
  .
\end{equation}
Reversing the change of variables
and averaging over the distribution of maps $\Psi$ with $|\rd\Psi|=1$,
whose expectation is denoted by $\E^\Psi$, we have 
\begin{equation} \label{e:qf-lb-torus-pf-jensen}
  \frac{1}{\beta} \log \int \ee^{-\beta H^\ell} \tilde \chi
  \ge
  \frac{1}{\beta} \log \int \ee^{-\beta \hat H_\alpha(\b v^\alpha) }   \prod_\alpha d \b v^\alpha
  +
  \E^{\Psi} \E^{\b \omega}  ( \hat H(\b v) - H^\ell(\Psi(\b v)) ).
\end{equation}
\newcommand{\hatE}{\E^{\Psi}\E^{\b\omega}}

{
It remains to estimate the second term on right-hand side of \eqref{e:qf-lb-torus-pf-jensen}.
Let $\mu_\alpha$ denote either the normalized uniform measure on the square $\alpha$ or the associated torus $\T_\alpha$
(the distinction will be clear from the context).
We write $\tilde \mu_\alpha = \hat \mu_\alpha - \mu_\alpha$
where
\begin{equation} \label{muhat-torus}
  \hat \mu_\alpha(\rd v)
  = \frac{1}{n_\alpha} \sum_{j: v_j \in \T_\alpha} \delta_{v_j}(\rd v).
\end{equation}
Note that $\int \rd \hat\mu_\alpha = 1$.
Define
\begin{equation} \label{e:Edef-torus}
  E = \sum_{\alpha} \bar n_\alpha^2 \hatE 
  \iint_{\T_\alpha\times \T_\alpha}
  (U_b^\ell(v-w) - Y^\ell(\Psi_\alpha(v)-\Psi_\alpha(w)) )
  \, \tilde \mu_\alpha(\dd v) \, \tilde \mu_{\alpha}(\dd w).
\end{equation}
The following two lemmas, which we will prove in the remainder of this section,
estimate the second term on right-hand side of \eqref{e:qf-lb-torus-pf-jensen}.

\begin{lemma} \label{lem:Hdec12bd-torus}
Assume $\ell \ll b \ll 1$.
Then
\begin{equation}
\label{e:Hdec1bd-torus}
\hatE(\hat H(\b v) - H^\ell(\Psi(\b v)))
=
E  + \Oinfty
.
\end{equation}
\end{lemma}

\begin{lemma} \label{lem:E-lb-torus}
  Assume $\ell \ll b \ll 1$. Then
  \begin{equation} \label{e:Etorus}
    E = N^\epsilon\OO(N^2\ell^3b^{-1}).
  \end{equation}
\end{lemma}

Given Lemmas~\ref{lem:Hdec12bd-torus}--\ref{lem:E-lb-torus},
the proof of Proposition~\ref{prop:qf-lb-torus}
is completed as follows.

\begin{proof}[Proof of Proposition~\ref{prop:qf-lb-torus}]
  By combining  \eqref{e:qf-lb-torus-trivial}--\eqref{e:Etorus},
  we have
  \begin{equation} \label{e:qf-lb-torus-pf-final}
    \frac{1}{\beta} \log \int \ee^{-\beta H^\ell(\b z)} \, m(\rd\b z)
    \geq
    \frac{1}{\beta} \log \binom{N}{\b n} +
    \sum_\alpha \frac{1}{\beta} \log \int \ee^{-\beta \hat H_\alpha(\b v^\alpha) } d \b v^\alpha
    +
    N^\epsilon\OO(N^2\ell^3b^{-1}).
  \end{equation}
  The first two terms on the right-hand side of \eqref{e:qf-lb-torus-pf-final}
  together equal $F(\bar {\b n})$, completing the proof.
\end{proof}

To complete the proof
of Proposition~\ref{prop:qf-lb-torus},
it still remains to prove Lemmas~\ref{lem:Hdec12bd-torus}--\ref{lem:E-lb-torus}.

\begin{proof}[Proof of Lemma~\ref{lem:Hdec12bd-torus}]
We must prove that
\begin{equation} \label{eq:2ptdiff1-torus} 
  \sum_{\alpha, \beta}
  \hatE \Big [
  \sum_{i \neq j}
  {\bf 1}_{v_i \in \T_\alpha} {\bf 1}_{v_j \in \T_\beta} \pB{
    U_b^\ell(v_i-v_j)\mathds{1}_{\alpha=\beta}
  -Y^\ell(\Psi_\alpha(v_i)-\Psi_\beta(v_j))}   \Big ]
  = E
  +\Oinfty
  .
\end{equation}
We first note that 
the contribution of the nonadjacent squares on the left-hand side   is bounded by
$\OO(\ee^{-c\ell/b}) = \Oinfty$,
by \eqref{e:Ubell} and \eqref{Ya}.
For any $\alpha, \beta$ (including $\alpha = \beta$), define 
\begin{equation}
  \bar Y_{\alpha\beta}
  = \iint_{\T_\alpha \times \T_\beta} Y^\ell(\Psi_\alpha(v)-\Psi_\beta(w)) \, \mu_\alpha(\dd v) \, \mu_\beta(\dd w)
  = \iint_{\alpha \times \beta} Y^\ell(v-w) \, \mu_\alpha(\dd v) \, \mu_\beta(\dd w)
  ,
\end{equation}
where the equality follows from $\hatE \hat\mu_\alpha(\dd v) = \mu_\alpha(\dd v)$
and $|\rd\Psi_\alpha|=|\rd\Psi_\beta|=1$.
Denoting by $\alpha \sim \beta$ that the squares $\alpha$ and $\beta$ are adjacent,
therefore 
\begin{align}
  &\sum_{\alpha, \beta} \hatE \Big [ \sum_{i \neq j}
    {\bf 1}_{v_i \in \T_\alpha} {\bf 1}_{v_j \in \T_\beta} \pB{ U_b^\ell(v_i-v_j)\mathds{1}_{\alpha=\beta}
    -Y^\ell(\Psi_\alpha(v_i)-\Psi_\beta(v_j))} \Big ] \label{eq:2ptdiff-torus}
  \nonumber\\
  &= 
    \sum_\alpha \bar n_\alpha^2 \hatE \Big [
  \iint_{\T_\alpha\times \T_\alpha} (U^\ell_b(v-w) - Y^\ell(\Psi_{\alpha}(v)-\Psi_{\alpha}(w)))  \hat \mu_\alpha(\dd v) \hat  \mu_\alpha(\dd w) \Big ]
  \nonumber\\
  &\qquad\qquad -\sum_{\alpha \sim \beta} \bar n_\alpha \bar n_\beta \bar Y_{\alpha\beta}
    +\Oinfty. 
\end{align}
Using again $|\rd\Psi_\alpha|=1$     
and $\hatE \tilde\mu_\alpha =0$,
we can rewrite the  difference of the first term in the last line and $E$ as
\begin{equation}
  \sum_\alpha \bar n_\alpha^2 \qa{
    \iint_{\T_\alpha^2} U^\ell_{ b} (v-w) \, \mu_\alpha(\dd v) \, \mu_\alpha(\dd w) - \iint_{\alpha^2} Y^\ell(v-w) \, \mu_\alpha(\dd v) \, \mu_\alpha(\dd w)
    }.
\end{equation}
We then apply the cancellation \eqref{e:Ycancel-torus} below to this term
and the last term on the right-hand side of \eqref{eq:2ptdiff-torus},   
i.e., $-\sum_{\alpha \sim \beta} \bar n_\alpha \bar n_\beta \bar Y_{\alpha\beta}$.
Finally, we sum over the squares $\alpha$ of which there are $\OO(b^{-2})$ many.
Using $\OO(b^{-2})\Oinfty = \Oinfty$
and the definition of $E$, the claim follows.
\end{proof}

\begin{lemma}\label{lem:Ycancel-torus}
For any square $\alpha$ of side length $b \gg \ell$ fixed,
\begin{align}\label{e:Ycancel-torus}
\bar n_\alpha^2 \Big [ \iint_{\T_\alpha^2} { U_{b}^\ell}(v-w)  \, \mu_\alpha(\dd v) \, \mu_\alpha(\dd w) -   \iint_{\alpha^2} Y^\ell(v-w) \,  & \mu_\alpha(\dd v) \, \mu_\alpha(\dd w)\Big ]  \nonumber \\
& - \sum_{\beta: \beta \sim \alpha} \bar n_\alpha \bar n_\beta \bar Y_{\alpha \beta}
= \Oinfty.
\end{align}
\end{lemma}

\begin{proof}
Using that contributions of pairs with distance $\gg \ell /b$ are negligible in $U_b^\ell$
and unfolding the periodization in the definition of $U_b^\ell$, we have
\begin{equation*}
  \iint_{\T_\alpha^2} U^\ell_b(u-v) \, m(\dd v) \, m(\dd w)
  =
  \int_{\alpha} \int_{(\cup_{\beta \sim \alpha} \beta) \cup \alpha} Y^\ell(z-w) \, m(\dd z) \, m(\dd w)
  + \Oinfty.
\end{equation*}
Thus the first two terms on the left-hand side of \eqref{e:Ycancel-torus} are given by
\begin{equation*}
  \sum_{\beta \sim \alpha} \iint_{\alpha \times \beta} Y^\ell(z-w) \, \mu_\alpha(\rd z) \, \mu_\beta(\dd w) + \Oinfty
  ,
\end{equation*}
i.e., the left-hand side of \eqref{e:Ycancel-torus} equals
\begin{equation*}
  \bar n_\alpha \sum_{\beta \sim \alpha} \pb{ \bar n_\alpha - \bar n_\beta }
  \iint_{\alpha \times \beta} Y^\ell(z-w) \mu_\alpha(\rd z) \mu_\beta(\dd w) + \OO(\ee^{-N^\varepsilon})
  = \Oinfty
\end{equation*}
as needed.
\end{proof}

\begin{proof}[Proof of Lemma~\ref{lem:E-lb-torus}]
By the definition \eqref{e:Edef-torus}, $\bar n_\alpha = Nb^2$,
and since there are $b^{-2}$ squares $\alpha$, we must bound
\begin{equation} \label{e:Edef-torus-bis}
  E = b^{-2} (Nb^{2})^2 \hatE  \iint_{\T^{(b)}\times \T^{(b)}}
  (U_b^\ell(v-w) - Y^\ell(\Psi_\alpha(v)-\Psi_\alpha(w)) )
  \, \tilde \mu_\alpha(\dd v) \, \tilde \mu_{\alpha}(\dd w),
\end{equation}
where $\alpha$ is any of the squares in $\b S$.
Recall that the expectation $\hatE$ averages over the parameter $u \in [-b/2,b/2)^2$ in the definition of the embedding $\Psi_\alpha^u$
(see Section~\ref{sec:qf-def-torus}) and over the Yukawa gas in $\T_\alpha$.
We estimate this expectation in two steps.

\smallskip\noindent
\emph{Step 1.} 
For $v,w \in \T^{(b)}$, we write $v-w=(x,y)$ in the sense that $(x,y) \in \R^2$ is a representative for $v-w\in\T^{(b)}$
chosen such that $|x|\leq b/2$ and $|y|\leq b/2$.
We claim that
\begin{equation}  \label{e:tildeY0Y0}
  \E^\Psi
  \pa{U_b^\ell(v-w) - Y^\ell(\Psi_\alpha(v)-\Psi_\alpha(w)) }
  = \frac{b|x|+b|y|-|xy|}{b^2} U^\ell_b(v-w) + \OO(\ee^{-cb/\ell}).
\end{equation}

The proof of \eqref{e:tildeY0Y0} uses exactly 
the same reasoning as that of Lemma~\ref{lem:gb-torus}.
Namely, by the definitions of $\Psi_\alpha$ and $U_b^\ell$,
the difference $U_b^\ell(v-w) - Y^\ell(\Psi_\alpha(v)-\Psi_\alpha(w))$ is $\OO(\ee^{-cb/\ell})$
unless
$v$ and $w$ have periodic distance of order $\ell$ and $\Psi_\alpha(v)$ and $\Psi_\alpha(w)$ have Euclidean distance order~$b$
(i.e., $\Psi_\alpha(v)$ and $\Psi_\alpha(w)$ are on opposite sides of the square $\alpha$).
Assuming that the difference is not $\OO(e^{-cb/\ell})$, the $Y$ term itself is $\OO(\ee^{-cb/\ell})$, and only the $U$ term contributes.
The prefactor $(b|x|+b|y|-|xy|)/b^2$
on the right-hand side of \eqref{e:tildeY0Y0}
is the probability that $\Psi_\alpha(v)$ and $\Psi_\alpha(w)$ fall on opposite sides of the torus
under the randomness of $\E^\Psi$, i.e., when the center of the square $\alpha$ is chosen uniformly.

\smallskip\noindent
\emph{Step 2.} We claim that
\begin{equation}
  \label{e:UY2-torus}
  \E^{\b\omega} \iint \frac{b|x|+b|y|-|xy|}{b^2} U^\ell_b (v-w) \, \tilde \mu(\dd v) \, \tilde \mu(\dd w)
  =
  n^\epsilon\OO(\ell/b)^3.
\end{equation}

This estimate does not use any cancellations due to the difference in the definition of $\tilde\mu_\alpha$ as $\hat\mu_\alpha-\tilde\mu_\alpha$.
We may therefore replace $\tilde\mu$ by $\hat\mu$; the terms involving $\mu_V$ obtained
when expanding $\tilde\mu$ are analogous.
Moreover, since the left-hand side of \eqref{e:UY2-torus} does not depend on the position of the center of the square $\alpha$,
the expectation $\hatE$ can be replaced by the expectation of the Yukawa gas on the torus $\T^{(b)}$.
Furthermore, by rescaling it suffices to assume that $b=1$, i.e., that the torus $\T^{(b)}$ is the unit torus $\T$.
With $\gamma=\ell/b$ and denoting by $\E^\gamma$ the expectation of the Yukawa gas on the unit torus with $n$ particles and range $\gamma$,
it is then sufficient to to show that
\begin{equation} \label{e:UY2-torus-pf}
  n^2 \E^\gamma  \iint |v-w| \, U^\gamma (v-w) \, \hat \mu(\dd v) \, \hat \mu(\dd w)
  =
  n^\epsilon\OO(n^2\gamma^3).
\end{equation}

Note that $\OO(n^2\gamma^3)$ is the order of the left-hand side when we replace $\hat \mu$
by the uniform measure. So the proof 
of the last bound can be understood by the simple heuristic that the density of the measure   $\hat \mu$ is bounded by uniform measure 
at the scale $\gamma$ provided by the regularization of the interaction $U^\gamma$.
We now give the formal proof
by using the local density bound for $\hat\mu$ stated in Theorem~\ref{thm:YTdensity}.
More precisely, dividing the unit torus into squares of length $\tilde b=n^\epsilon\gamma$,
by the local density estimate,
each square contains $\OO(n\tilde b^2)$ particles, with high probability.
Thus, denoting the squares by $\alpha$ and $\beta$, the left-hand side of \eqref{e:UY2-torus} is
bounded by
\begin{equation}
  n^2 \E^\gamma \sum_{\alpha,\beta} \iint_{\alpha \times \beta} |v-w| \, U^\gamma(v-w)  \, \hat \mu(\dd v) \, \hat \mu(\dd w)
  .
\end{equation}
Using the exponential decay of $U^\gamma(v-w)$, up to an error of order $\OO(\ee^{-cn^\epsilon})$,
only the neighboring or equal pairs of squares $\alpha,\beta$ contribute to this sum.
For each such pair, the contribution is $\OO(n^2 \tilde b^5)$ with two factors of $\tilde b^2$ arising from
the integrals over $z$ and $w$ and one from the factor $|z-w|$.
Summing over the $\OO(\tilde b^{-2})$ terms and using that $\tilde b = n^\epsilon \gamma$,
the left-hand side of \eqref{e:UY2-torus-pf} is bounded by
$\OO(n^2 \tilde b^3) = \OO(n^{2+3\epsilon}\gamma^3)$.
Finally, replacing $3\epsilon$ by $\epsilon$, the estimate \eqref{e:UY2-torus-pf} follows.
\end{proof}

}

\subsection{Consequence of quasi-free approximation}

The main consequence of the quasi-free approximation for the torus is
Proposition~\ref{prop:qf-density-torus} below. 
In preparation, we need two elementary lemmas.
The quasi-free approximation upper bound \eqref{e:qf-ub-torus}
and the lower bound \eqref{e:qf-lb-torus} are slightly different in that the upper bound is summed over all possible particle numbers in every small tori while the lower bound 
contains only the term that the number of particle in every small tori is identically its mean. Due to the convexity of free energy, it is not difficult to show that the fluctuations of number of particles can be estimated and they will be of lower order. This is the content of the next two lemmas. Once this is achieved, the quasi-free approximation 
upper and lower bound match up to a lower order terms.
This  establishes the  additivity of 
the  free energy up to negligible errors except that the range of Yukawa interactions are different for the gas in the original torus and the smaller one. However, the scaling 
of the free energy is given in Lemma~\ref{lem:scaling} and the error due to the change of Yukawa range is easy to estimate. Thus we obtain that the existence of 
the specific free energy with effective error estimate in Proposition~\ref{prop:qf-density-torus}.

\begin{lemma} \label{lem:F1-torus}
Let
\begin{equation} \label{e:hlong}
  h_\alpha(\b n)
  =      
  2\pi \gamma^2 (n_\alpha  - \bar n)^2
    - n_\alpha \zeta^{(\gamma)} (n_\alpha)
    - \frac {1}{2} n_\alpha   \log  n_\alpha
    -  \Big ( \frac 1 2 - \frac 1 \beta  \Big )   n_\alpha \log  b^{-2}
  .
\end{equation}
Then the quasi-free free energy with  particle profile $\b n$ defined in \eqref{e:Fdef-torus} can be  written as
\begin{equation} \label{e:F1-torus}
  F(\b n)
  =
  \frac{1}{\beta} \log \binom{N}{\b n}
  +
  2\pi \ell^2 N^2 + \sum_{\alpha} h_\alpha({\bf n}) -N \log \ell.
\end{equation}
\end{lemma}

\begin{proof}
From \eqref{e:Fdef-torus} and \eqref{e:xib},
recall that
$F(\b n)
=
\frac{1}{\beta} \log \binom{N}{\b n}
-
\sum_{\alpha} T_\alpha(n_\alpha)
$,
where
\begin{equation*}
  T_\alpha(n_\alpha)
  :=- \frac{1}{\beta} \log \int_{\T_\alpha^{n_\alpha}}
    \ee^{- \beta \sum_{j\neq k} U_\alpha^\ell(z_j-z_k)}
    \, m(\rd \b z)
  = 
  2\pi \gamma^2 n_\alpha^2 - n_\alpha \log \ell - \xi_b^{(\gamma)}(n_\alpha)
  .
\end{equation*}
By the scaling relation \eqref{xizeta}, we also have
$h_\alpha (\b n)
  =
  2\pi \gamma^2 (n_\alpha  - \bar n)^2
  - \xi^{(\gamma)}_{b} (n_\alpha)    
$.
The equality
\begin{equation*}
\sum_{\alpha}
2\pi \gamma^2 n_\alpha^2 
=  2\pi \gamma^2    \sum_{\alpha} (n_\alpha  - \bar n)^2 
+ 2\pi \ell^2 N^2
\end{equation*}
therefore implies 
\begin{align*}
  \sum_{\alpha} T_\alpha(n_\alpha) 
  = 2\pi \ell^2 N^2 + \sum_{\alpha} h_\alpha({\bf n}) -N \log \ell.
\end{align*}
This completes the proof.
\end{proof}

\begin{lemma}\label{lem:F-torus}
For any functions $\cal E_\alpha : \N \to \R$
satisfying  $|\cal E_\alpha(n)-\cal E_\alpha(m)| \leq { \OO(}|n-m|(n+m)^\epsilon)$,  
with $\gamma = \ell/b \geq N^{-C}$, we have
\begin{equation} \label{e:F2-torus}
  \frac{1}{\beta}
  \log \sum_{\b n} \ee^{\beta \cal E(\b n)}
  \le
  \cal E(\b {\bar n})
  + N^{\OO(\epsilon)} \OO(\ell^{-2})
  , 
  \qquad
  \cal E( \b n)
  :=
  \sum_\alpha \Big [
  - 2\pi  \gamma^2  (n_\alpha  - \bar n_\alpha)^2
  + \cal E_\alpha(n_\alpha)
  \Big ]
  .
\end{equation}
\end{lemma}

\begin{proof}
By definition,
\begin{multline}\label{e:calF-torus}
  \frac{1}{\beta}  \log  \sum_{\b n}\ee^{\beta \cal E( \b n)
  }
  - \cal E  (\b {\bar n})
  =
  \frac{1}{\beta}  \log  \sum_{\b n}\ee^{\beta (\cal E( \b n) - \cal E(\b {\bar n}))
  }
  \\
  =
  \frac{1}{\beta}  \log  \sum_{\b n} 
  \exp \qa{
    \sum_\alpha \beta
    \qB{
      - 2\pi  \gamma^2  (n_\alpha  - \bar n_\alpha)^2
      + (\cal E_\alpha ( n_\alpha ) - \cal E_\alpha ( \bar n_\alpha ) 
    }
  }.
\end{multline}  

To get an upper bound,
we drop the constraint $\sum_\alpha n_\alpha = N$ on $\b n$,
and sum each $n_\alpha$ independently.
Using the assumption
$|\cal E_\alpha (n ) - \cal E_\alpha (m) | \le  \OO(|n-m| (n + m)^{\epsilon})$,
the elementary inequality that for any positive fixed numbers $C,c> 0$ and all integers $m \ge 0$, 
\begin{equation} \label{e:F-torus-elineq}
  \sum_{n =0 }^\infty    
  \exp \Big [  C|n-m| (n + m)^{\epsilon} - c \gamma^2     (n-m )^2    \Big ]
  \le \OO(\gamma^{-1})  (m + \gamma^{-2})^{2\epsilon}  \ee^{ \OO(\gamma^{-2}) (m + \gamma^{-2})^{2\epsilon}},
\end{equation} 
and that $\bar n_\alpha = Nb^2$,
the left-hand side of \eqref{e:calF-torus} is bounded by
\begin{equation*}
  \OO (\log N) \sum_{\alpha}   \gamma^{-2}   (\bar n_\alpha + \gamma^{-2})^{2\epsilon}
  \le N^{\OO(\epsilon)} \OO(\ell^{-2}).
\end{equation*} 
This completes the proof of the lemma. 
\end{proof}

\begin{proposition} \label{prop:qf-density-torus}
For any $\sigma>0$, there is $\tau>0$ such that
if $\ell \geq N^{-1/2+\sigma}$ and $1 \geq b \geq N^{1+\sigma}\ell^3$,
\begin{equation} \label{e:zetab}
  \zeta^{(\ell)}(N) = \zeta^{(\ell/b)}(Nb^2) + \OO(N^{-\tau}).
\end{equation}
More precisely, $\OO(N^{-\tau})$ is $N^{\varepsilon}\OO(N\ell^3/b+1/(N\ell^2))$.
\end{proposition} 

\begin{proof}
The assumptions on $\ell$ and $b$ imply that the error terms in
\eqref{e:qf-ub-torus}, \eqref{e:qf-lb-torus} are $\OO(N^{1-\tau})$.
By Propositions~\ref{prop:qf-ub-torus}, \ref{prop:qf-lb-torus}, together with Lemma~\ref{lem:F1-torus}, therefore
\begin{align} \label{pa-lb-torus}
  \frac{1}{\beta} \log \int \ee^{- \beta H^\ell(\b z)} \, m(\rd\b z)
  &\geq -  2\pi \ell^2 N^2  + N   \log \ell
    +  \frac{1}{\beta}  \log   {N \choose \bar{\b n}}  \ee^{-\beta \sum_\alpha h_\alpha (\bar {\b n})}
    - \OO( N^{1-\tau} ),
  \\    
  \label{pa-ub-torus}
  \frac{1}{\beta} \log \int \ee^{- \beta H^\ell(\b z)}  \, m(\rd\b z)
  &\leq - 2\pi \ell^2 N^2   + N   \log \ell
  +  \frac{1}{\beta}  \log  \sum_{\b n}   {N \choose \b n}
  \ee^{-\beta \sum_\alpha h_\alpha (\b n)}
  + \OO( N^{1-\tau} ).
\end{align} 
We compute the sums on the right-hand sides of  \eqref{pa-lb-torus}, \eqref{pa-ub-torus}.
By Stirling's formula,
\begin{equation} \label{e:Stirling}
 \log {N \choose \b n} = N \log N - \sum_\alpha n_\alpha \log n_\alpha +  \OO(\log N).
\end{equation}
With the definition of $h$ from \eqref{e:hlong}, and
$\cal E_\alpha(n_\alpha) = ( \frac 1 2 - \frac 1 \beta)  n_\alpha \log (n_\alpha b^{-2}) + n_\alpha \zeta^{(\gamma)} (n_\alpha)$ and $\cal E$ of \eqref{e:F2-torus},
we rewrite \eqref{pa-lb-torus}, \eqref{pa-ub-torus} as 
\begin{align*}
\frac{1}{\beta} \log \int \ee^{- \beta H^\ell(\b z)} \, m(\rd\b z)
+ 2\pi \ell^2 N^2 - N \log \ell
&\geq \cal E(\b{\bar n}) +  \frac{1}{\beta} N \log N
+\OO(N^{1-\tau}),
\\
\frac{1}{\beta} \log \int \ee^{- \beta H^\ell(\b z)} \, m(\rd\b z)
+ 2\pi \ell^2 N^2 - N \log \ell
&\leq \frac{1}{\beta} \log \sum_{\b n} \ee^{\beta \cal E(\b n)} +  \frac{1}{\beta} N \log N
+\OO(N^{1-\tau}).
\end{align*}
By Lemma~\ref{xic}, $\cal E_\alpha$ satisfies the assumption of Lemma~\ref{lem:F-torus}.
Lemma~\ref{lem:F-torus} then shows that the sum over $\b n$
can be estimated by its dominant term $\b {\bar n}$ with error
$N^\epsilon\OO(\ell^{-2}) = \OO(N^{1-\tau})$.
Since
\begin{equation*} 
  \cal E(\b{\bar n})
  =
  \pa{\frac12 - \frac1\beta}
  \sum_\alpha \bar n \log (\bar n b^{-2})
  + \sum_\alpha  \bar n \zeta^{(\gamma)}(\bar n)
  =
  \pa{\frac12 - \frac1\beta}
  N \log N
  + N\zeta^{(\gamma)}(Nb^2),
\end{equation*}
this replacement yields
\begin{equation*}
\frac{1}{\beta} \log \int \ee^{- \beta H_{V}^\ell(\b z)} \, m(\rd\b z)
+ 2\pi \ell^2 N^2
- N \log \ell
=
\frac{1}{2} N \log N
+ N\zeta^{(\gamma)}(Nb^2)
+\OO(N^{1-\tau})
,  
\end{equation*}
which completes the proof of \eqref{e:zetab}.
\end{proof}

\subsection{Existence of torus residual  free energy: proof of Theorem~\ref{thm:Z-torus}}

We now prove Theorem~\ref{thm:Z-torus}.
The main ingredient is the next lemma which combines
the previously proven estimates.

\begin{lemma} \label{lem:nnt}
For any $\sigma>0$ there exists $\tau>0$ such that
for $\nu$ with $n^{-1/2+\sigma} \leq \nu \leq n^{-1/3-\sigma}$,
\begin{equation} \label{e:nnt}
  \max_{\tilde n \in [n, 2n]} |\zeta^{(\nu)}  (n) -  \zeta^{(\nu)}   (\tilde n)|
  = \OO(n^{-\tau}).
\end{equation} 
More precisely, $\OO(n^{-\tau})$ is $n^{\varepsilon}\OO(n\nu^3+1/(\nu\sqrt{n}))$.
\end{lemma}

\begin{proof}
We will find $u_0(\sigma)>0$ such that the following statements hold.
Given $n \in \N$, let $\tilde n \in [n,2n] \cap \N$.
Choose $0 \leq u,\tilde u \leq u_0(\sigma)$
such that $B = n^u$ and $\tilde B = \tilde n^{\tilde u}$ are both integers and that $|\tilde B\sqrt{M} - B|\leq 1$ where $M = \tilde n/n \in [1,2]$.
We also set $\ell=n^{-u} \nu$ and $\tilde \ell = \tilde n^{-\tilde u} \nu$.
We claim that the following statements hold:
\begin{align}
  \label{zetaconstant1}
  &\zeta^{(\ell)} (B^2 n)
  = \zeta^{(\nu)}(n) + \OO(n^{-\tau}),
  \\
  \label{zetaconstant2}
  &\zeta^{(\tilde\ell)} (\tilde B^2 \tilde n)
  = \zeta^{(\nu)}(\tilde n) + \OO(n^{-\tau}),\\
\label{e:zetacont1}
  &\zeta^{(\tilde\ell)}(\tilde B^2 \tilde n)
  - \zeta^{(\ell)}(\tilde B^2 \tilde n)
  = \OO(n^{-\tau}),
  \\
  &\label{e:zetacont2}
  \zeta^{(\ell)} (B^2 n )-  \zeta^{(\ell)} (\tilde B^2 \tilde n )
  = \OO(n^{-\tau}).
\end{align}
By combining the estimates \eqref{zetaconstant1},
\eqref{zetaconstant2}, \eqref{e:zetacont1}, \eqref{e:zetacont2},
we obtain \eqref{e:nnt}.

To prove \eqref{zetaconstant1}, we apply Proposition~\ref{prop:qf-density-torus}
with $N=B^2n$ and $b=1/B$. For $u$ sufficiently small,  the assumptions of this lemma imply that 
the assumptions of Proposition~\ref{prop:qf-density-torus} are satisfied.  Thus 
the resulting error estimate of Proposition~\ref{prop:qf-density-torus} becomes 
$
N^{\varepsilon}\OO(N\ell^3/b+1/(N\ell^2))
= n^{2\varepsilon}\OO(n^{1-2u}\nu^3+1/(n\nu^2))
\leq n^{2\varepsilon}\OO(n^{1-2u}\nu^3+1/(n^{1/2+\sigma}\nu))
= \OO(n^{-\tau})
$.
This completes the proof of \eqref{zetaconstant1}. 
The proof of \eqref{zetaconstant2} can be done  analogously.

To prove \eqref{e:zetacont1}, we apply \eqref{zetas} with $N = \tilde B^2 \tilde n$.
This gives the needed bound
since 
$1/(N \ell^2) \leq n^{2(u-\tilde u)-1} \nu^{-2} = \OO(n^{2(u-\tilde u)-\sigma}/(\sqrt{n}\nu)) = \OO(1/(\sqrt{n}\nu))$
when $u$ and $\tilde u$ are small depending on $\sigma$.

To prove \eqref{e:zetacont2}, we apply \eqref{xi0}.
Since $|\tilde B^2 M - B^2  | \le \OO(B)$, \eqref{xi0} implies
\begin{equation} \label{e:zetacont2-bis}
  | \zeta^{(\ell)} (B^2 n )-  \zeta^{(\ell)} (\tilde B^2 \tilde n ) |
  = \OO\pa{\frac { |\tilde B^2 \tilde n - B^2 n | }  {|\tilde B^2 \tilde n + B^2 n |^{1 -\epsilon} } }
  = \OO\pa{\frac { |\tilde B^2 M - B^2 | n^\epsilon}  {|\tilde B^2 M + B^2 |^{1 -\epsilon} }}
  \le  \frac {\OO(n^\epsilon)}  { B^{1 - 2\epsilon} }
  = \OO(n^{-\tau}),
\end{equation}
where the last inequality follows
from $n^{\epsilon}\OO(1/B) = n^{\epsilon} \OO(1/(\sqrt{n}\nu))= \OO(n^{-\tau})$.
\end{proof}

\begin{proof}[Proof of Theorem~\ref{thm:Z-torus}]
Recall $\zeta$ is defined in \eqref{zetadef}. We  set a constant $c=3/8$ and 
for $j \in \N$, 
define the sequences $n_j = 2^j$, $\nu_j = 2^{-cj}$ and $\zeta_j = \zeta^{(\nu_j)}(n_j)$.
For $k \geq i$, we then write
\begin{equation} \label{e:Z-torus-pf1}
  \big | \zeta_i - \zeta_k \big |
  \leq
  \sum_{j=i}^{k-1} |\zeta^{(\nu_j)}(n_j)-\zeta^{(\nu_{j})}(n_{j+1})|
  +
  \sum_{j=i}^{k-1} |\zeta^{(\nu_j)}({ n_{j+1}})-\zeta^{(\nu_{j+1})}(n_{j+1})|
  .
\end{equation}
We estimate the first sum in \eqref{e:Z-torus-pf1} by using Lemma~\ref{lem:nnt}.
Note that since $c=3/8 \in (1/3,1/2)$,
the assumptions of Lemma~\ref{lem:nnt} are satisfied with $n=n_j$ and $\nu=\nu_j$.
Thus the first sum can be bounded by
$\sum_{j=i}^{k-1}n_j^{\epsilon} \OO(n_j\nu_j^3+1/(\nu_j\sqrt{n_j})) =\OO(2^{-(1/8-\epsilon)i})$.
For the second sum in \eqref{e:Z-torus-pf1}, we use \eqref{zetas}
and obtain the estimate $\sum_{j=i}^{k-1} \OO(n_j^{-1+\epsilon} \nu_j^{-2}) = \OO(2^{-(1/4-\epsilon)i})$.

In summary, with $\tau = 1/8-\epsilon$,
we have shown that $\zeta_k  = \zeta_i + \OO(2^{-i\tau})$ for $k\geq i$ sufficiently large.
This implies the existence of
the limit $\lim_{j\to\infty} \zeta_j = \zeta$ with the estimate $\zeta_j = \zeta + \OO(2^{-j\tau})$.

Finally,
it remains to pass from the limit along the dyadic sequence above to that for general $N$ and $\ell$ as in the
statement of the theorem.
Given $N$ large, we let $j_N$ be the smallest integer $j$ such that $2^j \geq N$.
By \eqref{zetas} respectively \eqref{e:nnt}, then
\begin{align}
  \zeta^{(\ell)}(N) &= \zeta^{(\nu_{j_N})}(N) + \OO(N^{-2\sigma+\epsilon}),
  \\
  \zeta^{(\nu_{j_N})}(N) &= \zeta_{j_N} + \OO(N^{-1/8+\epsilon}).
\end{align}
Combining these two inequalities then gives the claim with $N^{\varepsilon}\OO(N^{-1/8} +N^{-2\sigma}) = \OO(N^{-\kappa})$.
\end{proof}

\section{Proof of Theorem~\ref{freeasy}:
quasi-free approximation of free energy}
\label{sec:quasifree}

In this and the following three sections, we prove Theorem~\ref{freeasy}
and its local version Theorem~\ref{rk:quasifree-meso}.
The proofs of both theorems will be parallel and we will give detailed arguments
for the proof of Theorem~\ref{freeasy} and remark
along the way the modifications needed for the proof of its local version.

We follow the strategy of \emph{quasi-free approximation} analogously as for the torus in Section~\ref{sec:Z-torus}.
The main differences are  that the equilibrium measure can have a non-constant density
and that its support has a boundary.
In this section, we present the set-up of the quasi-free approximation,
and give the proof of Theorem~\ref{freeasy} assuming 
using Propositions~\ref{prop:qf-ub}--\ref{prop:qf-lb} which are proved
subsequently in Sections~\ref{sec:quasifreeub}--\ref{sec:quasifreelb}.
In Section~\ref{sec:quasifreeub},
we prove the upper bound on the partition function.
As in the torus case,
this upper bound can essentially be established using the 
Jensen inequality and the positive definiteness of the Coulomb potential.
In Section~\ref{sec:quasifreelb}, we prove the lower bound.
The lower bound involves estimating the Coulomb energy near
the boundary of the support of the equilibrium measure and
is the main difficulty of the proof.

\subsection{Main result}

We recall the definition of the two-dimensional Yukawa gas
with range $R$
and external potential $V$
as well as the related potential theory from Section~\ref{sec:prelim}.
In particular,
\begin{equation*} 
  H_{V}^R(\b z) = N \sum_{j} V(z_j) +  \sum_{j \neq k} Y^R (z_j-z_k)
\end{equation*}
is the corresponding Hamiltonian,
$\mu_V^R$ is the equilibrium measure,
$\rho_V^R$ denotes the density of its absolutely continuous part,
and $I_V^R$ is the minimizing energy of the variational functional.

\begin{theorem}\label{FA} 
Assume that $V$ satisfies the assumptions of Theorem~\ref{freeasy}
  or more generally those stated in Theorem~\ref{rk:quasifree-meso} below.
  Then for any $\sigma>0$, there exists a constant $\kappa> 0$ such that,
  for all  $R\geq N^{2}$,
\begin{equation} \label{pa2}
  \frac{1}{\beta N} \log \int_{\C^N} \ee^{- \beta H_{V}^R(\b z)} \, m(\rd\b z)
  = - N    I_V^R  +  \log R  + \frac 1 2  \log N 
  + \zeta   
  + \Big (\frac 1 2 -  \frac 1 \beta \Big ) \int_\C \rho_V^R \log \rho_V^R \, \rd m
  + \OO(N^{-\kappa}), \nonumber 
\end{equation} 
where $\zeta$ is the residual torus free energy of Theorem~\ref{thm:Z-torus}.
For $R\geq 1$, any $\kappa<1/24$ is admissible.
\end{theorem} 

The remainder of Section~\ref{sec:quasifree} is devoted to the proof of
Theorem~\ref{FA}, which is concluded in Section~\ref{sec:qf-pf-fa},
subject to the proofs of Propositions~\ref{prop:qf-ub}--\ref{prop:qf-lb}, which 
will be proved in Sections~\ref{sec:quasifreeub} and \ref{sec:quasifreelb}.
Theorem~\ref{freeasy} for the Coulomb gas
is then a direct consequence, by taking $R\to \infty$,
which we do in detail in Section~\ref{sec:coulombapprox}.

Throughout this section, we make the standing assumptions
$R \geq N^2$,
and that $V$ satisfies the
assumptions of Theorem~\ref{freeasy}, i.e., the asymptotic condition \eqref{e:Vgrowth} and \eqref{Vcondition1},
or more generally that the conditions of the remark below hold.
We denote the empirical measure by $\hat \mu$
and its difference with equilibrium measure by $\tilde \mu = \tilde\mu^R_V$, i.e., 
\begin{equation}\label{tildemu}
\hat \mu = N^{-1} \sum_j \delta_{z_j}, \qquad
\tilde\mu = \tilde \mu^R_V = \hat \mu  -  \mu_V^R.
\end{equation}
We denote the expectation of the Coulomb gas with density $\ee^{-\beta H_V^R}$ by $\E_V^R$. The following Theorem is a local version of Theorem~\ref{freeasy}. 
A more precise statement  with precise scaling and notation  will be given in Theorem \ref{freeasy-local}.  
We choose to present it in the following way so that it is easier to digest in the first reading.

\begin{theorem} [A local version of  Theorem~\ref{freeasy}]   \label{rk:quasifree-meso}
Consider the setting of Theorem~\ref{freeasy}. Then the good boundary conditions hold with high probability. 
Furthermore, Theorem \ref{freeasy}  hold with respect to the conditional measure with 
the error term $\OO(N^{-\kappa})$ replaced by  
\begin{equation} \label{e:quasifree-meso-bound}
  C(\Omega,A)(1+K^2) (Nr^2)^{-(\kappa \wedge a')} 
\end{equation}
for  any $a' < a$,  where $a$ is the constant in \eqref{e:quasifree-meso-cond}. 
\end{theorem}

\subsection{Short-range Yukawa approximation} 

The first step
is a decomposition of the Yukawa potential
into a short-range and a long-range part.
This is similar to our strategy in the proof of Lemma~\ref{on} for the torus.
However, due to the presence of a boundary of the support of the equilibrium measure
and lack of rigidity estimates there, we cannot prove an analog of Lemma~\ref{on}.
Therefore we subsequently cannot drop the long-range part of the interaction near the boundary.

Given $0 < \ell < R$,
we decompose the Yukawa potential as
$Y^R  = Y^\ell(z) + L^\ell_R(z)$.
The formula \eqref{Yg} shows that
the Fourier transform of $L^\ell_R$ is positive and thus that
$L^{\ell}_R$ is a positive definite function on $\C$.
The next lemma expresses the long-range contribution to the interaction
in terms of a effective potential $Q$ and an error. We set
 \nomPot{$Q = Q_R^\ell$}{effective potential for the Yukawa gaz with range $\ell$}
\begin{align} 
Q(z) = Q_R^\ell(z) &= V(z)   + 2   \int L^\ell_R (z-w) \, \mu^R_V(\rd w), \label{Q}
\\
{\b L} = {\b L}_R^\ell &= \int L^\ell_R (z-w) \, \tilde\mu^R_V(\rd w) \, \tilde  \mu^R_V(\rd z),
\\
\label{e:Kdef}
K = K_R^\ell &= \int L^\ell_R (z-w) \, \mu^R_V(\rd w) \, \mu^R_V(\rd z),
\end{align}
where $  K_R^\ell  $ is the equilibrium interaction energy of the potential difference $L_R^\ell = Y^R-Y^\ell$.
\nomHam[11]{$K_R^\ell$}{equilibrium energy of the Yukawa potential from scale $\ell$ to $R$}

\begin{lemma} \label{lem:LOmegaK}
Let $0 < \ell < R$, and let $Q$, $\b L$, and $K$ be as above. Then we have the identity 
\begin{equation} \label{e:LOmegaK}
\sum_{j \neq k} L^\ell_R (z_j-z_k) + N \sum_{j} V(z_j)  
= N \sum_j Q_{R}^{\ell}(z_j) + N^2 {\b L}_R^\ell - N \log(R/\ell) -N^2 K^\ell_R,
\end{equation}
and in particular
\begin{equation} \label{e:HellHR}
H^\ell_Q(\b z)
= H^R_V(\b z) - N^2 {\b L}_R^\ell + N \log (R/\ell) +  N^2 K_R^\ell
 .
\end{equation}
Moreover, the minimizers of the variational functionals $\cal I_Q^\ell$ and $\cal I_V^R$ coincide,
i.e., $\mu_Q^\ell = \mu_V^R$, and their energies satisfy $I_Q^\ell = I_V^R + K_R^\ell$.
The Euler--Lagrange equation for the measure $\mu_Q^\ell$ is
\begin{align}
  \label{e:qf-EL}
  \int Y^\ell(z-w) \, \mu_Q^\ell(\rd w) + \tfrac12 Q(z) = c_V & \quad \text{q.e.\ in $S_V^R$} \quad \text{and}\\
  \int Y^\ell(z-w) \, \mu_Q^\ell(\rd w) + \tfrac12 Q(z) \geq c_V & \quad \text{q.e.\ in $\C$}, \nonumber
\end{align}
with the same constant $c_V$ as in the Euler--Lagrange equation for $\mu_V^R$.
\end{lemma}

\begin{proof}
The proof of \eqref{e:LOmegaK} is a direct calculation.
Indeed, using that $L_R^\ell(0) = \log(R/\ell)$ by \eqref{Ya}, it follows that
\begin{equation*}
  \int_{z \neq w} L^\ell_R (z-w) \, \tilde   \mu^R_V(\rd w) \, \tilde  \mu^R_V(\rd z)
  =  \int L^\ell_R (z-w) \, \tilde\mu^R_V(\rd w) \, \tilde  \mu^R_V(\rd z) -  \frac{1}{N} \log(R/\ell).
\end{equation*}
The equilibrium measures (minimizers) of $\cal I_V^R$ and $\cal I_Q^\ell$
are characterized by the Euler--Lagrange equations \eqref{e:EL},
which state that in the supports of the measures, the equalities
\begin{equation*}
\frac12 V + Y^R \ast \mu_V^R = c_V^R, \quad \frac12 Q + Y^\ell \ast \mu_Q^\ell = c_Q^\ell
\end{equation*}
hold, and that equality is replaced by inequality outside the supports of the equilibrium measures.
By definition of $Q$ and the Euler--Lagrange equation for $\mu_V^R$,
the solution $\mu_Q^\ell$ satisfies \eqref{e:qf-EL}.
By the uniqueness of the minimizers, we thus conclude that $\mu_Q^\ell = \mu_V^R$ and $S_V^R= S_Q^\ell$, i.e., 
the two minimizers coincide.
Moreover, a simple computation yields that   $I_V^R= I_Q^\ell + K_R^\ell$. 
\end{proof}

In view of the above lemma, we write $\mu_V$ instead of $\mu_V^R = \mu_Q^\ell$ from now on,
we write $\rho_V$ for the density of the absolutely continuous part of $\mu_V$,
and $S_V$ for its support.
The next lemma gives an elementary estimate on $Q$ that will be useful later.

\begin{lemma}\label{lem:QzQalpha}
For $z \in S_V$ with distance $\gg \ell$  to the complement of $S_V$,
\begin{equation} \label{e:Qz}
  Q(z) = 2c_V - 4\pi \ell^2 \rho_V (z) 
  +  N^\epsilon \OO(\ell^4) \| \nabla^2 \rho_V \|_\infty + \Oinfty 
  .
\end{equation}
\end{lemma}

\begin{proof}
By Lemma~\ref{lem:LOmegaK}, for $z \in S_V$, we have 
\begin{align*}
  Q(z)
  &=  2c_V - 2 \int Y^\ell(z-w) \, \mu_V(\rd w)
  \\
  &= 2c_V - 2 \rho_V (z)  \int  Y^\ell(z-w) \,  m( \rd w)
    + N^\epsilon\OO(\ell^4) \| \nabla^2 \rho_V \|_\infty
    + \Oinfty. 
\end{align*}
In the second equality, we used that, by the exponential decay of $Y^\ell$, 
we may restrict the integral over $w$ a disk of radius $\OO(\ell N^{\epsilon})$ around $z$,
up to an error $\Oinfty$. 
Moreover, since $z$ is in the support of the absolutely continuous part of
$\mu_V$ with distance $\gg \ell$ to its complement,
we may Taylor expand the equilibrium density to second order
and use that the first-order term vanishes after integration.
The definition of the Yukawa potential \eqref{Yg} implies 
$\int Y^\ell (z-w) \, m(\rd w) = 2\pi \ell^2$.
This implies \eqref{e:Qz}.
\end{proof}

\subsection{Quasi-free approximation}
\label{sec:quasifreedef}

In this and the next subsections,
we approximate the partition function of the (long-range) Yukawa gas
in terms of the {\it quasi-free Yukawa approximation}, which we now define.
The idea is the same as in Section~\ref{sec:qf-def-torus},
with the additional element that now the boundary requires a special treatment.

Given parameters which we will chose later on with the constraint 
\be\label{b-condition}
 N^{-1/2 +\sigma} \ll \ell \ll b \leq b' \ll 1,  \qquad \ell < R,
\eeq
we divide $\C$ into a grid of squares $\alpha$ of side length $b$ with centers $c(\alpha) \in (b\Z)^2 \subset \C$.
The last constraint  \eqref{b-condition} will be assumed through out this paper.  
It will also be useful to also consider the shifted grid, in which all squares are translated by $u \in [-b/2,b/2)^2$
so that their centers are $u+c(\alpha)$.
We write $\b S_u$
\nomOth[16]{$\b S_u$}{set of squares such that square containing $0$ has center $u$}%
for the set of squares partitioning $\C$ such that the square containing $0$ has $u$ as its center.
We say that the square $\alpha \in \b S_0$ is in the bulk if it and its translates by $u \in [-b/2,b/2)^2$
have distance at least $b'$ to the complement of $S_V$
(respectively $\Omega$ in the situation of Theorem~\ref{rk:quasifree-meso}).
Denote by  $D_0$ the union of the bulk squares  in $\b S_0$ 
and by $B_0= S_V \setminus D_0$ the remaining boundary region.
For a square $\alpha \in \b S_u$ we define that $\alpha$ is in the bulk  if $\alpha-u \in \b S_0$ is the bulk. 
Similarly, we define $D_u $ the union of the bulk squares satisfying the previous condition 
and denote by $B_u= S_V \setminus D_u$  the boundary region in this case.  
We will  use the notation $\alpha \in D_u$  (or $\alpha \subset D_u$) to denote that $\alpha$ is a bulk cube. 
Throughout Section~\ref{sec:quasifree}, we  assume in addition to \eqref{b-condition} the following condtion: 
\be\label{b'}
 b' \gg N^{-1/4}.
\eeq
In the context of Theorem~\ref{rk:quasifree-meso}, we assume that   $b' \geq N^{-a}$ instead of $b' \gg N^{-1/4}$.

Given parameters as above,
we consider the quasi-free Yukawa gas obtained
by removing the interaction  between particles  in 
a bulk square with particles outside that square,
and replacing the interactions between particles in the same
square by a periodic one inside each bulk square. 
For the particles in the boundary region, we will use independent particle approximation 
with density given by the equilibrium density $\rho_V$ near the boundary. Since the boundary region $B$ has an area of smaller order when compared with the interior domain $D$,
the independent particle approximation is already sufficient to approximate the
log partition function to order $N^{1-c}$. 

Fix $u \in [-b/2,b/2)^2$.
The following definitions depend on $u$, but we do not make this explicit in the notation.
Firstly, we write $\bar{\b S} = \bar{\b S}_u$ for the set
\begin{equation}
  \bar {\b S}_u = \{u+\alpha: \alpha \in \b S_0, \alpha \subset D_u\} \cup \{B_u\},
\end{equation}
i.e. the set of bulk squares together with the boundary region.
Let $\b n = (n_\alpha)$ be a particle profile with $\sum_\alpha n_\alpha = N$.
Similarly as in \eqref{e:Fdef-torus},
we define the quasi-free free energy for particle profile $\b n$ by
\begin{equation} \label{e:Fdef}
  F(\b n)
  =
  \frac{1}{\beta} \log \binom{N}{\b n}
  +
  \frac{1}{\beta} \sum_{\alpha\subset D} \log \int_{\T_\alpha^{n_\alpha}} \ee^{-\beta \hat H_{\alpha}(\b v) } \, m(\rd\b v) 
  - \hat H_B
  ,
\end{equation}
with
\begin{equation} \label{e:hatHalpha}
\hat H_\alpha(\b v)
= \sum_{i \neq j} U^\ell_\alpha (v_i-v_j) +  N n_\alpha Q(\alpha),
\quad
\hat H_B
=
N^2 I_{Q,B} + 2c_V N (n_B - N\mu_V(B))
,
\end{equation}
where $Q(\alpha):= Q(c(\alpha))$ and $I_{Q,B}$ is a constant defined in \eqref{e:IQB} below.
We denote by $\b {\bar n} = (\bar n_\alpha)$ 
the approximate mean number of particles in $\alpha$,
where $\alpha$ is either a square or the boundary region.
More precisely, we choose $\bar n_\alpha$ to be an integer at distance at most $1$ to $N \mu_V(\alpha)$;
we assume that this rounded choice is such that $\sum_\alpha \bar n_\alpha = N$.
The precise choice of $\bar n_\alpha$ is not important as long as it is within  $1$ distance to  $n_\alpha$  and $\sum_\alpha \bar n_\alpha = N$. 
We also impose the convention that sums over $\b n$ will always be over all particle profiles with $\sum_\alpha n_\alpha = N$.

We will prove the following upper and lower bounds on the partition function in terms of the quasifree free energy.

\begin{proposition}[Upper Bound]\label{prop:qf-ub}
Assume that the parameters $b, b'$ satisfy  \eqref{b-condition} and \eqref{b'}.
Then there exists $u \in [-b/2,b/2)^2$ such that
\begin{multline} \label{e:qf-ub}
  \frac{1}{\beta} \log \int \ee^{- \beta H^R_{V}(\b z)}
  \, m(\rd\b z)
  - N \log (R/\ell) - N^2 K_R^\ell
  \leq
  \frac{1}{\beta} \log \sum_{\b n} \ee^{\beta F(\b n)}
  \\
  + N^\epsilon \OO(N^2 \ell^3 b^{-1} + N^2 \ell^2b )\|\rho_V\|_{\infty,2}
  + \OO(n_B \log N), 
\end{multline} 
where  $\|\rho_V\|_{\infty,2} $ is defined in \eqref{normt}. 
\end{proposition}

The error terms in \eqref{e:qf-ub} can be understood as follows.
The error $N^2 \ell^3 b^{-1} = (N \ell^2)  (N^2 \ell b^{-1})$
is the number of pair interactions via a Yukawa gas of range $\ell$ for particles in neighboring squares; the error 
$N^2 \ell^2 b$ is the variation of the effective potential $Q$ over a square of size $b$.
The error terms in the following lower bound cannot be obtained by a simple counting
as the bound relies on higher order cancellations which we will explain later on.

\begin{proposition}[Lower Bound] \label{prop:qf-lb}
  Assume
  $N^{-1/2+\sigma} \ll \ell \ll b \ll N^{-2c/5}$ for some small $c>0$, $\ell<R$,
  and $1 \gg \ell/b \gg (Nb^2)^{-1/4}$. 
  Then, with $\tau = 2 \sigma /5$, for all $u \in [-b/2,b/2)^2$,
  \begin{multline}  \label{e:qf-lb}
    \frac{1}{\beta} \log \int \ee^{- \beta H^R_{V}(\b z)} m(\rd\b z)
    - N \log (R/\ell) - N^2 K_R^\ell
    \ge 
    F(\bar {\b n})
    + N^\epsilon \OO(N^{1-\tau} + b^2\ell^{-4})
    \\ 
    + \OO(N^2(b^4+\ell^2 b)) (\|\rho_V\|_{\infty,3}+\|\rho_V\|_{\infty,3}^2)
    + \OO\p{b^{-2}\log N  + \bar n_B \log N}
    .
  \end{multline}
  More precisely, $\OO(N^{1-\tau})$ is $N^{\varepsilon}\OO(N^{4/5}\ell^{-2/5}+Nb)$.
\end{proposition}

Propositions~\ref{prop:qf-ub} and \ref{prop:qf-lb}
will be  proved in Sections~\ref{sec:quasifreeub}--\ref{sec:quasifreelb}.
In the remainder of Section~\ref{sec:quasifree}, we complete the proof of Theorem~\ref{FA}
assuming these propositions.
They assert that the free energy of a Yukawa gas with (long) range~$R$
can be approximated by that of the quasi-free Yukawa gases with range~$\ell \ll R$,
for appropriate choices of the parameters $b,b'$ and $\ell$.
These propositions are analogous to Propositions~\ref{prop:qf-ub-torus} and \ref{prop:qf-lb-torus},
with the additional treatment of the boundary
and taking into account that the density of the equilibrium measure is in general not constant.

\bigskip
We end this subsection by recording the following simple estimates for the bulk and boundary regions.
In the following, we usually omit the parameter $u$ from $D_u$ (the union of bulk squares)
and write $B = S_V \setminus D$ to denote the boundary region.
We write $\bigcap D=\bigcap_u D_u$ and $\bigcup D = \bigcup_u D_u$.

\begin{lemma} \label{lem:nbar-bd}
The following bounds hold uniformly in the shift parameter $u\in [-b/2,b/2)^2$.
The number of bulk squares (which is independent of $u$) is $\OO(b^{-2})$,
the number of bulk squares touching the boundary region is $\OO(b^{-1})$,
and the equilibrium mass covered by the bulk squares is 
$\mu_V(\bigcap D) \geq  1-\OO(b')$. In addition, for any $\alpha \subset D$,
\begin{equation} \label{e:nbarbd}
  \bar n_\alpha = \OO(Nb^2) \|\rho_V\|_\infty,
  \quad
  \bar n_\alpha = Nb^2 \rho(\alpha) + \OO(Nb^3)\|\nabla\rho_V\|_\infty,
  \quad
  \bar n_B = \OO(Nb').
\end{equation}
\end{lemma}

\begin{proof}
The claim about the number of  bulk squares follows immediately from the fact
that the support of $S_V$ has diameter of order $1$.
The statements about the number of squares touching the boundary region
and the mass not covered by the squares follow from the assumption that the
the boundary of $S_V$ is piecewise $C^1$.
In the more general situation of Theorem~\ref{rk:quasifree-meso}, the estimates hold by the assumption stated in the remark.
Finally, \eqref{e:nbarbd} follows immediately from the fact that, by construction,
$\rho_V$ is $C^1$ on the squares $\alpha$.
\end{proof}

\subsection{Consequence of quasi-free approximation}

With the upper and lower bounds established in Propositions~\ref{prop:qf-ub}
and~\ref{prop:qf-lb}, the remainder of the proof is similar to  that for the torus.
In fact, the following proof is simpler since the limit of the torus free energy has already been established.

First, analogously to \eqref{e:hlong}, we define
\begin{equation} \label{e:hdef}
  h_\alpha(\b n)
  =      
  2\pi \gamma^2 (n_\alpha  - \bar n_\alpha)^2
    - n_\alpha \zeta
    - \frac {1}{2} n_\alpha   \log  n_\alpha
    -  \Big ( \frac 1 2 - \frac 1 \beta  \Big )   n_\alpha \log  b^{-2}, \quad \text{where } \gamma = \ell /b.
\end{equation}
Then, similarly to Lemma~\ref{lem:F1-torus},
we have the following estimate for $F(\b n)$ defined in \eqref{e:Fdef}.

\begin{lemma} \label{lem:F1}
Assume that $b$ satisfies   \eqref{b-condition}.
There exists $\tau>0$ such that
$$
  F(\b n) + N \log \ell - N^2 K_R^\ell - N^2 I_V^R
  =
  \frac{1}{\beta} \log \binom{N}{\b n}
  - \sum_{\alpha \subset D} h_\alpha(\b n)
  + \OO(N^2 \ell^2 b)(1+\|\rho_V\|_{\infty,1})^2 + \OO(N^{1-\tau}), 
$$
where $\|\rho_V\|_{\infty,1} $ is defined in \eqref{normt}. 
More precisely, the error $\OO(N^{1-\tau})$ is $N^{\varepsilon}\OO(N^{7/8}/b^{1/4}+\ell^{-2})$.
\end{lemma}

\begin{proof}
From \eqref{e:Fdef}, recall that
\begin{equation*}
  F(\b n)
  =
  \frac{1}{\beta} \log \binom{N}{\b n}
  - \sum_\alpha T_\alpha(n_\alpha)
  - \hat H_B,
  \quad
  T_\alpha(n_\alpha)
  :=
  -\frac{1}{\beta} \log \int_{\T_\alpha^{n_\alpha}} \ee^{-\beta \hat H_{\alpha}(\b z) } \, m(\rd\b z)
  ,
\end{equation*}
where here and in the rest of this proof, all summations over $\alpha$ are over $\alpha \subset D$.
Recall the defintion of $\hat H_B$ from \eqref{e:hatHalpha},
hence Lemma \ref{lem:F1} follows from
\begin{multline*}
  \sum_{\alpha} T_\alpha(n_\alpha)
  + N^2 I_{Q,B} + 2c_VN(n_B-N\mu_V(B))
  - N^2 K_R^\ell
  \\
  = N^2 I_V^R
  -N \log \ell 
  + \sum_{\alpha \subset D} h_\alpha({\bf n}) 
  + \OO(N^2 \ell^2b (1+\|\rho_V\|_{\infty,1})^2) + \OO(N^{1-\tau}),
\end{multline*}
which we now prove.
By definition of $T_\alpha$,
\begin{align*} 
  T_\alpha(n_\alpha)
  &= N n_\alpha Q(\alpha) + 2\pi \gamma^2 n_\alpha^2 - n_\alpha \log \ell - \xi_b^{(\gamma)}(n_\alpha)
  \nonumber\\
  &= N n_\alpha Q(\alpha) + 2\pi \gamma^2 n_\alpha^2 - n_\alpha \log \ell
    - n_\alpha \zeta^{(\gamma)} (n_\alpha)
    - \frac {1}{2} n_\alpha   \log  n_\alpha
    -  \Big ( \frac 1 2 - \frac 1 \beta  \Big )   n_\alpha \log  b^{-2}
  .
\end{align*}
By Theorem~\ref{thm:Z-torus}, 
$n_\alpha \zeta^{(\gamma)} (n_\alpha) = n_\alpha \zeta + N^\varepsilon\OO(n_\alpha^{7/8}+1/\gamma^2)$
so that $\sum_\alpha n_\alpha \zeta^{(\gamma)} (n_\alpha)
= N\zeta +b^{-2}N^\varepsilon \OO( (Nb^2)^{7/8}+(\ell/b)^{-2})
= N \zeta +N^\varepsilon\OO(N^{7/8}/b^{1/4}+\ell^{-2})$.
Therefore
\begin{multline*}
\sum_\alpha T_\alpha(n_\alpha)
=
\sum_\alpha \pa{ N n_\alpha Q(\alpha) + 2\pi \gamma^2 n_\alpha^2 
- \frac {1}{2} n_\alpha   \log  n_\alpha
-  \Big ( \frac 1 2 - \frac 1 \beta  \Big )   n_\alpha \log  b^{-2}}
\\
- N \log \ell
- N \zeta
+ N^\varepsilon\OO(N^{7/8}/b^{1/4}+\ell^{-2}).
\end{multline*}
By definition of $h_\alpha(\b n)$ in \eqref{e:hdef} and since
\begin{equation*}
\sum_{\alpha}
2\pi \gamma^2 n_\alpha^2 
=  2\pi \gamma^2    \sum_{\alpha} (n_\alpha  - \bar n_\alpha)^2 
+ 4\pi \gamma^2 \sum_\alpha n_\alpha \bar n_\alpha
- 2\pi \gamma^2 \sum_\alpha \bar n_\alpha^2,
\end{equation*}
we obtain
$$
  \sum_\alpha T_\alpha(n_\alpha) - \sum_\alpha h_\alpha(n_\alpha) + N \log \ell
  =
  \sum_\alpha \pa{
  N n_\alpha Q(\alpha)
  + 4\pi \gamma^2 n_\alpha \bar n_\alpha
  - 2\pi \gamma^2 \bar n_\alpha^2
  }
+
  \OO(N^{\frac{7}{8}+\varepsilon}/b^{\frac{1}{4}}+\ell^{-2}).
$$
We now compute the right-hand side of the last equation.
Using that $\gamma = \ell/b$, that
$2\pi \ell^2 = \int Y^\ell(z) \, m(\rd z) = \int_\alpha Y^\ell(z) \, m(\rd z) + \Oinfty$,
and that
$\bar n_\alpha
= Nb^2\rho_V(z) + \OO(Nb^3) \|\nabla\rho_V\|_\infty
= N \int_{\alpha} \rho_V(w) \, m(\rd w) + \OO(Nb^3) \|\nabla\rho_V\|_\infty$ for any $z\in \alpha$,
we obtain
\begin{align}   \label{e:C1rewrite}
  2\pi \gamma^2 \sum_\alpha \bar n_\alpha^2
  &= N^2 \sum_{\alpha} \iint_{D \times\alpha} Y^\ell(z - w) \,  \rho_V (z) \, \rho_V (w) \, m(\rd z)\, m(\rd w)
    +\OO(N^2 \ell^2 b) \|\rho_V\|_\infty \|\nabla\rho_V\|_\infty
    \nonumber\\
  &= N^2 \iint_{D\times D} Y^\ell (z-w) \, \mu_V (\rd z) \, \mu_V(\rd w)
    +\OO(N^2 \ell^2 b) \|\rho_V\|_\infty \|\nabla\rho_V\|_\infty
    .
\end{align}
Analogously, we have
$
  4\pi \gamma^2 \bar n_\alpha
  = 2N \int Y^\ell(\alpha-z) \rho_V(z) \, m(\rd z)  + \OO(N \ell^2b) \|\nabla \rho_V\|_\infty.
$
It follows that
\begin{align*}
  \sum_\alpha [Nn_\alpha Q(\alpha) + 4\pi \gamma^2 n_\alpha \bar n_\alpha]
  &=
    N \sum_\alpha n_\alpha \qa{ Q(\alpha) + 2 \int Y^\ell(\alpha-z) \, \mu_V(\rd z) }
    + \OO(N^2 \ell^2 b)\|\nabla\rho_V\|_\infty
    \nonumber\\
  &\qquad + 2N \sum_\alpha n_\alpha \int Y^\ell(\alpha-z) \qa{\rho_V(z) \, m(\rd z)-\mu_V(\rd z)}
    \nonumber\\
  &=
    2c_V N(N-n_B) + \OO(N^2 \ell^2 b)\|\nabla\rho_V\|_\infty
    \nonumber\\
  &\qquad -2N^2 \iint_{D \times B} Y^\ell  (z-w) \, \mu_V (\rd z) \, \mu_V (\rd w) + \OO(N^2\ell^3 \|\rho_V\|_\infty^2)
    ,
\end{align*}
where the second equality follows from the Euler--Lagrange equation \eqref{e:qf-EL}
and $\sum_\alpha n_\alpha = N - n_B$,
and using that in the computation of $\iint_{D \times B} Y^\ell  (z-w) \, \mu_V (\rd z) \, \mu_V (\rd w)$,
the contribution of the absolutely continuous part of $\mu_V$ in $B$ is of order $N^2\ell^2\|\rho_V\|_\infty^2$.
Using also that
$
  I_{Q,B} - 2c_V \mu_V(B)
  =
  - \iint_{B \times B} Y^\ell  (z-w) \, \mu_V (\rd z) \, \mu_V (\rd w)
  $ by \eqref{e:IQB},
in summary, we have proved
\begin{multline*}
  \sum_\alpha T_\alpha(n_\alpha)
  + N^2 I_{Q,B} + 2c_VN(n_B-N\mu_V(B))
  - \sum_\alpha h_\alpha(n_\alpha) + N \log \ell
  \\
  =
  2c_V N^2
  - N^2 \iint_{\C^2} Y^\ell (z-w) \, \mu_V (\rd z) \, \mu_V(\rd w)
  + \OO(N^2 \ell^2 b)(\|\rho_V\|_\infty+\|\nabla\rho_V\|_\infty)^2
  .
\end{multline*}
Lemma \ref{lem:F1} now follows from the Euler--Lagrange equation \eqref{e:qf-EL}, which implies
\begin{align*}
  2c_V
  &=
  2 \iint Y^R (z-w) \, \mu_V (\rd z) \, \mu_V(\rd w) + \int V(z) \, \mu_V(\rd z)
 \nonumber\\
  &=
  \iint Y^R (z-w) \, \mu_V (\rd z) \, \mu_V(\rd w) + I_V^R
  =
  \iint Y^\ell (z-w) \, \mu_V (\rd z) \, \mu_V(\rd w) + K_R^\ell + I_V^R
  .
\end{align*}
This completes the proof.
\end{proof}

We need the following bound showing that in the sum over $\b n$
the dominant term is $\b n = \b {\bar n}$.
The  torus version of this lemma was given in  Lemma~\ref{lem:F-torus}.

\begin{lemma}\label{lem:F2}
Recall the condition \eqref{b-condition}. Suppose that we have a collection of functions $\cal E_\alpha : \N \to \R$
satisfying  $|\cal E_\alpha(n)-\cal E_\alpha(m)| \leq { \OO(}|n-m|(n+m)^\epsilon)$.
Define
\begin{equation}
  \cal E( \b n)
  = \sum_{\alpha\subset D} \Big [
   - 2\pi \gamma^2  (n_\alpha  - \bar n_\alpha)^2
   + \cal E_\alpha(n_\alpha)
   \Big ].
   \label{pa10}
 \end{equation}
 Assume that  $\bar {\b n}$ 
 satisfies \eqref{e:nbarbd} and that $\gamma = \ell/b \geq N^{-C}$. Then
\begin{equation} \label{e:F2}
  \frac{1}{\beta}  \log  \sum_{\b n}\ee^{\beta \cal E( \b n) + \beta \OO(n_B \log N)}
  \leq 
  \cal E(\b {\bar n})
  + N^\epsilon \OO(Nb' + \ell^{-2}\|\rho_V\|_\infty),
\end{equation}
where the sum on $\b n$  is under the constraint $N = \sum_\alpha n_\alpha = n_B + \sum_{\alpha\subset D} n_\alpha$.
Notice that $ \cal E$ contains only contribution from the squares in the bulk.
\end{lemma}

\begin{proof}
By definition,
\begin{multline}\label{e:calF}
  \frac{1}{\beta}  \log  \sum_{\b n}\ee^{\beta \cal E( \b n) + \beta \OO(n_B \log N)}
  - \cal E  (\b {\bar n})
  =
  \frac{1}{\beta}  \log  \sum_{\b n}\ee^{\beta (\cal E( \b n) - \cal E(\b {\bar n}))  + \beta \OO(n_B \log N)}
  \\
  =
  \frac{1}{\beta}  \log  \sum_{\b n} 
  \exp \qa{
    \sum_\alpha \beta
    \qB{
      - 2\pi  \gamma^2  (n_\alpha  - \bar n_\alpha)^2
      + (\cal E_\alpha ( n_\alpha ) - \cal E_\alpha ( \bar n_\alpha ) 
    }
    + \OO(\beta n_B \log N)
  }.
\end{multline}  
By the constraint $N = \sum_\alpha n_\alpha = n_B + \sum_{\alpha\subset D} n_\alpha$, we can add the factor
\begin{align*}
\mathds{1} \pB{n_B - \bar n_B = \sum_{\alpha\subset D} (\bar n_\alpha -n_\alpha)}
&\leq \mathds{1} \pB{|n_B - \bar n_B| \leq \sum_{\alpha\subset D} |\bar n_\alpha -n_\alpha|}
\\
&\leq \exp\qa{-\frac{\beta 2\pi \gamma^2}{2 \#\{\alpha \subset D\}} (n_B - \bar n_B)^2
  + \frac{\beta 2\pi \gamma^2}{2} \sum_{\alpha \subset D} (\bar n_\alpha -n_\alpha)^2},
\end{align*}
where we used $\mathds{1}(a \leq b) \leq \ee^{-A  a^2+A b^2}$ for any constant $A > 0$ and
$(\sum_{\alpha \subset D} x_\alpha)^2 \leq \#\{\alpha \subset D\} \sum_{\alpha \subset D} x_\alpha^2$ where
$\#\{\alpha \subset D\} = \OO(b^{-2})$ is the number of squares. 
Thus, at the cost of replacing $2\pi\gamma^2$ by $\pi\gamma^2$ in \eqref{e:calF},
we can add the following factor to the right hand side of \eqref{e:calF}: 
\begin{equation*}
\exp\qa{-c\beta b^2\gamma^2 (n_B - \bar n_B)^2}
= \exp\qa{-\beta c \ell^2 (n_B - \bar n_B)^2}, 
\end{equation*}
where $c$ is a constant of order one.
With this preparation, to get an upper bound,
we now drop the constraint $\sum_\alpha n_\alpha = N$ on $\b n$,
and sum each $n_\alpha$ independently.
For the bulk squares, we use
$|\cal E_\alpha (n ) - \cal E_\alpha (m) | \le \OO(|n-m| (n + m)^{\epsilon})$
and the elementary inequality \eqref{e:F-torus-elineq}, as in the torus case.
For the boundary layer $B$, we similarly use
\begin{equation*} 
  \sum_{n =0 }^\infty    
  \exp \Big [ C n \log N- c \ell^2     (n-m )^2    \Big ]
  \leq \OO(\ell^{-1}) \ee^{\OO(m+\ell^{-2}) (\log N)^2}.
\end{equation*} 
In summary, using $\bar n_\alpha = \OO(Nb^2) \|\rho_V\|_\infty$
for $\alpha \subset D$ and $\bar n_B = \OO(N b')$ by \eqref{e:nbarbd},
the left-hand side of \eqref{e:calF} is of order
\begin{equation*}
  (\log N) \sum_{\alpha \subset D}   \gamma^{-2}   (\bar n_\alpha + \gamma^{-2})^{2\epsilon}\|\rho_V\|_\infty
  + (\log N)^2 (\bar n_B + \ell^{-2})
  \le \OO(N^{2\epsilon} \ell^{-2})\|\rho_V\|_\infty + \OO( N b' (\log N)^2).
\end{equation*} 
This completes the proof of the lemma. 
\end{proof}

\subsection{Existence of free energy of Yukawa gas: proof of Theorem~\ref{FA}}
\label{sec:qf-pf-fa}

The proof of Theorem~\ref{FA} below is analogous to that of Proposition~\ref{prop:qf-density-torus}.

\begin{proof}[Proof of Theorem~\ref{FA}]
We first show that if $1\geq R\geq N^{-1/2+\sigma}$ there is some $\kappa=\kappa(\sigma)>0$ such that \eqref{pa2} holds.
  Subsequently we will observe that any $\kappa<1/24$ is admissable if $R \geq 1$.
To do this, we apply Propositions~\ref{prop:qf-ub}, \ref{prop:qf-lb}
and consider the different error terms.

First, for any choice $N^{-1/4} \ll b' \ll 1$ the error
terms involving $b'$ are $N^\epsilon\OO(n_B) = N^\epsilon\OO(Nb')$ using \eqref{e:nbarbd}.
In particular, in the situation of Theorem~\ref{freeasy}, we can choose $b' < N^{-\kappa}$ as needed.
In the situation of Theorem~\ref{rk:quasifree-meso}, this error term is $N^\epsilon\OO(N^{1-a})$
as claimed in the remark.

Next we emphasize that, in the upper and lower bounds (Propositions~\ref{prop:qf-ub} and \ref{prop:qf-lb}),
the range parameter $\ell$ is not required to be the same,  but we always require $\ell \leq R$.
We denote the value of $\ell$ by $\ell_+$ for the upper bound
and by $\ell_-$ for the lower bound.

We first consider the case $N^{-1/2+\sigma}\leq R\leq 1$. Take $b=N^{-1/2+\sigma/10}$. 
For $\ell_+=N^{-1/2+\sigma/100}$, the error terms in \eqref{e:qf-ub} are bounded by $N^{1-\sigma/1000}$.
For $\ell_-=N^{-1/2+9\sigma/100}$, the error terms in \eqref{e:qf-lb} are also bounded by $N^{1-\sigma/1000}$
(we used $\bar n_B = \OO(N b')$).

With Lemma~\ref{lem:F1}, for some $\kappa=\kappa(\sigma)>0$ we therefore obtain
\begin{align} \label{pa-lb}
  \frac{1}{\beta} \log \int \ee^{- \beta H_{V}^R(\b z)}  m(\rd\b z)
  &\geq -   N^2   I^R_V + N   \log R
     +  \frac{1}{\beta}  \log {N \choose \bar{\b n}}
    \ee^{-\beta \sum_\alpha h_\alpha(\bar{\b n}) }
    - \OO(N^{1-\kappa}),
  \\    
  \label{pa-ub}
  \frac{1}{\beta} \log \int \ee^{- \beta H_{V}^R(\b z)}   m(\rd\b z)
  & \leq - N^2   I^R_V  + N   \log R 
    +  \frac{1}{\beta}  \log  \sum_{\b n}   {N \choose \b n}
    \ee^{-\beta \sum_\alpha h_\alpha(\b n) }
    + \OO(N^{1-\kappa}).
\end{align}
For the rest of this proof, all summations of $\alpha$ are over $\alpha \subset D$.

By Stirling's formula as in \eqref{e:Stirling},
and using the definitions of $h$ in \eqref{e:hdef} and of $\cal E$ in \eqref{pa10} with
$\cal E_\alpha(n_\alpha) = ( \frac 1 2 - \frac 1 \beta)  n_\alpha \log (n_\alpha b^{-2})$,
we can rewrite \eqref{pa-lb}, \eqref{pa-ub} as 
\begin{align}   \label{pa0a}
\frac{1}{\beta} \log \int \ee^{- \beta H_{V}^R(\b z)} \, m(\rd\b z)  + N^2   I_V^R
&\geq
\cal E( \b {\bar n})
+ \zeta + N   \log R +  \frac{1}{\beta} N \log N
+\OO(N^{1-\kappa}),
\\
\label{pa0b}
\frac{1}{\beta} \log \int \ee^{- \beta H_{V}^R(\b z)} \, m(\rd\b z)  + N^2   I_V^R
&\leq
\frac{1}{\beta}  \log  \sum_{\b n} \ee^{\beta \cal E( \b n)}
+ \zeta + N   \log R +  \frac{1}{\beta} N \log N
+\OO(N^{1-\kappa})
.
\end{align}
By Lemma~\ref{lem:F2}, we can replace the sum over $\b n$ in \eqref{pa0b}
by the dominant term $\b{\bar n}$ with error smaller than $\OO(N^{1-\kappa})$.
By a Riemann sum approximation using that $\rho_V$ is $C^1$ in $D$,
\begin{align*} 
  \cal E(\b {\bar n})
  &= \pa{\frac12-\frac1\beta} \sum_\alpha \bar n_\alpha \log (\bar n_\alpha b^{-2})
  \\
  &= \pa{\frac12-\frac1\beta} N  \int \rho_V(z) \log \rho_V(z) \, m(\rd z)
  + \pa{\frac12-\frac1\beta} N \log N + \OO(N(b+b')) \|\rho_V\|_{\infty,1}
  .
\end{align*}
This completes the proof of \eqref{pa2} when $N^{-1/2+\sigma}\leq R\leq 1$.

To show that if $R \geq 1$ then any $\kappa<1/24$ is admissible,
we consider all error terms in details. In the upper bound \eqref{e:qf-ub},
the error is 
\begin{equation}
\OO(N^\varepsilon)\qB{N^2 \ell_+^3 b^{-1} + N^{2}\ell_+^2 b},
\end{equation}
while, in the lower bound \eqref{e:qf-lb}, it is of order 
\begin{equation}
\OO(N^\varepsilon)\qB{ N^2 b^4
  + N^2\ell_-^2b +Nb+ (b^2 \ell_-^{-4} + N^{4/5}/\ell_-^{2/5})
  }.
\end{equation}
Lemma \ref{lem:F1} gives analogues of \eqref{pa-lb} and \eqref{pa-ub} with
an error term 
$$\OO(N^\varepsilon)\qB{N^{7/8}/b^{1/4}+1/\ell_-^2+1/\ell_+^2}.$$
Optimizing the parameters yields
$b=N^{-1/3}$, $\ell_+=N^{-23/48}$, $\ell_-=N^{-7/18}$.
Note that this choice of parameters satisfies the hypothesis $\ell_-/b \gg (Nb^2)^{-1/4}$
and $\ell_\pm \leq R$.
The common error then becomes $\OO(N^{23/24+\varepsilon})$ for arbitrarily small $\varepsilon>0$.
The rest of the proof is unchanged.
\end{proof}

\subsection{Existence of free energy of Coulomb gas: proof of Theorem~\ref{freeasy}}
\label{sec:coulombapprox}

We now choose $R = N^2$ to deduce Theorem~\ref{freeasy} from Theorem~\ref{FA}.
 
\begin{proof}[Proof of Theorem~\ref{freeasy}]
The equilibrium measure $\mu_V$ of the Coulomb gas in Theorem~\ref{freeasy}
is characterized by the Euler--Lagrange equation
\begin{equation}
  U^{\mu_V} + \frac12 V = c_V
\end{equation}
in its support $S_V$ and inequality in all of $\C$.  Define the potential $V_R$ via the equation 
\begin{equation}
  V_R(z) = V(z) + 2 \int \pa{\log\frac{1}{|z -w|} - Y^R(z -w) + Y_0 + \log R} \mu_V(\rd w).
\end{equation}
Explicitly, one can check that in $S_V$, 
\begin{equation}
  U_R^{\mu_V} + \frac12 V_R = c_V^R, \qquad c_V^R = c_V + Y_0 +  \log R, 
\end{equation}
holds and   with the inequality $\ge  c_V^R$ outside the support of $S_V$. 
Thus $\mu_V$ is also the equilibrium measure with respect to the Yukawa interaction
and external potential $V_R$.
Moreover, by \eqref{Ya},
\begin{align}
  I_{V_R}^R
  &= \int U_R^{\mu_V}(z) \mu_V(\rd z) + \int V_R(z) \, \mu_V(\rd z)
    \nonumber\\
  &= \int U_R^{\mu_V}(z) \mu_V(\rd z) + \int V(z) \, \mu_V(\rd z) + 2\int (U^{\mu_V}(z)-U^{\mu_V}_R(z) + Y_0 + \log R)
    \nonumber\\
  &= I_V^{\cal C} + (Y_0 + \log R)  + \OO(\frac{1}{R}).
\end{align}
Thus we have 
\begin{equation*}
\label{CYapp}
\frac{1}{\beta} \log \int  \ee^{-\beta H_{V_R}^{Y^R}}  \, m(\rd \b z)
= \frac{1}{\beta} \log  \int  \ee^{-\beta H_{V}^ \cal C } \, m(\rd \b z)
- N(N-1)(Y_0+\log R)
+\OO\left(\frac{N^2}{R}\right)
.
\end{equation*}
Moreover, \eqref{Ya} and an analogous estimate for derivatives of \eqref{Yg} imply
\begin{equation}
  \max_{k\leq 5} \|\nabla^k (V_R-V)\|_\infty = \OO\left(\frac{1}{R}\right).
\end{equation}
Thus, we may apply Theorem~\ref{FA} with $V$ replaced by $V_R$ and with $R=N^2$,
and Theorem~\ref{freeasy} then follows with $\zeta^{\cal C}_{\beta} = \zeta - Y_0$.
\end{proof}

\section{Proof of Proposition~\ref{prop:qf-ub}: free energy upper bound}
\label{sec:quasifreeub}

\subsection{Upper bound: proof of Proposition~\ref{prop:qf-ub}}

In this section, the condition $R\geq N^{2}$ is imposed. 
Recall from the proof of Proposition~\ref{prop:qf-ub-torus}  that, to each square $u+\alpha$, 
we associate a map $\Phi_\alpha^{u} : u+\alpha \to \T^{(b)}$ defined by \eqref{Phidef}--\eqref{Phidef2}.
Analogously to \eqref{e:tildeYellu-torus}, we define a two-body potential
\begin{equation} \label{eq:tildeYellu}
\tilde Y_u^\ell(z,w)
= \sum_{\alpha \in \b S_0} U^{\ell}_b(\Phi_{\alpha}^u(z)-\Phi_{\alpha}^u(w)) \mathds{1}_{z\in u+\alpha}\mathds{1}_{w\in u+\alpha} + Y^\ell(z-w) \mathds{1}_{z \not\in D_u,w \not \in D_u},
\end{equation}
and $\tilde Q_u$ by replacing $Q$ in the bulk squares $u+\alpha \subset D_u$
by its value at the centers of the squares,
and outside $D_u$ by adding the equilibrium contribution from the pair interaction with the bulk, i.e.,
\begin{equation} \label{e:tildeQ}
  \tilde Q_u(z) = \sum_{\alpha \in \b S_0} Q(c(u+\alpha)) \mathds{1}_{z \in u+\alpha} +
  \pa{Q(z) + 2N \int_{D_{u}} Y^\ell(z-w) \, \mu_V(\rd w)} \mathds{1}_{z \not\in D_u}
  .
\end{equation}
Denote by $\tilde H^\ell_u$ the corresponding Hamiltonian on $\C^{N}$:
\begin{equation} \label{e:qf-H-ub}
\tilde H^\ell_u(\b z)
= N \sum_j \tilde Q_u(z_j) + \sum_{i \neq j} \tilde Y_u^\ell(z_i,z_j).
\end{equation}
The main work towards Proposition~\ref{prop:qf-ub} is contained in
the proof of Proposition~\ref{prop:qf-ub-1} below.

\begin{proposition}\label{prop:qf-ub-1}
Under the  assumptions of Proposition~\ref{prop:qf-ub}, 
there exists $u \in [-b/2,b/2)^2$ such that  (the constant $K_R^\ell$ is defined in \eqref{e:Kdef})
\begin{multline} \label{e:qf-ub-1}
  \frac{1}{\beta} \log \int \ee^{- \beta H^R_{V}(\b z)} \, m(\rd\b z)
  \leq \frac{1}{\beta}    \log   \int  \ee^{- \beta  \tilde   H^{\ell}_u(\b z) } \, m(\rd\b z) 
  + N \log (R/\ell) + N^2 K_R^\ell
  \\
  + N^\epsilon \OO(N^2 \ell^3 b^{-1} + N^2 \ell^2 b  )\|\rho_V\|_{\infty,2}
  .
\end{multline}
\end{proposition}

In preparation of the proof, we collect some notation and bounds.
We write $\E^u f(u) = b^{-2} \int_{[-b/2,b/2)^2} \rd u f(u)$ for the average over $u$, and 
analogously to \eqref{e:gdef-torus}, we denote 
\begin{equation}  \label{e:gdef}
  \bar Y(z,w) =
  \E^u
  \tilde Y^\ell_u(z,w),
  \qquad
  \bar Q(z) =
  \E^u 
  \tilde Q_u(z) .
\end{equation}
The following lemma provides estimates on $\bar Y$,
extending the analogous Lemma~\ref{lem:gb-torus} for the torus.
The estimates are stated in terms of the function $g$ defined in \eqref{e:by}.

\begin{lemma} \label{lem:gb}
  Assume that $\ell \ll b$. Then 
  \begin{enumerate}
    \item 
      Inside the bulk, i.e., for $z,w\in \bigcap D$, we have $\bar Y(z,w) = g(z-w) + \OO(\ee^{-cb/\ell})$
      and $g(z-w) - Y^\ell(z-w) = \OO(\ell/b)$.
    \item
      Away from the bulk, i.e., for $z,w \not\in \bigcup D$, by definition we have $\bar Y(z,w) = Y^\ell(z-w)$.
    \item
      In general, and in particular near the boundary, we have the inequalities
      \begin{equation}
        g(z-w) + \OO(\ee^{-cb/\ell}) \leq \bar Y(z,w) \leq Y^\ell(z-w)  + \OO(\ee^{-cb/\ell}) \quad \text{if }  |z-w|_\infty \leq b/2.
      \end{equation}
\end{enumerate}
\end{lemma}

\begin{proof}
(i) This case is exactly the same as Lemma~\ref{lem:gb-torus}.
\smallskip

\noindent (ii)
In this case, since $z,w \notin D_u$, 
by the definition \eqref{eq:tildeYellu} we directly have $\bar Y(z,w) = Y^\ell(z,w)$.

\smallskip
\noindent (iii)
By the exponential decay of $Y^\ell$, the definition \eqref{eq:tildeYellu}  and using that $U^\ell$  is the periodization of $Y^\ell$, 
we  have the bound 
$\bar Y(z,w) \leq Y^\ell(z-w) + \OO(\ee^{-cb/\ell})$
for $|z-w|_\infty \leq b/2$.

For the lower bound on $\bar Y$ for $|z-w|_\infty \leq b/2$,
we notice that $\tilde Y_u^\ell(z,w) = Y^\ell(z-w) + \OO(\ee^{-cb/\ell})$ if and only if either
$z$ and $w$ belong to the same square $\alpha \subset D_u$ or $z,w \notin D_u$,
and in other cases $\tilde Y_u^\ell(z,w) = 0$.
The probability of first event, with respect to the $u$-average, is bounded below
by that of the event that $z$ and $w$ are both in the same square,
irregardless of whether the square is in $D_u$ or not.
This probability is $(b-x)(b-y)/b^2$,
and therefore
\begin{equation*}
\tilde Y(z,w) \geq \frac{(b-x)(b-y)}{b^2} Y^\ell(z-w) + \OO(\ee^{-cb/\ell}) = g(z-w) + \OO(\ee^{-cb/\ell}).
\end{equation*}
This completes the proof.
\end{proof}

\begin{proof}[Proof of Proposition~\ref{prop:qf-ub-1}]
By Jensen's inequality,
\begin{align} 
  \frac{1}{\beta}   \log \int  \ee^{- \beta  H^R_{V}(\b z)  }   \, m(\rd \b z)    
  \le
  \frac{1}{\beta}  \log   \int  \ee^{- \beta   \tilde  H^{\ell}_u(\b z) }   \, m(\rd \b z)
  + \E^R_V (\tilde H^{\ell}_u- H^R_{V} ),
\label{11}
\end{align}
where we recall that $\E^R_V$ denotes the expectation of the probability measure with density $\ee^{-\beta H_V^R}$.
The last term can be rewritten as 
\begin{equation}
\label{e:tildeHellHR}
 \E^R_V (\tilde H^{\ell}_u
    - H^R_{V} )  =   \E^R_V (\tilde H^{\ell}_u
    -  H_{V}^{\ell} )    +  \E^R_V (  H_{V}^{\ell}
    - H^R_{V} ) 
    .
\end{equation}
Using that $L^\ell_R$ is positive definite, ${\b L}_R^\ell \geq 0$,
and by \eqref{e:HellHR},
the last term in \eqref{e:tildeHellHR} is bounded by
\begin{equation*}
  \E^R_V (   H_{V}^{\ell}  - H^R_{V} )
  = 
  - N^2 \E_V^R {\b L}_R^\ell 
  + N \log (R/\ell) + N^2 K_R^\ell
  \le  N \log (R/\ell) + N^2 K_R^\ell
  .
\end{equation*} 
To bound the first term in \eqref{e:tildeHellHR} for some $u$,
it suffices to bound the average of \eqref{11} over $u$ in the square $[-b/2,b/2]^2$.
Indeed, by the mean-value theorem for continuous functions,
there then exists a choice of $u$ that achieves the bound of the average.
By the definition of $\bar Y$ and $\bar Q$ in \eqref{e:gdef}, we have
\begin{equation} \label{e:qf-ub-pf1}
 \frac{1}{b^2} \int_{[-b/2,b/2]^2} \rd u \,
  \E^R_V  ( \tilde H^{\ell}_u - H_{V}^{\ell} )
  =
  \E^R_V \qB{ N \sum_j (\bar Q(z_j)- Q(z_j))}
  +
  \E^R_V \qB{ \sum_{i \neq j} (\bar Y(z_i,z_j) - Y^\ell(z_i - z_j))}
  .
\end{equation}
For the particles in the bulk,
the term involving $Q$ is bounded using \eqref{e:Qz}.
Indeed, the term $2c_V$ in \eqref{e:Qz} cancels and using that $\ell \leq b$ and $N^\epsilon \ell^4 \leq \ell^2b$
the difference of the other two terms in \eqref{e:Qz} is estimated by
\begin{equation} \label{e:QzQalphasum-ub}
N \sum_j (Q(z_j)-\bar Q(z_j)) \mathds{1}_{z_j \in D}
= \OO(N^2 \ell^2 b) (\|\nabla\rho_V\|_\infty + \|\nabla^2\rho_V\|_\infty)
.
\end{equation}
For the particles {outside the bulk}, 
the difference of $Q$ and $\tilde Q$ is by the definition \eqref{e:tildeQ} equal to
\begin{equation*}
  2N \int_D Y^\ell(z_j-w) \, \mu_V(\rd w)
  = \OO(N\ell^2)\|\rho_V\|_\infty.
\end{equation*}
 By the decay of the Yukawa potential, only particles $z_j$ within distance $N^\epsilon\ell$
  to $D$ give a nonnegligible contribution to this term.
By the local density estimate, Theorem~\ref{thm:Cdensity},
there are $\OO(N^{1+\epsilon} \ell)$ such particles,
so that the sum of the last expression over the particles $z_j$ in the boundary region
is bounded by  $N^{2+\epsilon}\ell^3 \leq N^{2}\ell^2b$.

Similarly, dividing the sum over $i \neq j$ for the pair interaction in \eqref{e:qf-ub-pf1}
into bulk and boundary contribution,
Lemma~\ref{lem:gYub} below implies
\begin{equation*} 
  \E^R_V  \sum_{i \neq j} [\bar Y(z_i,z_j) - Y^\ell(z_i - z_j)]
  = N^\epsilon \OO(N^2 \ell^3 b^{-1}).
\end{equation*}
Here we used that the contribution for the above sum where
both $z_i$ and $z_j$ are outside $\bigcup D$ vanishes since then
$\bar Y(z_i,z_j) = Y^\ell(z_i-z_j)$.
Moreover, for the contributions where at least one of the particles is in the bulk,
we may assume with negligible error
that the other particle is at most distance $N^\epsilon \ell$ from it and thus
also far from the boundary so that the local density estimate is applicable.
This completes the proof.
\end{proof}

The following lemma is analogous to Lemma~\ref{lem:gYub-torus} for the torus.

\begin{lemma} \label{lem:gYub}
For any $u$,
\begin{align}
  \label{e:gYub1}
  \E^R_V \sum_{i, j} \mathds{1}_{z_i,z_j \in \bigcap D} [g(z_i-z_j) - Y^\ell(z_i - z_j)] 
  &=
  \OO( N^\epsilon N^2 \ell^3b^{-1}),
  \\
    \label{e:gYub2}
  \E^R_V \sum_{i, j} \mathds{1}_{z_i \in B,z_j \in \bigcup D} [\tilde Y(z_i, z_j) - Y^\ell(z_i - z_j)] 
  &= \OO(N^\epsilon N^2 \ell^3) 
  .
\end{align}
\end{lemma}

\begin{proof}
We use the local density for the Yukawa gas, Theorem~\ref{thm:Cdensity},  
stating that balls of radius $r \gg N^{-1/2}$ 
contain $\OO(Nr^2)$ particles with high probability
(provided that the distance to the boundary is at least $b' \gg N^{-1/4}$).
In addition, we use that for $|z_i-z_j| \geq \ell N^\epsilon$
we have $Y^\ell(z_i-z_j) \leq \ee^{-cN^\epsilon}$
so that contributions to the corresponding contributions to the double sums in the statement
contribute lower order errors.
As a consequence, {exactly as in the proof of \eqref{e:gYub1-torus},}
\begin{equation*} 
  \E^R_V \sum_{i, j} [g(z_i-z_j) - Y^\ell(z_i - z_j)] 
  = \OO( N^\epsilon N (N \ell^2) (\ell /b))
\end{equation*}
since each of the at most $N$ particles $z_i$ interacts with $\OO(N^\epsilon N\ell^2)$.
particles $z_j$, and the difference $g-Y^\ell$ is of order $\ell/b$ by Lemma~\ref{lem:gb}~(i).
This proves \eqref{e:gYub1}.

The estimate for the boundary layer \eqref{e:gYub2} is analogous.
Indeed, by definition,
the boundary layer has distance at least $b'$ to the boundary of the support
of the equilibrium measure, so that the local density estimate can still be applied.
Then we similarly have
\begin{equation*} 
  \E^R_V \sum_{i, j} \mathds{1}_{z_i \in B,z_j \in \bigcup D} [\tilde Y(z_i, z_j) - Y^\ell(z_i - z_j)] 
  = \OO((N\ell)(N \ell^2)).
\end{equation*}  
To see that this inequality holds, we note that, up to exponentially small errors, the only pairs we need to consider are  that one particle is in the boundary and the other one is in 
 the bulk with the distance of these two particles of order $\ell$. 
Since, by the local density estimate, with high probability
the total number of particle near boundary  corridor of width $\ell$ is $N \ell$ and  each particle interacts with $N \ell^2$ particles,  
the left side of the last inequality  is of order  $N \ell N \ell^2$.
\end{proof}

To bound the boundary contribution, we will need the following estimate.
For $z \in S_V \setminus D$, recall from \eqref{e:tildeQ} and the
Euler--Lagrange equation \eqref{e:qf-EL} that
\begin{equation*}
  \tilde Q(z) = Q(z) + 2\int_D Y^\ell(z-w) \, \mu_V(\rd w) = 2c_V - 2 \int_B Y^\ell(z-w) \, \mu_V(\rd w)
  ,
\end{equation*}
and define the constant
\begin{align} \label{e:IQB}
  I_{Q,B}
  &=
  \int_{B}  \tilde Q(z) \, \mu_V (\rd z) + \iint_{B^2} Y^\ell  (z-w) \, \mu_V (\rd z) \, \mu_V (\rd w)
  \nonumber\\
  &=
  2c_V \mu_V(B) - \iint_{B^2} Y^\ell  (z-w) \, \mu_V (\rd z) \, \mu_V (\rd w)
  .
\end{align}

\begin{proposition} \label{prop:qf-ub-B}
For any $u$,
\begin{multline} \label{e:TBbd-ub}
  \frac{1}{\beta} \log \int_{(\C \setminus D_u)^{n_{B}}}
  \ee^{-\beta N \sum_{j} \tilde Q (z_j) - \beta \sum_{j\neq k} Y^\ell(z_j-z_k)}
  \, m(\rd \b z)
  \\
  \leq
  - N^2 I_{Q,B} - 2c_V N (n_B - N\mu_V(B_u))
  + \OO(n_B \log N).
\end{multline}
\end{proposition}

\begin{proof}
We fix $u$ and abbreviate $D = D_u$ and $B = B_u$ throughout the proof.
Let
\begin{equation*}
E(n_{B}) = \inf_{\int \omega = n_{B}} \qa{
  N \int_{\C\setminus D} \tilde Q(z) \, \omega(\dd z) + \iint_{(\C\setminus D)^2} Y^\ell(z-w) \, \omega(\dd z) \, \omega(\dd w) },
\end{equation*}
where $\omega$ is a positive measure of total mass $n_{B}$ supported on $\C \setminus D$.
Using the standard technique to replace point particle by a smooth distribution of radius $1/N$, 
the left-hand side of \eqref{e:TBbd-ub} is bounded  above by
\begin{equation} \label{e:TBEB}
  -  E(n_{B}) + \OO(n_{B} \log N).
\end{equation}
(A more sophisticated form of this method will be presented in the proof of Proposition~\ref{prop:Ystep}, where 
the  regularity of the equilibrium measure was  used only in the proof of  the lower bound of the partition function.)

It thus suffices to show that
\begin{equation*}
  E  (n_{B}) - N^2 I_{Q, B}
  \geq 2 c_V N ( n_{B}- N \mu_V(B) ).
\end{equation*}
To do so, with $\tilde \omega = \omega - N\mu_V$ inside the infimum, we write
\begin{align*}
&  E (n_{B}) - N^2 I_{Q,B}\\
& = \inf_{\int \omega = n_{B} }
\bigg[ N   \int_{D^c} \tilde Q (z) \, \omega (\rd z)  +   \iint_{(D^c)^2} Y^\ell (z-w) \,\omega (\rd z) \,\omega (\rd w) \bigg] \\
& \qquad  -  N^2  \int_{{D^c}} \, \tilde Q(z) \, \mu_V (\rd z) -  N^2 \iint_{(D^c)^2} Y^\ell  (z-w) \, \mu_V (\rd z) \, \mu_V (\rd w)\\
& = \inf_{\int \omega = n_{B} }  \bigg[ N  \int_{D^c} \tilde \omega (\rd z) \, \Big [ \tilde Q (z) +   2\int_{D^c} Y^\ell (z-w)  \, \mu_V (\rd w) \Big ]
+  \int_{D^c} Y^\ell (z-w) \, \tilde \omega (\rd z) \, \tilde\omega (\rd w) \bigg]
  .
\end{align*} 
The last term on the right-hand side is nonnegative, and can therefore be dropped.
By definition of $\tilde Q$ and the Euler--Lagrange equation \eqref{e:qf-EL}, also
\begin{equation*}
  \tilde Q (z) +   2 \int_{D^c} Y^\ell (z-w) \, \mu_V (\rd w)
  = Q (z) +   2 \int Y^\ell (z-w) \, \mu_V (\rd w)
  \geq 2c_V
  .
\end{equation*}
Since the same relation holds with equality on the support of $\mu_V$,
therefore
\begin{equation*}
  E (n_{B}) - N^2 I_{Q,B} \geq
  2c_V N  \int_{D^c}  \tilde \omega (\rd z)
  = 2 c_V N ( n_{B}- N\mu_V(B) ).
\end{equation*} 
This completes the proof.
\end{proof}

\begin{proof}[Proof of Proposition~\ref{prop:qf-ub}]
Summing over the possible particle profiles, we have
\begin{equation*}
  \int  \ee^{- \beta  \tilde   H^{\ell}_u(\b z) } \, m(\rd\b z) 
  =
  \sum_{\b n} \binom{N}{\b n} \int  \ee^{- \beta  \tilde   H^{\ell}_u(\b z) } \, \mathds{1}_{\b n(\b z) = \b n} \, m(\rd\b z),
\end{equation*}
where $\b n(\b z)$ is the particle profile of the configuration $\b z \in \C^N$.
By definition of $\tilde H$, for any $u$,
the integral on the right-hand side factorizes as
\begin{equation*}
  \pa{\prod_{\alpha\subset D} \int_{\alpha^{n_\alpha}} \ee^{-\beta \hat H_\alpha(\b z)} \, m(\rd z)}
  \times
  \pa{\int_{(\C \setminus D)^{n_{B}}}
  \ee^{-\beta N \sum_{j} \tilde Q (z_j) - \beta \sum_{j\neq k} Y^\ell(z_j-z_k)} \, m(\rd \b z)}.
\end{equation*}
The claim now follows from Propositions~\ref{prop:qf-ub-1} and \ref{prop:qf-ub-B}.
\end{proof}

\subsection{Summary}

We summarize some of the key facts  used in the proof of the upper bound of the partition function of the  Yukawa gas:

\begin{enumerate}

\item The local densities are bounded at the scale $\ell$ of the interaction. 

\item The solution of the ground state is regular  in terms of derivatives of $\rho_V$; this is reflected  in the estimate  \eqref{e:qf-ub-1}.

\item We used the independent particle approximation for particles within distance $b'$ to the boundary of the support of the equilibrium measure.    In order to control the error due to the 
interactions between 
boundary particles and bulk particles, we used  that  the local density at the scale $\ell$  for particles at a  distance of order $b'$ to the boundary
is bounded.
\end{enumerate}

\section{Decoupling estimate}
\label{sec:decoupling}

The proof for the lower bound on the partition function, Proposition~\ref{prop:qf-lb}, 
will be presented in Section~\ref{sec:quasifreelb}. This proof 
is  based as on a trial state similar to the one used  in Section~\ref{sec:qf-lb-torus} for the torus case.
Notice that the Yukawa potential has range $R$ in the current setting instead of $\ell$ 
in the torus case.
Since our grid size $b$ satisfies $\ell \ll b \ll R$, many error terms which are negligible  in the torus case need  now to be 
estimated carefully. In particular, 
the embedding map $\Psi$ has to be chosen differently from the simple average used in \eqref{e:Psidef-torus}--\eqref{e:Psidef-choice-torus}.
In preparation of Section~\ref{sec:quasifreelb}, we construct this choice in Proposition~\ref{prop:decoupling-1} below.
We call it a decoupling estimate because it allows us  to pass from the original Yukawa gas to
the quasi-free Yukawa gas in which cubes are decoupled.
By rescaling, we state the  estimates for the Yukawa gas on the unit torus.

More precisely, the next proposition asserts the existence of a random choice of maps
\begin{equation}
  \Psi : \T \to [-1/2,1/2)^2 \subset \R^2 \qquad \text{with Jacobian $|\rd \Psi|=1$,}
\end{equation}
such that the estimates stated in the proposition hold.
Here  $\T$ is the unit torus.
The expectation corresponding to the randomness defining the maps $\Psi$ 
is denoted by $\E^{\Psi}$ (and is independent of everything else).
In the statement of the estimate in the following proposition,
$U^\gamma$ is the Yukawa potential on the unit torus \eqref{e:Uell},
$Y^\gamma$ is the Yukawa potential on the plane \eqref{Yg},
but $\E^\gamma$ denotes the expectation of the Yukawa gas with $N$ particles and range $\gamma$ on the unit torus.
As usually, we also denote $\hat\mu$ the empirical measure and $\tilde \mu = \hat\mu-m$
where $m$ is the uniform probability measure on $\T$. In the statement below and this section,
it is understood that all double integrals are evaluated on $\{z\neq w\}$.

\begin{proposition} \label{prop:decoupling-1}
  Assume that $N^{-1/4} \ll \gamma \ll 1$ and $\gamma\leq R$.
  Let $\E^\gamma$ denote the expectation of the $N$-particle Yukawa gas of range $\gamma$ on the unit torus $\T$.
  There is a random choice of $\Psi: \T \to [-1/2,1/2)^2$ with $|\rd \Psi|=1$ such that
  \begin{align}
      \label{e:decoupling-1-UY}
  N^2\E^\Psi \E^\gamma  \iint_{\T \times \T}
  (U^\gamma(v-w) - Y^\gamma(\Psi(v)-\Psi(w)) )
  \, \tilde \mu(\dd v) \, \tilde \mu(\dd w)
  =
  N^{\varepsilon}\OO(N^{4/5}/\gamma^{2/5}+\gamma^{-4}),\\
    \label{e:decoupling-1-YY}
  N^2 \E^\Psi \E^\gamma \iint_{\T \times \T} (Y^R (\Psi(v)-\Psi(w))-Y^\gamma (\Psi(v)-\Psi(w))) \, \tilde\mu(\rd v) \, \tilde\mu(\rd w)
  =
  N^{\varepsilon}\OO(N^{4/5}/\gamma^{2/5}+\gamma^{-4}).
  \end{align}
\end{proposition}

The remainder of this section is devoted to the proof of this proposition.
The main reason to introduce randomness into  $\Psi$ is to resolve the issue that the torus distance and Euclidean distance are incompatible.  
The  range of the Yukawa interaction  $\ell$, appearing in the quasi-free gas, is small. On the other hand, we wish to use it to approximate 
the Coulomb energy which corresponds to $Y^R$ with $R\gg 1$.
The Coulomb interaction will be pushed back to the torus;  this creates discontinuities since the torus is periodic.
The naive embedding of the square onto the torus used in Section~\ref{sec:qf-def-torus}
is discontinuous along a horizontal and a vertical line.
This discontinuity can be averaged out  using the translational invariance of the torus,
but the resulting interaction on the torus is still not smooth enough to apply the  rigidity estimate.
Therefore we now choose $\Psi$ to involve a more sophisticated average than the simple mean over
the discontinuity lines so as to  have a smooth interaction 
after pushing back the Coulomb interaction to the torus. 

\bigskip
In Section~\ref{sec:quasifreelb}, we will apply this estimate with the unit torus $\T$ rescaled to the torus $\T^{(b)}$
of side length $b$. For later reference, we state the rescaled version below.

\begin{corollary} \label{cor:decoupling}
  Let $\ell \ll b \ll 1$, $\ell\leq R$, and assume that $\gamma := \ell/b$ satisfies $n^{-1/4} \ll \gamma \ll 1$.
  Let $\E^\ell_b$ denote the expectation of the $n$-particle Yukawa gas of range $\ell$ on the torus $\T^{(b)}$.
  Then there is a random choice of $\Psi = \Psi^{(b)}: \T^{(b)} \to [-b/2,b/2)^2$ with $|\rd \Psi|=1$ such that
  \begin{align}
  \label{e:decoupling-UY}
  n^2\E^\Psi \E^\ell_b  \iint_{\T^{(b)} \times \T^{(b)}}
  (U_b^\ell(v-w) - Y^\ell(\Psi(v)-\Psi(w)) )
  \, \tilde \mu(\dd v) \, \tilde \mu(\dd w)
  =
  n^{\varepsilon}\OO(n ^{4/5}/\gamma^{2/5}+\gamma^{-4}),
  \\
  \label{e:decoupling-YY}
  n^2 \E^\Psi \E^\ell_b \iint_{\T^{(b)} \times \T^{(b)}} (Y^R (\Psi(v)-\Psi(w))-Y^\ell (\Psi(v)-\Psi(w))) \, \tilde\mu(\rd v) \, \tilde\mu(\rd w)
  =
  n^{\varepsilon}\OO(n ^{4/5}/\gamma^{2/5}+\gamma^{-4}).
  \end{align}
\end{corollary}

\begin{proof}
The corollary is immediate from Proposition~\ref{prop:decoupling-1} by rescaling.
\end{proof}

\subsection{Choice of the maps $\Psi_\alpha$}\label{sub:psi}
 \nomOth[13]{$\Psi$}{A distorted map defined along subsection \ref{sub:psi}}

To define the maps $\Psi$, we define $[u]$ through
\begin{equation}\label{eqn:modulo}
  -\frac{1}{2}\leq [u]<\frac{1}{2},\ u-[u]\in  \mathbb{Z} \quad \text{for $u \in \R$},
  \qquad
  [z]=([z_1],[z_2])\in\mathbb{T} \quad \text{for $z \in \C \cong \R^2$}.
\end{equation}
Then we define maps $\Phi_1,\Phi_2:\mathbb{T}\to\mathbb{T}$ by
\begin{equation} \label{eqn:Philbdef}
 \Phi_1(z) = ([z_1+m_1s(z_2)], z_2),\quad \Phi_2(z) = (z_1, [z_2+m_2 s(z_1)]), 
\end{equation}
where we will choose $ s(z)=\sin(2\pi x)$ (or any smooth periodic function with similar oscillation).
Let  $\Phi = \Phi_1\circ \Phi_2$. 
We choose $m_1, m_2$ as independent random variables with the distribution of
$tX$ with $X$ a random variable with smooth and compactly supported density, $\E(X)=0$, and $N^{-1/2}\ll t\ll 1$ is some mesoscopic scale.
Eventually, we will choose 
\begin{equation}\label{eqn:t}t=N^{-1/4}.\end{equation}

Finally, let $\Psi_z= [\Phi(z) + (a_1,a_2)]$, where  $(a_1,a_2)$ is a random shift, with $a_1$ and $a_2$ independent and uniform on $[-1/2,1/2)$.
Note that $\Phi$ and $\Psi$ are smooth function on the torus and they preserve volumes: 
\begin{equation}\label{eqn:Jacobian}
|\rd\Phi|=|\rd\Psi|=1.
\end{equation}

\subsection{From euclidean to periodic interaction}
\label{sec:average}

All terms we need to bound can be written as in the left-hand side of (\ref{e:Grigi}), so  Proposition~\ref{prop:Grigi} will be our main tool.
However, these terms involve
interactions for the Euclidean distance on the square while Proposition \ref{prop:decoupling-1} applies to the unit torus.
We therefore first need the next Lemma \ref{lem:toTorus} to turn the Euclidean interaction into a periodic one;
subsequently, we decompose the resulting singularities carefully. For the lemma, we first need the following definition of an average of interaction over translations.

\begin{definition}\label{dfn:averageInteraction}
For any $G:\mathbb{T}^2\to\mathbb{R}$ and $h\in\mathbb{C}$, we define
$$\cal T_G(h)= \int_{\mathbb{T}} G(z,[z+h]) \, m(\rd z),$$
\nomInt[040]{$\cal T_G$}{Interaction $G$ averaged over translations of the unit torus}
where $m$ is the Lebesgue measure on $\mathbb{T}$ and we used the notation (\ref{eqn:modulo}).

If $G(z,w)=g(|z-w|)$ is a function of the Euclidean distance, $\cal T_G$ will also be denoted 
by $\cal T_g$ (and is obviously equal to  $\cal T_g (h) = \int_{\mathbb{T}} g(|[z+h]-z|) m(\rd z)$). 
\end{definition}

 We remark that in the above definition and below,  $z - w$ for $(z, w) \in \T\times \T$ is defined
as the difference of two elements in  $\C^2$  through the identification of $\T = [-1/2, 1/2)^2$.

\begin{lemma} \label{lem:toTorus}
Consider a Yukawa gas on the unit torus $\mathbb{T}$, $G:\mathbb{T}^2\to\mathbb{R}$, and assume all integrands below are integrable. 
The following holds:
\begin{equation}\label{eqn:lalalolo}
    \E \iint_{z\neq w} G(z,w) \, \hat \mu(\rd z) \, \hat \mu(\rd w)
        = \E \iint_{z\neq w} \cal T_G ([z - w]) \, \hat \mu(\rd z) \, \hat \mu(\rd w).
\end{equation}
Moreover, if $G(z,w)=g(|z-w|)$ is a function of the Euclidean distance, for any $h=(h_1,h_2)\in\mathbb{T}$ we have 
\begin{align}
  \cal T_g (h)
  &=   (1-|h_1|)(1-|h_2|) g_1(h)  
    + |h_1| (1-|h_2|) g_2 (h)
    \notag   \\
  & \qquad
    + |h_2| (1-|h_1|) g_3(h) + |h_1| |h_2| g_4(h), \label{eqn:torusForm}
\end{align}
where
\begin{alignat}{2}
 g_1(h) &:= g (\sqrt{|h_1|^2+|h_2|^2}), &\qquad g_2 (h) &:= g(\sqrt{(1-|h_1|)^2+|h_2|^2}), \notag\\
 g_3(h) &:= g(\sqrt{|h_1|^2+(1-|h_2|)^2}), &\qquad g_4(h) &:= g(\sqrt{(1-|h_1|)^2+(1-|h_2|)^2}).\label{eqn:g}
\end{alignat} 
\end{lemma}

\begin{remark}\label{rem:smooth}
The above calculation is stated for $h\in\mathbb{T}$, and it shows that $\cal T_G  = \cal T_g$ is not smooth for $h_1=0$ or $h_2=0$. This non-smoothness prevents us from using the rigidity estimate \eqref{e:Yrigi} and is the main source of difficulty we will address in this section.  In addition to the non-smoothness  for $h_1=0$ or $h_2=0$, 
one  may wonder if $\cal T_G$ has additional singularities (i.e., non-smoothness) at $h_1=\pm1/2$ or $h_2=\pm1/2$, as a function on the torus. It has not, as shown by the following argument.
Assume $-1/2\leq h_2<1/2$ is fixed. The right-hand side of  (\ref{eqn:torusForm}) admits an obvious smooth extension to $h_1\in(0,1)$, called $\tilde {\cal T}_G$.
One readily sees that for such $h_1\in(0,1)$, we have
$\tilde {\cal T}_G(h_1,h_2)=\tilde {\cal T}_G(1-h_1,h_2)$: $\tilde {\cal T}_G$ is smooth and symmetric with respect to $h_1=1/2$, so all its odd derivatives vanish there,
meaning ${\cal T}_G$ is smooth at $h_1=\pm1/2$. The same reasoning applies on  $h_2=\pm 1/2$.
\end{remark}

\begin{proof}
Recall that  $\rho_2$ is the two point correlation function for the Yukawa gas on $\mathbb{T}$. By translation invariance of the distribution of the Yukawa gas, we have
\begin{align*}
& \E \iint_{z\neq w} G(z,w) \hat \mu(\rd z) \hat \mu(\rd w)
  = \iint_{z\neq w} G(z,w) \rho_2([z-w]) m(\rd z) m(\rd w)
  \\
& = \iint G(z,[z+h]) \rho_2(h) m(\rd z) m(\rd h)
  = \int \rho_2(h) \left( \int G(z,[z+h]) m(\rd z) \right) m(\rd h)
  \\
& = \int \rho_2(h)  \cal T_G(h) m(\rd h)
  = \iint \rho_2(h)  \cal T_G(h) m(\rd h) m(\rd \tilde z)
  = \E \iint_{z\neq w}  \cal T_G([\tilde z - \tilde w]) \hat \mu(\rd \tilde z) \hat \mu(\rd \tilde w).
  \end{align*}
In the case $G(z,w)=g(|z-w|)$, the assertion follows from a direct calculation of $ \cal T_G(h)=\int_{\mathbb{T}} g(|[z+h]-z|) m(\rd z)$.
\end{proof}

Denote by $\mathbb{E}_{(a,b)}$ integration with respect to the shift $(a,b)$ of $\Psi$,
and write 
\begin{equation}\label{Delta}
\Delta_z^w=[\Phi_z-\Phi_w].
\end{equation}
\nomOth[04]{$\Delta_z^w$}{increments of the map $\Phi$, $\Delta_z^w=[\Phi_z-\Phi_w]$}
Then the functional $\cal T$ from Definition \ref{dfn:averageInteraction} naturally appears in the following calculation:
$$
\E_{(a,b)} \left(G(\Psi_z,\Psi_w) \right)
=
  \int_{\mathbb{T}} G([\Psi_z+\tilde z],[\Psi_w+\tilde z]) \, m(\rd \tilde z)
=
\int_{\mathbb{T}}G([\tilde z+\Delta_z^w],\tilde z) \, m(\rd \tilde z)
=  \cal T_{G} (\Delta_z^w).
$$
In particular, 
\begin{align}\label{eqn:useful}
\E_{(a,b)} \left(g ( |\Psi_z-  \Psi_w|) \right)
=  \cal T_{g} (\Delta_z^w).
\end{align}
This will be useful in the following proof of Proposition \ref{prop:decoupling-1}.

\subsection{Proof of estimate (\ref{e:decoupling-1-UY})} 
First note that, by (\ref{e:Uell}) we have $U^{\gamma}(z-w)=Y^{\gamma}([z-w])+\OO(\ee^{-N^c})$,
so that it will be sufficient to prove both of the following estimates:
\begin{align}
\E^\Psi \E^\gamma N^2 \iint (Y^\gamma([z-w]) - Y^\gamma([\Psi_z-\Psi_w]) )
  \, \tilde \mu(\dd z) \, \tilde \mu(\dd w) \label{eqn:int1}
 &= \OO(N^{1+\varepsilon}t),\\
\E^\Psi \E^\gamma N^2 \iint ( Y^\gamma([\Psi_z-\Psi_w])- Y^\gamma(\Psi_z-\Psi_w))
  \, \tilde \mu(\dd z) \, \tilde \mu(\dd w) \label{eqn:int2}
 &= \OO(N^{\varepsilon})\left(\frac{1}{\gamma^4}+\frac{\sqrt{N}}{t}\right).
\end{align}
From (\ref{eqn:t}) and the hypothesis $\gamma\ll 1$, the sum of both error terms above is dominated by the right hand side of (\ref{e:decoupling-1-YY}).

For the proof of (\ref{eqn:int1}), let $N^{-1/2}\ll r\ll \gamma$ be some intermediate scale. Let  $\chi:\mathbb{R}_+\to[0,1]$ be a smooth function such that $\chi(z)=1$ on $[0,1]$, $\chi(z)=0$ on $[2,\infty)$  and define $q=Y^{\gamma}$, $\widetilde q(z)=q(z)\chi(|z|/r)$.
The proof of (\ref{eqn:int1}) will consist of the following two estimates (note that $[\Psi_z-\Psi_w]=[\Phi_z-\Phi_w]$):  
\begin{align}
\label{e:UYY1}
  \E^\Phi    \E^{\gamma}  N^2 \iint ( \widetilde q([\Phi_z- \Phi_w]) - \widetilde q([z- w]))  \;  \wt \mu(\dd z) \, \wt  \mu(\dd w)
 &=
   N^{\varepsilon}\,  \OO\left(N^2 t^2 r^2\right),
  \\
      \label{e:UYY}
  \E^\Phi    \E^{\gamma}  N^2 \iint ((q-\widetilde q)([\Phi_z- \Phi_w]) - (q-\widetilde q)([z- w]))  \; \wt \mu(\dd z) \, \wt \mu(\dd w)
 &=
N^{\varepsilon}\,  \OO\left(\frac{1}{r^2}\right).
\end{align}
Optimization over $r$ shows that  (\ref{eqn:int1}) holds
(note that the optimum $r^*=(Nt)^{-1/2}$ is smaller than $\gamma$ because $t=N^{-\frac{1}{4}}<\gamma$.

For (\ref{e:UYY1}), we
proceeds by Taylor expansion
around  $|z-w+a|<N^{\varepsilon}r$, where
$a=\pm 1,\pm {\rm i}$. We treat the case $a=0$, the other ones being identical.
As $\tilde q$ is supported on $|x|<N^{\varepsilon}r$, for all $z,w$ contributing to (\ref{e:UYY1})
we have $[z-w]=z-w$ and  $[\Phi_z-\Phi_w]=\Phi_z-\Phi_w$.
For such $z,w$, from the definition $\Phi=\Phi_1 \circ \Phi_2$ with \eqref{eqn:Philbdef},
we have
\be\label{hu}
[\Phi_w-\Phi_z]
=
[w-z]+
\left(
\begin{array}{l}
m_1\left(s(w_2+m_2s( w_1))-s(z_2+m_2s(z_1))\right) \\
m_2\left(s(w_1)-s(z_1)\right)
\end{array}
\right).
\eeq
Expanding \eqref{hu},
\begin{align*}
\frac{([\Phi_w-\Phi_z])_2-([w-z])_2}{m_2}
&=s'(w_1)(w_1-z_1) +\OO(|w-z|^2),
\\
\frac{([\Phi_w-\Phi_z])_1-([w-z])_1}{m_1}
&=\,s'(w_2)(w_2-z_2)+\OO(|w-z|^2)+m_2\, \OO(|w-z|)+m_2^2\, \OO(|w-z|^2), 
\end{align*}
where,  here and in the following,  the $\OO$ error terms are non-random, namely,  
they do not depend on  $m_1,m_2$.
Denoting
$$
\Delta=
\left(
\begin{array}{l}
m_1\left(s'(w_2)(w_2-z_2)+\OO(|w-z|^2)+m_2\, \OO(|w-z|)+m_2^2\, \OO(|w-z|^2)\right)\\
m_2\left(\,s'(w_1)(w_1-z_1)+\OO(|z-w|^2)\right)
\end{array}
\right),
$$
we  have
$$
\widetilde q([\Phi_z- \Phi_w])-\widetilde q ([z- w])
=
\nabla \widetilde q([z-w])\cdot\Delta
+\OO\left(\sup_{[|z-w|/2,2|z-w|]}|\nabla^2 \widetilde q|\right)|\Delta|^2.
$$
As $m_1,m_2$ are centered (under the random choice of $\Phi$),  the linear terms vanish under expectation.
This gives, for any fixed small $\varepsilon>0$,
\begin{equation*}
\E^{\Phi}\left(\widetilde q ([\Phi_z- \Phi_w])-\widetilde q ([z- w])\right)
=
\E(m^2)
|(\nabla^2 \widetilde q)(z-w)|\OO(|z-w|^2)
=
\OO(N^{2}t^2)\mathds{1}_{|z-w|\leq N^\varepsilon r}.
\end{equation*}
We have  therefore proved that 
\begin{equation}\label{eqn:final}
 \E^\Phi    \E^{\gamma}  N^2 \iint  (\widetilde q([\Phi_z- \Phi_w]) - \widetilde q ([z- w]))  \;  \wt \mu(\dd z) \, \wt  \mu(\dd w)
 \leq
 \OO(N^2t^2)\E\left(\sup_{z\in\mathbb{T}}\hat\mu (\{w:|[z-w]|\leq N^\varepsilon r\})\right).
\end{equation}
From the local density estimate for the Yukawa gas on the torus implied by Theorem \ref{thm:YTdensity},
the above parenthesis is bounded by $(N^\epsilon r)^2$ with high probability.
We have therefore proved  (\ref{e:UYY1}).

Equation (\ref{e:UYY}) is a consequence of Proposition~\ref{prop:Grigi}.
Indeed, let $(\chi_k)_{k\geq 1}$ be a partition of unity in the sense that $\sum_k\chi_k(x)=1$ for any $x>r$, $\chi_k$ is supported on $[2^{k-1}r,2^{k+1}r]$, and 
$\|\chi_k^{(n)}\|_{\infty}\leq C_n (2^k r)^{-n}$. We denote $f=q-\wt q$ and apply Proposition \ref{prop:Grigi} to $G_k(z,w)=f([z- w])\chi_k(|[z-w]|)$ and $s=s_k=N^{-\varepsilon}2^kr$, for some fixed small  $\varepsilon>0$.
For any $k$ such that $2^k r<\gamma N^{\varepsilon}$,
we have
$$
|\nabla^jG_k(x,y)|=\OO\left(|[x-y]|^{-j}\mathds{1}_{|[x-y]|\in[2^{k-2}r,2^{k+2}r]}\right)
$$
and the same estimate holds for $(G_k)_{{\rm B}_s}^{(j)}(x,y)$, defined in (\ref{eqn:nablaj}). Proposition \ref{prop:Grigi}  gives
$$
N^2 \iint  G_k (z, w) \, \tilde\mu(\rd z) \, \tilde\mu(\rd w)
=  \OO(N^\varepsilon)  \left( \frac {1} {s_k^4}
\sum_{j=0}^{p-1} s_k^j (s_kN^{\varepsilon})^{2-j}+
N^2 s_k^p (s_kN^{\varepsilon})^{2-p}\right)
=
\OO\left(\frac{N^{3\varepsilon}}{s_k^2}\right),
$$
where we chose $p=\lfloor10/\varepsilon\rfloor$. Summation of the above estimate over $1\leq k\leq \log N$ gives
\begin{equation}\label{eqn:int44}
N^2 \iint  G (z, w) \, \tilde\mu(\rd z) \, \tilde\mu(\rd w)
= 
\OO\left(\frac{N^{5\varepsilon}}{r^2}\right),
\end{equation}
which proves (\ref{e:UYY}) for the term involving $f([z-w])$. The same estimate holds for the integral of $f([\Phi_z-\Phi_ w])$, because for any fixed $m_1,m_2=\OO(t)$ the function
$(z,w)\mapsto f([\Phi_z-\Phi_w])$ has the same regularity properties as  $(z,w)\mapsto f([z-w])$ ($\Phi={\rm Id}+t\phi$ for some function $\phi$ smooth on a scale $1$). This concludes the proof of (\ref{eqn:int1}).\\

\medskip 
For equation (\ref{eqn:int2}), we first consider the averaging in the shift $(a,b)$ from $\Psi$: denoting  $|h|=(|h_1|,|h_2|)$,  
for any $h\in\mathbb{T}$ we have
\begin{align*}
\E_{(a,b)} \left(Y^\gamma([[h+(a,b)]-(a,b)]) \right)
&=Y^{\gamma}(|h|),\\
\E_{(a,b)} \left(Y^\gamma([h+(a,b)]-(a,b)) \right)
&=
(1-|h_1|)(1-|h_2|) Y^{\gamma}(|h|)
+
\OO(\ee^{-N^c}),
\end{align*}
where the second equation comes from (\ref{eqn:useful})
and the fact that $Y^\gamma$ is essentially supported on $|x|<N^{\varepsilon}\gamma$. Equation (\ref{eqn:int2}) is therefore equivalent to
$$
\E^\Phi \E^\gamma N^2 \iint \left(|(\Delta_z^w)_1+|(\Delta_z^w)_2|-|(\Delta_z^w)_1||(\Delta_z^w)_2|\right)Y^\gamma (|\Delta_z^w|)
  \, \tilde \mu(\dd z) \, \tilde \mu(\dd w)
 = \OO(N^{\varepsilon})\left(\frac{1}{\gamma^4}+\frac{\sqrt{N}}{t}\right).
$$
On the left  side of the above equation, we would like to calculate the interaction after the $\Phi$-averaging.
However, this expression is not a function of $[z-w]$, which  would be convenient for our proof. We therefore perform an additional averaging over the torus: the above equation is equivalent to
\begin{equation}\label{eqn:needsmall}
\E^\gamma N^2 \iint K^t([z-w]) \, \tilde \mu(\dd z)  \, \tilde \mu(\dd w)
 = \OO(N^{\varepsilon})\left(\frac{1}{\gamma^4}+\frac{\sqrt{N}}{t}\right)
\end{equation}
where $K^t(h)=\E^{\Phi} \int f(\Delta_{v}^{[v+h]}) m(\rd v)$ and $f(h)=(|h_1|+|h_2|-|h_1h_2|)Y^\gamma(h)$.

The proof of (\ref{eqn:needsmall}) is  delicate, so before giving the technical details,
we  list below  the main difficulties and ingredients.
\begin{enumerate}
\item The function $K^0$ is smooth on $\mathbb{T}$ except on $h_1=0$ or $h_2=0$, as explained in Remark~\ref{rem:smooth}.
  This prevents a direct application of Proposition \ref{prop:Grigi} and  is the motivation for our averaging over $\Phi$.
\item The function $K^t$ now gained some smoothness in neighborhoods of $h_1=0$ and $h_2=0$ thanks to the convolution with the distribution of $tX$. For example, around  $h_1=0$,
$K^t$ is smooth on a scale $|th_2|$: for $k\geq 1$, $\partial_{h_1}^kK^t(h)=\OO(|h_1| |th_2|^{-k+1})$, thanks to the definition of $\Phi_1$ in (\ref{eqn:Philbdef}) and 
the asymptotics $tX s(h_2)\sim 2\pi t Xh_2$. Proposition \ref{prop:Grigi} can now be applied for the function $K^t$, after some surgery removing some small singular set corresponding to the cross 
$\{|th_2|\leq N^{-1/2}, h_1\leq N^{-1/2}\}\cup \{|th_1|\leq N^{-1/2}, h_2\leq N^{-1/2}\}$.
\end{enumerate}

We now implement the above outline. The function $K^t$ is a linear combination of the
terms
\begin{equation}\label{eqn:defF}
  F_j^t(h)=\E^{\Phi} \int f_j(\Delta_{v}^{[v+h]}) m(\rd v),
\end{equation}
where $\Delta_{v}^{[v+h]}$ is defined in \eqref{Delta} and 
\begin{equation}\label{eqn:deff}
f_1(h)=|h_1|Y^{\gamma}(h),\quad f_2(h)=|h_2|Y^{\gamma}(h), f_3(h)=|h_1h_2|Y^{\gamma}(h).
\end{equation}

\begin{figure}[h] \centering
  \includegraphics[width=5cm]{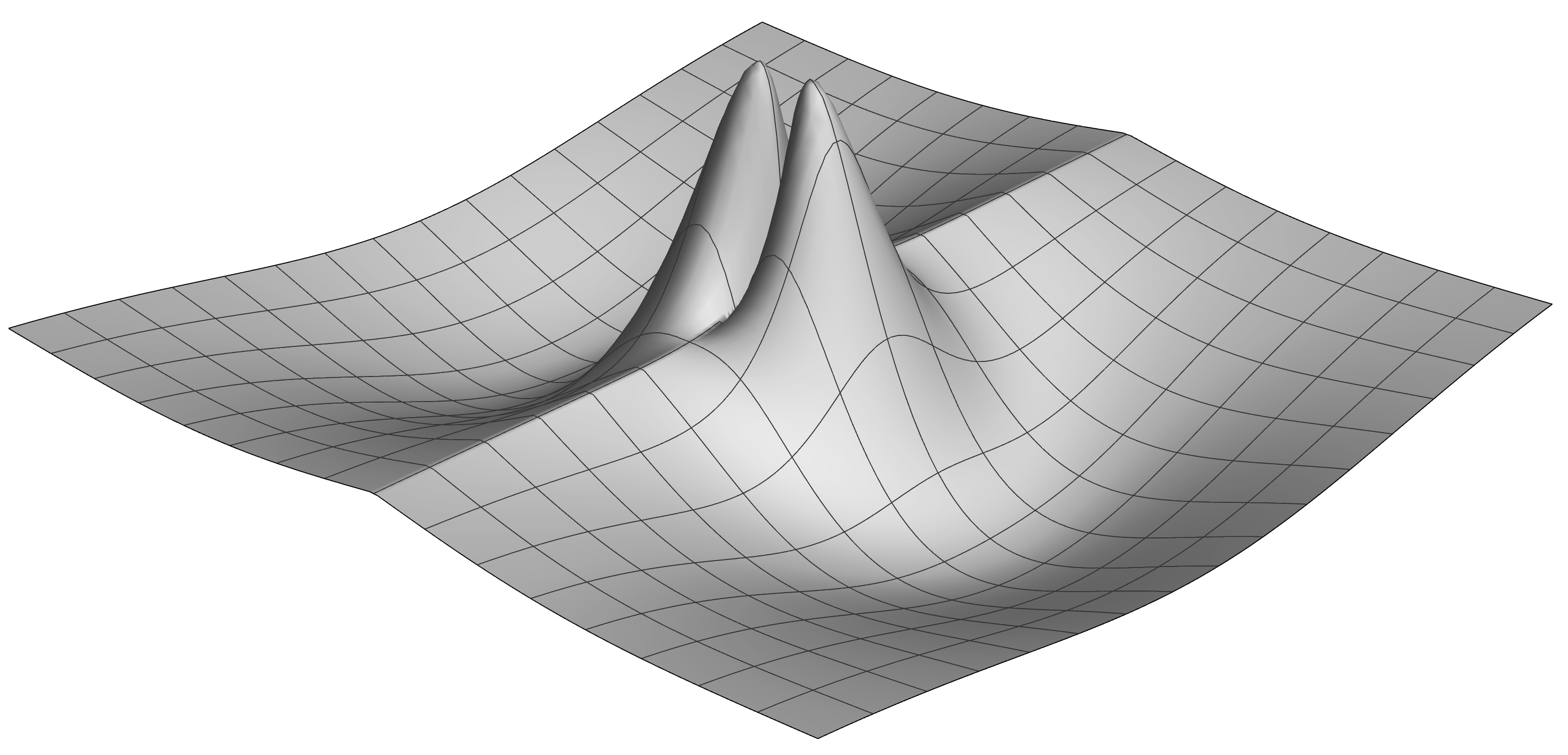}
  \hspace{1cm}
  \includegraphics[width=5cm]{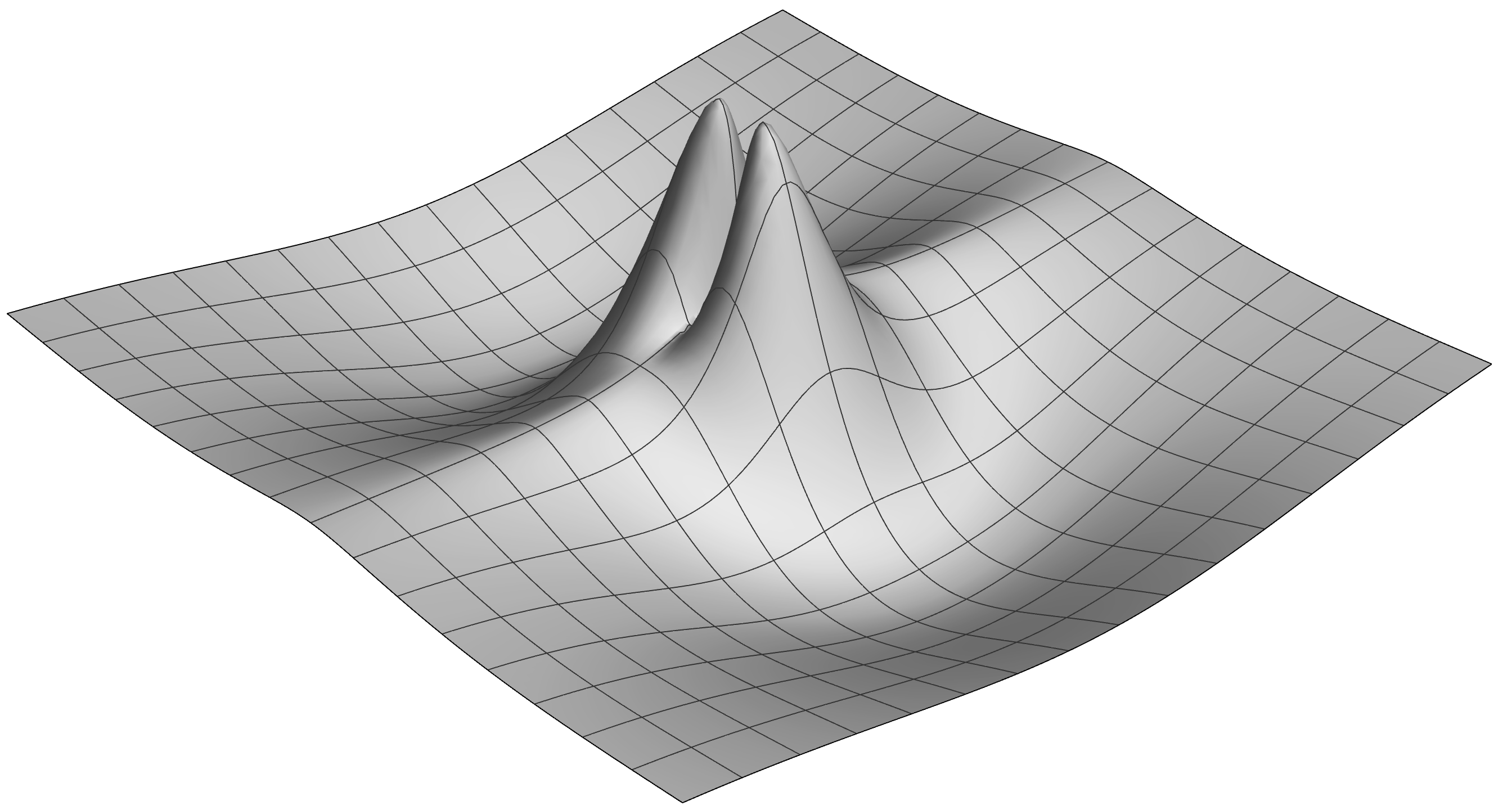}
  \caption{The functions $F_1^0$ (left) and $F_1^t$(right) in $[-1/5,1/5]^2$, $\gamma=t=1/20$}
\end{figure}

We first bound the contribution from $F_1^t$ (and therefore $F_2^t$ by a similar argument), by exploiting its smoothness properties. 
Thus a calculation on $\{h_1\neq 0\}$ gives 
\begin{align}
&|\partial_{h_1}F_1^0(h)|\leq C\left(\frac{|h_1|}{|h|}+|Y^{\gamma}(h)|\right)\notag,\\
&|\partial_{h_2}F_1^0(h)|\leq C\frac{|h_1|}{|h|},\notag\\
&|\partial_{h_1}^{k_1}\partial_{h_2}^{k_2}F_1^0(h)|\leq\frac{ C_{k_1,k_2}}{|h|^{k_1+k_2-1}}, \quad 
{\rm if} \; k_1+k_2\geq 2.\label{eqn:partder0}
\end{align}
Here we have used $|\partial_{h_1}^{k_1}\partial_{h_2}^{k_1}Y^{\gamma}(h)|\leq C_{k_1,k_2}|h|^{-(k_1+k_2)}$ and $[\Phi_w-\Phi_z]=[w-z]$ when $t=0$.

Recall that $\Phi$ is defined in \eqref{eqn:Philbdef} with  $m_1, m_2$ given by  independent random variables with  distribution $tX$ with $X$ a compactly supported  random variable of order one. 
Thus $\Delta_{v}^{[v+h]} \sim h $ with a  distortion of order $t$, which  makes the function 
 $|(\Delta_{v}^{[v+h]})_1| $
differentiable on $\{h_1=0\}$. With this in mind
and  explicitly writing $F_1^t$ with the definition (\ref{hu}),  
a simple calculation and  the Taylor expansion extends  the estimate \eqref{eqn:partder0}
to 

\begin{align}\label{eqn:partder}
&|\partial_{h_1}F_1^t(h)|\leq C\left(\frac{|h_1|+|th_2|}{|h|}+|Y^{\gamma}(h)|\right)\notag,\\
&|\partial_{h_2}F_1^t(h)|\leq C\frac{|h_1|+|th_2|}{|h|},\notag\\
&|\partial_{h_1}^{k_1}\partial_{h_2}^{k_2}F_1^t(h)|\leq C_{k_1,k_2}\left( \frac{1}{|h|^{k_1+k_2-1}}+\frac{1}{|t h_2|^{k_1-1}|h_2|^{k_2}}\mathds{1}_{|h_1|<t|h_2|,k_1\geq 1}\right)\ (k_1+k_2\geq 2).
\end{align}

For some mesoscopic scale $N^{-1/2+c}\leq r\ll t$, define
a partition of unity  $\mathds{1}_{[0,1]}(x)=\sum_{i=0}^n\wt\chi_i$ where $n$ is of order $\log N$,
$\wt \chi_0$ is supported on $[0,2r]$, $\chi_i$ is supported on $[2^{i-1}r, 2^{i+1}r]$,
and $\|\chi_i^{(m)}\|_\infty\leq C_m(2^{i}r)^{-m}$.
We define $F^t_{ij}(h)=F_1^t(h)\wt \chi_{|i|}(|h_1|)\wt \chi_{|j|}(|h_2|)$ for $|i|,|j|\leq n$.
By symmetry, we only need to bound each $\iint F_{ij}^t$ in one quadrant: we now assume $0\leq i,j\leq n$.

First, for $i=j=0$ (in fact for $i+j$ bounded), the local density estimate and $\|F^{t}_{ij}\|_\infty = \OO(r^{1+\varepsilon})$ give
\begin{equation}\label{eqn:F_ijclose}
N^2\E^{\gamma}\iint F^{t}_{ij}([z-w]) \,\tilde\mu(\rd z) \, \tilde \mu(\rd w)=\OO(N^\varepsilon)N^2 r^3.
\end{equation}
We now assume $i+j>0$.

For $2^{i} > t2^{j}$ (in other words $|h_1|>t|h_2|$),  (\ref{eqn:partder}) yields
$$
|\partial_{h_1}^{k_1}\partial_{h_2}^{k_2}F_{ij}^t(h)|\leq C_{k_1,k_2} \frac{1}{\max(2^{i} r,2^{j} r)^{k_1+k_2-1}}.
$$
The area of the support of $F_{ij}$ is $\OO(r^22^{i+j})$. We proceed as in the proof of (\ref{eqn:int44}). We use Proposition~\ref{prop:Grigi}
with the parameter $s$  in the  Proposition  chosen to be $\max(2^{i} r,2^{j} r)N^{-\epsilon}$. 
Thus \eqref{e:Grigi}  gives 
$$
N^2\E^{\gamma}\iint F^{t}_{ij}([z-w]) \, \tilde\mu(\rd z) \, \tilde \mu(\rd w)
=
\OO(N^\varepsilon)\left(\frac{1}{\gamma^4}+\frac{1}{(\max(2^{i} r,2^{j} r))^4}\right)\max(2^{i} r,2^j r)(r^2 2^{i+j}).
$$
Notice that the error term in \eqref{e:Grigi} is negligible here  by choosing $p \epsilon >10$, say. 
We will often use this argument and from now no we will not repeat it in details.

After summation over $i,j,$ we obtain
\begin{equation}\label{eqn:F_ijffar}
\sum_{2^{i} > t2^{j}}N^2\E^{\gamma}\iint F^{t}_{ij}([z-w]) \, \tilde\mu(\rd z) \, \tilde \mu(\rd w)
=\OO(N^\varepsilon)\left(\frac{1}{\gamma^4}+\frac{1}{r}\right).
\end{equation}

For $i>0$ and $2^{i} < t2^{j}$ ($|h_1|<t|h_2|$),  from (\ref{eqn:partder}) we have
$$
|\partial_{h_1}^{k_1}\partial_{h_2}^{k_2}F_{ij}^t(h)|\leq C_{k_1,k_2} \frac{1}{(t 2^j r)^{k_1+k_2-1}}.
$$
The area of the support of $F_{ij}$ is still of order $r^22^{i+j}$, so that Proposition \ref{prop:Grigi} now yields
$$
N^2\E^{\gamma}\iint F^{t}_{ij}([z-w]) \, \tilde\mu(\rd z) \, \tilde \mu(\rd w)
=\OO(N^\varepsilon)\left(\frac{1}{\gamma^4}+\frac{1}{(t 2^j r)^4}\right)(t2^{j} r)(r^2 2^{i+j}).
$$
The contribution of such terms is therefore
\begin{equation}\label{eqn:F_ijfar}
\sum_{2^{i} < t2^{j}, i>0}N^2\E^{\gamma}\iint F^{t}_{ij}([z-w])\tilde\mu(\rd z) \tilde \mu(\rd w)=\OO(N^\varepsilon)\left(\frac{t^2}{\gamma^4}+\frac{1}{rt}\right).
\end{equation}

For $i=0$ and $t 2^j r\geq N^{-1/2+\varepsilon}$ ($|t h_2 |\geq N^{-1/2+\varepsilon}$), we have  from (\ref{eqn:partder}) 
$$
|\partial_{h_1}^{k_1}\partial_{h_2}^{k_2}F_{ij}^t(h)|\leq C_{k_1,k_2}  \frac{1}{(t 2^j r)^{k_1+k_2-1}}
$$
so that Proposition \ref{prop:Grigi} gives
$$
N^2\E^{\gamma}\iint F^{t}_{0j}([z-w]) \, \tilde\mu(\rd z) \, \tilde \mu(\rd w)
=\OO(N^\varepsilon)\left(\frac{1}{\gamma^4}+\frac{1}{(t 2^j r)^4}\right)(t2^jr)(r^2 2^{j}),
$$
and therefore 
\begin{equation}\label{eqn:F_0j}
\sum_{N^{-1/2+\varepsilon}< t2^{j}r}N^2\E^{\gamma}\iint F^{t}_{0j}([z-w]) \, \tilde\mu(\rd z) \, \tilde \mu(\rd w)
=
\OO(N^\varepsilon)\left(\frac{rt}{\gamma^4}+\frac{rN}{t}\right).
\end{equation}

For the terms corresponding to $i=0$, $j > 0$ and $t 2^j r\leq N^{-1/2+\varepsilon}$, 
we need one more decomposition. Let $r'=N^{-1/2+\varepsilon'}\ll r$, and decompose $F^{t}_{0j}=A_j+B_j$ with 
\begin{enumerate}[(i)]
\item $A_j$ supported on $\{|h_1|<2 r' \}\cap\{2^{j-1}r<h_2<2^{j+1}r\}$,
$\|A_j\|_{\infty}\leq r'$, 
\item $B_j$  smooth, supported on $\{|h_1|<2 r\}\cap\{2^{j-1}r<h_2<2^{j+1}r\}$, satisfying $\|B_j\|_\infty\leq r$, 
and
$$
|\partial_{h_1}^{k_1}\partial_{h_2}^{k_2}B_j(h)|\leq  C_{k_1,k_2}\left(
 \frac{r}{|h_2|^{k_1+k_2}}+\frac{r}{(r')^{k_1+k_2}}\mathds{1}_{|h_1|<r'}\right).
$$
\end{enumerate}
More explicitly, $A_j$ and $B_j$ can be constructed from $F_{0j}$ as follows. Let $a\geq 0$ be smooth on $\mathbb{R}_+$, $a=1$  on $[0,1]$, $a>0$ on $[1,2]$ and $a=0$
on $[2,\infty)$. Let $g\geq 0$ be a smooth, compactly supported function on $\mathbb{C}$ with $\int g=1$.
Define $$g_\eta(z)=\frac{1}{(r' a(|\eta|))^2}g\left(\frac{z}{r' a(|\eta|)}\right),$$
with the convention $g_\eta=\delta_0$ when $|\eta|>2$, and 
$$
B_j(z)=\left(F_{0j}* g_{h_2/r'}\right)(z),\ A_j=F_{0j}-B_j.
$$
Then the functions $A_j$ and $B_j$ satisfy (i) and (ii):
for example, note that  that in  the region 
$\{ r'  <  h_1< r \}\cap\{2^{j-1}r<h_2<2^{j+1}r\}$,  $B \sim h_1Y^{\gamma}(h) $, and this
function satisfies the  estimates in (ii).  

The function  $A_j$ is supported on a domain of area $\OO(2^j r r')$ and $\|A\|_\infty\leq r'$,
and the local density implied by Theorem~\ref{thm:YTdensity} gives
$$
N^2\E^{\gamma}\iint A_j([z-w]) \, \tilde \mu(\rd z) \, \tilde \mu(\rd w)=\OO(N^{1+2\varepsilon}2^j r).
$$
The contribution of all $A_j$ terms is therefore
\begin{equation}\label{eqn:A_j}
\sum_{j>0:\, t 2^j r\leq N^{-1/2+\varepsilon}}N^2\E^{\gamma}\iint A_j([z-w]) \, \tilde \mu(\rd z) \, \tilde \mu(\rd w)
=
\OO\left(\frac{N^{1/2+3\varepsilon}}{t}\right).
\end{equation}
For the contribution from $B_j$, consider the following partition of $\mathbb{T}$:
$1=\sum_{-1/(2r)\leq a,b\leq 1/(2r)}\chi_{ab}$ where $\chi_{ab}$ is supported on a disk of radius $10r$ around $(ar,br)$,
and $\|\chi_{ab}^{(n)}\|_\infty\leq C_n r^{-n}$.
The contribution of $B_j$ is of order at most 
$$
r^{-2}N^2\E^{\gamma}\iint\sum_{|a|\leq 5,2^{j-1}\leq b\leq 2^{j+1}} B_j([z-w])\chi_{00}(z)\chi_{ab}(w)\tilde \mu(\rd z) \tilde \mu(\rd w).
$$
Let $E$ be the event that all particles at distance $4r$ from 0 are known. Then 
$$
N^2\E^{\gamma}\iint B_j([z-w])\chi_{00}(z)\chi_{ab}(w)\tilde \mu(\rd z) \tilde \mu(\rd w)
=
\E^{\gamma}\E^{\gamma}\left(\int f(z)\chi_{00}(z) N\tilde \mu(\rd z)\mid E\right)
$$
where $f(z)=\int B([z-w]\chi_{ab}(w) N\tilde\mu(\rd w)$. By the local law Theorem~\ref{thm:YTdensity},
the set of $E$ such that 
\begin{align*}
&f(z)=\OO(Nr^3),\\
&\nabla f(z)=\int \nabla (B([z-w]\chi_{ab}(w)) N\tilde\mu(\rd w)=\OO(Nr^2),\\
&\Delta f(z)=\int \Delta (B([z-w])\chi_{ab}(w)) N\tilde\mu(\rd w)=\OO(Nr),
\end{align*}
has measure at least $1-N^{-100}$.
Using the (conditioned version of the) local law, Theorem \ref{thm:YTdensity-cond},
for $E$ in such a good set  we therefore have 
\begin{multline*}
|\E^{\gamma}\left(\int f(z)\chi_{00}(z) N\tilde\mu(\rd z)\mid E\right)|
\\
\leq \left(N r^2\left(\int |\nabla(f\chi_{00})|^2+\frac{1}{\gamma^2}\int (f\chi_{00})^2\right)\right)^{1/2}
+N^{\varepsilon} r^2\|\Delta (f\chi_{00})\|_\infty
= \OO(N^{3/2} r^4).
\end{multline*}
Hence the contribution of $B_j$ is at most 
$$
N^2\E^{\gamma}\iint B_j([z-w]) \, \tilde \mu(\rd z) \, \tilde \mu(\rd w)
=\OO(N^\varepsilon)2^j N^{3/2}r^2.
$$
All $B_j$ terms are therefore bounded by
\begin{equation}\label{eqn:B_j}
\sum_{j>0:\, t 2^j r\leq N^{-1/2+\varepsilon}}N^2\E^{\gamma}\iint B_j([z-w]) \, \tilde \mu(\rd z) \,  \tilde \mu(\rd w)
=
\OO(N^\varepsilon)\frac{Nr}{t}.
\end{equation}
Equations (\ref{eqn:F_ijclose}), (\ref{eqn:F_ijffar}), (\ref{eqn:F_ijfar}), (\ref{eqn:F_0j}), (\ref{eqn:A_j}) and (\ref{eqn:B_j}) show that $r= N^{-1/2+c}$ for arbitrarily small $c$ 
is the best choice. We therefore proved
\begin{equation}\label{eqn:f2}
N^2\E^{\gamma}\iint F_1^t([z-w]) \, \tilde \mu(\rd z) \,\tilde \mu(\rd w) 
=
\OO(N^{\varepsilon})
\left(
\frac{1}{\gamma^4}+\frac{\sqrt N}{t} \right).
\end{equation}
The contribution from $F_3^t$ can be bounded following the same method, and the resulting estimate is smaller due to the extra small $|h_2|$ factor in $f_3$.
Inserting the estimate (\ref{eqn:f2})  for  $F_i^t, i=1, 2, 3$ into 
 (\ref{eqn:defF}), we have  completed the proof of 
(\ref{eqn:needsmall}) and thus  (\ref{e:decoupling-1-UY}).

\subsection{Proof of estimate (\ref{e:decoupling-1-YY})}
By \eqref{eqn:useful} we need to bound
\begin{equation}\label{eqn:tobebounded}
  N^2 \E^{\Phi} \E^\gamma \iint \cal T_{L} ([\Delta_z^w]) \, \tilde \mu(\rd z) \, \tilde \mu(\rd w), \quad L=L_R^\gamma+c,
\end{equation}
where $c$ is an arbitrary constant. Without loss of generality, we  choose $c=\log R-\log \gamma$, so that from (\ref{Ya}) we have 
\begin{equation}\label{eqn:Lb}
L(z)=\OO(|z|/\gamma)
\end{equation}
for $|z|\ll \gamma$.
Equation  (\ref{eqn:tobebounded}) is equivalent to
\begin{equation}\label{eqn:Dt}
N^2 \E^\gamma \iint  D^t([z-w]) \, \tilde \mu(\rd z) \, \tilde \mu(\rd w),
\quad
{\rm where}\ D^t(h)=\E^{\Phi} \int\cal T_{L}(\Delta_{v}^{[v+h]}) m(\rd v),
\end{equation}
and we remind the reader that $\Phi$ depends on $t$ (note that we introduced an additional averaging
over $v$ for the same reasons as in (\ref{eqn:needsmall})).

Our estimate  of (\ref{eqn:Dt}) is similar to (\ref{eqn:needsmall}), up to two differences. 
First, it is easier to bound (\ref{eqn:Dt}) when the contributions from $g_1$, $g_2$, $g_3$, $g_4$ from (\ref{eqn:g}) are isolated, but this cannot be performed directly: smoothness of $D_t$ across $h_1,h_2=\pm 1/2$ requires the combination of these four terms. In the  first step, we therefore prove that the  long-range  contribution of $D_t$,  which we will denote by  $E_t$,  is negligible (this problem was not present for $K^t$, which is essentially supported in a small neighborhood of $0$). 

Second, the most delicate decompositions of $K_t$ are not necessary for $D_t$ as we have the additional small factor (\ref{eqn:Lb}) for the interaction at small distance.
\smallskip

\noindent{\it First step.} In this paragraph we prove that the contribution of the long range part in $D^t$  is  of order
\begin{equation}\label{eqn:longrange1}
  N^2 \E^\gamma \iint  E^t([z-w]) \, \tilde \mu(\rd z) \, \tilde \mu(\rd w)
  =\OO\left(N^\varepsilon\right)\left(\frac{1}{t^3}+\frac{1}{\gamma^4}\right),
\end{equation}
where 
$
E^t(h)=\E^{\Phi} \int\cal T_{L}(\Delta_{v}^{[v+h]})(1-\chi)(\Delta_{v}^{[v+h]}) \, m(\rd v)
$
and  $\chi$ is a smooth cutoff function  equal to 1 on  $|h_1|+|h_2|\leq 1/10$, $0$ on $|h_1|+|h_2|>1/5$.

The function $\cal T_L(h)$ has discontinuous derivative on $\{h_1=0\}\cup\{h_2=0\}$, which imposes a detailed analysis around these axes. 
We first gain some order of magnitude of $\cal T_L(1-\chi)$ around these singularities by removing the following function,  
$$
A(h)=\Big(\tilde\chi(h_1)\big [ (1-|h_2|)g_1(h)+|h_2|g_3(h)\big ]
+\tilde\chi(h_2)\big[ (1-|h_1|)g_1(h)+|h_1|g_2(h)\big ] \Big)(1-\chi(h)),
$$
where $\tilde\chi$ is a smooth cutoff equal to 1 on $[0,1/200]$ and vanishing outside 
$[0,1/100]$. The function $A$ is smooth on $\mathbb{T}$ for the following two reasons.
First, the function is smooth
on $h_1=0$  because the following three estimates cannot be simultaneously satisfied: $|h_1|<1/1000$, $\tilde \chi(h_2)\neq 0$ and  $(1-\chi)(h)\neq 0$.
Similarly, the function is smooth on $h_2=0$.
Second, $A$ is smooth on $h_1=\pm 1/2$ and $h_2=\pm 1/2$. Indeed, assume $-1/2\leq h_2<1/2$ is fixed. Then $(1-|h_1|)((1-\chi)g_1)(h)+|h_1|((1-\chi)g_2)(h)$ admits an obvious smooth extension to $h_1\in(0,1)$, and this extension is symmetric in a neighborhood of $h_1=1/2$, hence all its odd derivatives vanish there, so that $A$ is smooth at $h_1=\pm1/2$. The same reasoning applies to $h_2=\pm 1/2$.
 
We define
$
\cal A(h)=\E^{\Phi} \int A(\Delta_{v}^{[v+h]}) \, m(\rd v).
$
As $A$ is smooth at the scale of order one,  from Proposition \ref{prop:Grigi} 
 with $t$ chosen to be $N^{-\epsilon}$ and $p$ large enough, we obtain
\begin{equation}\label{eqn:longrange2}
  N^2 \E^\gamma \iint  \cal A([z-w]) \, \tilde \mu(\rd z) \, \tilde \mu(\rd w)
  =\OO\left(\frac{N^\varepsilon}{\gamma^4}\right).
\end{equation}
It thus remains to estimate $N^2 \E^\gamma \iint  H^t([z-w]) \, \tilde \mu(\rd z) \, \tilde \mu(\rd w)$ where $H^t=E^t-\mathcal{A}$.
To understand the regularity properties of $H^t$,  assume first that the distortion vanishes, i.e., $t=0$. We have $H^0={\cal T}_L(1-\chi)-A$, so that
$H^{0}$ is smooth on $\{h_1\neq 0\}\cap\{h_2\neq 0\}$, vanishes on $\{h_1= 0\}\cup\{h_2= 0\}$ (this is the purpose of removing the contribution from $A$) and 
satisfies (on the smoothness domain $\{h_1\neq 0\}\cap\{h_2\neq 0\}$)
$$
\sup_{k_1+k_2=k}|\partial_{h_1}^{k_1}\partial_{h_2}^{k_2}H^0(h)|\leq C_k.
$$
The distortion $t$ smooths the singularities on $\{h_1= 0\}\cup\{h_2= 0\}$ as follows:
\begin{equation}\label{eqn:dist1}
  |H^t(h)|\leq C(t+\min(|h_1|,|h_2|)),\quad
  |\partial_{h_1}^{k_1}\partial_{h_2}^{k_2}H^t(h)|\leq  C_{k_1,k_2}\left(1+\frac{\mathds{1}_{|h_2|<t}}{t^{k_2-1}}+\frac{\mathds{1}_{|h_1|<t}}{t^{k_1-1}}\right).
\end{equation}
The above bounds are elementary after writing $H^t$ explicitly in terms of $g_1,g_2,g_3,g_4,\chi$ and $\wt \chi$.
It amounts to the observation that the function $r_t(x)=\E_{X} |x+t X|$ satisfies $|r_t^{(k)}(x)|\leq C_k(1+\mathds{1}_{|x|<2t}t^{1-k})$.
Intuitively,  $r_t(x)$ is a regularized absolute value function which is  smooth at the scale $t$ and $r_t(x) = |x|$ for $|x| \ge 2 t$.   

Let 
$
\Omega_t=\{|h_1|<t\}\cup\{|h_2|<t\}.
$
Consider a partition of unity $1=\sum \chi_i$ on the torus with $\OO(\log N)$ summands, $\chi_0$ with support on
$\Omega_{t}$, $\chi _i$ $(i>0)$ supported on $(2^{i+1}\Omega_{t})\backslash (2^{i-1}\Omega_{t})$, and  the derivative bound  
$\|\chi_i^{(n)}\|_\infty\leq C_n (2^it)^{-n}$  for all integer $n$. 
Note that for $H=H^{t}\chi_i$ we have $|\nabla^jH(x,y)|=\OO((2^it)^{-j+1})$,
and the same estimate holds for $H_s^{(j)}$ when $s=(2^i t)N^{-\varepsilon}$.
Moreover, $(2^{i+1}\Omega_{t})\backslash (2^{i-1}\Omega_{t})$ has area $\OO(2^i t)$,
so that Proposition~\ref{prop:Grigi} gives (take $p=\lfloor 10/\varepsilon\rfloor$)  
$$
N^2\E^{\gamma}\iint H^{t}([z-w])\chi_i([z-w])\tilde\mu(\rd z) \tilde \mu(\rd w)=\OO(N^\varepsilon)\left(\frac{1}{(2^it)^3}+\frac{2^it}{\gamma^4}\right),\quad {\rm for}\ 2^i t<10.
$$
Summation of the above  equations over $i$ gives
\begin{equation}\label{eqn:longrange3}
N^2 \E^\gamma \iint  H^t([z-w])(1-\chi)([z-w]) \, \tilde \mu(\rd z) \, \tilde \mu(\rd w)
=\OO(N^{\varepsilon})\left(\frac{1}{t^3}+\frac{1}{\gamma^4}\right).
\end{equation}
Equations (\ref{eqn:longrange2}) and (\ref{eqn:longrange3}) prove (\ref{eqn:longrange1}).
\smallskip

\noindent{\it Second step.} In this paragraph we prove that the contribution of the short range is 
\begin{equation}\label{eqn:longrange4}
  N^2 \E^\gamma \iint  U^t([z-w]) \, \tilde \mu(\rd z) \, \tilde \mu(\rd w)
  =\OO(N^{\varepsilon})\left(\frac{1}{\gamma^4}+\frac{N^{1/2}}{t}+\frac{N^{4/5}}{\gamma^{2/5}}\right), 
\end{equation}
where
$
U^t(h)=\E^{\Phi} \int\cal T_{L}(\Delta_{v}^{[v+h]})\chi(\Delta_{v}^{[v+h]}) \, m(\rd v)
$. 
From our expression (\ref{eqn:torusForm}) for $\cal T_L$, we only need to bound
$N^2 \E^\gamma \iint  U_j^t([z-w]) \tilde \mu(\rd z) \tilde \mu(\rd w)$ ($1\leq j\leq 3$) where
\begin{equation} \label{eqn:defF-bis}
  U_j^t(h)=\E^{\Phi} \int u_j(\Delta_{v}^{[v+h]}) m(\rd v),
\end{equation}
with
\begin{equation*}
u_1(h)=L(h)\chi(h),\quad u_2(h)=|h_1|L(h)\chi(h),\quad u_3(h)=|h_1h_2|L(h)\chi(h).
\end{equation*}
The terms above all involves $g_1$; the other ones involving $g_2,g_3,g_4$    can be bounded  in an easier way, because $g_2,g_3,g_4$ are smooth on scale $1$ with no singularity at $0$.

We first consider $U_1^t$. Let 
$N^{-1/2}\ll u\ll \gamma$ be some intermediate scale. Let  $\chi$ be as before
and define $\tilde L(h)=f_1(h)\chi(h/u)$.
Then the local law, Theorem \ref{thm:YTdensity} and the bound (\ref{eqn:Lb}) give
\be\label{66}
N^2 \E^\gamma \iint  \tilde U_1^t([z-w]) \, \tilde \mu(\rd z) \, \tilde \mu(\rd w)
=\OO\left(\frac{N^{2+\varepsilon}u^3}{\gamma}\right),
\quad
\text{where } \tilde U_1^t(h)=\E^{\Phi} \int \tilde L(\Delta_{v}^{[v+h]})\, m(\rd v).
\eeq
On the other hand, the same reasoning as the paragraph from (\ref{eqn:final}) to (\ref{eqn:int44}) gives
$$
N^2 \E^\gamma \iint  (U_1^t-\tilde U_1^t)([z-w]) \, \tilde \mu(\rd z) \, \tilde \mu(\rd w)
=\OO\left(\frac{N^{\varepsilon}}{u^2}\right).
$$
Optimization the parameter $u$ in both previous estimates  shows that the contribution from  $U_1^t$ is 
\begin{equation}\label{eqn:F1}
  N^2 \E^\gamma \iint  U_1^t([z-w]) \, \tilde \mu(\rd z) \, \tilde \mu(\rd w)
  =\OO\left(\frac{N^{4/5+\varepsilon}}{\gamma^{2/5}}\right).
\end{equation}

We now consider the most delicate 
terms $U_2^t$ and $U_3^t$. 
We decompose $L^{\gamma}_R=Y^R-Y^\gamma+\OO(e^{-N^c})$, we can just replicate the proof for 
$F_1^t$ and $F_3^t$ and get the same estimates as (\ref{eqn:f2}) (note that, as for $U_1^t$, we could also have used the short range bound (\ref{eqn:Lb}) for an improved but unnecessary estimate). This concludes the proof.

\section{Proof of Proposition~\ref{prop:qf-lb}: free energy lower bound}
\label{sec:quasifreelb}

In this section, we construct a trial state to give a correct lower bound for the free energy,
and thus prove Proposition~\ref{prop:qf-lb}. Recall the assumptions of the proposition
\be\label{lb-assumption} 
1 \gg \ell/b \gg (Nb^2)^{-1/4},  \quad  N^{-1/2+ c}  \ll \ell \ll b \ll  N^{-\tau}, \quad 
\ell < R,   \quad \tau = 2c/5, 
\eeq
which we will assume throughout this section. 

\subsection{The trial state and embedding of the torus}\label{torus}

Throughout the proof of the lower bound the parameter $u \in [-b/2,b/2)^2$ is fixed
arbitrarily, and all estimates will be uniform in the choice of $u$.
To obtain a lower bound on the partition function,
we first restrict the particle profile to $\b {\bar n}$.
For this, we define the indicator function
\begin{equation} \label{chi1}
  \hat\chi(\b z)
  =
  \mathds{1} \Big( n_B(z) = \bar n_B \Big)
  \prod_\alpha \mathds{1} \Big (n_\alpha(z) = \bar n_\alpha \Big )
  \prod_{j} \mathds{1}\Big( z_j \in D \cup B \Big)
\end{equation}
where $\b n(\b z)$ is the particle profile of the configuration $\b z \in \C^N$,
i.e., $\b n (\b z) = (n_\alpha(\b z))$
where $n_\alpha(\b z)$ is the number of particles $z_j \in \alpha$
(with  $\alpha$ either  a bulk square or the boundary region $B$).

We then start with the trivial bound
\begin{equation} \label{e:qf-lb-pf-trivial}
  \frac{1}{\beta} \log \int \ee^{-\beta H^R_V(\b z)} \, m(\rd\b z)
  \ge \frac{1}{\beta} \log \int \ee^{-\beta H^R_V(\b z) } \, \hat\chi(\b z) \, m(\rd \b z)
  .
\end{equation}
Next we break the permutation symmetry of the particles.
We order the squares $\alpha$ arbitrarily
as $\alpha_1, \alpha_2, \dots$
and write $\tilde \chi(\b z)$ for $\hat \chi(\b z)$
multiplied with the indicator function of the event in which the 
particles $z_1, \dots, z_{\bar n_{\alpha_1}}$ are in $\alpha_1$, the particles
$z_{\bar n_{\alpha_1}+1}, \dots, z_{\bar n_{\alpha_1}+\bar n_{\alpha_2}}$ are in $\alpha_2$,
and so on. Then
\begin{equation*}
  \frac{1}{\beta} \log \int \ee^{-\beta H^R_V(\b z) } \, \hat\chi(\b z) \, m(\rd \b z)
  =
  \frac{1}{\beta} \log \binom{N}{\b n} +
  \frac{1}{\beta} \log \int \ee^{-\beta H^R_V(\b z) } \, \tilde\chi(\b z) \, m(\rd \b z).
\end{equation*}

\nomOth[15]{$\Psi_\alpha^u$}{another map from $\mathbb{T}_\alpha$ to $\alpha$ based on $\Psi$, with distortions}
As in the lower bound for the torus in Section~\ref{sec:qf-lb-torus},
to each bulk square $\alpha$, we associate a map
\begin{equation}
  \Psi_\alpha : \T^{(b)} \to \alpha.
\end{equation}
The main difference between these two settings   
is  the choice of the embeddings $\Psi_\alpha$.
We now choose $\Psi_\alpha$ as the  re-centered version of the map defined by Corollary~\ref{cor:decoupling}:
\begin{equation}
  \Psi_\alpha(v)
  = c(\alpha) + \Psi^{(b)}(v)
  = c(\alpha) + b \Psi(v/b) \qquad (v \in \T^{(b)}).
\end{equation}
This choice of the maps will enter in this section only through the estimates
given by Corollary~\ref{cor:decoupling} and the fact that $|\rd \Psi|=1$.

The remaining set-up is parallel to that for the torus in Section~\ref{sec:qf-lb-torus}.
Let $\omega_\alpha$ be the measure of the Yukawa gas on $\T_\alpha^{n_\alpha}$, with density
\begin{equation}
  \omega_\alpha(\rd \b v^\alpha) = \frac{1}{Z_\alpha} \ee^{- \beta \hat H_\alpha (\b v^\alpha)  }   m(\rd \b v^\alpha), 
\end{equation}
where $\hat H_\alpha$ was defined in \eqref{e:hatHalpha} as the energy of a torus Yukawa gas in $\T_\alpha$ with range $\ell$ 
(in principle, there is an  external potential $Q(\alpha)$. Since it is a constant, we remove it). 
For the boundary, we take $\omega_B$ to be the measure under which the particles are
independently distributed according to the equilibrium measure, i.e.,
 $\omega_B = \mu_V|_B^{\otimes n_B}$ on $B$ and $\Psi_B : B\to B$ to be the identity map.
With the fixed particle profile $\b n = \bar {\b n}$,
the quasi-free approximation is the product measure $\b \omega = \prod_\alpha \omega_\alpha$
(where the product also includes $\alpha=B$).
Given the maps $\Psi_\alpha$, define $\Psi$ by
\begin{equation} \label{e:Psiproddef}
\Psi : \prod_\alpha \T_\alpha^{n_\alpha} \times B^{n_B}\to \C^N, \quad
\Psi (\{\b v\}) = ( \{ \Psi_\alpha \b v^\alpha  \} ) \in \C^N.
\end{equation}
In particular,
$\Psi^* \b \omega = \prod_\alpha \Psi_\alpha^* \b \omega_\alpha$ is a measure on configurations of $N$ particles in $\C$.
Under the map $\Psi_\alpha$, using that $|\rd\Psi_\alpha|=1$, the measure $\omega_\alpha$ transforms to 
\begin{equation*}
  \frac{1}{Z_\alpha} \ee^{-\beta \hat H_\alpha (\Psi_\alpha^{-1} (\b z) ) }   \prod_i \rd \Psi_\alpha^{-1} (z_i) 
=
\frac{1}{Z_\alpha} \ee^{-\beta \hat H_\alpha (\Psi_\alpha^{-1} (\b z) )} \prod_i \rd  z_i
,
\end{equation*}
where we write $\Psi_\alpha^{-1}(\b z) = (\Psi_\alpha^{-1}(z_1), \Psi_\alpha^{-1}(z_2), \dots)$ if $\b z = (z_1, z_2, \dots)$.
Thus defining $\hat H_\Psi(\b z) = \sum_\alpha \hat H_\alpha \circ \Psi_\alpha^{-1}(\b z^\alpha)$,
by Jensen's inequality,
\begin{equation}
\frac{1}{\beta} \log \int \ee^{-\beta H_V^R(\b z)} \tilde\chi(\b z) \, m(\rd \b z)
\ge
\frac{1}{\beta} \log \int \ee^{-\beta \hat H_\Psi(\b z)} \tilde\chi(\b z) \, m(\rd \b z)
+ \E^{\Psi^* \b \omega} (\hat H_\Psi-H_V^R) 
  .
\end{equation}
Reversing the change of variables
and averaging over the distribution of maps $\Psi$ with $|\rd\Psi|=1$,
whose expectation is denoted by $\E^\Psi$, 
\begin{equation}
  \frac{1}{\beta} \log \int \ee^{-\beta H_V^R} \tilde \chi
  \ge
  \frac{1}{\beta} \log \int \ee^{-\beta \hat H_\alpha(\b v^\alpha) }   \prod_\alpha \rd \b v^\alpha 
  + \Omega
\end{equation}
where $\Omega = \E^{\Psi} \E^{\b \omega}  ( \hat H(\b v) - H_V^R(\Psi \b v) )$ with $\hat H(\b v) = \sum_\alpha \hat H_\alpha (\b v^\alpha)$.
We abbreviate by $\hatE$ the expectation $ \E^\Psi  \E^{\b \omega} $ with $\b \omega = \prod_\alpha \b \omega_\alpha$, 
where $\b \omega_\alpha$ is the measure of a Yukawa gas of range $\ell$ defined in the previous paragraph.
Then, in summary, we need to estimate $\Omega = \Omega_1+ \Omega_2$ where 
\begin{equation}  \label{56}
  \Omega_1 : = \hatE (\hat H (\b v) - H_{Q}^{\ell}(\Psi \b v) ),
  \quad
  \Omega_2:=  \hatE (H_{Q}^{\ell}    (\Psi \b v) -   H_{V}^R (\Psi \b v)),
\end{equation}
and we recall that $H_{Q}^{\ell}$ is the Yukawa energy of range $\ell$ with potential $Q$. 
Thus $\Omega_1$  is the error of a short range Yukawa gas in the quasifree approximation and
is  similar to that in \eqref{e:qf-lb-torus-pf-jensen} for the torus.
The second term was absent for the torus because it was essentially handled by
Lemma~\ref{on} at an earlier stage; the choice of  $\ell \ll b$ in \eqref{e:qf-lb-torus-pf-jensen} 
is the  key reason that this term is much simpler on the torus than the current general setting. 
The control of this term requires a more careful choice of the maps $\Psi$.

\subsection{Lower bound I: the short-range term $\Omega_1$} 
\label{sec:quasifreelb1}

In the next two lemmas we estimate the short-range contribution $\Omega_1$.
These lemmas are analogous to Lemmas~\ref{lem:Hdec12bd-torus}--\ref{lem:E-lb-torus} for the torus setting.
Besides the density of the equilibrium measure is not constant,
that there is a small contribution from the boundary, we also need precise estimate on $\Omega_1$ in the dependence 
of the parameters $\ell$ and $b$. With the current more sophisticated choice of the  map  $\Psi$ and 
the decoupling estimate, Corollary~\ref{cor:decoupling}, we will be able to estimate $\Omega_1$  effectively. 

First recall notation similar to that discussed around \eqref{muhat-torus}.
As previously, $U_b^\ell$ is the Yukawa potential on the torus $\T^{(b)} \cong \T_\alpha$.
Also, $\tilde \mu_\alpha = \hat \mu_\alpha - \mu_\alpha$
where $\mu_\alpha$ is the normalized uniform measure on $\T_\alpha$
and where $\hat\mu_\alpha$ is defined  in \eqref{muhat-torus}.
We observe that, by construction,
the expected empirical measure $\hat \mu$ under $\hatE$
in each square $\alpha$ is uniform with total mass $n_\alpha /N$:
\begin{equation} \label{e:hatEalpha}
  N \hatE(\hat \mu|_\alpha) = n_\alpha \mu_\alpha,
  \quad \text{where }
  \mu_\alpha(dz)
  =
  b^{-2} \mathds{1}_{z \in\alpha} \, m(\dd z)
  .
\end{equation}

The next lemma replaces Lemma~\ref{lem:Hdec12bd-torus} for the torus.
Similarly as in \eqref{e:Edef-torus}, we define
\begin{equation} \label{e:Edef}
  E = \sum_{\alpha \subset D} n_\alpha^2 \hatE  \iint_{\T_\alpha\times\T_\alpha}
  (U_b^\ell(v-w) - Y^\ell(\Psi_\alpha(v)-\Psi_\alpha(w)) )
  \, \tilde \mu_\alpha(\dd v) \, \tilde \mu_{\alpha}(\dd w).
\end{equation}
We also write
$\hat H_D = \sum_{\alpha \subset D} \hat H_\alpha$
and decompose $H^\ell_Q(\b z)$ into bulk and boundary contributions as
\begin{equation} \label{e:HQD}
  H_{Q,D}^\ell(\b z)
  = N \sum_{j} Q(z_j) \mathds{1}_{z_j \in D}  + \sum_{j \neq k} Y^\ell(z_j-z_k) \mathds{1}_{z_j,z_k \in D},
  \quad
  H_{Q,B}^\ell(\b z)
  = H^\ell_Q(\b z) - H_{Q,D}^\ell(\b z)
  .
\end{equation}

\begin{lemma} \label{lem:Hdec12bd}
Assume $1\ll \ell \ll b$ and recall that $E$ is defined in \eqref{e:Edef}. Then 
\begin{align}
\label{e:Hdec1bd}
  \hatE(\hat H_{D}(\b u) - H^\ell_{Q,D}(\Psi \b u))
  &=
  E   +
  N^\epsilon\OO \p{ N^2 (\ell^3+b^2\ell^2)} \|\rho_V\|_{\infty,2}^2,
  \\
  \label{e:Hdec2bd}
  N^2 I_{B,Q} -  \hatE (H^\ell_{Q,B}(\b u))
  &= N^\epsilon\OO(N^2 \ell^2 b) \|\nabla \rho_V\|_\infty +  \OO \p{\bar n_B \log N}
.
\end{align}
\end{lemma}

The proof of the above lemma occupies the remainder of this subsection.
Before proceeding with the proof, we state the estimate for $E$ in the following lemma.

\begin{lemma}\label{C1} 
  Assume the parameters $b$ and $\ell$ satisfy the condition  \eqref{lb-assumption}.
  Then $E$ defined in \eqref{e:Edef}  satisfies
  \begin{equation}\label{eqn:EEE}
    E
    = N^\epsilon \OO \Big ( b^{-2}\left((Nb^2)^{4/5}/(\ell/b)^{2/5} + (\ell/b)^{-4}\right)\Big ) 
    = N^\epsilon \OO(N^{4/5}/\ell^{2/5} + b^2\ell^{-4}) = \OO(N^{1-\tau} + N^\epsilon b^2\ell^{-4}). 
  \end{equation}
\end{lemma}

\begin{proof}
  This is \eqref{e:decoupling-UY} of Corollary~\ref{cor:decoupling} 
  and the fact that there are $\OO(b^{-2})$ bulk squares $\alpha$ according to \eqref{e:nbarbd}.
\end{proof}

We now prove Lemma~\ref{lem:Hdec12bd}.
The main error in \eqref{e:Hdec1bd} is the one with  the factor  $N^2 (\ell^3+b^2\ell^2)$, which  is of order $b$ smaller than the 
main error term in  the upper bound \eqref{e:qf-ub-1}. 
The reason we gain an additional factor $b$ here, roughly speaking, is due to the fact  that the leading error from the left side of a square is canceled by that from the right side 
provided that  the densities of the two neighboring squares are the same. 
Since the density variation is of order $b$, the next order error carries an additional  factor  $b$. 
(A similar cancelation could have been  obtained also in the upper bound \eqref{e:qf-ub-1}. 
Since this refined estimate is not needed in this paper,
we chose  not to present it for the sake of simplicity.)

\begin{proof}[Proof of \eqref{e:Hdec1bd}]
Estimating $Q$ by \eqref{e:Qz} and $\bar n_\alpha$ by \eqref{e:nbarbd},
the difference of the contributions of the external potential is 
\begin{align*}
& \Big| \hatE \qB{ N \sum_{\alpha \subset D} \sum_i (Q(z_i) - Q(\alpha) \mathds{1}_{z_i \in \alpha} ) } \Big | 
 = \Big | N \sum_{\alpha \subset D} n_\alpha \int (Q(z) - Q(\alpha) )  \, \mu_\alpha ( \dd z)  \Big |
\nonumber
\\
& \leq 4 \pi \ell^2 N \sum_{\alpha \subset D} n_\alpha \Big | \int (\rho_V (\alpha) - \rho_V (z) ) \, \mu_\alpha ( \dd z)   \Big | + N \sum_{\alpha \subset D} \OO( N b^2 \|\rho_V\|_\infty )  \OO(N^\varepsilon \ell^4 \|\nabla^2 \rho_V\|_\infty )
\nonumber
\\
& \leq  \OO( N^2 b^2   \ell^2  \|\nabla \rho_V\|_\infty^2 )
+  \OO( N^\varepsilon N^2 \ell^4 \|\rho_V\|_\infty \|\nabla^2 \rho_V\|_\infty)
.
\end{align*}
To estimate the two-particle interactions,
it suffices to show that
\begin{multline} \label{eq:2ptdiff1} 
  \sum_{\alpha, \beta \subset D} \hatE \Big [ \sum_{i \neq j} {\bf 1}_{v_i \in \T_\alpha} {\bf 1}_{v_j \in \T_\beta} (U_\alpha^\ell(v_i-v_j)\mathds{1}_{\alpha=\beta}-Y^\ell(\Psi_\alpha(v_i)-\Psi_\beta(v_j)))   \Big ]
  \\
  = E
  + \OO(N^2 \ell^3) (\|\rho_V\|+\|\nabla\rho_V\|_\infty)^2
  .
\end{multline}
The outline of the proof is analogous to that of \eqref{eq:2ptdiff1-torus} for the torus.
Again, the contribution of the nonadjacent pairs of squares on the left-hand side is bounded by $\OO(\ee^{-c\ell/b}) = \Oinfty$.
For any squares $\alpha,\beta$, define
\begin{equation}
  \bar Y_{\alpha\beta}
  = \iint_{\T_\alpha \times \T_\beta} Y^\ell(\Psi_\alpha(u)-\Psi_\beta(v)) \, \mu_\alpha(\dd u) \, \mu_\beta(\dd v)
  = \iint_{\alpha \times \beta} Y^\ell(u-v) \, \mu_\alpha(\dd u) \, \mu_\beta(\dd v)
  .
\end{equation}
Denoting by $\alpha \sim \beta$ that the squares $\alpha$ and $\beta$ are adjacent,
exactly as in \eqref{eq:2ptdiff-torus}, therefore
\begin{align}
  & \sum_{\alpha, \beta \subset D} \hatE \Big [ \sum_{i \neq j} {\bf 1}_{v_i \in \T_\alpha} {\bf 1}_{v_j \in \T_\beta} (U^\ell_b(v_i-v_j)
  \mathds{1}_{\alpha=\beta}-Y^\ell(\Psi_\alpha(v_i)-\Psi_\beta(v_j))) \Big ] \nonumber \\
& = 
\sum_{\alpha \subset D} \bar n_\alpha^2 \hatE \Big [
\iint_{\T_\alpha \times \T_\alpha} (U^\ell_b(v-w) - Y^\ell(\Psi_\alpha(v)-\Psi_\alpha(w))) \, \hat \mu_\alpha(\dd v) \, \hat  \mu_\alpha(\dd w) \Big ]
  \nonumber\\
  & \qquad\qquad
-\sum_{\alpha \sim \beta} \bar n_\alpha \bar n_\beta \bar Y_{\alpha\beta}  + \Oinfty
\label{eq:2ptdiff}.
\end{align}
The difference between $E$ and  the first term on the right-hand side of the above equation is
\begin{equation}
  \sum_{\alpha\subset D} \bar n_\alpha^2 \qa{
    \iint_{\T_\alpha^2} U^\ell_{ b} (v-w) \, \mu_\alpha(\dd v) \, \mu_\alpha(\dd w) - \iint_{\alpha^2} Y^\ell(v-w) \, \mu_\alpha(\dd v) \, \mu_\alpha(\dd w)
    }, 
\end{equation}
where we have used that $\tilde \mu_\alpha = \hat \mu_\alpha - \mu_\alpha$.
For the squares $\alpha \subset D$ not touching the boundary, 
we use the cancellation \eqref{e:Ycancel} below
and for the squares touching the boundary instead the weaker estimate \eqref{e:Yboundary}.
By these estimates, and summing over $\alpha$
using that there $\OO(b^{-2})$ squares $\alpha$ not touching a boundary square
and $\OO(b^{-1})$ squares touching the boundary, it follows
that the last display equals
\begin{equation*}
  \sum_{\alpha \subset D} \sum_{\beta: \beta \sim \alpha} \bar n_\alpha \bar n_\beta \bar Y_{\alpha \beta}
  +
  \OO(N^2 \ell^3) (\|\rho_V\|_\infty + \|\nabla \rho_V\|_\infty)^2.
\end{equation*}
This proves \eqref{eq:2ptdiff1}.
\end{proof}

The following lemma replaces Lemma~\ref{lem:Ycancel-torus} for the torus.
The argument requires more care since we here do not have $\bar n_\alpha = \bar n_\beta$.

\begin{lemma}\label{lem:Ycancel}
Assume that $b \gg \ell$.
For any bulk square $\alpha$ whose neighboring squares do not touch the boundary region $B$,
\begin{multline}\label{e:Ycancel}
  \bar n_\alpha^2 \qB{ \iint_{\T_\alpha \times \T_\alpha} U^\ell(v-w) \, \mu_\alpha(\dd v) \, \mu_\alpha(\dd w)
  - \iint_{\alpha \times \alpha} Y^\ell(v-w) ) \, \mu_\alpha(\dd v) \, \mu_\alpha(\dd w)}
  - \sum_{\beta \sim \alpha} \bar n_\alpha \bar n_\beta \bar Y_{\alpha \beta}
\\
= \OO\pb{ N^2 b^2 \ell^3 \|\rho_V\|_\infty \|\nabla \rho_V\|_\infty } + \Oinfty.
\end{multline}
For all other  squares $\alpha$, we still have
\begin{multline}\label{e:Yboundary}
  \bar n_\alpha^2 \qB{ \iint_{\T_\alpha \times \T_\alpha} U^\ell(v-w) \, \mu_\alpha(\dd v) \, \mu_\alpha(\dd w)
  - \iint_{\alpha \times \alpha} Y^\ell(v-w) ) \, \mu_\alpha(\dd v) \, \mu_\alpha(\dd w)}
- \sum_{\beta \sim \alpha} \bar n_\alpha \bar n_\beta \bar Y_{\alpha \beta}
\\
= \OO(N^2 b \ell^3) \|\rho_V\|_\infty^2.
\end{multline}
\end{lemma}

\begin{proof}
For any fixed square $\alpha$ of side length $b \gg l$,
using that contributions for distances $\gg \ell$ are negligible,
by unfolding the periodized interaction we have
\begin{equation*}
  \int_\alpha \int_{\alpha} U^\ell(u-v) \, m(\dd u) \, m(\dd v)
  =
  \int_{\alpha} \int_{\cup_{\beta \sim \alpha} \beta \cup \alpha} Y^\ell(z-w) \, m(\dd u) \, m(\dd v)
  + \Oinfty,
\end{equation*}
and thus
\begin{equation*}
   \iint_{\alpha^2} (U^\ell(u-v) - Y^\ell(u-v) ) \, \mu_\alpha(\dd u) \, \mu_\alpha(\dd v)
  =
    \sum_{\beta \sim \alpha} \iint_{\alpha \times \beta} Y^\ell(z-w) \, \mu_\alpha(\rd z) \, \mu_\beta(\dd w) + \Oinfty    .
\end{equation*}
Therefore the left-hand side of \eqref{e:Ycancel} equals
\begin{equation*}
  \bar n_\alpha \sum_{\beta:\beta \sim \alpha} \pb{ \bar n_\alpha - \bar n_\beta }
  \iint_{\alpha \times \beta} Y^\ell(z-w) \mu_\alpha(\rd z) \mu_\beta(\dd w) + \Oinfty
  .
\end{equation*}
Note that for $\alpha \neq \beta$,
\begin{equation*}
  \iint_{\alpha \times \beta} Y^\ell(z-w)  \mu_\alpha(\rd z) \mu_\beta(\dd w)
  = \OO(b^{-2})\OO(\ell b^{-1}) \sup_{z\in\alpha} \int Y^\ell(z-w) \, m(\rd w)
  = \OO(b^{-3}\ell^3)
  .
\end{equation*}
Using $\abs{ \bar n_\alpha - \bar n_\beta } = \OO(N b^3) \|\nabla \rho_V\|_\infty$
and $\bar n_\alpha = \OO(N b^2) \|\rho_V\|_\infty$, the claim \eqref{e:Ycancel} follows.

For the boundary squares, we do not use any cancellation
between the difference of $U^\ell$ and $Y^\ell$ and $\bar Y$ in \eqref{e:Ycancel},
but we still use the  cancellation between $U^\ell$ and $Y^\ell$.
Analogously to the above, the difference between $U^\ell$ and $Y^\ell$ and the $\bar Y$ terms
are each bounded by
\begin{equation*}
\OO(Nb^2)^2 \OO(b^{-3} \ell^3) \|\rho_V\|_\infty^2
=
\OO(N^2 b \ell^3) \|\rho_V\|_\infty^2
.
\end{equation*}
This completes the proof.
\end{proof}

\begin{proof}[Proof of \eqref{e:Hdec2bd}]
By definition, we have
\begin{equation*}
\hatE[H^\ell_{Q,B}]
= \hatE\qa{ \sum_{i \neq j} Y^\ell(z_i-z_j) \mathds{1}_{z_i,z_j \in B} + N \sum_{z_i \in B} Q(z_j) + 2 \sum_{z_i \in D} \sum_{z_j \in B} Y^\ell(z_i-z_j) }.
\end{equation*}
Moreover, by definition of the expectation $\hatE$,
the particles in $B$ are distributed independently according to the restriction of the equilibrium measure $\mu_V$.
If the particles in $D$ were also distributed independently according to the equilibrium measure,
the above right-hand side would be $N^2 I_{Q,B} + \OO(\bar n_B \log N)$,
with the error term $\OO(\bar n_B \log m(B)) = \OO(\bar n_B \log N)$ resulting from the inclusion of the diagonal $i=j$ in the first sum.
In reality, the particles in $D$ are distributed according to the periodic Yukawa gas in the squares $\alpha$;
under this measure the expected empirical measure is uniform on the squares $\alpha$ with
constant density $\bar n_\alpha/N$;
we may replace this constant density in the bulk squares by the density of the equilibrium measure 
with an error $\OO(N n_B \ell^2 b b')\|\nabla\rho_V\|_\infty = \OO(N n_B \ell^2 b)\|\nabla\rho_V\|_\infty$.
In summary, we have
\begin{equation*}
  \hatE[H^\ell_{Q,B}] = N^2I_{Q,B} +  \OO(N^2\ell^2b)\|\nabla \rho_V\|_\infty + \OO(\bar n_B \log N)
\end{equation*}
as claimed.
\end{proof}

\subsection{Lower bound II: the long-range term $\Omega_2$ and conclusion} \label{sec:quasifreelb2}

In the next lemma we estimate the term $\Omega_2$.
It is in this estimate where the randomness of the $\Psi_\alpha$ enters in an essential way
through the decoupling estimate,  \eqref{e:decoupling-YY} of  Corollary~\ref{cor:decoupling}.
We recall the decomposition $H_Q^\ell = H_{Q,D}^\ell + H_{Q,B}^\ell$ of the energy into bulk
and boundary part from \eqref{e:HQD} and decompose $H_V^R$ analogously.

\begin{lemma} \label{lem:Hdec3bd}
Assume the parameters $b$ and $\ell$ satisfy the condition \eqref{lb-assumption} and recall $K_R^\ell$ from \eqref{e:Kdef}. 
Then
 \begin{equation}
 \Omega_2 = N \log(R/\ell) + N^2 K_R^\ell
  + N^\epsilon\OO \p{ N^{1-\tau} + b^2\ell^{-4}}
  + \OO(N^2b^4)\|\nabla\rho\|_{\infty,2}^2
  + \OO\p{(\log N) b^{-2} + n_B \log N}.
\end{equation}
More precisely, with $\OO(N^{1-\tau})= N^\epsilon\OO(N^{4/5}/\ell^{2/5})$, we have
\begin{align}
 \hatE(H^\ell_{Q,D}(\Psi \b v) &  -  H_{V,D}^R(\Psi \b v)) 
   - N \log (R/\ell)  - N^2 K_R^\ell   \nonumber \\
   & =
  N^\epsilon\OO \p{ N^{1-\tau} + b^2\ell^{-4}}
     + N^\epsilon\OO(N^2b^4)\|\nabla\rho\|_{\infty,2}^2
     + \OO((\log N)b^{-2}),
   \label{e:Hdec3bd}
   \\
  \hatE(H^\ell_{Q,B}(\Psi \b v) &  - H_{V,B}^R(\Psi \b v))
   =
  \OO\p{(\log N) b^{-2} + n_B \log N}.
  \label{e:Hdec3bd-B}
\end{align}
\end{lemma}

Assuming this Lemma, we can now prove Proposition~\ref{prop:qf-lb}. 
\begin{proof}[Proof of Proposition~\ref{prop:qf-lb}]
  By \eqref{e:qf-lb-pf-trivial}--\eqref{56},
  \begin{multline}
    \frac{1}{\beta} \log \int \ee^{-\beta H^R_V(\b z)} \, m(\rd\b z) \\
    \geq
    \frac{1}{\beta} \log \binom{N}{\b n}
    + \sum_{\alpha \subset D} \frac{1}{\beta} \log \int \ee^{-\beta \hat H_\alpha(\b u^\alpha) } \rd \b u^\alpha
    + \frac{1}{\beta} \log \int \ee^{-\beta \hat H_B(\b u^B) } \rd \b u^B + \Omega_1 + \Omega_2
  \end{multline}
where the combinatorial factor $\frac{1}{\beta} \log \binom{N}{\b n}$ arises analogously as in \eqref{e:qf-lb-torus-binom}.
By definition, the first three terms on the right-hand side give $F(\b {\bar n})$.
By Lemmas~\ref{lem:Hdec12bd} and \ref{lem:Hdec3bd}, the last two terms on the right-hand side, 
$\Omega_1 + \Omega_2$,  contribute the explicit terms $N \log (R/\ell) + N^2 K_R^\ell$ as well as the 
error terms  in the statement of the proposition.
\end{proof}

The rest of this section is devoted to a proof of Lemma \ref{lem:Hdec3bd}. We start with a proof of \eqref{e:Hdec3bd}. 
\begin{proof}[Proof of \eqref{e:Hdec3bd}]
We start with  \eqref{e:HellHR}, which implies that
\begin{equation}  \label{e:HellHROmega}
  \hatE \p{ H_{Q}^{\ell} \circ \Psi - H_{V}^R \circ \Psi}
  - N \log (R/\ell) - N^2 K_R^\ell
  = -N^2 \hatE({\b L_R^\ell} \circ \Psi) 
  = -N^2 \hatE \sum_{\alpha , \beta}(\Omega_{\alpha \beta} \circ \Psi)
  ,
\end{equation}
where  we define (note that $\Omega_{\alpha\beta}$ should  not be confused with $\Omega_1$ and $\Omega_2$)
\begin{equation}
  \Omega_{\alpha \beta}(\b z)
  = \iint_{\alpha \times \beta}  L^\ell_R (v-w)\, \tilde\mu^{\b z}( \rd v) \, \tilde \mu^{\b z}(\rd w);
\end{equation}
here  $\tilde \mu^{\b z} = \hat\mu^{\b z} -\mu_V^R$
and we have made its dependence on $\b z \in \C^N$ through the empirical measure $\hat \mu  = \hat\mu^{\b z}$ explicit.
Recall $\Psi$ from \eqref{e:Psiproddef}
and note that $\Omega_{\alpha\beta} \circ \Psi$ is a function on $\prod_\alpha \T_\alpha^{n_\alpha}$. 
For any $\b v \in \prod_\alpha \T_\alpha^{n_\alpha}$,
denote by $\hat\mu_\alpha^{\b v}(\rd v)  = \bar n_\alpha^{-1} \sum_{j: v_j \in \T_\alpha} \delta_{v_j} (\rd v)$
the normalized empirical measure on $\T_\alpha$ as in \eqref{muhat-torus}.
We also denote by $\mu_\alpha(\rd v) = b^{-2} m(\rd v)$ the normalized uniform measure on $\T_\alpha$,
and set $\tilde\mu_\alpha = \hat\mu_\alpha- \mu_\alpha$.
We rewrite $\Omega_{\alpha\beta} \circ \Psi$ as
\begin{align}
  \Omega_{\alpha \beta}(\Psi (\b v))
  =
  \iint_{\T_\alpha\times \T_\beta}  L^\ell_R (\Psi_\alpha(v)-\Psi_\beta(w))\,
  &\Big [  \frac{\bar n_\alpha}{N} \hat\mu_\alpha^{\b v}(\rd v)   -  b^2 \rho_V(\Psi_\alpha(v))  \, 
   \mu_\alpha (\rd v)   \Big ]
    \nonumber\\
  \times
  & \Big [  \frac{\bar n_\beta}{N}  \hat\mu_\beta^{\b v}(\rd w)   -  b^2 \rho_V(\Psi_\beta(w))  \, \mu_\beta (\rd w)   \Big ]
 ,
\end{align}
where we  changed variables $v  \to \Psi_\alpha(v)$ and used \eqref{eqn:Jacobian}.
We then rewrite
\begin{align} \label{e:barnalphaexpand}
  [ \bar n_\alpha \hat\mu_\alpha^{\b v} (\rd v )   - N b^2 \rho_V(\Psi_\alpha(v)) \, \mu_\alpha(\rd v)]
  &= \bar n_\alpha \tilde\mu_\alpha^{\b v} (\rd v) - [Nb^2\rho_V(\Psi_\alpha(v)) - \bar n_\alpha] \, \mu_\alpha(\rd v).
\end{align}
Using that $\E^{\b \omega} \tilde\mu_\alpha^{\b v}(\rd v)=0$
to see that the cross terms between
the $\tilde\mu_\alpha^{\b v}$ and $\mu_\beta$ or $\tilde\mu_\beta^{\b v}$ terms vanish.
This gives
\begin{multline} \label{89}
  N^2 \hatE \Omega_{\alpha\beta}
  =  \b 1_{\alpha=\beta}\bar n_\alpha \bar n_\beta  \hatE \iint_{\T_\alpha \times \T_\beta}   L^\ell_R (\Psi_\alpha(v)-\Psi_\beta(w)) \, \tilde\mu_\alpha^{\b v} (\rd v) \, \tilde\mu_\beta^{\b v} (\rd w)
  \\
  + N^2 \hatE \iint_{\T_\alpha \times \T_\beta}  L^\ell_R (\Psi_\alpha(v)-\Psi_\beta(w))  \Big [\rho_V(\Psi_\alpha(v)) - \frac{\bar n_\alpha}{Nb^2}   \Big ] \, m(\rd v)  \,  \Big [\rho_V(\Psi_\beta(w)) - \frac{\bar n_\beta}{Nb^2} \Big ] \, m(\rd w)
  .
\end{multline}

The proof is completed by bounding the sums of the two term in \eqref{89} over $\alpha,\beta$.
The first  term with $\alpha = \beta$ is the key difficulty 
requiring the sophisticated averaging over $\Psi$.
Indeed, by \eqref{e:decoupling-YY} of Corollary~\ref{cor:decoupling} and using that there are $\OO(b^{-2})$ bulk squares $\alpha$, we have
\begin{multline}\label{C2} 
  \sum_{\alpha \subset D} \bar n_\alpha^2 \hatE \iint_{\T_\alpha\times \T_\alpha}   L^\ell_R (\Psi_\alpha(v)-\Psi_\alpha(w)) \, \tilde\mu_\alpha^{\b v} (\rd v) \, \tilde\mu_\alpha^{\b v} (\rd w)
  \\
=
N^{\varepsilon}b^{-2}\OO((Nb^2)^{4/5}/(\ell/b)^{2/5}+b^4\ell^{-4})
=
N^{\varepsilon}\OO(N^{4/5}/\ell^{2/5}+b^2\ell^{-4})
\end{multline}%
as needed.
The sum over $\alpha,\beta$ of the second term on the right-hand side of \eqref{89}
is bounded as needed in Lemma~\ref{lem:LTay} stated below.
\end{proof}

In the statement of the following lemma, a naive bound of the left-hand side is $N^2 b^2$.
We gain an extra factor $b$ for each integration variable due to the cancellation of the
integrand, and thus obtain the resulting stronger estimate.

\begin{lemma} \label{lem:LTay}
Assume the parameters $b$ and $\ell$ satisfy the condition  \eqref{lb-assumption}. Then we have
\begin{multline} \label{89xxx}
  N^2 \sum_{\alpha,\beta} \hatE \iint_{\T_\alpha\times\T_\beta}  L^\ell_R (\Psi_\alpha(v)-\Psi_\beta(w))
  \Big[\rho_V(\Psi_\alpha(v)) - \frac{\bar n_\alpha}{Nb^2}\Big] \, m(\rd v)
  \Big[\rho_V(\Psi_\beta(w)) - \frac{\bar n_\beta}{Nb^2} \Big] \, m(\rd w)
  \\= 
  N^\epsilon \OO(N^2 b^4) \|\nabla \rho_V\|_{\infty,3}^2
   +\OO((\log N)b^{-2})
\end{multline}
where the sum is over bulk squares $\alpha,\beta$.
\end{lemma}

\begin{proof}
By changing variables, the claim is equivalent to
\begin{multline} \label{89xx}
  N^2 \sum_{\alpha,\beta \subset D} \hatE \iint_{\alpha\times\beta}  L^\ell_R (z-w)
  \Big[\rho_V(z) - \frac{\bar n_\alpha}{Nb^2} \Big] \, m(\rd z)
  \Big[\rho_V(w) - \frac{\bar n_\beta}{Nb^2} \Big] \, m(\rd w)
  \\= 
  N^\epsilon \OO(N^2 b^4) \|\nabla \rho_V\|_{\infty,3}^2
   +\OO((\log N)b^{-2})
  .
\end{multline}
By definition of $\bar n_\alpha$ we have $|\bar n_\alpha - N\mu_V(\alpha)| \leq 1$, and thus (recalling that $\mu_V$ has density $\rho_V$)
\begin{equation}\label{e:rhoVbarn1}
  \rho_V(z) - \frac{\bar n_\alpha}{Nb^2}
  =
  \rho_V(z) - \frac{\mu_V(\alpha)}{b^2} + O(\frac{1}{Nb^2})
  =
  \frac{1}{b^2} \int_\alpha m(\rd \zeta) \pa{
  \rho_V(z) - \rho_V(\zeta) }  \rd\zeta + O(\frac{1}{Nb^2})
\end{equation}
as well as
\begin{equation} \label{e:rhoVbarn2}
  \rho_V(w) - \frac{\bar n_\alpha}{Nb^2}
  =
  \frac{1}{b^2} \int_\alpha m(\rd \xi) \pa{
    \rho_V(w) - \rho_V(\xi) }  \rd\xi + O(\frac{1}{Nb^2})
  .
\end{equation}
We will use these bounds below and also bound $L_R^\ell(z-w)$ by $\OO(\log \ell) = \OO(\log N)$ in the integral.

We first consider the diagonal terms $\alpha=\beta$ on the left-hand side of \eqref{89xx}.
We claim that the contribution of each such term is
$\OO(\log N) [ N^2 b^6\|\nabla \rho_V \|_\infty^2+1]$.
To see this, note that the factor $\log N$ is due to $L^\ell_R$,
and a factor $b^2$ arises from each of the integration of $z$ and $w$.
The first terms on the right-hand sides of
\eqref{e:rhoVbarn1}--\eqref{e:rhoVbarn2} contribute
a factor $b \|\nabla \rho_V\|_\infty$ each from bounding $[\rho_V(z) - \rho(\zeta) ]$
respectively $[\rho_V(w) - \rho(\xi) ] $,
while the error term are of order $1/(Nb^2)$.
Since there are $\OO(b^{-2})$ many bulk squares $\alpha$,
this bounds the sum over the terms $\alpha=\beta$ as claimed.

Next we consider the off-diagonal terms $\alpha \neq \beta$;
we drop  sub- and superscripts from $L$ and $\rho$ to shorten notations. 
We use a Taylor expansion to find that the sum of these terms is bounded,
up to remainder terms, by
\begin{multline*}
N^2  \sum_{\alpha\neq\beta}  \iint_{\alpha\times \beta} \iint_{\alpha\times \beta}
\Big [ \nabla \rho(\alpha) (z- \zeta)  + \nabla^2 \rho(\alpha)(z- \zeta)^2 \Big ]
\Big [ \nabla \rho(\beta) (w- \xi)  + \nabla^2 \rho(\beta) (w- \xi)^2 \Big ]
\\
\times \Big [ L (\alpha-\zeta)  +  \nabla L (\zeta-\xi)  (z-\zeta -w + \xi)
+ \nabla^2 L (\alpha-\xi)   (z-\zeta -w + \xi)^2 \Big ]
\, m(\rd z) \, m(\rd  w)
\, m(\rd \zeta) \, m(\rd  \xi).
\end{multline*} 
The remainder term are bounded similarly without using symmetry
and produce the error terms depending on $\|\nabla^3 \rho\|_\infty$.
We return to the main terms.
By symmetry, the odd terms in $(z- \xi)$ and $(w- \zeta)$ do not contribute.
The leading terms are therefore the quartic terms. These terms are bounded by
$$
N^2 b^4 (\|\nabla \rho\|_\infty + \|\nabla^2\rho\|_\infty)^2
.
$$
The factor $b^4$ comes from $b^{-4} b^4 b^4$ with the factor $b^{-4}$ coming from the summation over squares; 
the  $b^4$ factor coming from the volume of the integration of $z$ and $w$,
and the last $b^4$ factor comes  from the size of products of  $(z- \zeta)$ and $(w- \xi)$ in the formula.
This concludes the proof.
\end{proof}

\begin{proof}[Proof of \eqref{e:Hdec3bd-B}]
We now bound $\hatE\Omega_{\alpha\beta}$
for $\beta=B$ and $\alpha \subset D$.
Since $\tilde \mu_V = \hat \mu - \mu_V$ and $\hatE\hat\mu|_\alpha$
is the uniform measure on $\alpha$ with total mass $\bar n_\alpha/N$, we have
\begin{align*}
  \hatE(\tilde \mu_V|_\alpha(\rd z))
  &= \pa{\frac{\bar n_\alpha}{Nb^2\rho_V(z)}-1} \, \mu_V|_\alpha(\rd z),
  \\
  \hatE(\tilde \mu_V|_B(\rd z))
  &= \pa{\frac{\bar n_B}{N\mu_V(B)}-1} \mu_V|_B(\rd z).
\end{align*}
Since $\hat\mu|_B$ and the $\hat\mu|_\alpha$ are independent under $\hatE$,
and since the number of squares $\alpha$ is $\OO(b^{-2})$
and bounding $L_R^\ell$ by $\OO(\log N)$, therefore
\begin{align*}
  &N^2 \sum_\alpha \hatE \int_{z \in B} \int_{w \in \alpha}  L^\ell_R (\Psi_B(z)-\Psi_\alpha(w)) \, \tilde\mu_V(\rd z) \, \tilde \mu_V(\rd w)
  \\
  &=\pa{\frac{\bar n_B}{\mu_V(B)} - N} \sum_\alpha \int_{z \in B} \int_{w \in \alpha}  L^\ell_R (z-w) \pa{ \frac{\bar n_\alpha}{b^2 \rho_V(z)} - N}  \, \mu_V(\rd z) \, \mu_V(\rd w)
  \\
  &=\OO\pa{ (\log N) \absb{\bar n_B - N \mu_V(B)} \sum_\alpha \absb{\bar n_\alpha - N \mu_V(\alpha)} }=\OO\pa{ (\log N) b^{-2}}
  .
\end{align*}
Similarly, for $\alpha=\beta=B$, we have
\begin{align*}
    &N^2\hatE\int_{z \in B} \int_{w \in B}  L^\ell_R (z-w)\,  \mathds{1}_{z \neq w} \, \tilde\mu(\rd z) \, \tilde \mu(\rd w)
  \\
    &=
    \pa{\frac{\bar n_B}{\mu_V(B)} - N}^2
    \iint_{B \times B} L_R^\ell(z-w) \, \mu_V(\rd z) \, \mu_V(\rd w)
  \\
  &= 
    \pa{\frac{\bar n_B}{\mu_V(B)} - N}^2
    \OO(\log N)\mu_V(B)^2
    = \OO(\bar n_B \log N).
\end{align*}
This completes the proof.
\end{proof}

\subsection{Summary}
In order to prove  the lower bound for the partition function  of the Coulomb gas,  we used a quasi-local approximation whose main building 
blocks are 
Yukawa gases on torus. Due to the natural that the lower bound is proved via  a variational trial state,  all estimates  needed are 
with respect to a Yukawa gas on tori. 
In particular, the  rigidity estimate  needed in the lower bound is with respect to a Yukawa gas on a torus.
This rigidity estimate is done in  Appendix~\ref{app:Yrigi}.

\section{Proof of Theorem~\ref{clt}: central limit theorem}
\label{sec:clt}

In this section, we prove Theorem~\ref{clt}.
Our proof uses a modification of the \emph{loop equation},
 which is a relation between one- and two-point correlations.
It allows to obtain the moment generating function for linear statistics
of the Coulomb plasma in terms of expectations of terms involving
one-point and two-point correlations with respect to a tilted measure.
The density estimates  of Section~\ref{sec:prelim} provide sufficient control on the one-point terms
in the loop equation.
The two-point correlation term in the loop equation is singular and can be decomposed into short- and long-range contributions.
The long-range part can be decomposed further into scales which can
then be bounded using local density estimates for all scales.
Thus the short-range contribution, which we call the local \emph{angle term},
is the main difficulty to obtain the CLT.

In \cite{MR3342661}, the loop equation was used to prove a central limit theorem for $\beta=1$,
by bounding the two-point contribution using the determinantal structure of the microscopic point process
(which holds only for $\beta=1$).
In \cite{MR3694026}, we used the loop equation approach for Coulomb plasma for any $\beta>0$ to obtain rigidity estimates,
by introducing the local density estimates to estimate the long-range part of the two-point contribution
and bounding the angle term by a trivial bound.
In this section, we obtain an effective estimate for the angle term for general $\beta>0$.
We deduce this estimate from Theorem~\ref{freeasy} and the fact that the estimates for the remaining terms in the loop equation
can be obtained also for version of the Coulomb plasma that is tilted by a small two-body interaction.

We remind the reader that all estimates in this section are with respect to a Coulomb gas  with or without an angle correction term; the Yukawa gas 
is used only in the approximation of the free energy of the Coulomb gas,
 in Sections~\ref{sec:quasifree}-\ref{sec:quasifreelb}.
In particular,  the estimate of the angle term,  to be presented in this section,
is with respect to a Coulomb gas. This estimate requires not just the local density bound, 
but the  sophisticated rigidity estimate which is a consequence of the loop equation. The rigidity  estimate will also be needed for Yukawa gases on a torus, to be presented 
in the Appendix~\ref{app:Yrigi}.

\subsection{CLT for macroscopic test functions}

We first prove Theorem~\ref{clt} for macroscopic test functions $f$.
For this, we first prove that a version of Theorem~\ref{clt} holds up to certain random shift,
the local angle term $\hat A^f_{V}$ defined by
\begin{equation}  \label{e:Adef}
  \hat A^f_{V}
  =  \frac{N}{2} \re \iint_{z\neq w}
  \frac{h(z)-h(w)}{z-w} \ee^{-\frac{|z-w|^2}{2\theta^2}}
  \,  \tilde \mu_V (\rd z) \, \tilde \mu_V (\rd w)
  , 
  \qquad
  h(z) = \frac{\bar \partial f(z)}{\partial \bar \partial V(z)}
  ,
\end{equation}
\nomHam[01]{$\hat A^f_{V}$}{local angle term for the test function $f$}%
where $\theta = N^{-1/2+\sigma}$. Note the integrand is singular at $z=w$ since
\begin{equation*}
  \frac{h(z)-h(w)}{z-w} = \partial h(z) + \bar\partial h(z) \frac{\bar z - \bar w}{z-w} + \OO(|z-w|)
  .
\end{equation*}
We recall the definitions of $X_V^f$ and $Y_V^f$ from \eqref{eqn:Xf} and \eqref{eqn:Yn},
as well as the norms from \eqref{normt}, and that we write $\|f\|_{\infty,k}$ for $\|f\|_{\infty,k,b}$ with $b=1$.

In the proof of \cite[Theorem~1.2]{MR3694026}, more precisely in \cite[Lemma~7.5]{MR3694026},
we showed that \eqref{e:Adef} is bounded by $\OO(N^\epsilon)$ with very high probability.
Assuming this term was $\ll 1$ instead of $\OO(N^\epsilon)$, a small modification of the argument in \cite[Section~7]{MR3694026} would already
imply Theorem~\ref{clt}.
A similar strategy was used in \cite{MR3342661,MR2817648}, where a version of \eqref{e:Adef} was
shown to be approximately equal to $-\frac12 Y^f_V$ for $\beta=1$, by using the exactly known
correlation kernel for the microscopic correlation functions in this integrable case.
Our strategy now is to first prove a version of Theorem~\ref{clt} in which the contribution
of the angle term \eqref{e:Adef} has been removed (in Proposition~\ref{prop:clt-angle1} below),
and then subsequently, by combining this argument with  Theorem~\ref{freeasy},
prove that the angle term \eqref{e:Adef} is in fact negligible up to the constant $-\frac12 Y^f_V$
(in Proposition~\ref{prop:anglevanish}).

\begin{proposition} \label{prop:clt-angle1} 
Suppose that $V$ satisfies conditions \eqref{e:Vgrowth} and \eqref{Vcondition1},
or more generally the conditions stated in Remark~\ref{rk:quasifree-meso}.
Then for any small $\sigma$, the following holds.
For any function $f$ satisfying the same assumptions as in Theorem~\ref{clt}
(in particular the support of $f$ has distance of order $1$ to $\partial S_V$), 
for  small $\varepsilon$ and
$ t b^{-2} N^{2 \sigma} + t b^{-2} \|f\|_{\infty,4,b} \ll 1$,
we have for any $0 \le |u| \le \OO(t)$
\begin{multline}\label{p1}
\frac{1}{t \beta N} \log \E_V \ee^{  - \beta  N t (X_{V} ^f- \hat A^{f}_{V+ u f}) }
=
\frac{tN}{8 \pi} \int  | \nabla  f (z) |^2 \, m(\rd z) 
- \frac { 1} {\beta}Y_V^f
\\
+ \OO(N^{-1/2 +  3\sigma+ \epsilon} b^{-1} + N^{-\sigma+\epsilon}) \norm{f}_{\infty,3, b}
+ \OO(t N^{2\sigma+\epsilon} b^{-2}) \norm{f}_{\infty, 3,b}^2
\ + \OO(N^{-1/2+\epsilon}b\norm {f}_{4, b})
.
\end{multline}
\end{proposition} 

\begin{proposition}\label{prop:anglevanish}
There exists $\kappa>0$ such that
if $\sigma=\kappa/6$ and $0 \le |u|,  t \leq N^{-2\kappa/3}$,
\begin{equation}\label{35} 
  \frac{1}{t\beta N}
  \log  \E_{V }  \ee^{\beta N t \hat A^{f}_{V +uf}} 
  =
  -\frac{1}{2}Y_{V}^f 
  +
  \OO(N^{-\kappa/3}) (1+\|f\|_{\infty,5})^2.
\end{equation}
\end{proposition}

The above two propositions will be  proved in Sections~\ref{sec:cltangle}-\ref{sec:anglevanish} below.
We first note that the estimate \eqref{p1} given by Proposition~\ref{prop:clt-angle1}
without the angle term $\hat A^{f}_{V+uf}$ on the left-hand side
would imply a CLT for $X_V^f$.  This angle term 
is controlled by \eqref{35} of Proposition~\ref{prop:anglevanish}. 
By combining the two estimates,
we can  complete the proof of Theorem~\ref{clt} for macroscopic test functions.
For mesoscopic test functions, a similar argument applies after conditioning
(see Section~\ref{sec:meso}).

\begin{proof}[Proof of Theorem~\ref{clt} for macroscopic test functions]
By assumption, $f$ is a macroscopic test function with $\normt{f}_{\infty,5}$ bounded.
Let $\sigma$ and $\kappa$ be as in Proposition~\ref{prop:anglevanish}.
Then, with $\lambda=Nt$ in the identity
\begin{equation} \label{e:logpm}
\frac{1}{t\beta N}\log \E_{V}\ee^{-\beta N t X^f_{V}}
=\frac{1}{t\beta N}\left(
\log \E_{V}\ee^{-\beta N t(X^f_{V}-\hat A^f_{V})}
-
\log \E_{V+tf} \ee^{\beta N t \hat A^f_{V}}\right),
\end{equation}
the claim follows by using 
the estimates \eqref{p1}, \eqref{35} for the two terms on the right-hand side of \eqref{e:logpm},
and finally replacing $\kappa$ by $3\kappa$. 
\end{proof}

\subsection{Loop equation with angle term}
\label{sec:cltangle}

We start the proof with an integration by parts formula. 
Consider a smooth bounded function $v:\mathbb{C}\to\mathbb{C}$, and $G$ smooth, defined on $z_1\neq z_2$ such that $G(z_1,z_2)=G(z_2,z_1)$, and 
\begin{equation}\label{eqn:Ggrowth}
\limsup_{|z_2|\to\infty}(|G(z_1,z_2)|/\log|z_2|)\leq 1
\end{equation}
for any fixed $z_1$.
For any $\bold z\in\mathbb{C}^N$ we define  
\begin{equation}\label{e:wvplus}
W_V^{G, v}({\bf z})= -\sum_{j \neq k} (v(z_j)-v(z_k)) \partial_{z_j} G(z_j,z_k) + \frac{1}{\beta} \sum_j \partial_j v(z_j) - N \sum_j v(z_j) \partial V(z_j).
\end{equation}
\nomHam[13]{$W_V^{G, v}$}{evaluated Hamiltonian in the Ward identity}
The following elementary lemma is often referred to as Ward identity or loop equation.
For example, it was used in \cite{MR3342661} to study fluctuations of the empirical measure
when $\beta=1$, 
and in \cite{MR3694026} to prove rigidity for all $\beta>0$, with in both cases the interaction $G$ being
the Coulomb potential $\mathcal{C}$.
Its relation to Conformal Field Theory is discussed in \cite{MR3052311}.
In this work we need a perturbation $G$ of the Coulomb interaction by the local angle term.

\begin{lemma}\label{lem:expectation0}
Under the above assumptions, we have
$$
\E^G_V\left(W_V^{G, v}\right)= \frac{1}{2}\E^G_V\Big (\sum_{j\neq k}(v(z_j)+v(z_k))(\partial_{z_k}+\partial_{z_j})G(z_j,z_k)\Big),
$$
where the expectation is with respect to $P_{N, V}^G$ defined in (\ref{e:Pdef}).
\end{lemma}

\begin{proof}
The proof is a classical simple integration by parts: for any $j\in\qq{1,N}$, we have
$$
\E\left(\partial_{z_j} v(z_j)\right)
=
\beta\E\left(v(z_j)\partial_{z_j} H(\bold z)\right),
$$
where both terms are absolutely summable and the boundary terms vanishes because (i) with probability 1, no two $z_i$'s have the same real or imaginary part, (ii) $v$ is bounded, $G$ satisfies the growth condition (\ref{eqn:Ggrowth}), $V$ satisfies the growth condition (\ref{e:Vgrowth}).
Summation of the above equation over all $j\in\llbracket 1,N\rrbracket$
therefore gives
\begin{align*}
&\frac{1}{\beta N}\sum_{j=1}^N\E(\partial_{z_j} v(z_j))
=\E\Big (\sum_{j=1}^Nv(z_j)\Big ( \partial_{z_j}V(z_j)+\sum_{k\neq j}(\partial{z_j}G(z_j,z_k)+\partial{z_j}G(z_k,z_j))\Big )\Big )\\
& =\E\Big (\sum_{j=1}^Nv(z_j)\Big (\partial_{z_j}V(z_j)+\sum_{k\neq j}(\partial{z_j}-\partial{z_k})G(z_j,z_k)\Big)\Big )
+
\E\Big (\sum_{j=1}\sum_{k\neq j}v(z_j)(\partial{z_j}+\partial{z_k})G(z_j,z_k)\Big ).
\end{align*}
Using $G(z_j,z_k)=G(z_k,z_j)$, we can continue the equation with 
\begin{align*}
&=\E\Big (\sum_{j=1}^Nv(z_j)\partial_{z_j}V(z_j)+\frac{1}{2}\sum_{j\neq k}\Big (v(z_j)(\partial{z_j}-\partial{z_k})G(z_j,z_k)+v(z_k)(\partial{z_k}-\partial{z_j})G(z_k,z_j)\Big )\Big )\\
&\qquad+
\frac{1}{2}\E\Big (\sum_{j\neq k}(v(z_j)+v(z_k))(\partial{z_j}+\partial{z_k})G(z_j,z_k)\Big )\\
&=\E\Big (\sum_{j=1}^Nv(z_j)\partial_{z_j}V(z_j)+\frac{1}{2}\sum_{j\neq k}\Big (v(z_j)-v(z_k)\Big )(\partial{z_j}-\partial{z_k})G(z_j,z_k)\Big )\\
&\qquad+
\frac{1}{2}\E\Big (\sum_{j\neq k}(v(z_j)+v(z_k))(\partial{z_j}+\partial{z_k})G(z_j,z_k)\Big ).
\end{align*}
This concludes the proof.
\end{proof}

Before considering the interaction $G$ with additional angle term,
we temporarily restrict our attention to the Coulomb case, where
$\partial_{z_j} \cal C (z_j-z_k)  =  -\frac{1}{2}(z_j-z_k)^{-1}$.

\begin{lemma}\label{lem:Wdef} For any $f:\mathbb{C}\to\mathbb{R}$  of class $\mathscr{C}^2$ supported on $S_V$ and  $\bold z\in\mathbb{C}^N$, recall $X_V^f$ defined  in \eqref{eqn:Xf}
and  $h$, depending on $f$ and $V$,   defined in \eqref{e:Adef}. With these notations, 
we have
\begin{equation}\label{Wdef}
X_V^f
=
- \frac{1}{N} W_V^h ({\bf z})
+ \frac{1}{N\beta }  \sum_k \partial h(z_k)
+ \frac{N}{2} \iint_{z\neq w} \frac{h(z)-h(w)}{z-w} \, \tilde \mu_V(\rd z) \, \tilde \mu_V(\rd w),
\end{equation}
where  $\tilde \mu_V= \hat \mu-\mu_V$ and we used the notation $W_V^h = W_V^{\cal C, h}$.
\end{lemma}

\begin{proof}
First remember the following two identities:
\begin{align}
\int \frac{\mu_V(\rd w)}{z-w}=\partial V(z), \quad f(z)=\frac{1}{\pi}\int \frac{\bar\partial f(w)}{z-w}m(\rd w)\label{eqn:alg1}.
\end{align}
The first equation holds for $z\in S_V$ and is obtained by the Euler-Lagrange equation, the second equation is a simple integration by parts.
We therefore can write 
\begin{align*}
X_V^f&= \sum_j \int \frac{h(w)}{z_j-w}\mu_V(\rd w)-N\iint\frac{h(w)}{z-w}\mu_V(\rd w)\mu_V(\rd z)\\
&=N\iint \frac{h(w)-h(z)}{z-w}\hat \mu_V(\rd z)\mu_V(\rd w)+\sum_j h(z_j)\partial V(z_j)-\frac{N}{2}\iint\frac{h(w)-h(z)}{z-w}\mu_V(\rd w)\mu_V(\rd z)\\
&=-\frac{1}{2N}\sum_{j\neq k}\frac{h(z_j)-h(z_k)}{z_j-z_k}+\sum_j h(z_j)\partial V(z_j)+\frac{N}{2} \iint_{ z \not = w } \frac{h(z)-h(w)}{z-w}  \tilde \mu_V(\rd z) \tilde \mu_V(\rd w),
\end{align*}
which is equivalent to (\ref{Wdef}). In the first equation we used (\ref{eqn:rhoV}) and (\ref{eqn:alg1}), and in the second equation we used (\ref{eqn:alg1}).
\end{proof}

We now decompose  the last term in \eqref{Wdef} into a sum of the long-range and short-range terms.
For this purpose, let  $\varphi(z)=\ee^{-|z|^2}$ and, given a mesoscopic scale 
$\theta=N^{-\frac{1}{2}+\sigma}$,  
we define
\begin{align}
\Phi(z-w,r)&=\frac{2}{\pi}\int \varphi\left(\frac{|z-\xi|}{r}\right) \varphi\left(\frac{|\xi-w|}{r}\right)\rd m(\xi)=r^2\ee^{-\frac{|z-w|^2}{2 r^2}},\notag\\
\Phi_\theta^- (z-w)&= \int_0^\theta   \Phi (z-w,  r)  \frac{\rd r }{r^5}
=\frac{\ee^{-\frac{|z-w|^2}{2\theta^2}}}{|z-w|^2},\label{eqn:explicit}\\
\Phi_\theta^+ (z-w)&= \int_\theta^\infty   \Phi (z-w,  r)  \frac{\rd r }{r^5}=\frac{1-\ee^{-\frac{|z-w|^2}{2\theta^2}}}{|z-w|^2},  \notag \\
 \Psi_h^\pm (z, w) &= \Phi_\theta^\pm (z-w)  (\bar z-\bar w)(h(z)-h(w)),   \quad 
 \Psi^h (z, w) = \Psi_h^+ (z, w)+  \Psi_h^- (z, w).  \label{psidef}
\end{align}  
As in the proof of \cite[Lemma~7.5]{MR3694026} (see also \cite{MR864658}),
we have decomposed the last term in \eqref{Wdef} into a relatively long range part and, essentially, a local angle term:
$$
\frac {N} 2  \iint_{z\neq w}\frac{h(z)-h(w)}{z-w}  \tilde \mu_V (\rd z) \, \tilde \mu_V (\rd w) 
=A_V^{h, +} +  A^{h, -}_V,
$$
where
\begin{align}\label{Ap}
  A_V^{h, +} &=   \frac{N}{2}   \iint_{z\neq w} \Psi^+_h (z, w) \, \tilde \mu_V (\rd z) \, \tilde \mu_V (\rd w),  \\
  \label{Am}
 A^{h, -}_V&=   \frac{N}{2} \iint_{z\neq w}  \Psi^-_h (z, w) \, \tilde \mu_V (\rd z) \, \tilde \mu_V (\rd w). 
\end{align}
\nomHam[02]{$  A_V^{h, +} $}{long-range angular term}
\nomHam[03]{$ A^{h, -}_V$}{short-range angular term}
 By definition  \eqref{e:Adef}, we also have 
\begin{equation}  \label{e:Adef1}
  \hat A^f_{V}
  =  \frac{N}{2} \re \iint_{z\neq w}
  \frac{h(z)-h(w)}{z-w} \ee^{-\frac{|z-w|^2}{2\theta^2}}
  \,  \tilde \mu_V (\rd z) \, \tilde \mu_V (\rd w) =   \re A^{h, -}_V,  
\end{equation}
i.e., $ \hat A^f_{V}$ is just $ \re A^{h, -}_V$ with $h$ chosen according to \eqref{e:Adef}.

Note that, in the above decomposition, we could have considered  any fixed
non-negative function $\varphi \in C^\infty(\mathbb{C})$ with compact support or fast decay at infinity,
as in \cite[Lemma~7.5]{MR3694026}.
We here chose the Gaussian scale function for the sake of concreteness and some convenient simplifications.
Compared with \cite{MR3694026}, we also write the mesoscopic scale as $\theta$
rather than $N^{-1/2}\theta$.

\subsection{Coulomb gas with angle perturbation}

We now define the perturbed Coulomb gas.
The Coulomb gas, exponentially tilted by the real-part of the local angle term,
is defined to have pair interaction and potential given by
\begin{equation} \label{e:GVangle}
G_t = \cal C  - \frac{t}{2} \re  \Psi_{h}^-, \quad
V_t = V+tf + t F, \quad F=\re \int \Psi_{h}^-(\cdot, w) \, \mu_V(\rd w), 
\end{equation}
where $h =\frac{\bar\partial f}{\partial\bar\partial V}$ (we will see that $h=h_0$ defined in \eqref{ht} below).
We also include a $t$-dependent constant in the perturbed Hamiltonian and define 
\begin{align}
\label{eqn:HamiltonAngle}
 H_t :=
H_{V_t}^{G_t}
- \frac{t}{2} { N^2} \re \iint  \Psi^-_h (z, w) \, \mu_{V} (\rd z) \, \mu_{V} (\rd w)
= H_{V + tf}^{\cal C }  -  N t \hat   A^{f}_{V}.
\end{align}
Notice that the interaction term involving $ \re  \Psi_{h}^-$, the potential term involving $tF$ and the constant term in \eqref{eqn:HamiltonAngle}
were recombined  into  $\hat   A^{f}_{V}$ which was  defined in \eqref{e:Adef}.
Notice further that  the subscript $V$ is different from 
the subscript $V + tf$ in the Hamiltonian in \eqref{eqn:HamiltonAngle}. 

For the proof of Proposition~\ref{prop:clt-angle1}, we require the following Proposition \ref{prop:CAdensity2} regarding 
a local density
estimate for this interaction.

\begin{proposition}\label{prop:CAdensity2}
Consider the Coulomb  gas with Hamiltonian \eqref{eqn:HamiltonAngle}, with
$V,f\in \mathscr{C}^2$
and $tN^{2\sigma} \leq 1$ and $\|\nabla h\|_\infty \leq 1$ and $t\in[0,1]$.
For all $s \in (0,\frac12)$,
for all $f$ supported in ball of radius $b=N^{-s}$ contained in $S_V$
with distance of order $1$ to the boundary,  
we have the local density estimate
\begin{equation}
\label{Sec2rig}
X_f^{V_t} \prec  \sqrt{Nb^2} \|  f  \|_{\infty,2,b}
\end{equation}
with respect to the measure $P_{V_t}^{G_t}$.
In particular, for any ball as above, the number of particles
in that ball is bounded by $\OO(Nb^2)$ with high probability.
\end{proposition}

\begin{proof}
  The proposition is a direct consequence of Theorem~\ref{thm:Cdensity}.
  Indeed, \eqref{eqn:HamiltonAngle}
  corresponds to the choice $\tilde G(z,w) = \frac{h(z)-h(w)}{z-w} \ee^{-|z-w|^2/(2\theta^2)}$
  in \eqref{e:C-perturb-def}
  which satisfies \eqref{e:C-perturb} since $\|\nabla h\|_\infty \leq 1$.
\end{proof}

For $0 \le t \ll1$, we define 
\begin{equation}\label{ht}
  h_t(z)=\frac{\bar\partial f(z)}{\partial\bar\partial(V(z)+t f(z))}, \qquad h = h_0.
\end{equation}
In the next lemma, we collect some elementary estimates for $h_t$ and $F_t$.
Recall that $V$ satisfies the growth condition  \eqref{e:Vgrowth} and the regularity assumption  \eqref{Vcondition1}.

\begin{lemma}\label{lem:86}
Assume that the support of $f$ has distance $\gg N^{-1/2+\sigma}$ to $\partial S_V$,
and that
\begin{equation} \label{Af}
  t b^{-2} \| f \|_{\infty,4, b} \ll 1.
\end{equation}
Recall $\theta = N^{-1/2 + \sigma}$.   Then the following estimates hold:  
\begin{align}\label{hbound}
  \|h_t\|_{k, b}
  &\le b^{-1}  \| f \|_{\infty,k+1, b}  \big [   1 +    t   b^{-2} \| f \|_{\infty,k+2, b}     \big ]
  ,\\
  \label{Fb}
  t F(z)
  &= \OO(  N^{-1 + 2 \sigma})  t b^{-2} \norm {f}_{\infty,2, b},
  \\
  \label{Fb2}
  t\Delta F(z)
  &= \OO(  N^{-1/2 + \sigma})  t b^{-3} \norm {f}_{\infty,4, b}.
\end{align}
\end{lemma}

\begin{proof}
Using that $t\|\Delta f\|_{\infty} \ll 1$ and \eqref{Af}, we have  
\begin{align*}
 \|\nabla h_t\|_{\infty} 
  \le \frac{ \| \nabla \bar \partial f \|_\infty}{ \norm{\partial \bar \partial (V + t f )}_\infty } 
  +  \frac{\norm{ \bar \partial f  \nabla (\partial \bar \partial (V + t f ))}_\infty }{\norm{\partial \bar \partial (V + t f )}_\infty^2}   
 \le b^{-2}  \| f \|_{\infty,2, b}  \big [   1 +    t   b^{-2} \| f \|_{\infty,3, b}     \big ] 
  .
\end{align*} 
Similar estimates hold for higher derivatives and we get in general \eqref{hbound}.
We can bound $tF$ by
\begin{equation*}
  t F(z)
  = t \int \frac{h(z)-h(w)}{z-w} \ee^{-\frac{|z-w|^2}{2\theta^2}}  \, \mu_{V}(\rd w) 
  =\OO(  N^{-1 + 2 \sigma})  t  \norm {\nabla h}_\infty =\OO(  N^{-1 + 2 \sigma})  t b^{-2} \norm {f}_{\infty,2, b},
\end{equation*} 
which is a  small correction to $V+tf$.  
Similarly, we have 
\begin{align*}
  t \Delta F(z)   
  &= t \Delta h (z) \int \frac{  \ee^{-\frac{|z-w|^2}{2\theta^2}}}{z-w}    \mu_{V}(\rd w)
    - 2  t    \nabla h(z)  \int \pbb{\nabla_w \frac{  \ee^{-\frac{|z-w|^2}{2\theta^2}}}{z-w}}  \, \mu_{V}(\rd w)  
  \\ &\qquad
       +  t  h (z) \int \pbb{ \Delta_w \frac{  \ee^{-\frac{|z-w|^2}{2\theta^2}}}{z-w}}  \,   \mu_{V}(\rd w)
  = \OO(N^{-1/2+\sigma+\epsilon}),
\end{align*} 
where for the last estimate we integrated $w$ by parts to avoid the singularity.
\end{proof}

By using the local law of Proposition~\ref{prop:CAdensity2} in the loop equation,
as in \cite[Section~7]{MR3694026}, we obtain the following estimate. 
Recall that  $ \hat A^f_{V}$ was defined in  \eqref{e:Adef} and satisfies \eqref{e:Adef1}.  

\begin{lemma} \label{lem:cltangle-bis}
Suppose that the assumption \eqref{Af} holds. 
Recall $\sigma$ is the parameter in the definition \eqref{e:Adef}. Then for any $0 \le |u| \le  \OO(t)$
\begin{multline}
\frac{1}{t \beta N} \log \E_V \ee^{  - t \beta  N (X_{V} ^f- \hat A^{f}_{V + uf} ) }
=
\frac{tN}{8 \pi} \int  | \nabla  f (z) |^2 \, m(\rd z) 
- \frac { 1} {\beta}Y_V^f
+ \frac{1}{t} \re \int_0^t \E_{V_s}^{G_s} \left(A^{h_s, +}_{V_{s}}\right) \, \rd s\\
 + \OO(N^{-1/2+3\sigma+\epsilon}b^{-1}\|f\|_{3,b})
  + \OO(t N^{\epsilon}     b^{-2} N^{2\sigma}  \norm {f}_{3, b}^2)
  + \OO(N^{-1/2+\epsilon}b\norm {f}_{4, b}),  \label{e:cltangle-bis}
\end{multline}
\begin{multline}
\frac{1}{t \beta N} \log \E_V \ee^{  - t \beta  N X_{V} ^f }
 =
\frac{tN}{8 \pi} \int  | \nabla  f (z) |^2 \, m(\rd z) 
- \frac { 1} {\beta}Y_V^f
+ \frac{1}{t} \re  \int_0^t \E_{V + sf}^{\cal C} \left(A^{h_s, -}_{V+ s f} + A^{h_s, +}_{V+ s f}\right) \, \rd s\\
 + \OO(N^{-1/2+3\sigma+\epsilon}b^{-1}\|f\|_{3,b})
  + \OO(t N^{\epsilon}     b^{-2} N^{2\sigma}  \norm {f}_{3, b}^2)
  + \OO(N^{-1/2+\epsilon}b\norm {f}_{4, b}). \label{e:cltangle-bis-noangle}
\end{multline}
\end{lemma}

\begin{proof}
We focus on \eqref{e:cltangle-bis}; the second bound  \eqref{e:cltangle-bis-noangle} can be proved in a similar way.
Note that the expectation  on the right-hand side of
\eqref{e:cltangle-bis-noangle} is with respect to the standard Coulomb gas
without local angle term,
and that the terms $A^{h_s, \pm}_{V+ s f}$ are with respect to the external potential $V+sf$.  The estimate \eqref{e:cltangle-bis-noangle} was  essentially 
obtained in \cite[Section~7]{MR3694026} already. The short range angle term, $A^{h_s, -}_{V+ s f} $, was difficult to estimate in \cite[Section~7]{MR3694026}. In \eqref{e:cltangle-bis}, we added an angle term 
in the Hamiltonian so that there is no such short range angle term on the right side of \eqref{e:cltangle-bis}.
The following proof is written for $u=0$ for the simplicity of notations;
we will remark on the modification needed for the general case in the  proof. Furthermore,
the error $A^{h_t, -}_{V_t} - A^{h, -}_{V+uf}$ will be estimated in Lemma \ref{lem:Amdiff}.

We denote by $Z_t$ the partition function corresponding to the Hamiltonian \eqref{eqn:HamiltonAngle}.
Then the left-hand side of \eqref{e:cltangle-bis} can be written as
\begin{equation*}
  \frac{1}{t\beta N} \pa{ \log Z_{t} - \log Z_0 } + N \int f \, \rd \mu_V
  =
  \frac{1}{t}  \int_0^t \qa{ \partial_s \frac{1}{\beta N} \log Z_s + N \int f \, \rd \mu_V} \rd s .
\end{equation*}
Using the definition \eqref{eqn:HamiltonAngle} of $G_t$, we get 
\begin{equation*}
\partial_t \frac{1}{\beta N} \log Z_t  
+ N \int f \, \rd \mu_V
=  N \int f(\rd \mu_V-\rd \mu_{V_t})  + 
\re \E_{V_t}^{G_t}\Big (
- X^f_{V_t}
+ \hat  A^{f}_V       \Big )
.
\end{equation*}
As $t\ll 1$ and $\Delta f$ is bounded and supported in $S_V$,
the supports $S_{V}$ and $S_{V_t}$ coincide.
Together with the explicit formula for the equilibrium measure \eqref{eqn:rhoV}
and with \eqref{Fb}, we have 
\begin{align*}
N \int f(\rd \mu_V-\rd \mu_{V_t})
&= \frac{Nt}{4\pi}\int |\nabla f|^2 \, \rd m  +  \frac{N}{4\pi}\int  |f t \Delta F | \, \rd m
\\
&=   \frac{Nt}{4\pi}\int |\nabla f|^2 \, \rd m 
+ \OO(   t b^{-2}   N^{ 2 \sigma})   \| f \|_{\infty,2, b}^2 , 
\end{align*}  
where we have integrated  by parts twice in getting the last inequality
and also used that the support of the integrand has area $\OO(b^2)$.
Using \eqref{Wdef} with the choice $V_t$ for the external potential
(and the unperturbed Coulomb pair interaction), we have  
\begin{align}
 \E_{V_t}^{G_t}\Big (
- X^f_{V_t}
+   A^{h, -}_{V}  \Big ) = 
 \E_{V_t}^  {G_t}\Big ( \frac{1}{N}W_{V_t}^{h_t}
- \frac{1}{N\beta }  \sum_k \partial h_t(z_k)
- A^{h_t, +}_{V_t}-A^{h_t, -}_{V_t} + A^{h,-}_{V} \Big )  
\label{eqn:1}.
\end{align}
The perturbed interaction $G_t$ satisfies $G_t(z_j,z_k)=G_t(z_k,z_j)$
and the growth assumption \eqref{eqn:Ggrowth}, so
Lemma~\ref{lem:expectation0} applies.
Together with  the definition of $G_t$ and recalling $W_V^h = W_V^{\cal C, h}$, we have 
\begin{multline}
\E_{V_t}^  {G_t} \Big(W_{V_t}^{h_t}
+t \sum_{j\neq k}(h_t(z_j)-h_t(z_k))\partial_{z_j}\re \Psi^-_{h_t}(z_j,z_k)\Big )
= \E_{V_t}^{G_t}\Big (W^{ G_t, h_t}_{V_t}\Big )\\
= \frac{1}{2}\E_{V_t}^  {G_t}\Big (\sum_{j\neq k}(h_t(z_j)+h_t(z_k))(\partial_{z_k}+\partial_{z_j})G_t(z_j,z_k)\Big ).\label{eqn:2}
\end{multline}
In summary, equations (\ref{eqn:1}) and (\ref{eqn:2}) give
\begin{multline}
\partial_t \frac{1}{\beta N} \log Z_{t}
+ N \int f\, \rd \mu_V
= \frac{Nt}{4\pi}\int |\nabla f|^2 \, \rd m
+ \re \E_{V_t}^{G_t}\Big (
- \frac{1}{N\beta }  \sum_k \partial h_t(z_k)
- A^{h_t, +}_{V_t}-A^{h_t, -}_{V_t} + A^{h,-}_{V}
\Big .\\
 -\frac{t}{N} \sum_{j\neq k}(h_t(z_j)-h_t(z_k))\partial_{z_j}\re \Psi^-_{h_t}(z_j,z_k)
\Big .
+ \frac{1}{2N}\sum_{j\neq k}(h_t(z_j)+h_t(z_k))(\partial_{z_k}+\partial_{z_j})G_t(z_j,z_k)
\Big)\\
+ \OO(t b^{-2} N^{ 2 \sigma} )   \| f \|_{\infty,2, b}^2 .\label{eqn:3}
\end{multline}
We now evaluate all terms in the above expectation.
The difference $A^{h_t,-}_{V_t}-A^{h,-}_{V}$ is bounded in Lemma~\ref{lem:Amdiff} below. 
For the general cases with $u \not = 0$, $A^{h,-}_{V}$ should be replaced by $A^{h,-}_{V+ uf}$. 
Notice that Lemma~\ref{lem:Amdiff} is valid for  all  $0 \le |u| \le  \OO(t)$. 

The other terms are bounded as follows. By \eqref{Sec2rig},
\begin{equation} \label{eqn:est1}
\re \E_{V_t}^{G_t}\Big (
- \frac{1}{N\beta }  \sum_k \partial h_t(z_k)
\Big )
=
-\frac{1}{\beta}\re\int\partial h_t \, \rd\mu_{V_t} +  \OO(N^{-1/2+\epsilon}b)\|\nabla h_t\|_{\infty,2,b}.
\end{equation}
To compute the main term on the right-hand side, recall that $V_t = (V + t f ) + tF$. 
By integration by parts and the explicit formula for the equilibrium density,
\begin{align}
-\frac{1}{\beta}\re\int\partial h_t \, \rd\mu_{V_t}
&=-\frac{1}{4\pi\beta}\re \int \partial \Big (\frac{\bar \partial f}{\partial \bar \partial (V + t f )}\Big )  \Delta (V + tf ) \, \rd m
+ \OO \Big ( t \int\partial h_t \Delta  F \, \rd m \Big ) \notag
\\
&= -\frac{1}{4 \pi\beta} \int \Delta  f \log \Delta (V + tf) \, \rd m
+ \OO \Big ( t \int\partial h_t \Delta  F \, \rd m \Big )  \notag \\
& =-\frac{1}{\beta}Y_V^f+\OO\Big (t \int |\Delta f|^2 \, \rd m\Big )  +  
\OO( t b^{-2}   N^{ 2 \sigma})      \| f \|_{\infty,3, b}^2  \label{Yp}
\\
& = -\frac{1}{\beta}Y_V^f  + \OO(  t b^{-2}   N^{ 2 \sigma} )      \| f \|_{\infty,3, b}^2 \notag
.
\end{align}
Finally,  differentiating $\Psi$ and  using (\ref{eqn:explicit})  give
$$
\Big |\frac{t}{N} \sum_{j\neq k}(h_t(z_j)-h_t(z_k)) \partial_{z_j}\re \Psi^-_{h_t}(z_j,z_k)\Big |
\leq C
\frac{t}{N}\|\nabla h_t\|_\infty^2\sum_{j\neq k:z_j\in \Omega}\ee^{-\frac{|z_j-z_k|^2}{2\theta^2}}\Big (1+\frac{|z_j-z_k|^2}{\theta^2}\Big )+\ee^{-N^{\varepsilon}}, 
$$
where $\Omega$ is the $N^\varepsilon\theta$-neighborhood of the support of $h$.
Using the boundedness of the  local density, implied by \eqref{Sec2rig},  
we  have,  under the assumption \eqref{Af}, that 
\begin{align}\label{eqn:est3}
  \re \E_{V_t}^{G_t}\Big (\frac{t}{N} \sum_{j\neq k}(h_t(z_j)-h_t(z_k))\partial_{z_j}\re \Psi^-_{h_t}(z_j,z_k)\Big )
&=
  \OO( tN^{2\sigma+\epsilon} b^2) \|\nabla h_t\|^2_{\infty}
  \nonumber\\
&=
\OO( tb^{-2}N^{2\sigma+\epsilon}) \|f\|_{\infty,2 ,b}^2 
.
\end{align}
Similarly, (\ref{eqn:explicit}) yields
\begin{align*}
&\frac{1}{N} \sum_{j\neq k}(h_t(z_j)+h_t(z_k))(\partial_{z_j}+\partial_{z_k})G_t(z_j,z_k)\\
&=\frac{t}{N} \sum_{j\neq k}(h_{ t}(z_j)+h_{ t}(z_k)) \frac{\partial { h}(z_j)-\partial { h}(z_k)}{z_j-z_j}\ee^{-\frac{|z_k-z_j|^2}{2\theta^2}}\\
&=\OO\Big (
\frac{t}{N} \|h_t\|_\infty\|\nabla^2  { h}\|_\infty\sum_{j\neq k:z_j\in \Omega}\ee^{-\frac{|z_j-z_k|^2}{2\theta^2}}\Big )+\OO(\ee^{-N^{\varepsilon}})
\\
&=\OO\Big (
\frac{t}{N} b^{-4} \|f\|_{\infty,3,b}^2 \sum_{j\neq k:z_j\in \Omega}\ee^{-\frac{|z_j-z_k|^2}{2\theta^2}}\Big )+\OO(\ee^{-N^{\varepsilon}}).
\end{align*}
The local density estimate \eqref{Sec2rig} then again gives
\begin{align}\label{eqn:est4}
 \re \E_{V_t}^{G_t}\Big (\frac{1}{N} \sum_{j\neq k}(h_t(z_j)+h_t(z_k))(\partial_{z_j}+\partial_{z_k})G_t(z_j,z_k)\Big )
  &= \OO ( \frac{t}{N} b^{-4}  N^{2\sigma+\epsilon})  \|f\|_{\infty,3,b}^2
    \nonumber\\
&=  \OO ( t b^{-2} N^{2\sigma+\epsilon})  \|f\|_{\infty,3,b}^2
.
\end{align}
Collecting the error terms and using \eqref{e:Amdiff} and $b\ge \theta$,
we get the error terms 
\begin{equation} 
 N^{-1/2+3\sigma+\epsilon}b^{-1}\|f\|_{3,b}
  + t N^{\epsilon}     b^{-2} N^{2\sigma}  \norm {f}_{3, b}^2
  + N^{-1/2+\epsilon}b\norm {f}_{4, b}.
\end{equation}
This concludes the proof.
\end{proof}

\begin{lemma} \label{lem:Amdiff}
  Recall    assumption \eqref{Af} and that  $h_t$ is defined in \eqref{ht}.
  For any $0 \le |u| \le  \OO(t)$ we have the estimate
  \begin{equation} \label{e:Amdiff}
  \E^{G_t}_{V_t}\pa{A^{h_t, -}_{V_t} - A^{h, -}_{V+uf} }
  = \OO(N^{-1/2+3\sigma+\epsilon}b^{-1})\|f\|_{\infty,3,b}
  + \OO(t N^{2\sigma+\epsilon} b^{-2}) \|f\|_{\infty,3,b}^2.
\end{equation}
An analogous estimate holds with $\E_{V_t}^{G_t}$ replaced by $\E_{V}^{\cal C}$. 
\end{lemma}

\begin{proof}
To simplify notation, we set $u=0$ in the following proof as the general case is proved in the same way.
By definition,
\begin{align}
 A^{h_t, -}_{V_t} - A^{h,-}_{V} 
 =
 \frac{N}{2} \iint_{z\neq w}  \Big[ \Psi^-_{h_t} (z, w) \,  \tilde \mu_{V_t} (\rd z) \, \tilde \mu_{V_t} (\rd w)
- \Psi^-_{h} (z, w) \,  \tilde \mu_{V} (\rd z) \, \tilde  \mu_{V} (\rd w) \Big ]
.
\end{align}
Decompose the integrand into 
\begin{equation}\label{80}
 [\Psi^-_{h_t} - \Psi^-_{h} ]  (z, w) \,  \tilde \mu_{V_t} (\rd z)   \, \tilde \mu_{V_t} (\rd w)
+  \Psi^-_{h} (z, w) \, [ \tilde \mu_{V_t} (\rd z)   \, \tilde \mu_{V_t} (\rd w)-  \tilde \mu_{V} (\rd z) \, \tilde  \mu_{V} (\rd w) ].
\end{equation}
To estimate the first term,
using that
\begin{equation}
\partial_s \partial h_s(z)=\OO\left(\|\nabla f\|_\infty \|\nabla^3 f\|_\infty + \|\nabla^2 f\|_\infty^2\right)
= \OO(b^{-4}) \|f\|_{\infty,3,b}^2,
\end{equation}
with high probability with respect to the measure  $P_{V_t}^{G_t}$  we have
\begin{align}
 & N \iint_{z\neq w}  [\Psi^-_{h_t} - \Psi^-_{h} ]  (z, w) \,  \tilde \mu_{V_t} (\rd z)   \, \tilde \mu_{V_t} (\rd w)
\\ &  \le N \int_0^t \rd s \iint_{z\neq w}  \big |\partial_s \partial h_s(z)  \big |  \  \mathds{1} (  |z- w| \le \theta ) \,  \tilde \mu_{V_t} (\rd z)   \, \tilde \mu_{V_t} (\rd w)\\
&\leq \OO(t N^{2\sigma+\epsilon}b^{-2}) \|f\|_{\infty,3,b}^2
\end{align}
where we used the local density estimate Proposition~\ref{prop:CAdensity2},
and the factor $\theta^2 b^2$ comes from the integration restriction that $z$
is in the support of   $f$  and $|w-z| \lesssim  \theta$.

Similarly we can  estimate the second term in \eqref{80}. We start with the bound that, with high probability, 
\begin{align*}
  &N \iint \Psi^-_{h_t} (z, w) \, \big [   \tilde \mu_{V_t} (\rd z) 
  -  \tilde \mu_{V+tf} (\rd z)  \big ] \, \tilde \mu_{V_t} (\rd w)
 \\
 & = \OO(N) \iint  \Psi^-_{h_t} (z, w) \,   t\Delta F (z)  \, m ( \rd z)  \, \tilde \mu_{V_t} (\rd w)
 =  \OO(N^{1+\epsilon} \theta^2 b^2) \norm{\nabla h_t}_\infty  \norm{ t \Delta F}_\infty
 \\
 & =\OO(  N^{-1/2 +  3 \sigma + \epsilon})      \norm {f}_{\infty,2, b}  t b^{-3}  \norm {f}_{\infty,4, b}
 =\OO(  N^{-1/2 +  3 \sigma + \epsilon})    b^{-1}  \norm {f}_{\infty,2, b} 
\end{align*}
where we have used      Lemma \ref{lem:86}  to bound $\norm{\nabla h_t}_\infty  \norm{ t \Delta F}_\infty$ and assumption \eqref{Af} in the last step.  
Similar argument also leads to 
\begin{multline*}
 N \iint \Psi^-_{h_t} (z, w) \, \big [   \tilde \mu_{V} (\rd z) 
 -  \tilde \mu_{V+tf} (\rd z)  \big ] \, \tilde \mu_{V_t} (\rd w)
 = \OO(N t) \iint  \Psi^-_{h_t} (z, w) \,  \Delta f  (z) \, m ( \rd z)  \, \tilde \mu_{V_t} (\rd w)
 \\
 =  \OO(N^{1+\epsilon} t \theta^2 b^2) \norm{\nabla h_t}_\infty  \norm{  \Delta f}_\infty
 =   \OO(N^{2 \sigma +\epsilon} t     b^{-2}) \norm {f}_{\infty,2, b}^2.
\end{multline*}
Collecting all these bounds and using $\norm {f}_{\infty,2, b}^2 \le \norm {f}_{\infty,3, b}^2$, we have proved Lemma \ref{lem:Amdiff}. Notice that we have used assumption \eqref{Af} in the proof so that 
the right side of \eqref{e:Amdiff}  does not involve   $ \norm {f}_{4, b}$. This completes the proof.
\end{proof}

\subsection{Proof of Proposition~\ref{prop:clt-angle1}}

The proof of Proposition~\ref{prop:clt-angle1} follows the strategy in \cite{MR3694026} by first  estimating 
the  sum of the  long and short range angle terms with  the local law  Proposition~\ref{prop:CAdensity2}.

\begin{lemma}
\label{lem:Apbd-weak}
For any  $\varepsilon>0$, uniformly in $0 \leq t \ll 1$ with $t\|\Delta f\|_{\infty} \ll 1$,  we have
\begin{equation} \label{e:Apbd-weak}
 \E_{V_t}^{G_t}
 \left(A^{ g,+}_{V_{t}}+A^{ g,-}_{V_{t}}\right)
 =
 \OO(N^{\epsilon})
 b\|g\|_{\infty,2,b}.
\end{equation}
\end{lemma}

\begin{proof}
The proof is exactly the same as that of \cite[Lemma~7.5]{MR3694026},
using the local density estimate Proposition~\ref{prop:CAdensity2}.
Here $A^-$ corresponds to $t\leq N^{-1/2+\delta}$ in that proof and
$A^+$ to $t\geq N^{-1/2+\delta}$.
\end{proof}

Inserting  these bounds into  \eqref{e:cltangle-bis-noangle}, 
we obtain the following rigidity estimate. 
This estimate is essentially the same  as the rigidity estimate for 
the Coulomb gas, i.e.,   \cite[Theorem~1.2]{MR3694026}.
The only difference is that the estimate is  with respect to the Coulomb gas with an angle term, i.e., the  measure $\P_{V_t}^{G_t}$.

\begin{proposition} \label{prop:CArigi}
Assume the same conditions as in Proposition~\ref{prop:CAdensity2}.
For any $\epsilon>0$, $s \in (0,1/2)$,
for any $f$ supported in a ball of radius $b = N^{-s}$
contained in $S_V$ with distance of order $1$ to $\partial S_V$,
\begin{equation}\label{r6a} 
X_f \prec \normt{f}_{\infty,4,b}
\end{equation}
with respect to the measure $\P_{V_t}^{G_t}$.
\end{proposition}

\begin{proof}
The proof is exactly the same as the proof of \cite[Theorem~1.2]{MR3694026}.
\end{proof}

Finally, using this  rigidity estimate instead
of the local law of Proposition~\ref{prop:CAdensity2},
we obtain the following improved bound on  $A^+$, which consists of correlations at range longer than $N^{-1/2 + \sigma}$.
The proof of  Lemma \ref{lem:Apbd} uses a loop equation  and   will be given in Section~\ref{sec:Apbd-pf}, 
where a systematical treatment of  loop equation will be presented.
We remark that a similar estimate for Coulomb gas was already proved in   \cite{MR3694026}.

\begin{lemma}[Refined estimate on the long range angle term]  \label{lem:Apbd} 
  For any $\varepsilon>0$, uniformly in $0 \leq t \ll 1$ with $t\|\Delta f\|_{\infty} \ll 1$
  and for any function $g$, we have 
\begin{equation}
\label{e:Apbd}
 \E_{V_t}^{G_t}
 \left(A^{g, +}_{V_t}\right)=
 \OO(N^{-\sigma + \epsilon} )
b  \|g\|_{\infty,2,b}. 
\end{equation}
In particular, when $g=h_t$, the last term is bounded by $\OO(N^{-\sigma + \epsilon} )\|f\|_{3,b}$. 
For a Coulomb gas satisfying \eqref{e:quasifree-meso-cond} a similar estimate holds, i.e., 
\begin{equation}
\label{e:Apbdloc}
   \E_{V+ t f}^{\cal C}
 \left(A^{g, +}_{V+tf}\right)=
 \OO(N^{-\sigma + \epsilon} )
b  \|g\|_{\infty,2,b}.  
\end{equation}
\end{lemma}

\begin{proof}[Proof of Proposition~\ref{prop:clt-angle1}]
Proposition~\ref{prop:clt-angle1} follows immediately from \eqref{e:cltangle-bis} and Lemma~\ref{lem:Apbd}.
\end{proof}

\subsection{Concentration of the angle term (macroscopic case):
proof of Proposition~\ref{prop:anglevanish}}
\label{sec:anglevanish}
In this subsection, we assume $b$ is of order $1$.
The main input of the proof of Proposition~\ref{prop:anglevanish} is the following estimate of large deviations type,
which is a direct consequence of Theorem~\ref{freeasy}.

\begin{corollary} \label{LarDev}
Assume that $V$ satisfies the conditions of Remark~\ref{rk:quasifree-meso}.
Let $0< t\ll 1$ and $\kappa < 1/24$.
Then for any $f \in\mathscr{C}^{5}$ whose support is contained in  $S_V$ and
has distance of order $1$ to the boundary of $S_V$,
assuming that $t \|\Delta f\|_\infty \ll 1$, we have
$$
\frac{1}{t\beta N}  \log  \E_{V} \ee^{-\beta tN  X_V^f}
=\frac{t N}{ 8\pi} \int   |\nabla  f|^2 \, \rd m
 +\left(\frac 1 2 -  \frac 1 \beta \right)  Y_V^f
 + \OO(N^{-\kappa}/t) (1+\|\Delta f\|_{\infty,3})^2 + \OO(t)\|\Delta f\|_\infty^2.
$$
\end{corollary} 

\begin{proof}
By Theorem~\ref{freeasy}, we have
\begin{multline}\label{T3}
\frac{1}{t \beta N}  \log  \E_{V}\ee^{-\beta tN  X_V^f}
=N\int f \, \rd\mu_V
- \frac N t\left(I_{V+ tf}  -I_{V}\right) \\
+ \left(\frac 1 2 -  \frac 1 \beta \right)  \frac{1}{t}  \left(\int \rho_{V+tf} \log \rho_{V+tf}  - \int \rho_{V}  \log \rho_{V}\right)
+ \OO(t^{-1} N^{-\kappa}),
\end{multline}
with an $f$-dependent error term. More precisely, by Remark~\ref{rk:quasifree-meso},
with $V$ fixed, the $f$-dependence of the error term can be taken to be
$
\OO(t^{-1} N^{-\kappa})(1+\|\Delta f\|_{\infty,3})^2$.

Using that $\rho_V= \frac{1}{4\pi} \Delta V \mathds{1}_{S_V}$ and $\rho_{V+tf} = \frac{1}{4\pi} (\Delta V+\Delta f) \mathds{1}_{S_V}$ for $f$ with compact support contained in $S_V$
such that $t\Delta f < \Delta V$ in its support,
an explicit calculation (see, e.g., \cite[Proposition~3.1]{MR3694026}) shows that
\begin{equation}\label{eqn:el2}
I_{V+ tf}-I_{V}  =  t \int  f \, \rd\mu_{V}- \frac{t^2 }{ 8\pi} \int   |\nabla  f|^2 \, \rd m,
\end{equation}
and that
\begin{align}\label{eqn:el3}
\frac 1 t  \left( \int \rho_{V+tf}  \log \rho_{V+tf} - \int \rho_{V} \log \rho_{V} \right)
&=  
\frac 1{4\pi} \int \Delta f  \log \rho_{V}+\frac{1}{t}\int \rho_{V+tf}\log\left(\frac{\rho_{V+tf}}{\rho_V}\right)
\nonumber\\
&=\frac 1{4\pi} \int \Delta f  \log \rho_{V}+\OO\left(t\int(\Delta f)^2\right),
\end{align}
where for the last equality we expanded $\log(1+t\Delta f/\Delta V)$
to first order and used $\int \Delta f=0$.
Equations (\ref{eqn:el2}) and (\ref{eqn:el3}) in (\ref{T3}) conclude the proof.
\end{proof}

\begin{proof}[Proof of Proposition~\ref{prop:anglevanish}]
Let $\kappa$ be as in Corollary~\ref{LarDev}
and write $W= V-tf$.  
By an elementary identity as in \eqref{e:logpm},  we have
\begin{align*}
\frac{1}{t\beta N}\log \E_{V} \ee^{t\beta N \hat A^f_{V + u f} }
=\frac{1}{t\beta N}\left(
\log \E_{W}\ee^{-\beta N t (X^f_{V}-\hat A^f_{V + u f})}
-
\log \E_{W}\ee^{-\beta N t X^f_{V}}
\right).
\end{align*}
We can replace $X^f_{V}$ by $X^f_{W}$ in the two exponents in the above equation
since $X^f_{V}-X^f_{W}$ is a constant which cancels in the above expression. 
Also, $\hat A^f_{V + u f} = \hat A^f_{W + (t+u) f}$.
By Proposition~\ref{prop:clt-angle1},  
\begin{align*}
  \frac{1}{t\beta N}
\log \E_{W}\ee^{-\beta N t (X^f_{W}-\hat A^f_{W + (t+u) f})} =-\frac 1 2  Y^f_{W}
+ N^\epsilon\OO(t N^{2\sigma}+N^{ - \sigma} + N^{-1/2+3\sigma} ) (1+\|f\|_{\infty,4})^2.
\end{align*}
By Corollary~\ref{LarDev} with $V$ replaced by $W$,
we can estimate the last term $\log \E_{W}\ee^{-\beta N t X^f_{W}}$. 
Recall from \eqref{Yp} that $Y^f_{W}=Y^f_{V}+\OO\left(t  \int |\Delta f|^2 \, \rd m\right)$. 
Putting all these bounds together, we have arrived at 
\begin{align*}
\frac{1}{t\beta N}\log \E_{V} \ee^{t\beta N \hat A^f_{V + u f}}
=-\frac 1 2  Y^f_{V}
+ N^\epsilon \OO(t N^{2\sigma}+N^{-\sigma}+t^{-1} N^{-\kappa} + N^{-1/2+3\sigma})(1+\|f\|_{\infty,5})^2.
\end{align*}
This proves \eqref{35} in the specific case $t=N^{-4\sigma}=N^{-2/3\kappa}$.
Moreover,  the bound also holds as claimed for smaller $t$
by the monotonicity of $t\mapsto t^{-1}\log\E(\ee^{t X})$ applied
with the choice $X=\beta N(\hat A^f_{V + u f} + \frac12 Y^f_V)$.
\end{proof}

\subsection{CLT for mesoscopic test functions}
\label{sec:meso}

To extend the proof of the central limit theorem to test functions on mesoscopic scales,
it suffices to prove the following  estimate for the local angle term.
Recall that  $ \hat A^f_{V}$ was defined in  \eqref{e:Adef} and satisfies \eqref{e:Adef1}.

\begin{proposition} \label{prop:anglevanish-loc-bis}
Suppose that $V$ satisfies the conditions \eqref{e:Vgrowth} and \eqref{Vcondition1}.
Let $s \in (0,\frac12)$ and assume that $f$ is supported in a ball of radius $b = N^{-s}$
contained in $S_V$ with distance of order $1$  to the boundary $\partial S_V$.
Then there exists $\tau=\tau(s)>0$ such that with high probability under the measure $P_{V}^{\cal C}$,
\begin{equation}\label{eqn:concentration-loc-bis}
  \left| \hat A^{f}_V +  \frac {1}{ 2}Y_{V}^f\right|
  \prec (Nb^2)^{-\tau/3} \|f\|_{\infty,5,b}.
\end{equation}
\end{proposition}

This proposition can be proved by following the strategy used
in the proof of Proposition~\ref{prop:anglevanish}, after conditioning
on the particles outside a mesoscopic ball with radius of order $b$ containing the support of $f$.
Before implementing this, we complete the proof of Theorem~\ref{clt} using \eqref{eqn:concentration-loc-bis}.

\begin{proof}[Proof of Theorem~\ref{clt} for mesoscopic test functions]  
We apply \eqref{e:cltangle-bis-noangle} and we need to  estimate the term 
$\frac{1}{t} \re  \int_0^t \E_{V + sf}^{\cal C} \left(A^{h_s, -}_{V+sf} + A^{h_s, +}_{V+sf}\right) \, \rd s$ 
on  right hand side of \eqref{e:cltangle-bis-noangle}. 
The term $A^{+}$ is again bounded by Lemma~\ref{lem:Apbd}.  
To estimate the expectation of $A^{-}$, we now use \eqref{eqn:concentration-loc-bis}
which implies that with high probability 
\begin{align*}
  A^{h_s,-}_{V+sf}
  = -\frac12 Y_{V+sf}^f +\OO(M^{-\tau/3+\varepsilon}) \|f\|_{\infty,5,b}
  = -\frac12 Y_{V}^f +\OO(M^{-\tau/3+\varepsilon} + s b^{-2})  \|f\|_{\infty,5,b}
\end{align*}
where we have used $Y^f_{V+ sf}=Y^f_{V}+\OO\left( s \int |\Delta f|^2 \, \rd m\right)$ as in \eqref{eqn:el3}.
Clearly, the high probability estimate immediately implies the same estimate under expectation.  
Integrating $s$ from $0$ to $t$, this implies an estimate on the term 
$\frac{1}{t} \re  \int_0^t \E_{V + sf}^{\cal C} \left(A^{h_s, -}_{V+sf} + A^{h_s, +}_{V+sf}\right) \, \rd s$.
Inserting this estimate  into  \eqref{e:cltangle-bis-noangle}, 
we have  completed the proof of Theorem~\ref{clt}. 
\end{proof}

In the remainder of this section, we prove  Proposition~\ref{prop:anglevanish-loc-bis}.
For this, we use the approach of local conditioning of \cite{MR3694026}
and then proceed as in the proof of Proposition~\ref{prop:anglevanish}.
The local conditioning and its properties are given in Section~\ref{sec:prelim-cond}.
Relative to the conditioned measure, for $f$ compactly supported in $S_W \subset S_V$,
the definitions (\ref{eqn:Xf}), (\ref{eqn:Yn}) translate to
\begin{align*}
X^f_{N,V} =
X^f_{M,W} = \sum_j f(\tilde z_j) - M  \int f \,\rd \mu_W,
\quad 
Y^f_{V} =
Y^f_{W} = \frac { 1} { 4 \pi}   \int \Delta  f \log \rho_W \,\rd m,
\end{align*}
where $\rho_W$ is the density of the absolutely continuous part of $\mu_W$;
inside the support of $f$, this density equals that of $\mu_V$ up to rescaling.
The angle term relative to the conditioned measure is 
\begin{equation}\label{eqn:Ahat-local}
\hat A^f_V =
\hat  A^{f}_W = \frac{M}{2} \re \iint_{z\neq w}  \Psi^-_{h_W} (z, w)  \, \tilde \mu_W (\rd z) \, \tilde \mu_W (\rd w),
\quad h_W =\frac{\bar \partial f(z)}{\partial \bar \partial W(z)}.
\end{equation}

The following proposition is a conditioned version of Proposition~\ref{prop:anglevanish}.
Note that Lemma~\ref{lem:Wproperties} implies that the assumptions of this proposition holds with high probability. Thus 
by the Markov inequality, the following  Proposition~\ref{prop:anglevanish-loc}  immediately
implies Proposition~\ref{prop:anglevanish-loc-bis}.

\begin{proposition}\label{prop:anglevanish-loc}
Let $W$ be the conditional potential defined above
and assume that it satisfies the conclusions of Lemma~\ref{lem:Wproperties}.
Choosing the local angles cutoff $\theta=M^{-1/2+\sigma}$
with $\sigma=\tau/6$, for any $0\le   t \le  M^{- 2\tau/3}$ we have
\begin{equation}\label{35-loc} 
  \frac{1}{t\beta M} \log  \E_{W}  \ee^{t \beta M  (\hat A^{f}_{W}+\tfrac { 1}    {  2 }  Y_{W}^f ) } 
  = \OO( M^{-\tau/3} )  (1+\|f\|_{\infty,5,b})^2
  .
\end{equation}
\end{proposition}

To prove Proposition~\ref{prop:anglevanish-loc},
we need a version of Theorem~\ref{freeasy} for the conditioned measure.
Recall that $\mu_W$ denotes the unique minimizer of the energy functional
\begin{equation}\label{eqn:minimize}
\cal I_W(\mu) = \iint \log \frac{1}{|z-w|} \, \mu(\rd z) \, \mu(\rd w) + \int W(z) \, \mu(\rd z),
\end{equation}
defined for probability measures supported in $B$,
and that its minimum value is $I_W = \cal I_W(\mu_W)$.

\begin{theorem} \label{freeasy-local} 
Let $W$ be the conditional potential defined above
and assume that it satisfies the conclusions of Lemma~\ref{lem:Wproperties}.
Then there exists $\tau>0$
 (depending on the constant $\tau$ in Lemma~\ref{lem:Wproperties} but possibly smaller;  here we have abused 
 the notation and use the same symbol $\tau$)
such that
with $\zeta^{\cal C, \beta}  \in\mathbb{R}$ defined in Theorem \ref{freeasy}, 
\begin{multline*} 
  \frac{1}{\beta M} \log  \int_{B^M} \ee^{-\beta H_{W}^\cal C(\b z)} \, m^{\otimes M}(\rd \b z)
  \\
  =- M  I_W   +  \zeta^{\cal C, \beta}    + \frac 1 2 \log M
  +  \Big (\frac 1 2 -  \frac 1 \beta \Big ) \int_{B}  \rho_W(z)  \log \rho_W(z) \, m(\rd z)
  + \OO(M^{-\tau}),
\end{multline*}
where $\rho_W$ is the density of the absolutely continuous part of $\mu_W$.
\end{theorem} 

\begin{proof}
We apply  the local version of Theorem~\ref{freeasy}, i.e., Theorem~\ref{rk:quasifree-meso}, 
to the conditional Coulomb gas satisfying the properties stated in Lemma~\ref{lem:Wproperties}.
To  apply  Theorem~\ref{rk:quasifree-meso}, 
we first rescale and center the domain $B$,
which is a disk of radius $b$, to the unit disk $\D$ with center at $0$.
Since the translation is trivial, we will assume that the center of $B$ is already at the origin. 
Denote the  rescaling  by $\b z = b \b u$ and define the new Hamiltonian 
$  \hat H_{W}^\cal C(\b u)$ through the identity 
\be
  \int_{B^M} \ee^{-\beta H_{W}^\cal C(\b z)} \, m^{\otimes M}(\rd \b z)
=  \int_{\D^M} \ee^{-\beta \hat H_{W}^\cal C(\b u)}    \, m^{\otimes M}(\rd \b u).
\eeq
Hence $ \hat H_{W}^\cal C(\b u)$ is a Coulomb gas with external potential  $
  \tilde W(u) = W(b u )$ up to a constant.  More precisely, 
\[
\hat H_{W}^\cal C(\b u) = H_{W}^\cal C(\b u/b) - 2M \beta^{-1}   \log b =  H_{\tilde W}^{\cal C}(\b u) - M(M-1)\log b - \frac{1}{\beta} M \log b^2.
\]
By Theorem~\ref{rk:quasifree-meso},  there exists $\tau >0$ such that
\begin{multline*}
\frac{1}{\beta M} \log  \int_{B^M} \ee^{-\beta H_{W}^\cal C(\b z)} \, m^{\otimes M}(\rd \b z) = 
  \frac{1}{\beta M} \log \int_{\D^M} \ee^{-\beta H_{M,\tilde W}(\b u)} \, m^{\otimes M}(\rd \b u)
  + M\log b -  \Big (\frac 1 2 -  \frac 1 \beta \Big ) \log b^2
  \\
  = -M (I_{\tilde W}- \log b )  +\frac12 \log M
  +  \Big (\frac 1 2 -  \frac 1 \beta \Big ) \Big [  \int_{ \D}  \rho_{\tilde W}(u)  \log \rho_{\tilde W}(u) \, m(\rd u) - \log b^2 \Big ] 
  +  \OO(M^{-\tau}).
\end{multline*}
Recall the normalization conditions for the densities  $\int \rho_{\tilde W}(u)  m( \rd u) = 1= \int \rho_{W}(z)  m( \rd z)$. 
Hence $\rho_{\tilde W}(u)   = \rho_{W}(b u) b^2$ and we have 
\[
 \int_{ \D}  \rho_{\tilde W}(u)  \log \rho_{\tilde W}(u) \, m(\rd u) - \log b^2
 =  \int_{B}  \rho_{W}(z)  \log \rho_{\tilde W}(z) \, m(\rd z) .
\] 
A similar argument shows that $(I_{\tilde W}- \log b ) = I_{ W}$. We have thus proved Theorem~\ref{freeasy-local} . 
\end{proof}

\begin{proof}[Proof of Proposition~\ref{prop:anglevanish-loc}]
By assumption,
the potential $W$ satisfies the conditions of Theorem~\ref{rk:quasifree-meso}, and therefore
the assumptions of Proposition~\ref{prop:clt-angle1}.
Together with using Proposition~\ref{freeasy-local} to replace Theorem~\ref{freeasy},
the proof of Proposition~\ref{prop:anglevanish-loc}  follows in exactly the same way as that of Proposition~\ref{prop:anglevanish}.
\end{proof}

As in the proof of Theorem~\ref{freeasy-local}, one can also derive a
conditioned version of the CLT, stated below; we omit the details of the proof.

\begin{theorem} \label{cltloc1}
Suppose $W$ is the conditional potential defined above
and assume that it satisfies the conclusions of Lemma~\ref{lem:Wproperties}.
Then for any $\beta> 0$, $c\in(0,1)$ and large $C>0$,  there a positive constant $\tau> 0$ 
such that the following holds.
For any $f: \C \to \R$ supported   in the ball with same center as $B$ and radius $b (1-c)$  and 
satisfying  $\normt{f}_{4, b} < C$, and for any $0\leq \lambda \le M^{1-2 \tau}$, we have
\begin{equation} \label{e:clt-loc}
  \frac{1}{\beta \lambda}   \log   \left(\E_{M, W, \beta}^{\cal C} \ee^{ - \lambda  \beta \left(X^f_W -  (\tfrac 1 \beta -\tfrac 1 2 )  Y^f_W \right) } \right)
  =\frac { \lambda}  {8 \pi}     \int| \nabla  f (z) |^2  \, m(\rd z) + 
 \OO(  M^{-\tau}  ).
  \end{equation}
\end{theorem}

Note that the measure associated to the external potential $W+\frac{\lambda}{M} f$ is a perturbation of the original measure provided that 
$$
|\lambda \Delta f|  \ll |M \Delta W|= |N \Delta V|.
$$
Our assumptions  $\normt{f}_{4, b} < C$ and $\lambda \le M^{1-2 \tau}$ guarantee this condition.
Also note that, 
in the context of the above Theorem~\ref{cltloc1}, our test function has shrinking support so that 
$$
Y_W^f = \frac { 1} { 4 \pi}   \int \Delta  f(z) \log \rho_{W}(z) \,\rd m(z)   =\frac { 1} { 4 \pi}   \int \Delta  f(z) \log \frac{\Delta V(z)}{\Delta V(z_0)}\rd m(z)= \OO(b) \normt{f}_{\infty,b, 2} \normt{V}_{S_V,\infty,3}=\oo(1), 
$$
where we used (\ref{eqn:Delta}) and denoted the center of $\mathcal{J}$ by $z_0$.
Thus Theorem~\ref{cltloc1}, with $\lambda$ of order 1,
implies there is no shift of the mean in the convergence to the Gaussian free field for mesoscopic observables:
$$
X_W^f\overset{{\rm (d)}}{\underset{N\to\infty}{\longrightarrow}}\mathscr{N}\left(0,\frac{1}{4\pi\beta}\int |\nabla f|^2\right).
$$

\begin{appendices}

\section{Rigidity estimates for Yukawa gas on torus}
\label{app:Yrigi}

In this appendix, we prove Theorem~\ref{thm:Yrigi} and Proposition~\ref{prop:Grigi}.
The proofs use the same ideas as that of \cite[Theorem~1.2]{MR3694026}.
We will also prove Lemma~\ref{lem:Apbd} by the same argument.

\subsection{Loop equation}

As the  first step, we state the loop
equation for the Yukawa gas on the torus.
Given a function $v: \T \to \R$, the function $W_V^{U^\ell, v}: \T^N \to \C$ 
was defined by \eqref{e:wvplus}. For simplicity of notation, we denote it by $W_V^{v}({\bf z})$ in this section, i.e.,
\begin{equation}\label{e:wvplus-bis}
W_V^{v}({\bf z})= -\sum_{j \neq k} (v(z_j)-v(z_k)) \partial U^\ell(z_j-z_k) + \frac{1}{\beta} \sum_j \partial_j v(z_j) - N \sum_j v(z_j) \partial V(z_j).
\end{equation}
By Lemma \ref{lem:expectation0}, 
$\E_V^{U^\ell} (W_V^v)= 0$ since  the Yukawa interaction $Y^\ell$ (and therefore $U^\ell$) are functions of $z_j-z_k$.
Given $q: \C \to \R$ supported in $S_V$, further abbreviate
\begin{equation} \label{eqn:Yhg}
h(z) = \frac{1}{\pi} \frac{\bar \partial q(z)}{\rho_V(z)},
\quad
g(z) = \frac{1}{\pi} \frac{q(z)}{\rho_V(z)},
\end{equation}
where $\rho_V$ denotes the density of $\mu_V$   (\ref{e:equilibriumdensity}) with respect to the Lebesgue measure.  
The following lemma extends Lemma~\ref{lem:Wdef} from the Coulomb gas to the Yukawa gas.

\begin{lemma}\label{lem:YWdef}
  For any $q:\mathbb{T}\to\mathbb{R}$  of class $\mathscr{C}^2$ supported on $S_V$ and  $\bold z\in\mathbb{T}^N$, recall $X_V^q$ defined in  \eqref{eqn:Xf}. Then  we have
\begin{align}\label{e:YWdef}
X_V^q
&= - \frac{1}{N} W_V^h ({\bf z})+ \frac{1}{N\beta }  \sum_k \partial h(z_k)+ N \iint_{ z \not = w } (h(z)-h(w))\partial U^\ell (z-w)  \tilde \mu_V(\rd z) \tilde \mu_V(\rd w)\nonumber\\
&\qquad +\frac{Nm^2}{2} \iint_{ z \not = w } g(w)U^\ell(z-w)  \tilde \mu_V(\rd z) \mu_V(\rd w)
  ,
\end{align}
where $m = \ell^{-1}$ an $h$ is defined in \eqref{eqn:Yhg}.
Thus for any smooth enough $f: \T \to \R$ with  
\begin{equation}
 q = f - \int_\T f \rd m -m^2\Delta^{-1} (f- \int_\T f \rd m)
\end{equation}
supported in $S_V$,
where $\Delta$ is the Laplacian on the torus,   we have 
\begin{align}\label{e:YWdef2}  
X_V^f
= - \frac{1}{N} W_V^h ({\bf z})+ \frac{1}{N\beta }  \sum_k \partial h(z_k)
+ N \iint_{ z \not = w } (h(z)-h(w))\partial U^\ell(z-w)  \tilde \mu_V(\rd z) \tilde \mu_V(\rd w).
\end{align}
\end{lemma}

\begin{proof}
As in the proof of Lemma~\ref{lem:Wdef}, we have
\begin{align}
\label{eqn:Yalg1}
& 2\int \partial U^\ell(z-w) \, \mu_V(\rd w)=\partial V(z),
\\
\label{eqn:Yalg2}
& q(z)=\frac{1}{2\pi} \int (-4\partial\bar\partial q(w) + m^2 q(w)) U^\ell(z-w) m(\rd w)
\nonumber\\
&\quad =\frac{1}{2\pi}\int (4 \bar\partial q(w) \partial U^\ell(z-w) + m^2 q(w) U^\ell(z-w)) m(\rd w),
\end{align}
where again the first equation holds for $z\in S_V$ by the Euler-Lagrange equation,
the second equation holds by the definition of the Yukawa potential as the Green's function
of $-\Delta+m^2$ and integration by parts -- the boundary term vanishes by periodicity.
We therefore have
\begin{align*}
  X_V^q &= 2 \sum_j \int h(w)\partial U^\ell(z_j-w)\mu_V(\rd w) 
       + \frac{m^2}{2} \sum_j  \int g(w) U^\ell(z_j-w) \mu_V(\rd w)
  \\
       &         -2N\iint h(w)\partial U^\ell(z-w) \mu_V(\rd w)\mu_V(\rd z)
       -\frac{Nm^2}{2} \iint g(w)U^\ell(z-w) \mu_V(\rd w)\mu_V(\rd z)\\
&=2N\iint (h(w)-h(z))\partial U^\ell(z-w) \hat\mu(\rd z)\mu_V(\rd w)+\sum_j h(z_j)\partial V(z_j)\\
  &-N\iint (h(w)-h(z)) \partial U^\ell(z-w)\mu_V(\rd w)\mu_V(\rd z) + \frac{Nm^2}{2} \iint g(w) U^\ell(z-w) \tilde\mu_V(\rd z) \mu_V(\rd w). 
\end{align*}
In the first equation we used (\ref{eqn:Yhg}) and (\ref{eqn:Yalg2}),
and in the second equation we used (\ref{eqn:Yalg1}).
Since the integrands in the double integrals are symmetric, we arrive at
\begin{multline*}
  X_V^{q}=-\frac{1}{N}\sum_{j\neq k}(h(z_j)-h(z_k))\partial U^\ell(z_j-z_k)+ N \iint_{ z \not = w } (h(z)-h(w))\partial U^\ell(z-w)  \tilde \mu_V(\rd z) \tilde \mu_V(\rd w)\\
+\sum_j h(z_j)\partial V(z_j)
  +\frac{Nm^2}{2} \iint_{ z \not = w } g(w)U^\ell(z-w)  \tilde \mu_V(\rd z) \mu_V(\rd w),
\end{multline*}
which is equivalent to \eqref{e:YWdef}.

For the consequence, note that moving the last term on the right-hand side to the left-hand side,
the left-hand side becomes $X^f$ with
\begin{equation*}
  f(z) = q(z) - \frac{m^2}{2\pi} \int q(w) U^\ell(z-w) \, m(\rd w) = ((1-K) q)(z),
\end{equation*}
where
\begin{equation*}
  1-K = 1 - (1-\ell^2 \Delta)^{-1} = \frac{\ell^{2}\Delta}{\ell^{2}\Delta-1},
  \quad
  (1-K)^{-1} = 1 - m^2 \Delta^{-1}.
\end{equation*}
Therefore, given $f$ as in the assumption, we can choose $q=f-m^{2} \Delta^{-1}f$. 

Finally, since $\int \rd \mu_V =1$, we have  $X^f_V = X^{f-c}_V$ for any constant $c$.
Hence  the assumption $\int_\T f \rd m= 0$ is trivial to remove.
\end{proof}

\subsection{Estimate on two-point correlations} 

For the analysis of the loop equation, 
we need weak decorrelation estimates for two-point observables.
The following simple estimate based on Taylor expansion and the boundedness
of the local density. 
Let $\omega_t$ be a nonnegative mollifier such that $\int \omega_t(z) \, \rd z=1$,
$\omega_t$ has support in a square of side length $t$, and
$\|\omega^{(n)}_t\|_\infty\leq C_n t^{-2-n}$ for all $n\geq 0$.
In the lemma below $g$ is an arbitrary function on $\mathbb{T}^2$, unrelated to the normalization (\ref{eqn:Yhg}).

\begin{lemma} \label{lem:Taylor}
Consider the Yukawa gas on the unit torus with range $N^{-1/2+\sigma}\leq \ell\leq 1$.
Recall the definition \eqref{eqn:nablaj} and the notations of Proposition~\ref{prop:Grigi}.
Fix a scale $t$ with $N^{-1/2+\sigma} \leq t \leq N^{-\sigma}$.
Then for any fixed $p\in\mathbb{N}$ and $\varepsilon>0$, there exist functions $|F^{(j,\b k)}(x,y)|=\OO(|\nabla^j g(x,y)|)$
such that the  following bound holds with high probability:
\begin{align}
  &\iint  g (z, w) \, \tilde\mu(\rd z) \, \tilde\mu(\rd w)
    =\sum_{j=1}^p { \sum_{\sum k_i=j}} 
\iint F^{(j,\b k)}(x,y) \, m(\rd x) \, m(\rd y)  \nonumber \\   
& \times \left(\iint \varphi_{\b k}(x,y,z,w)\, \omega_t(z-x)\omega_t(w-y)\,\tilde\mu(\rd z) \, \tilde\mu(\rd w)\right)
+\OO( t^p\|g^{(p)}_{{\rm B}_t}\|_{1}),\label{eqn:nd}
\end{align} 
where $\|\cdot\|_1$ is the $L^1$-norm over $\T\times\T$, and
$$\varphi_{\b k}(x,y,z,w)=c_{\b k}(x-z)^{k_1}(\bar x-\bar z)^{k_2}(y-w)^{k_3}(\bar y-\bar w)^{k_4}.$$   
\end{lemma}

\begin{remark} \label{rem:Taylor}
  This lemma uses only that the density is locally bounded w.r.t.\ the Yukawa gas.
  In its application, we choose $t$ such that 
  $t^p\| g^{(p)}_{t}\|_{{\rm L}^1(\mathbb{T}\times\mathbb{T})} \le N^{-\epsilon p}$.
  If $g$ is a function smooth at the scale $W$, say, then $t = W N^{-\epsilon}$ is such a choice. 
\end{remark}

\begin{proof}
By Taylor expansion, for any $(x,y)\in{\rm B}(z,t)\times {\rm B}(w,t)$ (defined in \eqref{eqn:nablaj})
we can write
\be\label{B111}
g(z,w)=\sum _{j=0}^{p-1}\sum_{\sum k_i=j,a_i\in\{x,\bar x,y,\bar y\}}\left(\prod_{i=1}^4(\partial_{a_i})^{k_i}g(x,y)\right) \varphi_{\b k}(x,y,z,w)  + R_p (z, w: x, y),
\eeq
where 
$$
R_p(z, w: x, y) = C \int_0^1 (1-s)^p  \nabla^{\bf p} g (x+ s(z-x), y+s(w-y)) \ \rd s \  \varphi_{\b p}(x,y,z,w)
$$
is the remainder term. 
Here $\nabla^{\bf p}$ is understood as a multi-indices differentiation operators 
with total degree $p$ and
the right-hand side of $R_p$ is understood as a sum over all indices with $|\b p|=p$.
We now rewrite 
\begin{align} 
  \iint  g(z, w) \, \tilde\mu(\rd z) \, \tilde\mu(\rd w)
  =  \iint g(z, w) \, \tilde\mu(\rd z) \, \tilde\mu(\rd w) \, \omega_t(z-x) \, \omega_t(w-y) \, m(\rd x) \, m( \rd y),
\end{align}
and insert the equation \eqref{B111} into this identity.
This gives the sum in \eqref{eqn:nd} with
\begin{equation}
  |F^{(j,\b k)}(x,y)|
  = \sum_{a_i\in\{x,\bar x,y,\bar y\}}
  \left(\prod_{i=1}^4(\partial_{a_i})^{k_i}g(x,y)\right)
  =\OO(|\nabla^j  g(x,y)|).
\end{equation}

To complete the proof, it remains to bound the remainder term 
 \begin{align*} 
   &\Big |  \iint  R_p(z, w: x, y)  \, \tilde\mu(\rd z) \, \tilde\mu(\rd w) \Big |
  = C\Big |   \iint  \int_0^1 (1-s)^p  \nabla^{\bf p} g (x+ s(z-x), y+s(w-y)) \ \rd s   \nonumber \\
  & \qquad \qquad  \times  \Big [  \iint  \varphi_{\b p}(x,y,z,w)
 \, \tilde\mu(\rd z) \, \tilde\mu(\rd w) \, \omega_t(z-x) \, \omega_t(w-y) \Big ]  \, m(\rd x) \, m( \rd y) \Big |    \nonumber \\
   &\le   C   \iint  \OO\left( g_{{\rm B}_t}^{(p)}(x,y)\right) \, m(\rd x) \, m( \rd y) 
   \sup_{x, y} \Big |   \iint  \varphi_{\b p}(x,y,z,w)
     \, \omega_t(z-x) \, \omega_t(w-y)  \, \tilde\mu(\rd z) \, \tilde\mu(\rd w)  \Big |
     .
\end{align*}
By  the local law, Theorem~\ref{thm:YTdensity}, 
i.e., that the empirical density is bounded with high probability,
for any function $k$ supported in a square  of size $w\gg N^{-1/2}$, we have 
\be\label{B112}
\int |k(z)|  \, \hat \mu(\rd z) \le   C w^2 \| k \|_\infty,
\eeq
and the same estimate holds with $\hat\mu$ replaced by $\tilde \mu$ since it is trivially true for $\mu_V$.
Hence 
$$
 \sup_{x, y} \Big |    \iint  \varphi_{\b p}(x,y,z,w)
\, \omega_t(z-x) \, \omega_t(w-y)  \, \tilde\mu(\rd z) \, \tilde\mu(\rd w) \Big | =
\OO ( t^{p} ).
$$ 
This proves the the bound on the error term and completes the proof of the  lemma.
\end{proof}

\subsection{Analysis of loop equation and proof of Theorem~\ref{thm:Yrigi}}

We next analyze the terms in the loop equation.

\begin{lemma} \label{lem:hYdecomp}  
For any $A> 0$, there is a constant $C$ such that 
for any  smooth $f: \T \to \R$ supported in a ball of radius $b$ with $ b \ge N^{-1/2}$, there exists $f_s$ 
support in a ball of radius $b_s:=b+  C s \log N $ for $0\le s \le \log N$ such that $ h(z)  = { \frac{1}{\pi \rho_V(z)}} \bar\partial  (1-m^{2}\Delta^{-1})f(z)$
can be written as 
\begin{align} \label{e:hYdecomp}
h (z)&  =
    { \frac{1}{\pi \rho_V (z)}} 
    \pa{ \bar\partial f(z) - m^{2} \int_0^{\log N} 
    \frac{\rd s}{s}
    \bar\partial f_s(z) } + \OO(N^{-A} \norm{ f }_{\infty,1, b}),\\
     &
    \norm{f_s}_{\infty,k, b_s   } \le  C (b\wedge s)^2 N^\epsilon \|f\|_{\infty,k,b} \label{normbound}
  .
\end{align}
It is useful to keep in mind that if $f$ is dimensionless then $f_s$ has a linear dimension $2$. 
\end{lemma}

\begin{proof}
We write 
\begin{align}
 & \bar\partial  \Delta^{-1}f(z)=  
   \int_0^\infty
   \rd t \, \ee^{t \Delta}
    \bar\partial f (z)  
     =   \int_0^M  \rd t  \,  \ee^{t \Delta}
    \bar\partial f (z) +  \int_M^\infty  \rd t  \,  \ee^{-t \Delta}
    \bar\partial f (z).
\end{align}
Since  $\Delta$ has a spectral gap of order one w.r.t mean zero function and $ \bar\partial f$ is mean zero for any $f$ with compact support, 
there is  $c > 0$ such that 
\[
  \int_M^\infty  \rd t \, \ee^{t \Delta} \, \bar\partial f (z)  
  \le \int_M^\infty  \rd t \, \ee^{- ct } \norm{   \bar\partial f (z)}_2 \le  \norm{f}_{\infty,1, b} \ee^{-c M} .
\]
We choose $M =(\log N)^2$
so that this term is an error term of the form $ N^{-A} \norm{ f }_{1, b}$ for any $A > 0$. 

The heat kernel on the unit circle is given by 
\begin{align*}
 g_t(x) = 2 \pi \sum_{k \in \Z} k_t(x+ 2 \pi k)  = \sqrt { \frac { \pi} t } \ee^{-x^2/4 t} \Big [ 1 + 2 \sum_{k \ge 1} e^{-\pi^2 k^2/ t} \cosh ( \pi k x/t) \Big ]
\end{align*}
where $ k_t(x) = (4 \pi t)^{-1/2} \ee^{-x^2/4 t}$.
The heat kernel on the two dimensional unit torus is given by $G_{ t}(z) := g_t(x) g_t(y) $. 
Now change variables $s^2 = t$. 
Clearly, the function $G_{s^2}$   decays exponentially at scale $s$. Rewrite  $G_{s^2} = G_s^1+ G_s^2$ with 
$G_s^1(z) =  G_{s^2} (z) \eta_{s C \log N}  (z) $ where $C$ is a large constant and $\eta_{a}  (z)$ is a mollifier in a ball of radius  $a$ with 
$\eta_{a}  (z)=1$ if  $|z| \le a/2$ (note that we have changed the subscript from $s^2$ to $s$ in $G^i$ 
so that the subscript indicates the scale of the support for the  functions $G^i$).
Define 
\begin{equation*}
f_s (z) =  s^2  \int   G^1_{s} (z-w) f (w) \, m(\rd w).
\end{equation*}
Clearly,  $f_s$ is supported in a ball of radius  $b_s$ and satisfies the bound   \eqref{normbound}. 
Certainly,  when  $b_s \ge 1$,  $f_s$ is supported on the entire torus. 
The error term involving $ G^2_{s}$ can be trivially bounded and the constant $A$ can be arbitrary large by choosing $C$ large
depending on $A$. 
\end{proof}

In next lemma, which is parallel to \cite[Lemma~7.5]{MR3694026}, 
we estimate the last term in \eqref{e:YWdef2}. 
The proof of this lemma uses only the local law \eqref{e:YTdensity}
(In the later application, we only need $V = f/N$.)

\begin{lemma} \label{lem:Yanglebd}
For any $f: \T \to \R$,
define $h$ as in Lemma~\ref{lem:hYdecomp}, and $G(z, w) =  (h(z)-h(w))\partial U^\ell(z-w)$.
Then, for the Yukawa gas on the unit torus, we have
\begin{equation} \label{e:Yanglebd}
  \E^{U^\ell}_V \pa{ N \iint_{ z \neq w } G(z,  w) 
 \, \tilde \mu_{V}(\rd z) \, \tilde \mu_{V}(\rd w) }
  = \OO\pa{ N^\varepsilon \pa{ 1 + \frac{b^2}{\ell^2} } }  
  \normt{f}_{\infty,3, b}
  .
\end{equation}
\end{lemma}

\begin{proof}
We first write

\be\label{Udec}
\partial  U^\ell (z) = \sum_{i \le m}   U_i (z) +  U^{(m)} (z)
\eeq
where $ U_i $ is supported in $  \ell_i/2 \le |z| \le 2 \ell_i$ with  $\ell_i= 2^{-i} \ell$
and $U^{(m)} $ is supported in $  |z| \le 2 \ell_m = N^{-1/2 + \epsilon}$.\\

\noindent{{\bf Case 1}: $U^{(m)}$.}
For any function $k$ supported in a ball of radius $b$, 
using the fact that the empirical density is locally bounded up to a factor $N^{\epsilon}$,
with high probability we have 
\begin{equation} \label{e:H0bd}
  \left|N \iint  (\bar \partial k(z)-\bar \partial k (w))   U^{(m)} (z-w) \, \tilde \mu(\rd z) \, \tilde \mu(\rd w)\right|
  \le N^{1+ \epsilon}   \ell_m^{2}   b^2   \|\nabla^2 k\|_{\infty} \le N^{1+  \epsilon}   \ell_m^{2}     \|k\|_{\infty,2,b},
\end{equation}
where$ \ell_m^{2}   b^2$ comes from the volume in the integration
and  we have used $  |(z-w)|U^{(m)} (z-w) | \le C$. 

Recall that $h$ is defined in Lemma \ref{lem:hYdecomp}. We can apply the previous inequality to $k=f$. 
To bound the other contribution due to the integral of $\ell^{-2} \bar\partial f_s$,
we use that $f_s$ is supported in a ball of size  $b_s $  and apply \eqref{e:H0bd} and \eqref{normbound} to obtain (ignoring the small error $N^{-A}$ from (\ref{e:hYdecomp}))
\begin{multline*}
\left| \frac{N}{\ell^{2}} \int_0^\infty \frac{\rd s}{s}    \iint  (\bar \partial f_s(z)-\bar \partial f_s (w))   U^{(m)} (z-w) \, \tilde \mu(\rd z) \, \tilde \mu(\rd w)\right|
 \le 
\frac{N^{1+\epsilon}\ell_m^2}{\ell^{2}}\int_0^{\log N}  \frac{\rd s}{s}       (b\wedge s)^{2} \|f\|_{\infty,2,b} \\
\le
  N^{2 \epsilon} \big [ 1 + (\sqrt N  \ell_m)^{-1} \big ]^2 b \norm{h}_{\infty,2, b} \le N^{2 \epsilon} \big [ 1 + (\sqrt N  \ell_m)^{-1} \big ]^2  \norm{f}_{\infty,3, b}.
\end{multline*}

\noindent{{\bf Case 2}: $U_{i}$ for a scale $\ell_i :=q \le \ell$ (notice that $q$ is also used to denote the function in Lemma~\ref{lem:YWdef} and \eqref{eqn:Yhg})}.
Suppose that $k$ is  a function supported in a ball of radius $r$ (note that  $r$ can be either smaller or bigger than  $\ell$).  
We will prove
\begin{equation}
 \E^{U^\ell}_V \pa{ N \iint_{ z \neq w }  (k(z)-k(w))    U_i (z-w)
 \, \tilde \mu_{V}(\rd z) \, \tilde \mu_{V}(\rd w) }
  =  \OO(N^{2 \epsilon}) A(q)^2 r \norm{k}_{\infty,2, r},  \label{123}
\end{equation}
where $ A(q)   :=   1+ q\ell^{-1} + (\sqrt N  q)^{-1}$. Summation over $i\leq m$ will give an appropriate bound. 
Let $ M_i(z,  w) = (k(z)-k(w)) U_i (z-w)$.  Our goal is to bound
$N \iint_{ z \neq w } M_i(z,  w)   \, \tilde \mu_{V}(\rd z) \, \tilde \mu_{V}(\rd w)$.
Treating $k(z)-k(w)$ as a multiplicative factor,  we can apply   Lemma \ref{lem:Taylor} to the function $U_i (z-w)$ 
with the scale $s$ in the lemma replaced by $q N^{-\epsilon}$.  By applying the decomposition (\ref{eqn:nd}) for fixed and large enough $p$, for each $1\leq j\leq p$ and $\b k$
we need to estimate 
\begin{align}\label{B123}
 & \iint F^{(j,\b k)}(x,y)  \Omega_{x, y}m(\rd x) m(\rd y) , \notag \\
& \Omega_{x, y} = \iint (k(z) - k(w)) \varphi_{\b k}(x,y,z,w)   \omega_q(z-x)  \omega_q(w-y)  \ \tilde\mu(\rd z)\tilde\mu(\rd w), 
\end{align}
where $\omega_q$ is  a smooth mollifier at scale $q$ (more precisely, $q N^{-\epsilon}$ as mentioned in 
Remark~\ref{rem:Taylor}.  The reader can follow through this minor change in the following proof).

We now prove that the  contribution from $j=0$ in the decomposition of the left hand side of  \eqref{123} is bounded by the right hand side of  \eqref{123}. 
Recall that when $j=0$,  $ \varphi_{\b k}(x,y,z,w) = 1$ in   \eqref{B123}.
Rewrite $ k(z) - k(w)=  (k(z) - k(x)) +  (k(y) - k(w)) +  (k(x) - k(y))$ and we consider the term involving 
$k(z) - k(x)$.  
The other two terms can be estimated in a similar  way.
 Let $R_x(z) =   (k(z) - k(x)) \omega_q(z-x)$. 
To prove \eqref{123},  we apply
the local law  \eqref{e:YTdensity} to have 
\begin{equation}
 \Big |  \int (k(z) - k(x))\omega_q(z-x)\ \tilde\mu(\rd z) 
 \Big | 
\le  q    N^{-\frac{1}{2}+\epsilon} \Big  [ \|  \nabla_z R_x(z)\|_{L_2(z)} +  \|  R_x(z) \|_{L_2(z)}  \ell^{-1}  +  N^{ -\frac{1}{2} } q  \|\Delta_z R_x(z) \|_\infty \Big ],
\label{BBB}  
\end{equation}
and 
\begin{multline}
\Big |  \int \omega_q(w-y) \tilde\mu(\rd w)  
 \Big |
 \le  
  q    N^{-\frac{1}{2}+\epsilon} \Big  [ \|  \nabla_w \omega_q(w-y)\|_{L_2(w)} +  \|  \omega_q(w-y) \|_{L_2(w)}   \ell^{-1}  +  N^{ -\frac{1}{2}} q  \|\Delta_w \omega_q(w-y) \|_\infty \Big ]\\
   \le  
  q^{-1}    N^{-\frac{1}{2}+\epsilon} A(q).
  \label{BBB1}
\end{multline}
We now claim that 
\begin{multline}\label{45}
\left|\iint U_i (x,y) m(\rd x) m(\rd y) \Big  [ \|  \nabla_z R_x(z)\|_{L_2(z)} +  \|  R_x(z) \|_{L_2(z)}  \ell^{-1}  +  N^{ -1/2 } q  \|\Delta_z R_x(z) \|_\infty \Big ]\right|\\
 \le  q^2 r^2 q^{-1}     \Big [   q^{-1} +  \ell^{-1} + N^{-1/2} q^{-2}  \Big ] r^{-1} \norm{k}_{\infty,2, r} = A(q) r \norm{k}_{\infty,2, r}.
\end{multline}
If this holds, then the last three inequalities imply the $j=0$ case of \eqref{123}. 

To prove (\ref{45}), we first consider the case  $r \ge q$. 
Recall $|U_i(x,y)| \le q^{-1}  \mathds{1} (|x-y| \le q)  $. 
Furthermore, 
$$
 \| \nabla_z R_x(z)\|_{L_2(z)} \le  \norm{ (k(z) - k(x)) \nabla \omega_q(z-x) }_{L_2(z)} + \norm{ \omega_q(z-x) \nabla k(z)  }_{L_2(z)}.
$$
Clearly, the contribution of this first term involving $ \| \nabla_z R_x(z)\|_{L_2(z)}$ on the left hand side of (\ref{45}) is bounded by 
\begin{align*}
&\left| \iint U_i (x,y) m(\rd x) m(\rd y) 
 \norm{ (k(z) - k(x)) \nabla \omega_q(z-x) }_{L_2(z)}\right| \\
& \le   q   \norm{  \nabla k}_\infty   \int m(\rd x)  \mathds{1} ( {\rm dist} (x, {\rm supp} \ k)  \le q)
 \norm{ q \nabla \omega_q(z-x) }_{L_2(z)} 
 \le   r^2  \norm{ \nabla k}_\infty  \le  r \norm{k}_{\infty,2, r}.
\end{align*}
All other terms are bounded similarly and we obtain  \eqref{45} in this case $r \ge q$.

We now assume that $r \le q$.  In this case, we view  $k(z) \omega_q(z-x) \sim k(z)  \mathds{1}_{ {\rm dist} (x, {\rm supp} \ k)  \le q} q^{-2}$ and obtain
\begin{multline*}
\left|\int k(z) \omega_q(z-x)\ \tilde\mu(\rd z)\right|
\le   \mathds{1}_{{\rm dist} (x, {\rm supp} \ k)  \le q}   \frac{  r N^\epsilon }{\sqrt{N}q^2}
\Big  [ \|  \nabla k (z) \|_{L_2(z)} + \frac{ \|  k (z) \|_{L_2(z)}}{\ell}  + \frac{r}{\sqrt{N}} \|\Delta_z k(z) \|_\infty \Big ] 
\\ \le  \mathds{1}_{{\rm dist} (x, {\rm supp} \ k)  \le q}  
q^{-2}  r N^{-\frac{1}{2}+\epsilon}  \norm{k}_{\infty,2, r}, 
 \end{multline*}
where we have used $r \le q \le \ell$ and $N^{-1/2} \le r$. 
This implies that 
\begin{align*}
\iint U_i (x,y) m(\rd x) m(\rd y)   \int k(z) \omega_q(z-x)\ \tilde\mu(\rd z)  
=\OO\left(  q r N^{-\frac{1}{2}+ \epsilon} \norm{k}_{\infty,2, r}\right).
\end{align*}
Similar inequality holds with  $k(z)$ replaced by $k(x)$.  This concludes the proof of \eqref{45}, and therefore the contribution from $j=0$ in (\ref{123}). 

One can check in a similar way that the same  bound holds 
for any $j$ since  the  factor $q^{-j}$ induced by  the derivatives on $U_i$ is compensated by the size of the function 
$\varphi_{\b k}$. Notice that for all $j$, we  need at most two derivatives on $k$; all other derivatives will apply to explicit functions depending on 
$U_i$.  Summing over all $j$ and $i$ and using $q \le C \ell$, we have thus proved
\be\label{BBB0}
\left|\sum_{ i \le m} N \iint_{ z \neq w } M_i (z,  w)   \, \tilde \mu_{V}(\rd z) \, \tilde \mu_{V}(\rd w)\right|
\le N^{2 \epsilon} \sum_{i\le m}  A(\ell_i )^2 r \norm{k}_{\infty,2, r}
\le N^{2 \epsilon} \big [ 1 + (\sqrt N  \ell_m)^{-1} \big ]^2 r \norm{k}_{\infty,2, r}.
\eeq
From the definition of   $h$ in Lemma \ref{lem:hYdecomp}, we need to consider two contributions of $h$: one is $k=\bar \partial f$, 
the other involves the $s$ integration.  Since $f$ is supported in a ball of radius $b$, 
the contribution from $\bar \partial f$ can be trivially bounded by 
$$
N^{2 \epsilon} \big [ 1 + (\sqrt N  \ell_m)^{-1} \big ]^2 b \norm{h}_{\infty,2, b}
\le N^{2 \epsilon} \big [ 1 + (\sqrt N  \ell_m)^{-1} \big ]^2  \norm{f}_{\infty,3, b},
$$
where we have replaced $r$ in \eqref{BBB0} by $b$.  

Applying now \eqref{BBB0} to $k=\bar \partial f_s$, we bound
the other term of $h$ involving $s$ integration by 
\begin{align*}
& N^{2 \epsilon} \big [ 1 + (\sqrt N  \ell_m)^{-1} \big ]^2  \ell^{-2} \int_0^{\log N}  \frac{\rd s}{s} b_s   \norm{\partial f_s}_{\infty,2, b_s}
+N^{-A} \norm{ f }_{\infty,1, b}  \\
&\le   N^{2 \epsilon} \big [ 1 + (\sqrt N  \ell_m)^{-1} \big ]^2  \ell^{-2}
  \int_0^{\log N}  \frac{\rd s}{s}  (b\wedge s)^{2} \norm{ f}_{\infty,3, b} + N^{-A} \norm{ f }_{\infty,1, b}\\
  &  \le  N^{2 \epsilon} \big [ 1 + (\sqrt N  \ell_m)^{-1} \big ]^2  \frac{b^2}{\ell^2}
   \norm{ f}_{\infty,3, b} + N^{-A} \norm{ f }_{\infty,1, b}
  ,
\end{align*}
where we have again  used  \eqref{normbound}.

Combining Case 1 and 2, we have bounded \eqref{e:Yanglebd} by 
\be
  N^{2 \epsilon}  
  \Big [ \big (1 + (\sqrt N  \ell_m)^{-1} \big )^2  +  N  \ell_m^{2}  \Big ]  \Big [1+   \frac{b^2}{\ell^2} \Big ]  \|f\|_{\infty,3,b} +  N^{-A} 
  \|f\|_{\infty,3,b}, 
\eeq  
where the term $N  \ell_m^{2}    [1+   \frac{b^2}{\ell^2} ]$ comes from the Case 1 and the other terms come from Case 2. 
Recalling $\ell_m = N^{-1/2 + \epsilon}$, we have proved \eqref{e:Yanglebd}.
\end{proof}

\begin{proof}[Proof of Theorem~\ref{thm:Yrigi}] 
We will assume $V=0$; the general case can be proved in a similar way.
We again employ the loop equation and calculate  
\begin{align*}
& \frac{1}{\beta} \log \E_0^{U^\ell} \ee^{-t \beta X_0^f}
= \frac{1}{\beta} \int_0^t \rd s \, \frac{\partial}{\partial s} \log \E_0^{U^\ell} \ee^{-s \beta X_0^f}
= \int_0^t \rd s \, \pb{ - \E_{sf/N}^{U^\ell} X_{sf/N}^f + N \int f (\mu_0 - \mu_{sf/N}) } \\
& = \int_0^t \rd s \, \E_{sf/N}^{U^\ell} \bigg( \frac{1}{N} W_{sf/N}^{h_s} - \frac{1}{N\beta} \sum_j \partial h_s(z_j) \\
&  - N \iint (h_s(z)-h_s(w)) \partial U^\ell(z-w) \tilde \mu_{sf/N}(\rd z) \tilde \mu_{sf/N}(\rd w) + N \int f (\mu_0 - \mu_{sf/N}) \bigg),
\end{align*}
where $ h_s(z)  = { \frac{1}{\pi \rho_{sf/N}(z)}} \bar\partial  (1-m^{2}\Delta^{-1})f(z)$
 and $\mu_{sf/N}$ as in \eqref{eqn:Yhg}.
By Lemma \ref{lem:Yanglebd},
\begin{equation*}
  N \, \E_{sf/N}^{U^\ell} \iint (h_s(z)-h_s(w)) \partial U^\ell(z-w) \tilde \mu_{sf/N}(\rd z) \tilde \mu_{sf/N}(\rd w)
  = \OO\pbb{ N^\varepsilon \pb{ 1 + \frac{b^2}{\ell^2} } } \normt{f}_{\infty,3,b}.
\end{equation*}
By Lemma \ref{lem:expectation0}, $\E_{sf/N}^{U^\ell} W_{sf/N}^{h_s} = 0$.
Using $m = 1/\ell$, we have 
\begin{multline*}
  \int_0^t \rd s \, N \int f(\mu_0 - \mu_{sf/N})
  = - N \int_0^t \rd s \, \int f(z) \Big [ \frac{s}{4 \pi N} (\Delta-m^2) f(z) \, m(\rd z)
  +\frac{  m^2}{2 \pi }  (c_0 - c_{sf/N}) \Big ] \\
  =  \OO\Big (  \int_0^t \rd s \, \frac{s}{4 \pi} \int f (-\Delta+m^2) f  \Big ) = \OO\pb{t^2 ( 1 + b^2 \ell^{-2} )} \normt{f}_{\infty,2,b}^2,
\end{multline*}
 where we have used $  (c_0 - c_{sf/N}) = \frac s { 2 N} \int f (z)  m(\rd z) $
which is a consequence of \eqref{e:equilibriumdensity} and the normalization condition $\int \mu^\ell_V=1$. 
Finally, using the fact that the local density is bounded and \eqref{normbound}, we have  
\begin{align*}
   \int_0^t \rd s \, \frac{1}{N \beta} \E_{sf/N}^{U^\ell} \sum_j \partial h_s (z_j) = \OO(t) 
    \norm{\partial \rho_{sf/N}^{-1} \bar\partial ( f -  m^{2} \Delta^{-1} )  f}_{1}
   = \OO(t) \pb{ 1 + \frac{b^2}{\ell^2} } b^2   \normt{f}_{\infty,2,b}.
\end{align*}
Collecting these estimates gives
\begin{align}
  \frac{1}{\beta} \log \E_0^{U^\ell} \ee^{-t \beta X_0^f}
   = 
  \OO\pbb{  N^\varepsilon \pb{ 1 + \frac{b^2}{\ell^2} } } \big [ t \normt{f}_{\infty,3,b} + t^2  \normt{f}_{\infty,3,b}^2 \big ].
\end{align}
By Markov's inequality with  $t = 1/\normt{f}_{\infty,3,b}$ this implies 
$
X_0^f= \OO(N^\varepsilon) ( 1 + b^2/\ell^2 ) \normt{f}_{\infty,3,b} 
$
with probability at least $1 - \ee^{-N^\varepsilon}$, which proves Theorem~\ref{thm:Yrigi}. 
\end{proof}

With Theorem~\ref{thm:Yrigi} proved, we can now prove Proposition~\ref{prop:Grigi}.

\begin{proof}[Proof of Proposition~\ref{prop:Grigi}]
We apply the identity \eqref{B111}
which expands the left hand side of \eqref{e:Grigi} into a Taylor series with error term. 
To prove \eqref{e:Grigi}, we only have to estimate each term in the summation in (\ref{eqn:nd}).
From the rigidity estimate Theorem~\ref{thm:Yrigi}, we  have
\begin{equation}\label{eqn:st}
  N\int (x-z)^{k_1}(\bar x-\bar z)^{k_2} \omega_s(z-x) \tilde\mu(\rd z)
  = \OO(N^\epsilon) s^{k_1+k_2-2} \left(1+\frac{s}{\ell}\right)^2,
\end{equation}
and a similar estimate holds around $y$. 
These two bounds together imply that each term in the summation in (\ref{eqn:nd})
is bounded by $ \OO(N^\varepsilon)  \big (\frac { 1} {s^4 } +  \frac {1} { \ell^4}\big) s^j \|\nabla^j g\|_{1}$
and this  completes  the proof of the proposition. 
\end{proof}

\subsection{Proof of Lemma~\ref{lem:Apbd}}
\label{sec:Apbd-pf}

In this subsection, we  prove Lemma~\ref{lem:Apbd} which is an estimate with respect to the Coulomb gas with an angle correction term. 
Recall the definition of the long range interaction in \eqref{Ap}, given by 
\[
 A_V^{h, +} =   \frac{N}{2}   \iint_{z\neq w} \Psi^+_h (z, w) \, \tilde \mu_V (\rd z) \, \tilde \mu_V (\rd w), 
 \]
 where 
 \[
 \Psi_h^+ (z, w) = \Phi_\theta^+ (z-w)  (\bar z-\bar w)(h(z)-h(w)),   \quad 
 \Phi_\theta^+ (z-w)= \int_\theta^\infty   \Phi (z-w,  r)  \frac{\rd r }{r^5}=\frac{1-\ee^{-\frac{|z-w|^2}{2\theta^2}}}{|z-w|^2}.
\]
Our proof of Lemma~\ref{lem:Apbd} is based on transporting a proof for the Yukawa gas to the Coulomb setting. 
For this purpose, recall  that  the long range part in the  decomposition  \eqref{Udec} for the Yukawa gas 
is of  the following form 
\be\label{Um}
 G_\sigma(z,  w) =  (h(z)-h(w)) \sum_{i \le m}   U_i (z -w  ),
\eeq
with $\ell_m \sim N^{-1/2 + \sigma}$ for some $\sigma > 0$ fixed. 
Using the rigidity estimate  \eqref{e:Yrigi} for the Yukawa gas, 
we claim that 
\begin{equation} \label{e:anglebd-rigi}
 \left| \E_{V}^{U^\ell} \pa{ N \iint_{ z \neq w } G_\sigma(z,  w)   \, \tilde \mu_{V}(\rd z) \, \tilde \mu_{V}(\rd w) }\right|
  \leq N^{-\sigma+\epsilon} r\|h\|_{\infty,2,r}.
\end{equation}

\noindent 
To prove this bound, we keep the estimate \eqref{BBB} unchanged but for \eqref{BBB1},
instead of the local  law, we  apply the rigidity estimate  \eqref{e:Yrigi} 
\begin{align}
\Big |   & \int \omega_q(w-y) \tilde\mu(\rd w)  
 \Big |  \le   
  N^{-1+\varepsilon} \Big ( 1 +   q^2/\ell^2 \Big ) \|   \omega_q(\cdot -y)\|_{\infty,3, q}.
  \label{BBB2}
\end{align}
Using $\| \omega_q(\cdot -y)\|_{\infty,3, q}  \le q^{-2}$, we therefore
improved \eqref{BBB0} to  
\be
 \left|\sum_{ i \le m} N \iint_{ z \neq w } M_i (z,  w)   \, \tilde \mu_{V}(\rd z) \, \tilde \mu_{V}(\rd w)\right|   \leq  \frac{N^{2 \epsilon} }{\sqrt N  \ell_m}    \left( 1 + \frac{1}{\sqrt N  \ell_m}\right)^2 r \norm{h}_{\infty,2, r}, 
\eeq
gaining a factor $(\sqrt N  \ell_m)^{-1}$ over \eqref{BBB0}. This proves  \eqref{e:anglebd-rigi}. 
Notice that the extra derivative required in applying the rigidity estimate is performed on the test function $\omega$, 
so the number of derivatives required on $h$ remains the same when compared with the earlier results relying on the local law.

We return to estimating $A^+$, ie.e the proof of Lemma~\ref{lem:Apbd} 
 for Coulomb gases with or without an angle correction term.  Notice that $ \Psi_h^+ (z, w)$ is of the form 
 $ k (z-w) (h(z)-h(w))$ with $k (z-w) = \Phi_\theta^+ (z-w)  (\bar z-\bar w)$. Hence we can apply the
  decomposition  \eqref{Udec} and express $A^+$ similarly to  $G$ in \eqref{Um}. Due to the short range cutoff by $\theta$ in the definition 
 of $\Phi_\theta^+$, we effectively have a cutoff at the scale $\theta = N^{-1/2 + \sigma}$.  This is consistent with the choice of 
 $\ell_m \sim N^{-1/2 + \sigma}$ in \eqref{Um}.
Instead of the rigidity estimate 
\eqref{e:Yrigi} for the Yukawa gas, we apply the one with respect to the Coulomb gas with an angle correction term, i.e.,  \eqref{r6a}.
Notice that  in \eqref{r6a}, $ \|   \omega_q(\cdot -y)\|_{\infty,3, q} $ in \eqref{BBB2}  was  replaced by $\|   \omega_q(\cdot -y)\|_{\infty,4, q} $. 
Since $\omega$ is a smooth mollifier, 
both norms are of the same order.  Following the argument in the proof  of \eqref{e:anglebd-rigi}, we have therefore proved Lemma~\ref{lem:Apbd}.

The key observation in this proof  is that the application of the rigidity estimate yields an improvement over the local law for all functions of scales 
bigger than $N^{-1/2 + \epsilon}$. So any  estimates based on  the local laws can be improved by a factor   $N^{-\sigma}$ for functions 
at  scales greater than $ N^{-1/2 + \sigma}$.

\section{Local law for the Yukawa and Coulomb gas}
\label{app:yukawa}

In this appendix, we prove Theorems~\ref{thm:YTdensity}--\ref{thm:Cdensity}.
Our presentation follows closely that of \cite{MR3694026}
and we therefore mainly present the differences.
The interaction in  Theorem~\ref{thm:YTdensity}  is  a Yukawa potential instead of  the Coulomb potential in \cite[Theorem~1.1]{MR3694026}.
To allow for this change, we first develop generalizations
of the basic potential estimates used in \cite{MR3694026}  to  the Yukawa potential. 
Once these estimates are given, the rest of the proof is parallel to that of \cite[Theorem~1.1]{MR3694026}.
The proof of Theorem~\ref{thm:Cdensity} is  essentially the as same as that of 
\cite[Theorem~1.1]{MR3694026} under slightly generalized assumptions.
Its proof requires only minor adjustment to the original proof
which we will comment on later in this appendix.

\subsection{Some potential theory for the Yukawa potential}
\label{sec:yukawa-potential}

We start with properties of the Yukawa potential. They are parallel to those
of the Coulomb potential used in \cite{MR3694026}.

The following proposition characterizes the Yukawa potential of the equilibrium
measure in terms of an obstacle problem. The proposition is similar
to the analogous result for the Coulomb case, but requires a slightly different characterization of
the admissible potentials than the one stated for the Coulomb case
in \cite{MR3056295}, for example. We give a proof for completeness,
as we were unable to locate a suitable reference.

\begin{proposition} \label{prop:obstacle}
Under the assumptions of Theorem~\ref{thm:Yeqmeasure}, the following holds. Define
\begin{equation} \label{e:obstacle}
u_{V,\ell}(z) = \sup_{\nu,c} \{ -U_\nu^\ell(z) + c: -U_\nu^\ell + c \leq \tfrac12 V, \nu \geq 0, \nu(\C) \leq 1 \},
\end{equation}
where the supremum is over measures $\nu$ and constants $c$.
Then $
u_{V,\ell} = -U_{\mu_V}^\ell + c_V$ where $c_V$ is the constant in \eqref{e:EL}.
\end{proposition}

\begin{proof}
By definition, $u_{V,\ell} \geq -U_{\mu_V}^\ell + c_V$ since the right-hand side is a subsolution of the same
form as inside the supremum in  \eqref{e:obstacle}.
To prove that in fact equality holds,
suppose otherwise that $u_{V,\ell}(z_0) > -U_{\mu_V}^\ell(z_0) + c_V$ for some $z_0 \in \C$.
Then there exists some positive measure $\tilde \eta$ with $\tilde \eta(\C) \leq 1$ and constant $c \in \R$ for which $-U_{\tilde \eta}^\ell(z_0) + c > -U_{\mu_V}^\ell(z_0) + c_V$. By considering $\tilde \eta|_{B_R}$ for $R>0$ large enough we may suppose that $\tilde \eta$ is compactly supported, and by convolving with a smooth mollifier we may suppose $\tilde \eta$ has a smooth density. Consider the function
\begin{equation*}
g(z) = \max(-U_{\tilde \eta}^\ell(z) + \tilde c,-U_{\mu_V}^\ell(z) + c_V).
\end{equation*}
By writing $\max(a,b) = \tfrac{a+b}{2} + \tfrac{|a-b|}{2}$ and convolving the absolute value by a smooth,
compactly supported, symmetric mollifier, we may check that
$
g(z) = -U_\eta^\ell(z) + c
$
for some positive measure $\eta$, and necessarily $c = \max(\tilde c,c_V)$.
To show that $g$ is a subsolution of the form in \eqref{e:obstacle}
we need to show that $\eta(\C) \leq 1$.
For this, suppose without loss of generality that $c = \tilde c$.
Denote $D = \{z: -U_{\tilde \eta}^\ell(z) + \tilde c < -U_{\mu_V}^\ell(z) + c_V\}$. Then
\begin{align*}
& \eta(\partial D)  = \int_{\partial D} \partial_n (-U_{\tilde \eta}^\ell - (-U_{\mu_V}^\ell)) = \int_D \Delta (-U_{\tilde \eta}^\ell - (-U_{\mu_V}^\ell)) \\
& = \int_D (\tilde \eta - \mu_V) + m^2 \int_D (-U_{\tilde \eta}^\ell - (-U_{\mu_V}^\ell)) \\
& = \int_D (\tilde \eta - \mu_V) + m^2 \int_D (-U_{\tilde \eta}^\ell + \tilde c - (-U_{\mu_V}^\ell + c_V)) + m^2 \int_D (c_V - \tilde c) 
 \leq \tilde \eta(D) - \mu_V(D).
\end{align*}
Thus $\eta(D \cup \partial D) = \eta(\partial D) + \mu_V(D) \leq \tilde \eta(D)$.
Since clearly $\eta (\C \setminus (\partial D \cup D)) = \tilde \eta(\C \setminus (\partial D \cup D))$,
we have $\eta(\C) \leq \tilde \eta(C) \leq 1$.
Now,
\begin{align*}
g - (-U_{\mu_V}^\ell + c_V) & \geq 0, \\
(\Delta-m^2)(-U_\eta^\ell - (-U_{\mu_V}^\ell) = \eta - \mu & \geq m^2 (c-c_V) \geq 0.
\end{align*}
Since strict inequality holds in the first inequality for $z_0$ and the functions involved are continuous, equality (as distributions) cannot hold on the second line. But this implies $\eta(\C) > \mu_V(\C) = 1$, a contradiction.
\end{proof}

We also require the following properties of the Yukawa potential \eqref{Yg}.
Recall that
\begin{equation} 
  Y^{\ell} (z)
  = g(a),
  \quad a = \frac {|z|}{2\ell},
  \qquad\text{where }
  g(a) = \int_1^\infty  \ee^{- a ( s  + 1/s) }  \frac {\rd s}s.
\end{equation}
In fact, $g(a) = K_0(2a)$ where $K_0$ is a modified Bessel function of the second kind.
In particular, the gradient of the Yukawa potential has the expression:
\begin{equation} \label{e:dYf}
  \nabla Y^\ell(z)
  =g'(\frac{|z|}{2\ell}) \frac{\nabla |z|}{2\ell}
  =g'(\frac{|z|}{2\ell}) \frac{|z|}{2\ell} \frac{1}{\bar z}
  = - \frac{1}{\bar z} f(a), \quad a = \frac{|z|}{2\ell},
\end{equation}
where
\begin{equation*}
  f(a) = \int_1^\infty  a (s+1/s) \ee^{- a ( s  + 1/s) }  \frac {\rd s}s
  = \int_{a}^\infty  (1+a^2/s^2) \ee^{- ( s  + a^2/s) }  \rd s
  .
\end{equation*}
The function $f$ is smooth in $a >0$, satisfies $f(0) = 1$, and is positive and decreasing.
As a consequence we have $|\nabla Y^\ell(2r)| \leq |\nabla Y^\ell(r)|/2$.

Since $\nabla Y^\ell(z) \sim \nabla \log \frac{1}{|z|}$ for $z \to 0$,
the following formula \eqref{e:dUeta} follows as in \cite[(3.21)]{MR3694026}.
Let $\gamma \subset \C$ be a $C^1$ curve and $\eta$ a measure supported on $\gamma$
for which the potential $U_\eta^\ell$ is continuous on $\C$. Then for $z \in \gamma$ we have
\begin{equation} \label{e:dUeta}
  \partial_n^- U_\eta^\ell(z) = \pi \lim_{r \to 0^+} \frac{\eta(B_r(z))}{s(B_r(z))}
  +\int_\gamma \nabla Y^\ell(z-w) \cdot \bar{n} \, \eta(dw),
\end{equation}
where $\partial_n^-$ denotes a one-sided derivative
in the normal direction $\bar{n} = \bar{n}(z)$ and $s$ denotes the arclength measure of $\gamma$,
if the limit on the right-hand side exists.

The formula \eqref{e:dUeta} implies the following estimate for the
density of a measure supported on $\partial \D$. 
For the statement, define
\begin{equation}
  I^\ell := \frac{1}{2\pi} \int_{\partial \D} f(\frac{|1-w|}{2\ell}) \, s(\rd w) \in (0,1),
\end{equation}
and note that $I^\ell$ is increasing in $\ell$ with $I^\ell = 1+O(1/\ell)$ as $\ell\to\infty$ and $I^\ell = O(\ell)$ as $\ell\to 0$.
The proofs of the following Lemma~\ref{lem:domegacircle} and \ref{A9} are based
on elementary potential theory.

\begin{lemma} \label{lem:domegacircle}
  For any (signed) measure $\omega$ supported on $\partial \D$,
  denote by $\bar\omega =\frac{1}{2\pi} \int \rd \omega$ the constant part of $\omega$.
  Then
  \begin{align}
    \label{e:domegabd2-gt1}
    \norma{\frac{d\omega}{ds}- \bar\omega}_\infty
    &\leq \frac{2}{\pi I^\ell} \|\partial_n^- U_\omega^\ell\|_{\infty,\partial\D},
    \\
    \label{e:domegabd2-lt1}
    \norma{\frac{d\omega}{ds}}_\infty
    &
    \leq \frac{1}{\pi(1-I^\ell)} \|\partial_n^- U_\omega^\ell\|_{\infty,\partial\D},
  \end{align}
  and
  \begin{equation} \label{e:Ubaromega}
    \partial_n^- U_{\bar \omega}^\ell(1)
    = \frac{1}{2\pi} \int_{\partial \D} \partial_n^- U_\omega^\ell(z) \, s(\rd z)
    \leq \|\partial_n^- U_\omega^\ell\|_{\infty,\partial\D}.
  \end{equation}
\end{lemma}

\begin{proof}
By \eqref{e:dUeta}, we have
\begin{equation}
  \frac{d\omega}{ds}(z)
  = \frac{1}{2\pi} \pa{
    2\partial_n^- U^\ell_\omega(z) - 2 \int \nabla Y^\ell(z-w)\cdot \bar n(z) \, \omega(\rd w)
  }.
\end{equation}
For $z,w$ with $|z|=|w|=1$ and $z \neq w$,
\begin{equation}
  \frac{z-w}{|z-w|^2} \cdot \frac{z}{|z|}
  =
  \re \pa{\frac{z-w}{|z-w|^2} \bar z}
  =
  \re \pa{\frac{1-w/z}{|1-w/z|^2}} = \frac12,
\end{equation}
and, by \eqref{e:dYf}, therefore
\begin{equation}
  -2 \nabla Y^\ell(z-w) \cdot n(z) = f(\frac{|z-w|}{2\ell}).
\end{equation}
It follows that
\begin{equation} \label{e:domega-pf}
  \frac{d\omega}{ds}(z)
  = \frac{1}{2\pi} \pa{
    2 \partial_n^- U^\ell_\omega(z) + \int f(\frac{|z-w|}{2\ell}) \, \omega(\rd w)
  }.
\end{equation}

Integrating \eqref{e:domega-pf}, we obtain the identity
\begin{equation}
  (1-I^\ell) \int \omega =  \frac{2}{2\pi} \int_{\partial \D} \partial_n^- U_\omega^\ell(z) \, s(\rd z).
\end{equation}
Applying this identity to $\bar\omega$, since $\int \rd \omega = \int \rd \bar\omega$, we obtain
\begin{equation}
  \partial_n^- U_{\bar\omega}^\ell(1)
  = \frac{1}{2\pi} \int_{\partial \D} \partial_n^- U_{\bar \omega}^\ell(z) \, s(\rd z)
  = \frac{1}{2\pi} \int_{\partial \D} \partial_n^- U_\omega^\ell(z) \, s(\rd z).
\end{equation}
This shows \eqref{e:Ubaromega}.
Similarly, from \eqref{e:domega-pf}, we obtain 
\begin{equation} \label{e:domegabd2-pf}
  (1-I^\ell) \norma{\frac{d\omega}{ds}}_\infty
  \leq \frac{2}{2\pi} \|\partial_n^- U_\omega^\ell\|_\infty,
\end{equation}
which shows \eqref{e:domegabd2-lt1}, and also similarly,
\begin{equation} \label{e:domegabd1-pf}
  \norma{\frac{d\omega}{ds}- \frac{1}{2\pi} \int f(\frac{|\cdot-w|}{2\ell}) \, \omega(\rd w)}_\infty \leq \frac{2}{2\pi} \|\partial_n^- U_\omega^\ell\|_\infty.
\end{equation}
To show \eqref{e:domegabd2-gt1}, i.e.,
\begin{equation} \label{e:domegabd3-pf}
  I^\ell \norma{\frac{d\omega}{ds}- \frac{1}{2\pi}\int d\omega}_\infty
  \leq
  \frac{4}{2\pi} \|\partial_n^- U_\omega^\ell\|_\infty
  ,
\end{equation}
write
\begin{align*}
  &\frac{d\omega}{ds}- \frac{1}{2\pi} \int f(\frac{|\cdot-w|}{2\ell}) \, \omega(dw)
  \\
  &=
    \frac{d\omega}{ds}- \frac{1}{2\pi}\int d\omega
    +
    \frac{1}{2\pi} \int \pa{1-f(\frac{|\cdot-w|}{2\ell})} \, \omega(dw)
  \\
  &=
    \frac{d\omega}{ds}- \frac{1}{2\pi}\int d\omega
    + \frac{1}{2\pi} \int \pa{1-f(\frac{|\cdot-w|}{2\ell})} \, \pa{\frac{d\omega}{ds}(w)-\frac{1}{2\pi} \int d\omega}  \, s(dw)
  \\
  &\qquad
    + \frac{1}{2\pi} \int \pa{1-f(\frac{|\cdot-w|}{2\ell})}  \, s(dw) \cdot \frac{1}{2\pi} \int d\omega.
\end{align*}
Taking absolute values on the supremum over $\partial \D$, and using \eqref{e:domegabd2-pf}, therefore
\begin{align*}
  &\|\frac{d\omega}{ds}- \frac{1}{2\pi} \int f(\frac{|\cdot-w|}{2\ell}) \, \omega(dw)\|_\infty
  \\
  &\geq
    \|\frac{d\omega}{ds}- \frac{1}{2\pi}\int d\omega\|_\infty
    - (1-I^\ell)     \|\frac{d\omega}{ds}- \frac{1}{2\pi}\int d\omega\|_\infty
    - (1-I^\ell) \absa{\frac{1}{2\pi} \int d\omega},
  \\
  &\geq
    I^\ell
  \|\frac{d\omega}{ds}- \frac{1}{2\pi}\int d\omega\|_\infty 
    -  \frac{2}{2\pi} \|\partial_n^- U_\omega^\ell\|_\infty.
\end{align*}
Together with \eqref{e:domegabd1-pf}, we obtain \eqref{e:domegabd3-pf}.
\end{proof}

We will also need the following properties of the function
\begin{equation} \label{e:ldef}
  l_r(z) = \pa{ Y^\ell * \frac{1}{\pi r^2} 1_{B(0,r)}} (z).
\end{equation}
Clearly, $l_r(z)$ is radial, so we can define $h_r$ through $\nabla l_r(z) = -(z/|z|) h_r(|z|)$ for $z \neq 0$.

\begin{lemma}\label{A9}
For any $\ell>0$, the function $h_r(t)$is positive, increasing for $t \leq r$,
decreasing for $t \geq r$, and
\begin{equation} \label{e:lrYell}
  h_r(t) \geq |\nabla Y^\ell(t)| \quad \text{for $t \geq r$.}
\end{equation}
\end{lemma}

\begin{proof}
That $h_r(t)$ is increasing for $t<r$ can be seen as follows.
For $t>0$, since $Y^\ell$ is symmetric,
\begin{equation*}
  \nabla l_r(t)
  = \int_{|z|\leq r} \nabla Y^\ell(z-t) \, m(\rd z)
  = \int_{U_r(t)} \re \nabla Y^\ell(z-t) \, m(\rd z)
  = \int_{U_r(t)-t} \re \nabla Y^\ell(z) \, m(\rd z)
\end{equation*}
where $U_r(t)$ is $\{|z| \leq r\}$ minus the region $\{ |z| \leq r : \re z > t \}$
and the reflection of the latter region about the axis $\re z = t$.
In particular, the region $U_r(t)-t$ is increasing in $t$.

To prove  \eqref{e:lrYell}
and that $h_r(t)$ is decreasing for $t>r$,
we use the Yukawa version of Newton's shell theorem:
there is $M^\ell(r) \geq 1$ such that for $t \geq r$,
\begin{equation} \label{e:NewtonYukawa}
  \frac{1}{2\pi r} \int_{|z|=r} Y^\ell(t-z) \, s(\rd z) = M^\ell(r) Y^{\ell}(t).
\end{equation}
Denote the left-hand side by $f(t)$.
Then $f$ is a bounded and radially symmetric solution to $(-\Delta+1/\ell^2)f(z)=0$ for $|z|>r$.
Therefore, for $t>r$,
\begin{equation} 
  f''(t)+ \frac{1}{t} f'(t) - \frac{1}{\ell^2} f(t) = 0,
\end{equation}
and the solutions to this ODE are of the form
\begin{equation} \label{e:ODEBessel}
  f(t) = AI_0(t/\ell) + BK_0(t/\ell),
\end{equation}
where the $I_n$ are the modified Bessel functions of the first kind
and the $K_n$ are the modified Bessel functions of the second kind,
and $A,B$ are constants depending on $r$.
The Yukawa potential equals $Y^\ell(z) = K_0(|z|/\ell)$.
Since $I_0(t) \to \infty$ as $t\to\infty$, therefore $A=0$ and
thus $f(t) = BK_0(t/\ell) = BY^\ell(t)$ for some constant $B=M^\ell(r)$.

To see that $B \geq 1$, we assume that $r=1$ and $\ell=1/2$ to simplify the notation (the general case is analogous).
Denote by $\theta$ the angle of $z$ with respect to the real axis so that  $|t-z|^2 =   t^2 - 2t \cos \theta +1$.
Recall \eqref{Yg} and note that the function $\tilde g(x) = \int_1^\infty \ee^{-  \sqrt {x} ( s  + 1/s)} \frac {\rd s}s$ is convex for $x \geq 1$.
With  $x = t^2 - 2t \cos \theta +1$ and using  the Jensen inequality, we have 
\be
f(t) = \E \tilde g( t^2 - 2t \cos \theta +1)
\ge  \tilde g( t^2 - 2t \E \cos \theta +1) = \tilde g(t^2+1), \quad \E = (2 \pi)^{-1} \int  \rd \theta. 
\eeq
It is elementary to  check that 
\be
\lim_{t \to \infty} \frac { \tilde g(t^2+1)} { \tilde g(t^2)} = 1.
\eeq
Hence we have proved that $B \ge 1$. (In fact, $B > 1$ for any $r, \ell$ fixed, but we will not need this.)

In particular, for $t \geq r$,
\begin{equation}
  l_r(t)
  = \frac{1}{\pi r^2} \int_{|z|\leq r} Y^\ell(z-t)
  = \frac{1}{\pi r^2} \int_0^r  \int_{|z|=s} Y^\ell(z-t) \, s(\rd z) \, \rd r
  = \tilde M^\ell(t) Y^\ell(t)
\end{equation}
with
$\tilde M^\ell(t) = \frac{1}{\pi r^2} \int_0^r (2\pi r) M^\ell(r)  \, \rd r \geq 1$.
Thus, for $t \geq r$,
\begin{equation}
  |\nabla l_r(t)|
  = \tilde M^\ell(r) |\nabla Y^\ell(t)|
  \geq |\nabla Y^\ell(t)|.
\end{equation}
The first equality implies that $|\nabla l_r(t)|$ is
decreasing for $t>r$ since $|\nabla Y^\ell(t)|$ is decreasing.
The inequality implies that \eqref{e:lrYell} holds.
\end{proof}

In Section~\ref{sec:perteqmeas} below, we require the following two technical lemmas
to locate the bulk of the support of a perturbed equilibrium measure.
Lemma~\ref{lem:log-matching} is a small adaption of \cite[Lemma~3.6]{MR3694026}
to the Yukawa case;
Lemma~\ref{lem:G-matching} is a similar statement that applies to
a radially symmetric potential on the boundary of a disk instead of a point charge
outside a disk.

\begin{lemma} \label{lem:log-matching}
For any $z_0 \in \C, w \in \C$, $\sigma>\frac12$,
and $r \in (0,1)$ such that that $|z_0-w| \geq 2r$,
there exist $\tilde{z} \in \mathbb{C}$ and $k \in \mathbb{R}$ such that
\begin{equation} \label{e:log-matching}
\sigma \big(l_r(z_0-\tilde{z}) + k\big) = \frac{1}{2} Y^\ell(z_0-w)
\quad \textrm{and} \quad
\sigma \big(l_r(z-\tilde{z}) + k\big) \leq \frac{1}{2} Y^\ell(z-w) \textrm{ for all } z \in \mathbb{C}.
\end{equation}
Moreover, the point $\tilde z$ lies on the line passing through $z_0$ and $w$ at distance at most $r$ from $z_0$
between $z_0$ and $w$.
\end{lemma}

\begin{proof}
By \eqref{e:lrYell} and since $\sigma \geq \frac12$,
the map $z \mapsto \sigma \nabla l_r(z_0-z)$ takes
$B_{r}(z_0)$ onto
$B_{\sigma |\nabla l_r(r)|}(0) \supset B_{\sigma |\nabla Y^\ell(r)|} \supset B_{|\nabla Y^\ell(2r)|}(0)$,
where we also used $|\nabla Y^\ell(2r)| \leq \frac12 |\nabla Y^\ell(r)|$.
Therefore, as in \cite[Lemma~3.6]{MR3694026},
it follows there exists a unique choice of $\tilde z \in B_r(z_0)$ so that the gradients of
$\sigma l_r(\cdot - \tilde{z})$ and $\tfrac{1}{2} Y^\ell(\cdot -w)$ match at $z_0$.
By choice of $k$, we can in addition arrange
\begin{equation} 
\sigma \big(l_r(z_0-\tilde{z}) + k\big) = \frac{1}{2} Y^\ell(z_0-w).
\end{equation}
It remains to be shown that with the above choice 
it is in fact true that
\begin{equation} \label{e:lbdC}
\sigma \big(l_r(z-\tilde{z}) + k\big) \leq \frac{1}{2} Y^\ell(z-w) \quad \textrm{for all } z \in \mathbb{C}.
\end{equation}
As in the Coulomb case, the point must $\tilde z$ lie on the line between the points $z_0$ and $w$,
and it suffices to show the inequality on this line (by the same argument as in the Coulomb case,
\cite[Lemma~3.6]{MR3694026}).
Moreover, without loss of generality, we can assume
that $w=0$, $z_0>0$, $\tilde z > 0$, so that this line is $\R$.
Thus it needs to be shown that
\begin{equation*}
f(x) := \frac{1}{2} Y^\ell(x) \geq \sigma \big( l_r(x-\tilde{z}) + k \big) =: g(x), \quad x \in \R.
\end{equation*}
As in the Coulomb case,
denote by $h$ the common tangent of the graphs of $f$ and $g$ drawn at $x=z_0$.
Since $f$ is convex and $g$ is concave on $[\tilde{z}-r,\tilde{z}+r]$,
the graph of $f$ lies above $h$ and the graph of $g$ lies below $h$ on this interval.
Especially $g(x) \leq f(x)$ on $[\tilde{z}-r,\tilde{z}+r]$. Moreover, since $f'(x) < 0$ and $g'(x) > 0$ for $x \in (0,\tilde{z})$,
the inequality $g(x) \leq f(x)$ holds by these observations for $x \in (0,\tilde{z}+r]$.

To prove the inequality for $x \in [\tilde{z}+r,\infty)$,
we have  $g'(t) \leq f'(t+\tilde z) \leq f'(t)$ by \eqref{e:lrYell},
for $t \in [\tilde{z}+r,\infty)$. It follows that
\begin{equation*}
g(x) - g(\tilde{z}+r) = \int_{\tilde{z}+r}^x g'(t) \, dt \leq \int_{\tilde{z}+r}^x f'(t) \, dt = f(x) - f(\tilde{z}+r),
\end{equation*}
which by $g(\tilde{z}+r) \leq f(\tilde{z}+r)$ implies the desired inequality $g(x) \leq f(x)$, now proven for $x \in (0,\infty)$.
The case $x <0$ is actually not required for the application, but true. Indeed, 
for $x \in (-\infty,0)$ it also holds that $g'(x) \leq f'(x)$ and it is clear that $f(x) \geq g(x)$ as $x \to 0^-$,
so it remains to check the inequality as $x \to -\infty$. As in the Coulomb case, this follows
from $k<0$, which follows from
\begin{equation*}
  \sigma k
  = \frac{1}{2} Y^\ell(z_0) - \sigma l_r(z_0-\tilde{z})
  < \frac{1}{2} Y^\ell(2r) - \sigma l_r(r) < 0.
\end{equation*}
This completes the proof.
\end{proof}

\begin{lemma} \label{lem:G-matching}
  Let $r \in (0,\frac12)$ and $\sigma \geq \sigma_0$ and $\ell \geq \ell_0$,
  where $\sigma_0$ and $\ell_0$ are sufficiently large absolute constants.
  Then for any $z_0 \in \C$ with $|z_0|<1-2r$, there exists a constant $k \in \R$ and $\tilde z \in \C$
  with $|\tilde z|<1-r$ on the line through $0$ and $z_0$
  such that 
  \begin{equation} \label{e:G-matching}
    \sigma \big(l_r(z_0-\tilde{z}) + k\big) = \pm \ell^2 I_0(|z_0|/\ell)
    \quad \textrm{and} \quad
    \sigma \big(l_r(z-\tilde{z}) + k\big) \leq \pm \ell^2 I_0(|z|/\ell) \textrm{ for all } z \in \D,
  \end{equation}
  where $\pm$ is either always $+$ or always $-$, and $I_0$ is a modified Bessel function of the first kind.
\end{lemma}

\begin{proof}
  Throughout the proof, $x\gg 1$ means that $x$ is larger than a large absolute constant.
  Let
  \begin{equation}
    I(z) = \ell^2 (I_0(|z|/\ell) - 1).
  \end{equation}
  Replacing $k$ by $k-\ell^2/\sigma$, the claim \eqref{e:G-matching} is equivalent to the claim
  \begin{equation} \label{e:G-matching-eq}
    \sigma \big(l_r(z_0-\tilde{z}) + k\big) = I(z_0)
    \quad \textrm{and} \quad
    \sigma \big(l_r(z-\tilde{z}) + k\big) \leq I(z) \textrm{ for all } z \in \D.
  \end{equation}
  For the right-hand side, for $\ell \gg 1$, we have
  \begin{equation} \label{e:Ideriv}
    I(z) = \frac14 |z|^2(1+\OO(|z|/\ell)), \quad
    \nabla I(z) = \pa{\frac12 + O(1/\ell)} z, \quad
    \nabla^2 I(z) = \frac12 {\b 1}_{2 \times 2} + \OO(1/\ell).
  \end{equation}
  For $\ell \gg 1$,
  the map $z\mapsto \sigma \nabla l_r(z)$ takes $B_r(0)$ onto
  $B_{\sigma |\nabla Y^\ell(r)|}(0) \supset B_{\sigma (1-\epsilon)/r}(0) \supset B_{1}(0)$.
  Thus, by appropriate choice of $\tilde z$ and $k$,
  the derivatives of $\sigma l_r(z-\tilde z)$ and $\pm I$ can be matched at any $|z_0|<1$.
  It remains to show the inequality in \eqref{e:G-matching}.
  By definition of $l_r$ and since, by \eqref{e:dYf}, the derivatives of $Y^\ell(z)$ are well approximated by those of $-\log |z|$ for $\ell \gg 1$,
  we have
  \begin{equation}
    \nabla^2 l_r(z) = - \frac{1}{r^2}({\b 1}_{2 \times 2}+\OO(1/\ell)) \quad \text{for $|z|<r$.}
  \end{equation}
  Together with \eqref{e:Ideriv},
  using that $1/r^2 > 1 > 1/2$, it follows that
  the function $l_r(z-\tilde z)+k$ stays below $\pm I(z)$
  for $|z-\tilde z|<r$, provided that $\ell \gg 1$.
  Using further that $l_r(0)-l_r(r) = \frac12 + \OO(1/\ell)$,
  we can choose $\sigma \geq \sigma_0$ and $\ell \geq \ell_0$ large enough that
  $$
  \sigma (l_r(0)-l_r(r)) > \frac14 (1+\OO(1/\ell)) = \sup_{\D} (\pm I) - \inf_{\D} (\pm I).
  $$
  Since $\sigma (l_r(0)+k) \leq \sup_{\D} (\pm I)$,
  it follows that $\sigma(l_r(z-\tilde z)+k) \leq \inf_{\D} (\pm I)$ for $|z-\tilde z|=r$. Since $l_r(z-\tilde z)$ is decreasing in
  $|z-\tilde z|$ the inequality then holds on all of $\D$.
\end{proof}

\subsection{Perturbed Yukawa equilibrium measure}
\label{sec:perteqmeas}

As in \cite{MR3694026}, to prove the local law,
we will condition on the particles outside small disks.
To handle this conditioning, we next
state adaptations of the results of \cite[Section~3.3]{MR3694026} to the Yukawa case.
As in \cite[Section~3.3]{MR3694026},
we can assume here that $S_V = \rho\overline{\D}$ for some $\rho>0$, where $\D \subset \C$ is the open unit disk.
Furthermore, we assume the density of $\mu_V$ is bounded below by $\frac{1}{4\pi} \alpha$
in $\rho\D$ for some parameter $\alpha>0$.
The class of perturbed potentials $W$ that we consider is as follows.
Let $\nu$ be
a positive measure with $\supp \nu \cap \rho\D = \varnothing$, $t > 0$ and
let $R \in \mathscr{C}(\rho\overline{\mathbb{D}})$ satisfy $(\Delta-m^2) R = 0$ in $\rho\D$.
Then $W$ is given by
\begin{equation} \label{e:Winnerrad}
W(z) = \begin{cases}
tV(z) + 2 U_\nu^\ell(z),
+ 2R(z), & z \in \rho\overline{\mathbb{D}},\\
\infty, & z \in \rho\mathbb{D}^*,
\end{cases}
\end{equation}
where we write $\mathbb{D}^* = \C \setminus \overline{\mathbb{D}}$ for the open complement of the unit disk.
Both perturbations $U_\nu^\ell$ and $R$ are $m$-harmonic inside $\rho\D$, i.e., $(\Delta-m^2)R=0$ and analogously for $U_\nu^\ell$.
In particular, by \eqref{e:EL}, this implies that the density of $\mu_W$
is equal to $t \mu_V + \text{constant}$ in $S_W$.
For $z \in \partial (\rho\D)$ we write $\bar n = \bar n(z) = z/|z|$ for the outer unit normal, and
we write $\partial_n^-f(z) = \lim_{\varepsilon \downarrow 0} \frac{f(z)-f(z-\varepsilon \bar n)}{\varepsilon} $
for the derivative in the direction $\bar n$ taken from inside $\rho\D$.

The next two propositions show that the bulk of the equilibrium measure $\mu_V$
is stable under suitable perturbations $W$ of the form \eqref{e:Winnerrad},
and that the density of $\mu_W$ on the boundary remains bounded.
To prove the stability of the bulk we use the obstacle problem characterization
\eqref{e:obstacle} of the support.

\begin{proposition} \label{prop:innerradius}
Suppose that $V$ and $W$ are as above \eqref{e:Winnerrad}.
Then, for any $\ell>0$, the support $S_W$ of the equilibrium measure
with Yukawa interaction of range $\ell$ and potential $W$ satisfies
\begin{equation}
\label{e:kappabd}
S_W \supset \left\{ z \in \rho\D: \dist(z, \rho\D^*) \geq \kappa \right\},
\quad
\text{where } \kappa = C \sqrt{\frac{\max(\|\nu\|,\rho \|\partial_n^- R\|_{\infty,\partial \rho\D} + (t-1))}{\alpha t}}.
\end{equation}
\end{proposition}

\begin{proof}
As in the proof of \cite[Proposition~3.3]{MR3694026}, except that we must now replace $\ell$ by $\ell/\rho$, we may assume that $\rho=1$,
and we define $D = \{z \in \D: \dist(z,\D^*) \geq \kappa\}$.
The replacement of $\ell$ does not matter since the estimate is uniform in $\ell$.
By Proposition~\ref{prop:obstacle}, to prove the proposition, it suffices to exhibit,
for any $z_0 \in D$, a test function $v_{z_0} = v=-U_\nu^\ell(z) + c$ with $v(z_0) = \frac12 W(z_0)$ and satisfying
the requirements for the potential in \eqref{e:obstacle} with $W$ instead of $V$.

This test function is chosen almost exactly as in the Coulomb case, with the small difference
in the handling of the perturbation $R$.
Indeed, recall that by assumption $R= U_\mu^\ell$ for a (signed) charge distribution $\mu$ supported in $\D^*$.
Up to an additive constant, we may replace $\mu$ by its \emph{balayage} $\omega$ onto $\partial \D$,
i.e., we choose the measure $\omega$ supported on $\partial \D$ such that $R = U_{\omega}^\ell + c$ in $\D$.
The existence of $\omega$ follows as in the Coulomb case; see e.g.\ \cite{MR1485778}.
We choose $\ell_0$ to be the sufficiently large absolute constant from Lemma~\ref{lem:G-matching}.
For $\ell \geq \ell_0$,
we decompose $\omega = \omega_0 + \omega_+ - \omega_-$ with $\omega_0$ a measure of constant density
with respect to the arclength measure on $\partial \D$
such that $\int \rd\omega = \int \rd  \omega_0$ and with $\omega_\pm$ positive measures.
For $\ell < \ell_0$, we simply decompose $\omega = \omega_+ - \omega_-$ 
with $\omega_\pm$ positive measures and set $\omega_0=0$.
In both cases, Lemma~\ref{lem:domegacircle} implies that the total charge of $\omega_\pm$ is estimated by
\begin{equation} \label{e:balayagecharge}
  \|\omega_\pm\| = \OO(1) \|\partial^-_n R\|_{\infty, \partial \D}.
\end{equation}
Then, similarly as in \cite[Proposition~3.3]{MR3694026}, we will choose the function $v$ of the form
\begin{equation} \label{e:vdef}
  v(z)
  =
  t u_{V,\ell}(z)
  + \sigma L(z)
  + \gamma L_0(z)
  - U^{\omega_-}(z)
  ,
  \quad L(z) = \int \left( l_r\big(z-\tilde{z}(w)\big) + k(w) \right) (\nu+\omega_+)(\rd w)
,
\end{equation}
where $\sigma > 0$, $r > 0$, 
$k: \supp \nu \to \R$ and $\tilde{z}: \supp \nu \to \mathbb{D}$ are parameters,
and the function $l_r$ 
is now defined by \eqref{e:ldef}, and $L_0(z)$ is chosen of the form
\begin{equation*}
  L_0(z) = l_r(z-\tilde z_0)-k_0
\end{equation*}
 for some $\tilde z_0 \in \C$ and $k_0\in \R$ to be chosen later.

\smallskip
\noindent
{\it Step 1.}
With the choice
\begin{equation*}
\gamma = \OO(1) \|\partial_n^- R\|_{\infty,\partial \mathbb{D}}, \quad
\sigma = \max\left(\frac12,\frac{(t-1)-\gamma+\|\omega_-\|}{\|\nu+\omega_+\|}\right),
\quad
r = 2 \sqrt{\frac{\|\nu+\omega_+\| \sigma+\gamma}{\alpha t}} = \frac12 \kappa,
\end{equation*}
the function $v$ is of the form $-U_\mu^\ell +c$ for a positive measure $\mu$ of total mass at most
$t+ \|\omega_-\| -\gamma-\sigma\|\nu+\omega_+\| \leq 1$.
Indeed,
by definition, $-t u_{V,\ell} + U_{\omega_-}$ is the potential of a positive measure of mass $t+\|\omega_-\|$
and $-\sigma L-\gamma L_0$ is the potential
of a negative measure of total mass $-\sigma \|\nu+\omega_+\| - \gamma$.
Their sum is the potential of a positive measure since
\begin{equation} \label{e:vpos}
  (\Delta-m^2) (tu_{V,\ell} - U^{\omega_-} +\sigma L +\gamma L_0)
  \geq 2\pi t \rho_{V,\ell} + 2\pi \omega_- - \frac{2\sigma}{r^2}\|\nu+\omega_+\| - \frac{2\gamma}{r^2} \geq 0,
\end{equation}
where we used the assumption $\rho_{V,\ell} \geq \alpha/(4\pi)$.

\bigskip
\noindent
{\it Step 2.}
For appropriate choice of the parameters $\tilde z$ and $k$ (depending on $z_0$),
we have $v(z_0) = \frac12 W(z_0)$ and $v \leq \frac12 W$ in $\overline{\D}$.
Indeed,
replacing \cite[Lemma~3.6]{MR3694026} by Lemma~\ref{lem:log-matching} stated below the proof,
we choose the parameters $\tilde z$ and $k$ exactly as in the proof of \cite[Proposition~3.3]{MR3694026}
to achieve
\begin{align}
\label{e:Lz}
\sigma L(z) 
&\leq \frac{1}{2} \int Y^\ell(z-w) \, (\nu+\omega_+)(dw)
\quad \textrm{for all } z \in \overline{\D},
\\
\label{e:Lz0}
\sigma L(z_0)
&= \frac{1}{2} \int Y^\ell(z_0-w) \, (\nu+ \omega_+)(dw).
\end{align}
This concludes the proof for $\ell < \ell_0$. For $\ell\geq \ell_0$,
it remains to handle the remaining part of the perturbation, which is the potential $U^\ell_{\bar\omega}$
generated by the constant part $\bar\omega$ of $\omega$.
Since the Yukawa potential of $\bar\omega$ is $m$-harmonic in $|z|<1$,
radially symmetric and bounded as $|z|\to 0$,
as in \eqref{e:ODEBessel}, it is explicitly given inside $\D$ by
\begin{equation*}
  U^\ell_{\bar \omega}(z) = \pm A \ell^2 I_0(|z|/\ell) \quad (|z| < 1),
\end{equation*}
for some constant $A>0$ depending on $\ell$ and $\bar\omega$,
where $I_n$ are the modified Bessel functions of the first kind.
Using that $I_0' = I_1$ by general relations between Bessel functions,
\begin{equation*}
  \nabla U^\ell_{\bar \omega}(z)
  = \pm A 2 \ell I_1(|z|/\ell) \frac{z}{|z|}
  = \partial_n^- U^\ell_{\bar \omega}(1) \frac{I_1(|z|/\ell)}{I_1(1/\ell)} \frac{z}{|z|}.
\end{equation*}
The modified Bessel functions satisfy the asymptotics
\begin{gather}
  I_0(t) \sim 1+\frac14 t^2, \quad
  I_1(t) \sim \frac12 t, \quad \text{as $t\to 0$}.
\end{gather}
Therefore, with \eqref{e:Ubaromega}, the constant $A$ is given by
\begin{equation}
  A
  = \pm \frac{\partial_n^-U_{\bar \omega}^\ell(1)}{2 \ell I_1(1/\ell)}
  = \pm (1+\OO(1/\ell)) \partial_n^-U_{\bar \omega}^\ell(1)
  \leq (1+\OO(1/\ell)) \|\partial_n^-U_{\omega}^\ell\|_{\infty,\partial\D}
  = \OO(1) \|\partial_n^-R\|_{\infty,\partial\D}
  .
\end{equation}
By Lemma~\ref{lem:G-matching}, there exists a large constant $\sigma$ such that
we can choose $k_0$ and $\tilde z_0$ and $\gamma = \OO(1) \|\partial_n^- R\|_{\infty,\partial \D}$ such that
with $\gamma = \sigma A$, 
\begin{equation}
  \gamma L_0(z_0) = U_{\bar\omega}^\ell(z_0),
  \quad
  \gamma L_0(z) \leq U_{\bar\omega}^\ell(z) \quad \text{for all $z \in \D$}.
\end{equation}
This concludes the proof.
\end{proof}

\begin{proposition} \label{prop:bddensity}
Suppose that $V$ and $W$ are as above \eqref{e:Winnerrad} and assume in addition
that $\mu_V$ is absolutely continuous with respect to the 2-dimensional Lebesgue measure.
Then $\mu_W = \mu + \eta$, where $\mu$ is absolutely continuous with respect to $\mu_V$, and $\eta$ absolutely continuous
with respect to the arclength measure $s$ on $\partial \rho\mathbb{D}$ with the Radon--Nikodym derivative bounded by
\begin{equation} \label{e:etabd}
\rho
\normB{\frac{d \eta}{ds}}_\infty 
\leq 
C \pB{ \|\eta\| + \|\nu\| + 2\rho\|\partial_n^- R\|_{\infty,\partial \rho\D} + |1-t|\rho\|\partial_n^- V\|_{\infty,\partial\rho\D} }.
\end{equation}
\end{proposition}

\begin{proof}
The only change in the proof of Proposition~\ref{prop:bddensity} compared to
\cite{MR3694026} is the change of the logarithmic potentials to Yukawa potentials.
In particular, the formula \eqref{e:dUeta} holds and $\nabla Y^\ell(z)$ is proportional
to $\nabla \log \frac{1}{|z|}$.
\end{proof}

\subsection{One-step estimate for the Yukawa interaction}

As in \cite[Proposition~4.1]{MR3694026}, we use a simple mean-field partition function estimate
to obtain a bound on the fluctuations of smooth linear statistics.
In the following, $\rd m$ denotes the Lebesgue measure and is not related to the mass $m$.

\begin{proposition} \label{prop:Ystep}
  Let $\Sigma = \Sigma_W$ be a smooth domain 
  { with boundary $\partial \Sigma$
    or $\Sigma =\T$ (with $\partial \Sigma=\emptyset$).}
Given a potential $W \in C^{1,1}_{loc}(\Sigma_W)$ possibly depending on the number of particles $M$,
assume that there exist $u: \Sigma_W \to \R_+$ and $v: \partial \Sigma_W \to \R_+$
{ (if $\partial\Sigma_W \neq \emptyset$)}
such that
$d\mu_W = u \, \rd m + v \, \rd s$, where $\rd m$ is the 2-dimensional Lebesgue measure and $\rd s$
is the arclength measure on $\partial \Sigma_W$ { (if $\partial\Sigma_W \neq \emptyset$)}.
Assume the conditions (i)-(iv) as stated in \cite[Proposition~4.1]{MR3694026}
but replace the bounds on $\frac{1}{4\pi} \Delta W$ (which is the density in the Coulomb case)
{ more generally} by the same bound on the density { of the equilibrium measure} $u$
and also modify the assumption (iv) by replacing $\zeta$ by
$\zeta^\ell = U_{\mu_W}^\ell + \frac12 V - c_V$, where the constant $c_V$ is the one in \eqref{e:EL}.
Then, for any constant $A$, for any bounded $f \in C^2(\C)$ with compactly supported $(\Delta-m^2) f$,
\begin{align} \label{freeY}
\log \int \ee^{-\beta H_{M,W}(\b z) + \sum f(z_j)} \, m(\rd \b z)
&\leq -\beta M^2 I_W^\ell(\mu_V)
+ M (f,\mu_W)
+ \tfrac{1}{8\pi \beta}(f,-(\Delta-m^2)f)
\nonumber\\
&\qquad
+ \OO(M^{-A}) \|\Delta f\|_\infty
+ \OO(M \log M),
\\
\label{freeY2}
\log \int \ee^{-\beta H_{M,W}(\b z)}  \, m(\rd \b z)
&\geq -\beta M^2 I_W^\ell(\mu_V) + \OO(M \log M),
\end{align}
and consequently for any $\xi \geq 1+1/\beta$,
\begin{equation} \label{e:Ystep}
  \Big |\sum_j f(z_j) - M \int f \, \rd\mu_W^\ell \Big |
  = \OO(\xi)
  \pa{
    \sqrt{M \log M}(f, (-\Delta+m^2)f)^{1/2}
    + M^{-A} \|\Delta f\|_\infty
  }
  ,
\end{equation}
with probability at least $1-\ee^{-\xi \beta M \log M}$,
with the implicit constant depending only on the numbers $A$ in the assumptions (i)-(iv).
\end{proposition}

\begin{proof}
The probability estimate is obtained as in \cite{MR3694026}
from the partition function bounds \eqref{freeY} and \eqref{freeY2},
which are analogous to \cite[Lemmas~4.3~and~4.4]{MR3694026}
except that $\|\nabla f\|_2 = (f,(-\Delta)f)^{1/2}$ is replaced by $(f,(-\Delta+m^2)f)^{1/2}$.
The lower bound can be proved exactly the same way;
for the upper bound we may bound the energy slightly differently from below, as follows,
avoiding the need that the support of $(\Delta -m^2)f$ is contained in $S_V$.

All the properties of the Coulomb potential used in the proof of \cite[Lemmas~4.3]{MR3694026}
also hold for the Yukawa potential { and on the torus}.
Replacing the point charges by charged disks of radius $\varepsilon$,
and denoting by $D^\ell(\cdot,\cdot)$ the Yukawa analog of $D(\cdot,\cdot)$,
we get the bound
\begin{multline*}
 H_M^\ell ({\bf z}) - \frac{1}{\beta M} \sum_j f(z_j)  \geq M^2 D^\ell(\hat \mu^{(\varepsilon)}, \hat \mu^{(\varepsilon)}) + M^2 (W - \tfrac{1}{\beta M} f, \hat \mu) + \OO(M \log \frac{1}{\varepsilon}) \\
= M^2 \left( D^\ell(\hat \mu^{(\varepsilon)},\hat \mu^{(\varepsilon)}) + (W,\hat \mu) - (\tfrac{1}{\beta M} f, \hat \mu^{(\varepsilon)}) \right) + M^2 ( \tfrac{1}{\beta M} f, \hat \mu^{(\varepsilon)} - \hat \mu ) + \OO(M \log \frac{1}{\varepsilon}).
\end{multline*}
Writing
$$
D^\ell(\hat \mu^{(\varepsilon)},\hat \mu^{(\varepsilon)}) = D^\ell(\mu_W, \mu_W) + 2 D^\ell(\mu_W, \hat \mu^{(\varepsilon)} - \mu_W) + D^\ell(\hat \mu^{(\varepsilon)} - \mu_W,\hat \mu^{(\varepsilon)} - \mu_W)
$$
and further using the Euler--Lagrange equation \eqref{e:EL} to write
$$
2 D^\ell(\mu_W, \hat \mu^{(\varepsilon)} - \mu_W) + (W,\hat \mu) = (W,\mu_W) + 2 (\zeta^\ell, \hat \mu - \mu_W) + 2 (U_{\mu_W}^\ell, \hat \mu^{(\varepsilon)} - \hat \mu),
$$
where $\zeta^\ell = U_{\mu_W}^\ell + \tfrac{1}{2} W - c_W = 0$ on $S_W$, 
we therefore can the bound $H_M^\ell ({\bf z}) - \frac{1}{\beta M} \sum_j f(z_j)$ by
\begin{multline*}
 M^2 \left( I_W^\ell(\mu_W) + D^\ell(\hat \mu^{(\varepsilon)} - \mu_W,\hat \mu^{(\varepsilon)} - \mu_W) - (\tfrac{1}{\beta M} f, \hat \mu^{(\varepsilon)}) \right) + 2 M^2 (\zeta^\ell, \hat \mu - \mu_W) \\
+ M^2 ( \tfrac{1}{\beta M} f, \hat \mu^{(\varepsilon)} - \hat \mu ) + 2 M^2 ( U_{\mu_W}^\ell, \hat \mu^{(\varepsilon)} - \hat \mu ) + \OO(M \log \varepsilon^{-1}).
\end{multline*}
We write
$$
D^\ell(\hat \mu^{(\varepsilon)} - \mu_W,\hat \mu^{(\varepsilon)} - \mu_W) - (\tfrac{1}{\beta M} f, \hat \mu^{(\varepsilon)}) 
 =  \tfrac{1}{2\pi} (\tfrac{1}{\beta M} f + U_{\hat \mu^{(\varepsilon)}-\mu_W}^\ell, -(\Delta-m^2) U_{\hat \mu^{(\varepsilon)}-\mu_W}^\ell)- (\tfrac{1}{\beta M} f, \mu_W).
$$
The Yukawa potentials decay exponentially at infinity, so we may integrate by parts and use the elementary inequality $-|ab|+|b|^2 \geq -|a|^2/4$ to get
\begin{align*}
\tfrac{1}{2\pi} (\tfrac{1}{\beta M} f + U_{\hat \mu^{(\varepsilon)}-\mu_W}^\ell, (-\Delta) U_{\hat \mu^{(\varepsilon)}-\mu_W}^\ell)
& = \tfrac{1}{2\pi} (\tfrac{1}{\beta M} \nabla f + \nabla U_{\hat \mu^{(\varepsilon)}-\mu_W}^\ell, \nabla U_{\hat \mu^{(\varepsilon)}-\mu_W}^\ell) \\
& \geq - \tfrac{1}{8 \pi \beta^2 M^2} (\nabla f, \nabla f) = - \tfrac{1}{8 \pi \beta^2 M^2} (f, (-\Delta) f).
\end{align*}
By the same inequality we have
\begin{align*}
\tfrac{1}{2\pi} (\tfrac{1}{\beta M} f + U_{\hat \mu^{(\varepsilon)}-\mu_W}^\ell, m^2 U_{\hat \mu^{(\varepsilon)}-\mu_W}^\ell)
& \geq - \tfrac{m^2}{8 \pi \beta^2 M^2} (f, f).
\end{align*}
In conclusion,
\begin{align*}
& M^2 D^\ell(\hat \mu^{(\varepsilon)}, \hat \mu^{(\varepsilon)}) + M^2 (W - \tfrac{1}{\beta M} f, \hat \mu) \\
& \geq M^2 \left( I_W^\ell(\mu_W) - \tfrac{1}{\beta M} (f, \mu_W) - \tfrac{1}{8\pi \beta^2 M^2}(f,-(\Delta-m^2)f) \right) 
\\
& \quad
+ 2 M^2 (\zeta^\ell, \hat \mu - \mu_W)
+ M^2 ( \tfrac{1}{\beta M} f, \hat \mu^{(\varepsilon)} - \hat \mu ) + 2 M^2 ( U_{\mu_W}^\ell, \hat \mu^{(\varepsilon)} - \hat \mu ) + \OO(M \log \frac{1}{\varepsilon}).
\end{align*}
In the same way as in \cite{MR3694026},
for the error terms on the last line,
\begin{equation*}
  \frac{M}{\beta} |(f^{(\varepsilon)}-f,\hat\mu)|
  \leq \frac{M}{\beta} C \varepsilon^2 \|\Delta f\|_\infty
  \leq M^{-A} \|\Delta f\|_\infty
\end{equation*}
and
\begin{equation*}
2M^2 |(U_{\mu_{W}^\ell},\hatmueps-\hat\mu)|
\leq
C \varepsilon^2 M^{A_u} + C \sqrt{\varepsilon} M^{A_v}
\leq 
1,
\end{equation*}
by choosing $\epsilon$ sufficiently small depending on $A$ and such that $\log \frac{1}{\epsilon} = \OO(\log M)$.
Finally, we use that $2 M^2 (\zeta^\ell, \hat \mu - \mu_W) \geq 0$ by the Euler--Lagrange equation
to conclude the proof.
\end{proof}

\begin{remark} \label{rem:meanzero}
For test functions $f$ supported in $S_V^\ell$ and satisfying the condition $\int f \, \rd m = 0$,
\begin{equation*}
  \int f \, \rd\mu_V^\ell = \int f  \frac{1}{4\pi}(\Delta V - m^2 V) \, \rd m.
\end{equation*}
(We recall that $\rd m$ is the Lebesgue measure and not related to the mass $m$.)
Consequently, if $V$ is replaced by $V+R$ with $(\Delta -m^2)R=0$, and assuming that $f$ is supported in
the intersection of the supports of the equilibrium measures of $V$ and $V+R$,
and that $\int f\, \rd m =0$,
we have
\begin{equation*}
  \int f \, \rd\mu_V^\ell = \int f \, \rd\mu_{V+R}^\ell
  .
\end{equation*}
Since we are ultimately interested in test functions without the condition $\int f \, \rd m=0$,
some additional care is required.
(The condition was not necessary in the Coulomb case in \cite{MR3694026}.)
This problem will be addressed at the beginning of the proof of Proposition \ref{prop:Yind}.
\end{remark}

\subsection{Yukawa gas on the torus: proof of Theorem~\ref{thm:YTdensity}}

We follow the proof of \cite[Theorem~1.1]{MR3694026}
to improve the estimate of Proposition~\ref{prop:Ystep} to the stronger one asserted by
Theorem~\ref{thm:YTdensity} by using local conditioning.
Compared with \cite[Theorem~1.1]{MR3694026}, there are two main changes in Theorem~\ref{thm:YTdensity}: 
(i)  the domain is now a torus rather than the plane, (ii) the interaction is the Yukawa potential rather than the Coulomb potential.
The domain change is only visible in the first step of the induction in the proof; it does not have any effect after the first step when we take  local conditioning. 
The change from the Coulomb potential to the Yukawa potential does
require changes in the local conditioning; it will be  taken into account by replacing the
potential theory estimates  in \cite{MR3694026} by their  generalizations
in Sections~\ref{sec:yukawa-potential}--\ref{sec:perteqmeas}.

First, we note that \cite[Section 5]{MR3694026} applies without changes
except that the Coulomb potential $\log 1/|z|$ is replaced by the Yukawa potential $Y^\ell(z)$
in all expressions, and with the additional condition that $\int f \, \rd m=0$ in the
assumption of \cite[Proposition 5.3]{MR3694026}. This condition is necessary because, with the
$m$-harmonic perturbation $V_o$, inside the support of $\mu_W$ we now have 
\begin{equation*}
  \mu_W= \frac{N}{M} \mu_V+ \text{const.}
\end{equation*}
by \eqref{e:equilibriumdensity}.
As explained in Remark \ref{rem:meanzero},
the additional constant has no effect if both sides are integrated against a test function $f$
with support in the support of $\mu_W$ that satisfies $\int f \, \rd m = 0$.

Next, we adapt \cite[Section 6]{MR3694026} to the Yukawa case.
Here two modifications are required. First, the scaling of the Yukawa gas is different, which leads to
a different recursion of scales. Second, in the case of the Yukawa gas, as noted above,
the density of the equilibrium
is only stable under $m$-harmonic perturbations up to a constant, and thus a small extra argument is required
to remove the mean zero condition.

As previously, we write $\ell =N^{-1/2+\delta}$
for the range of the Yukawa potential.
Given $\varepsilon>0$ (and assuming $\varepsilon<\delta$), we set $s_0=0$ and
\begin{equation*}
  s_{j+1} = \left(\Big ( \frac14+\frac{s_j}{2}\Big ) \wedge \pa {s_j+\delta}\right) - \varepsilon,
\end{equation*}
for $\varepsilon >0$ fixed sufficiently small.
As long as the second term in the minimum above dominates,
the sequence $s_j$ grows linearly as $j(\delta-\varepsilon)$ until the scale $s=\frac12-2\delta$
is reached.
After that, the first term dominates. Then
$s_j$ evolves according to $\frac12-\delta-\epsilon$; then $\frac12 - \frac12 \delta$;
then $\frac12 - \frac14 \delta - \frac 32 \varepsilon$ and converges geometrically to
$\frac12 - 2\varepsilon$.
In particular, given $s \in (0,\frac12)$, 
we can fix $\varepsilon>0$
and $n<\infty$ such that $s_n = s$, and we will assume such a choice from now on.

The induction assumption (A$_{r}$) is modified as follows (as a formal remark, note that
compared to \cite{MR3694026}, we changed the index of the condition A$_{t}$ into A$_{r}$ as, in the current paper, $t$ refers to the argument of the Laplace transform).

\medskip
\noindent \textbf{Assumption (A$_{r}$).}
For any bounded $f \in C^2({ \T})$ with $\supp(\Delta-m^2) f \subset B_r^\circ \cap S_V$, we have
\begin{equation} \label{e:Ystep-prev}
  \Big |\frac{1}{N} \sum_j f(z_j) - \int f \, \rd\mu_V\Big |
  \prec
    N^{-\frac12-r} (f,(-\Delta+m^2)f)^{\frac12}
    +N^{-1-2r} \|\Delta f\|_\infty.
\end{equation}

\begin{proposition} \label{prop:Yind}
For arbitrary $\varepsilon>0$, (A$_r$) implies (A$_s$)
for any $0 \leq r \leq s \leq (\frac14 + \frac12 r) \wedge (r+\delta) - \varepsilon$
(with the implicit constants depending on $\epsilon$).
\end{proposition}

\begin{proof}
First, we show that, for any $s$ as asserted in the proposition,
it suffices to prove that (A$_r$) implies (A$_s'$), where (A$_s'$) is defined
exactly as (A$_s$) except that the test functions $f$ are required to obey
the additional mean zero condition $\int f\, \rd m=0$. Indeed, assume (A$_r$) and that
we have proved (A$_s'$) for all $s$ as in the statement of the proposition.
Recall from above that $B = B_s$ is a disk of radius $N^{-s}$
and that $B_s^\circ$ the disk with the same center and half the radius.
For any test function $f$ supported on $B_s^\circ$ we define $f_i(z) = 2^{-2i} f(2^{-i}z)$,
and write
\begin{equation*}
  f = f_k + \sum_{i=0}^{k-1} (f_i-f_{i+1}),
\end{equation*}
where $k$ is the largest integer such that $2^k N^{-s} \leq N^{-t}$. Then
\begin{equation}
\|\Delta f_i\|_\infty = 2^{-4i} \|\Delta f\|_\infty, \quad
(f_i, (-\Delta + m^2) f_i) \leq 2^{-2i} (f,(-\Delta+m^2)f).
\end{equation}
Therefore, with $s_i = s - (i+1)/\log_2 N$ for $i=0,1,\dots,k-1$,
applying (A$_{s_i}'$) to the mean zero function $f_i-f_{i+1}$, we obtain
\begin{align*}
&  \frac{1}{N} \sum_j (f_i(z_j)-f_{i+1}(z_j)) - \int (f_i - f_{i+1}) \, \rd\mu_V \\
& \prec { 2^{2i} N^{-1-2s} \| \Delta(f_i-f_{i+1}) \|_\infty
   + 2^i N^{-\frac12-s} (f_i-f_{i+1}, (-\Delta+m^2)(f_i-f_{i+1}))^{1/2} } \\
& \prec { 2^{-2i} N^{-1-2s} \| \Delta f \|_\infty + N^{-\frac12-s} (f,(-\Delta+m^2)f)^{1/2} }.
\end{align*}
Similarly, applying (A$_r$) to $f_k$, we have
\begin{align*}
 \frac{1}{N} \sum_j f_k(z_j) - \int f_k \, \rd \mu_V &
\prec  \pb{ N^{-1-2r} \| \Delta f_k \|_\infty + N^{-\frac12-r} (f_k, (-\Delta+m^2) f_k)^{1/2} } \\
& \prec \pb{ 2^{-4k} N^{-1-2r} \| \Delta f \|_\infty + 2^{-k} N^{-\frac12-r} (f,(-\Delta+m^2)f)^{1/2} } \\
& \prec \pb{ 2^{-4k} N^{-1-2s} \| \Delta f \|_\infty + 2^{-k} N^{-\frac12-s} (f,(-\Delta+m^2)f)^{1/2} }.
\end{align*}
Then
\begin{align*}
  \frac{1}{N} \sum_j f(z_j) - \int f 
 &= \frac{1}{N} \sum_j \pbb{ f_k(z_j) + \sum_{i=0}^{k-1} (f_i(z_j)-f_{i+1}(z_j)) } - \int \pbb{ f_k + \sum_{i=0}^{k-1} (f_i-f_{i+1}) }\\
& \prec \sum_{i=0}^k \pb{ 2^{-2i} N^{-1-2s} \| \Delta f \|_\infty + N^{-\frac12-s} (f, (-\Delta+m^2)f)^{1/2} } \\
& \prec {N^{-1-2s} \| \Delta f \|_\infty + N^{-\frac12-s} (f,(-\Delta+m^2)f)^{1/2}}.
\end{align*}

It remains to prove that (A$_r$) implies (A$_s'$) for $s$ as in the statement of the proposition.
This proof proceeds exactly as in \cite[Section~6.1]{MR3694026}, with the only essential changes
in \cite[Lemmas~6.2--6.3]{MR3694026}, since now $m^2>0$ in \eqref{e:Ystep-prev}.
Indeed,
the required properties of the conditional equilibrium measure follow from
Propositions~\ref{prop:innerradius}--\ref{prop:bddensity},
as soon as \cite[Lemmas~6.2--6.3]{MR3694026} are adapted.

In \cite[Lemma~6.2]{MR3694026}, which states that
$\tau = 1+\OO(N^{-c\varepsilon})$
(where we recall that $\tau=\frac{N}{M}\mu_V(B)$)
and $\nu(\C) = \OO(N^{-c\varepsilon})$, with high probability,
the following changes are necessary.
Recall that $\chi_\pm$ are smooth cutoff functions with
\begin{equation*}
  \chi_+|_{B} = 1, \quad \chi_+|_{B_+^c} = 0, \qquad
  \chi_-|_{B^c} = 0, \quad \chi_-|_{B_-} = 1,
\end{equation*}
obeying $\|\nabla^k \chi_\pm\|_\infty = \OO(N^{k s}/\eta^k)$ for $k=0,1,2$
(see \cite{MR3694026} for the definitions of the expressions).
We replace the estimates on $(\chi_\pm, -\Delta \chi_\pm)$ by
\begin{equation*}
  (\chi_{\pm},(-\Delta+m^2) \chi_{\pm})
  =
  \OO(\eta N^{-2s} N^{2s}/\eta^2)
  + \OO(N^{1-2\delta} N^{-2s})
  = \OO(1/\eta) + \OO(N^{1-2\delta-2s})
  ,
\end{equation*}
and thus
\begin{equation*}
  N^{-1-2r} (\chi_{\pm},(-\Delta+m^2) \chi_{\pm})
  = \OO(N^{-4s-4\varepsilon}/\eta) + \OO(N^{-4s-2\varepsilon})
  = \OO(N^{-4s-c\varepsilon})
  .
\end{equation*}
Using this, the rest of the proof of \cite[Lemma~6.2]{MR3694026} proceeds as in \cite{MR3694026}.

In \cite[Lemma~6.3]{MR3694026},
which states the estimate $N^{-s}\|\nabla \hat R\|_{L^\infty(B)} = \OO(N^{-c\varepsilon})$, with high probability,
we make the following changes.
We change the definition of $f$
from $f(w)=N^{-s}\nabla(\psi(w)\log \frac{1}{|z-w|})$
to $f(w)=N^{-s}\nabla(\psi(w) Y^\ell(z-w))$.
In particular, the property that $\Delta f = 0$ on $A^c$
is replaced by $(\Delta-m^2) f = 0$ on $A^c$,
and using this, the estimate on $(f,-\Delta f)$ is replaced by (here again we use a notation from \cite{MR3694026}, namely $a=N^{-c\varepsilon}$),
\begin{align*}
   & N^{-1-2r} (f,(-\Delta+m^2) f) \\
   & = N^{-1-2r} \OO(N^{-2s} N^{2s}|\log a|/a^{2}) + N^{-1-2r}\OO(N^{1-2\delta}N^{2s}|\log a|^2) 
   = \OO(N^{-4s-c\varepsilon}),
\end{align*}
so that, again, the rest of the proof of \cite[Lemma~6.3]{MR3694026} proceeds as in \cite{MR3694026}.
\end{proof}

\begin{proof}[Proof of Theorem~\ref{thm:YTdensity}]
Proposition~\ref{prop:Ystep} applied to the torus $\Sigma=\T$
and with $M=N$ verifies Assumption~(A$_0$).
We then apply local conditioning, exactly as in the proof of  \cite[Theorem~1.1]{MR3694026}.
For the conditioned measure, since $\ell \leq N^{-c}$, 
we may replace the torus Yukawa potential by
the full plane Yukawa potential since
\begin{equation*}
  H_{N,0}(z) = \sum_{j\neq k} U^{\ell}(z_j-z_k) + \Oinfty
  = \sum_{j\neq k} Y^\ell(z_j-z_k) + \Oinfty
\end{equation*}
with error bound uniform in $z \in \T^N$.
By inductive application of Proposition~\ref{prop:Yind}, the assumption (A$_s$) is verified
for all $s \in (0,\frac12)$. This completes the proof.
\end{proof}

\subsection{Coulomb gas on the plane: proof of Theorem~\ref{thm:Cdensity}}

Theorem~\ref{thm:Cdensity} is generalization 
 of \cite[Theorem~1.1]{MR3694026} 
in the following three ways:
(i)
the distance of the support of the test function to the boundary of the support
of the equilibrium measure can be $\gg N^{-1/4} + t^{1/4}$ rather than order $1$;
(ii) the Coulomb potential can be replaced by the perturbed Coulomb potential;  
(iii) or  replaced by the Yukawa potential $Y^\ell$ with $\ell \geq N^2$.
We will show that all these changes  have only  minor effects on the proof.
The condition $\gg N^{-1/4} + t^{1/4}$ arises  because $ N^{-1/4} + t^{1/4}$ is the scale that can be controlled without induction.
Indeed, the requirement of distance $\gg N^{-1/4}$ was already implicit in \cite{MR3694026}, but  the distance requirement was simply 
 estimated crudely  by order $1$ there.   When the perturbation is present, i.e., $t \neq 0$, there is   an additional error term 
which  leads to the  condition $\gg t^{1/4}$; see below.

We begin with the condition on the distance to the boundary.
In the proof of \cite[Theorem~1.1]{MR3694026},
in \cite[Section~6]{MR3694026},
by replacing $V(z)$ by $V(z-z_0)$ for some fixed $z_0 \in S_V$,
it was sufficient to restrict the induction to functions supported in the centered balls
$B_s^o = B(0,\frac12 N^{-s}) \subset B_s = B(0,N^{-s})$.
For test functions whose support has distance $ \gg N^{-1/4} + t^{1/4}$ to the boundary
of the support of the equilibrium measure,
we now choose $z_0$ to be $N$-dependent points
with $\dist(z_0,S_V^c)  \gg N^{-1/4} + t^{1/4}$.
This requires no changes in the proof because the initial estimate
(here given by Proposition~\ref{prop:CAstep} below) has no restriction on the
support of the test function $f$.
Writing $t=N^{-2a}$,
in the first inductive step, we can choose the scale as $N^{-s_1}$ with $s_1=(1/4 \wedge a/2)-\epsilon$.
By assumption the ball of this radius centered at $z_0$ is contained in the support
of the equilibrium measure and has density bounded below there.
Hence there is no change in the remaining steps.
Thus the condition $\gg N^{-1/4} + t^{1/4}$ arises  because $ N^{-1/4} + t^{1/4}$ is  the scale that 
the density in that scale can be controlled without induction.

As a preliminary step towards Theorem~\ref{thm:Cdensity},
we prove the following estimate,
which provides a weaker fluctuation bound than asserted in Theorem~\ref{thm:Cdensity}.
However, once this bound is established for all scales,
Theorem~\ref{thm:Cdensity} then follows from the same estimates.

\begin{proposition} \label{prop:CAdensity-first}
Assume the same conditions as in Theorem~\ref{thm:Cdensity}.
Write $t  = N^{-2a}$ and suppose that $\supp f$ has diameter at most $N^{-s}$. Then
\begin{equation*}
  \frac{X_f}{N} \prec 
  {(N^{-\frac12 -s} + N^{-a-2s}) \|\nabla f\|_2 + (N^{-1-2s} + N^{-2a-4s}) \|\Delta f\|_\infty}
  .
\end{equation*}
\end{proposition}

To prove this bound, we proceed as in the proof of \cite[Theorem~1.1]{MR3694026}.
The first ingredient is the following generalization of the one-step estimate
\cite[Proposition~4.1]{MR3694026}.

\begin{proposition} \label{prop:CAstep}
Assume that the potential $W$ and the number of particles $M$
satisfy the assumptions of \cite[Proposition~4.1]{MR3694026}.
Consider the probability measure on $\Sigma_W^M$ with density proportional to $\ee^{-\beta H(\b z)}$
where we assume that  for some constant $K$ the Hamiltonian $H :\Sigma_W^M \to \R$ satisfies the uniform estimate  
\begin{equation} \label{e:HbdK}
  |H(\b z) - H_{M,W}^{\cal C}(\b z)| \leq tMK.
\end{equation}
Then for any bounded $f \in C^2(\C)$ with $\supp \Delta f$ compact,
\begin{equation} \label{e:step}
  \sum_j f(z_j) - M \int f \, \rd\mu_W = \OO(\xi) \pa{
   \left( t M K + M \log M \right)^{1/2}  \|\nabla f\|_2
   +  M^{-A} \|\Delta f\|_\infty 
 }
\end{equation}
with probability at least $1 - \ee^{-\xi \beta (t M K + M \log M)}$ for any  $\xi \geq 1+1/\beta$.
The same estimate holds for the Yukawa gas with $\ell \geq M^{2}$.
\end{proposition}

\begin{proof}[Proof of Proposition \ref{prop:CAstep}]
The proof of the proposition is completely parallel to the one without the perturbation $\tilde G$,
only with an additional error term from \eqref{e:HbdK}.
Namely, by the assumption \eqref{e:HbdK}, we may trivially estimate
\begin{equation*}
\log \int \ee^{-\beta H_{M,W}^{\cal C}} \, m^{\otimes M}(\rd \b z) - t M K 
\leq \log \int \ee^{-\beta H} \,  m^{\otimes M}(\rd \b z)
\leq \log \int \ee^{-\beta H_{M,W}^{\cal C}} \, m^{\otimes M}(\rd \b z)+ t M K 
.
\end{equation*}
By \cite[Lemmas~4.3--4.4]{MR3694026}, the partition function of the Coulomb Hamiltonian
(without perturbation term)
can be estimated as
\begin{align*}
  \frac{1}{\beta} \log \int \ee^{-\beta H_W^{\cal C}} \, m^{\otimes M}(\rd \b z)
  &\geq M^2 I_W + \OO(M\log M),
  \\
  \frac{1}{\beta} \log \int \ee^{-\beta H_{W+f}^{\cal C}} \, m^{\otimes M}(\rd \b z)
  &\leq M^2 I_W
  + \frac{1}{8\pi \beta^2} (f,-\Delta f)
  + \OO(M^{-A}) \|\Delta f\|_\infty
  + \OO(M\log M).
\end{align*}
Here we have  used the improvement commented in the proof of Proposition~\ref{prop:Ystep},
which gives the improved factor for the error
term proportional to $\|\Delta f\|_\infty$ and avoids the restriction on $\|\Delta f\|_\infty$
 that was assumed in \cite[Lemmas~4.3--4.4]{MR3694026}. 
 From this and with $f$ replaced by $f/s$,  we obtain the estimate
\begin{equation*}
  \frac{1}{\beta} \log \E_{V_t}^{G_t} \ee^{X_{f/s}} = \frac{1}{8\pi s^2 \beta^2} (f,-\Delta f)
  + \OO(M^{-A})\frac{1}{s} \|\Delta f\|_\infty + \OO(E),
\end{equation*}
with $E =  tMK + M\log M$.
As in the proof of \cite[Lemmas~4.1]{MR3694026}, choosing
\begin{equation*}
  s = E^{-\frac12} \|\nabla f\|_2 + M^{-A} E^{-1} \|\Delta f\|_\infty,
\end{equation*}
this implies
\begin{equation*}
    \frac{1}{\beta} \log \E_{V_t}^{G_t} \ee^{X_f/s} = \OO(E).
\end{equation*}
By Markov's inequality,
$\P\p{X_f \geq \OO(s E)} \leq \ee^{-E}$,
and since the same estimate also holds with $f$ replaced by $-f$, we have
\begin{equation*}
  \P\pB{X_f = \OO\pa{E^{1/2}\|\nabla f\|_2 + M^{-A} \|\Delta f\|_\infty} } \geq 1- 2\ee^{-E},
\end{equation*}
which implies the claim \eqref{e:step}.

Finally, we note that for the Yukawa gas with $\ell \geq M^2$ we have
$Y^\ell(z) + \log |z|  = \text{constant} + O(|z|/M^2)$ by \eqref{Ya}.
The constant part of the energy does not affect the measure and the error term $O(1/M^2)$
is uniformly bounded by $O(1)$ when summed over all $M^2$ pairs of particles
(which we may assume to be at distance of order $1$ due to the growth of the external potential)
and therefore does not affect the estimate either.
\end{proof}

As in the proof of \cite[Theorem~1.1]{MR3694026},
the proof of Proposition~\ref{prop:CAdensity-first} now follows from iterated
applications of Proposition~\ref{prop:CAstep} to the conditioned measures
associated to increasingly small balls.
This induction proceeds almost exactly as in \cite[Section~6]{MR3694026},
with the additional element that, in each step, we improve also the bound $K$
for the conditioned measure. We first give an outline of this induction now.
Recall that we write $t = N^{-2a}$.

\paragraph{First step}

In the first step, using \eqref{e:C-perturb},
the difference $H - H_{W,N}^{\cal C}$ is bounded uniformly by
\begin{equation} \label{e:angletermbound0}
  \sum_{j, k: j \neq k} |\tilde G(z_j,z_k)|
  \leq 
  t\sum_{j \neq k} \ee^{-|z_j-z_k|^2/(2 \theta^2)}
  \leq tMK,
\end{equation}
with $M=N$ and $K = N$.
From Proposition~\ref{prop:CAstep},
we therefore get the high probability estimate
\begin{align*}
\frac{X_f}{N}
& \prec
  N^{-A-1} \|\Delta f\|_\infty + (N^{-a} + N^{-\frac12}) \|\nabla f\|_2.
\end{align*}
This estimate proves an effective estimate on the number of particles on scales $N^{-s}$
for $s < 1/4\wedge a/2$, i.e., $\gg N^{1/4}+t^{1/4}$.

\paragraph{Induction}
By induction, supposing we can control particle numbers on the distance scale $N^{-r}$,
in Proposition~\ref{prop:CAstep} applied to the conditional measure in a ball of the former scale,
we have $M \approx N^{1-2r}$ and $\alpha = N/M \approx N^{2r}$.
With the range of the perturbation in the interaction given by $\theta = N^{-1/2+\sigma}$,
it follows that \eqref{e:HbdK} holds (see Lemma~\ref{lem:CAimprovedK} below) with
$$
K = \OO(M \vee N^{2\sigma}).
$$
Using this estimate, by conditioning exactly as in the proof of \cite[Theorem~1.1]{MR3694026},
for any $f$ whose support has diameter at most $N^{-r}$,
we obtain from Proposition~\ref{prop:CAstep} the estimate
\begin{align*}
  \frac{X_f}{N}
  & \prec { \pb{ t \frac{M \vee N^{2 \sigma}}{\alpha N} + \alpha^{-1} N^{-1} } \|\Delta f\|_\infty    +  \pb{ t^\frac12 \frac{M \vee (M N^{2\sigma})^\frac12}{N} + M^\frac12 N^{-1} } \|\nabla f\|_2 } \\
  & \prec { \pb{ N^{-2a-4r} + N^{-1-2r} } \|\Delta f\|_\infty  +  \pb{ N^{-a-2r} + N^{-\frac12-r} } \|\nabla f\|_2}.
\end{align*}
This is an effective estimate on particle numbers on scales $N^{-s}$ for $s < (r + \tfrac{a}{2}) \vee (\tfrac{1}{4} + \tfrac{r}{2})+\epsilon$,
improving the assumed estimate.
We remark that, as far as the scales are concerned, this is the same recursion as in the case of the Yukawa gas, with $\delta$ replaced by $a/2$.
\bigskip

To set up the induction formally,
we replace the assumption (A$_r$) of \cite{MR3694026} by the following one.
(Note also that as before
 we changed the index $t$ from condition A$_{t}$ from \cite{MR3694026} into A$_{r}$ as, in the current paper, $t$ refers to the argument of the Laplace transform).

\medskip
\noindent \textbf{Assumption (A$_{r}$).}
For any bounded $f \in C^2(\C)$ with $\supp \Delta f \subset B_r^\circ \cap S_V$,
we have
\begin{equation} 
  \frac{X_f}{N} \prec
  { (N^{-1-2r} + N^{-2a-4r}) \|\Delta f\|_\infty + (N^{-\frac12-r} + N^{-a-2r}) \|\nabla f\|_2}.
  \label{e:step-prev}
\end{equation}
As shown above,
for $r = 0$ this is \eqref{e:step} applied with $M=N$ and $V=W$ and the trivial estimate $K=N$.
To prove Proposition~\ref{prop:CAdensity-first}, it is enough to prove the next proposition.

\begin{proposition} \label{prop:ind}
For arbitrary $\varepsilon>0$, (A$_r$) implies (A$_s$) for any
\begin{equation} \label{e:newscale}
0 \leq r < s \leq \pb{\frac{1}{4} + \frac{r}{2}} \wedge \pb{\frac{a}{2}+r} - \varepsilon,
\end{equation}
with the implicit constant in \eqref{e:step-prev} depending only on $\varepsilon$.
\end{proposition}

To prove Proposition~\ref{prop:ind}, exactly as in \cite[Sections 5-6]{MR3694026}, we condition on the outside of a ball $B_s$ on scale $s$
and replace the Coulomb potential of the outside charges with the Coulomb potential of the equilibrium measure.
To ensure that the equilibrium measure of the conditional system inside $B_s$ does not move much under this replacement,
we use \cite[Propositions 3.3 and 3.4]{MR3694026} and the analogues of \cite[Lemmas 6.2 and 6.3]{MR3694026},
where the input assumption is replaced by our new assumption (A$_r$);
the lemmas are checked exactly as in the case of the Yukawa gas.
The additional required estimate is the bound $K$ on \eqref{e:HbdK},
which is given by the following lemma.

\begin{lemma} \label{lem:CAimprovedK}
Assume (A$_r$).
Then, with high probability, uniformly for all configurations of the $M$ charges
inside $B_{s}$, the estimate \eqref{e:HbdK} holds with
\begin{equation*}
K = \OO(N^{1-2r} \vee N \theta^2).
\end{equation*}
In particular, if $B$ is at scale $N^{-r}$ and $\theta=N^{-1/2+\sigma}$ then the right-hand side is $\OO(M \vee N^{2\sigma})$.
\end{lemma}

\begin{proof}
Recall that the perturbation term in the Hamiltonian is bounded by
$
\sum_{j \neq k} \ee^{-|z_j-z_k|^2/(2 \theta^2)}.
$

We split this term into the three contributions: (1)  both particles are inside $B$,
(2) one particle is in $B$ and one outside $B$, and (3) both particles are outside $B$.
The contribution (3)
with both particles outside $B$ is a constant for the conditioned measure and thus
irrelevant for the estimate on the conditioned measure. Contribution (1)
is trivially estimated by $M^2$.
Contribution (2) is bounded by $\OO(M (N\theta^2 + N^{1-2r}))$ by the local density estimate,
with $r$-HP for the configurations outside $B$.
This gives the claimed estimate.
\end{proof}

\begin{proof}[Proof of Theorem~\ref{thm:Cdensity}.]
As in the proof of Proposition~\ref{prop:ind},
we condition on the particles outside a ball $B$ of radius $N^{-s}$
and assume that $f$ is supported in the ball with the same center and half of the radius.
However, since (A$_{1/2-\sigma}$) has already been proved, by Lemma~\ref{lem:CAimprovedK},
we now have the optimal estimate
$K = \OO(N^{2\sigma})$.
The theorem then follows directly from the one-step bound \eqref{e:step} on any scale $b$
as in the assumption of the theorem using this bound on $K$,
implying that $tK = t \OO(N^{2\sigma}) = \OO(1)$.
\end{proof}

\subsection{Conditioned versions: proof of Theorems~\ref{thm:YTdensity-cond}--\ref{thm:Cdensity-cond}}

The proofs of the conditioned versions of the local density estimates are analogous to
the original (unconditioned) versions. Namely, we prove the unconditioned versions
by inductive conditioning on increasing small balls.
The assumptions of the conditioned versions are exactly such that
the inductive assumption is satisfied for the conditional measure. We omit the details.

\end{appendices}

\section*{Notation index}
\addcontentsline{toc}{section}{Notation index}

\vspace{-1cm}
\setlength{\nomitemsep}{-0.1cm}
\printnomenclature[0.6in]

\bibliography{all}
\bibliographystyle{plain}
\end{document}